\documentclass[12pt, a4]{amsart}

\usepackage{latexsym,amsmath,amssymb,amsthm,mathrsfs,color}

\usepackage[bbgreekl]{mathbbol}
\usepackage{bbm}
\usepackage{tikz-cd}
\usepackage[all]{xy}
\usepackage{stmaryrd}
\usepackage{hyperref}
\hypersetup{colorlinks,%
    citecolor=brown,%
    filecolor=black,%
   linkcolor=blue,%
   urlcolor=black}

\usepackage[toc,page]{appendix}

\usepackage[shortalphabetic]{amsrefs}
 \makeatletter
\def\append@label@year@{%
    \safe@set\@tempcnta\bib@year
    \edef\bib@citeyear{\the\@tempcnta}%
    \ifnum\bib@citeyear>9
      \append@to@stem{%
          \ifx\bib@year\@empty
          \else
            \@xp\year@short \bib@citeyear \@nil
          \fi
      }%
    \fi
}
\makeatother
\let\oldtocsection=\tocsection
\renewcommand{\tocsection}[2]{\hspace{0em}\oldtocsection{#1}{#2}}

\usepackage[a4paper]{geometry}

\def\upddots{\mathinner{\mkern 1mu\raise 1pt \hbox{.}\mkern 2mu
\mkern 2mu \raise 4pt\hbox{.}\mkern 1mu \raise 7pt\vbox {\kern 7
pt\hbox{.}}} }

\numberwithin{equation}{section}

\begin{document}
\setlength{\unitlength}{2.5cm}

%%%%%%%%%%% theorem styles
\newtheorem{thm}{Theorem}[section]
\newtheorem{lm}[thm]{Lemma}
\newtheorem{prop}[thm]{Proposition}
\newtheorem{cor}[thm]{Corollary}
\newtheorem{conj}[thm]{Conjecture}

\theoremstyle{definition}
\newtheorem{dfn}[thm]{Definition}
\newtheorem{eg}[thm]{Example}
\newtheorem{rmk}[thm]{Remark}

\newcommand{\F}{\mathbf{F}}
\newcommand{\N}{\mathbf{N}}
\newcommand{\R}{\mathbf{R}}
\newcommand{\C}{\mathbf{C}}
\newcommand{\Z}{\mathbf{Z}}
\newcommand{\Q}{\mathbf{Q}}

\newcommand{\Mp}{{\rm Mp}}
\newcommand{\Sp}{{\rm Sp}}
\newcommand{\GSp}{{\rm GSp}}
\newcommand{\GL}{{\rm GL}}
\newcommand{\PGL}{{\rm PGL}}
\newcommand{\SL}{{\rm SL}}
\newcommand{\SO}{{\rm SO}}
\newcommand{\Spin}{\text{Spin}}
\newcommand{\Ind}{\text{Ind}}
\newcommand{\Res}{\text{Res}}
\newcommand{\Hom}{\text{Hom}}
\newcommand{\msc}[1]{\mathscr{#1}}
\newcommand{\mfr}[1]{\mathfrak{#1}}
\newcommand{\mca}[1]{\mathcal{#1}}
\newcommand{\mbf}[1]{{\bf #1}}

\newcommand{\mbm}[1]{\mathbbm{#1}}

\newcommand{\into}{\hookrightarrow}
\newcommand{\onto}{\twoheadrightarrow}

\newcommand{\s}{\mathbf{s}}
\newcommand{\cc}{\mathbf{c}}
\newcommand{\bfa}{\mathbf{a}}
\newcommand{\id}{{\rm id}}
\newcommand{\g}{\mathbf{g}_{\psi^{-1}}}
\newcommand{\w}{\mathbbm{w}}
\newcommand{\Ftn}{{\sf Ftn}}
\newcommand{\p}{\mathbf{p}}
\newcommand{\bq}{\mathbf{q}}
\newcommand{\WD}{\text{WD}}
\newcommand{\W}{\text{W}}
\newcommand{\Wh}{{\rm Wh}}
\newcommand{\ggma}{\omega}
\newcommand{\sct}{\text{\rm sc}}
\newcommand{\Of}{\mca{O}^\digamma}
\newcommand{\gk}{c_{\sf gk}}
\newcommand{\Irr}{ {\rm Irr}_{\rm gen}  }
\newcommand{\diag}{{\rm diag}}
\newcommand{\uchi}{ \underline{\chi} }
\newcommand{\Tr}{ {\rm Tr} }

\newcommand{\cu}[1]{\textsc{\underline{#1}}}
\newcommand{\set}[1]{\left\{#1\right\}}
\newcommand{\ul}[1]{\underline{#1}}
\newcommand{\wt}[1]{\overline{#1}}
\newcommand{\wtsf}[1]{\wt{\sf #1}}
\newcommand{\angb}[2]{\left\langle #1, #2 \right\rangle}
\newcommand{\wm}[1]{\wt{\mbf{#1}}}
\newcommand{\elt}[1]{\pmb{\big[} #1\pmb{\big]} }
\newcommand{\ceil}[1]{\left\lceil #1 \right\rceil}
\newcommand{\val}[1]{\left| #1 \right|}

\newcommand{\exc}{ {\rm exc} }

%%%%%%%%% Imported from Dani's paper %%%%%%
\newcommand{\nequiv}{\not \equiv}
\newcommand{\half}{\frac{1}{2}}
\newcommand{\psii}{\widetilde{\psi}}
\newcommand{\ab} {|\!|}
\newcommand{\mb}{{\widetilde{B(\F)}}}

\title[Local coefficients and Gamma factors]{Local coefficients and Gamma factors for principal series of covering groups}

\author{Fan Gao, Freydoon Shahidi, and Dani Szpruch}
\address{Fan Gao: School of Mathematical Sciences, Yuquan Campus, Zhejiang University, 38 Zheda Road, Hangzhou, China 310027}
\email{gaofan.math@gmail.com}
\address{Freydoon Shahidi: Department of Mathematics, Purdue University, 150 N. University Street, West Lafayette, IN 47907}
\email{shahidi@math.purdue.edu}
\address{Dani Szpruch: Department of Mathematics and Computer Science, Open University of Israel, Raanana 43107, Israel}
\email{dszpruch@openu.ac.il}

%\date{}
%\thanks{The second-named author is partially supported by the NSF grant DMS-1801273.
% The third-named author is partially supported by a Simons Foundation Collaboration Grant 426446.
%}
\maketitle

\begin{abstract} We consider an $n$-fold Brylinski-Deligne cover of a reductive group over a $p$-adic field. Since the space of Whittaker functionals of an irreducible genuine representation of such a cover is not one-dimensional, one can consider a local coefficients matrix arising from an intertwining operator, which is the natural analogue of the local coefficients in the linear case. In this paper, we concentrate on genuine principal series and establish some fundamental properties of such a local coefficients matrix, including the investigation of its arithmetic invariants. As a consequence, we prove a form of the Casselman-Shalika formula which could be viewed as a natural analogue for linear algebraic groups. We also investigate in some depth the behaviour of the local coefficients matrix with respect to the restriction of genuine principal series from covers of $\GL_2$ to $\SL_2$. In particular, some further relations are unveiled between local coefficients matrices and gamma factors or metaplectic-gamma factors.
\end{abstract}
\tableofcontents

%%%
\section{Introduction} \label{S:intro}

%%%%
Let $G$ be the $F$-rational points of a split connected reductive group $\mbf{G}$ over a non-archimedean local field $F$ of characteristic 0. Fix a Borel subgroup $B=TU$ of  $G$.
Let
$$\psi: U \to \C^\times$$ be a nondegenerate character. For any irreducible admissible representation $\pi$ of $G$, the space $\text{Wh}_\psi (\pi)$ of $\psi$-Whittaker functionals of $\pi$, which is the dual of the twisted Jacquet module, plays a crucial role in studying the representation. Indeed, motivated by some observations by Langlands, the second-named author has developed the (Langlands-Shahidi) theory of local coefficients and thus $\gamma$-factors and also $L$-functions, relying on an essential usage of the space $\Wh_\psi(\pi)$, whenever $\pi$ is generic, i.e., $\Wh_\psi(\pi)\ne 0$. See \cite{Sha1, Sha2, Sha85, Sha88, Sha3}. By results of Gelfand-Kazhdan \cite{GK}, Shalika \cite{Shal} and Rodier \cite{Rod1}, one has a uniform bound that
\begin{equation} \label{E:M1}
\dim \Wh_\psi(\pi)  \le 1 \text{ for every }  \pi \in \text{Irr}(G).
\end{equation}
What is equally important is the Casselman-Shalika formula \cite{CS} for local unramified representation, which relates value of the unique Whittaker function  to $L$-functions. This uniqueness (or multiplicity-one) property  has been applied extensively in various methods studying local representations of $G$, and global automorphic representations concerning $L$-functions for an adelic group. This is especially the case for the Langlands-Shahidi method mentioned above, which has many important applications in establishing various cases of the Langlands functoriality (see \cite{KSh2, KSh3, Kim4, ASh1, ASh2, CKPSS1, CKPSS2, KiKr}).

In this paper, we consider an $n$-fold Brylinski-Deligne cover $\wt{G}$ of $G$ arising from \cite{BD}. In this case, the space $\Wh_\psi(\pi)$ of Whittaker functionals of an irreducible genuine representation $\pi$ of $\wt{G}$ is in general not one-dimensional, i.e., \eqref{E:M1} fails. However, one can consider a \emph{local coefficients matrix} arising from an intertwining operator, which is the natural analogue of the local coefficients in the linear algebraic case mentioned above. We concentrate on genuine principal series of $\wt{G}$, study the arithmetic invariants associated to the local coefficients matrix, and investigate some consequences including proving a natural analogue of the Casselman-Shalika formula in the covering setting. This paper could be viewed as a continuation of our previous work \cite{GSS1}. In loc. cit. we studied the local coefficients matrix associated to unramified principal series of  the degree $n$-fold cover of $\SL_2$. In this paper, we carry out an extensive study for unramified principal series of covers of a general split connected reductive group. We also deal with ramified principal series for covers of $\GL_2$ and $\SL_2$. Such an investigation is directed towards a potential theory of $\gamma$-factors and $L$-functions for covering groups in the framework of the Langlands-Shahidi method (see \cite{Szp4, Ga1, GSS1} and references therein).
\vskip 5pt

Now we explain more precisely the setting and problems we study in this paper. Assume that $F^\times$ contains the full group $\bbmu_n$ of $n$-th roots of unity. Arising from certain $\mbf{K}_2$-extensions classified by Brylinski and Deligne \cite{BD}, there are natural $n$-fold central covers:
$$\begin{tikzcd}
\bbmu_n \ar[r, hook] &  \wt{G}  \ar[r, two heads] & G.
\end{tikzcd}$$
Denote by $\Irr(\wt{G})$ the set of isomorphism classes of irreducible admissible representations of $\wt{G}$ such that $\bbmu_n$ acts by a fixed embedding $\bbmu_n \into \C^\times$. Such a representation is called genuine. Since the cover $\wt{G}$ splits uniquely over the unipotent radical $U$ of the Borel subgroup, we identify $U$ as a subgroup of $\wt{G}$. For every $\pi\in \Irr(\wt{G})$, we denote again by $\dim \Wh_\psi(\pi)$ the space of $\psi$-Whittaker functionals of $\pi$, where $\psi: U \to \C^\times$ is the nondegenerate character as above.

For a fixed cover $\wt{G}$, it is shown in \cite{Szp1, GSS1} that $\dim \Wh_\psi(\pi) \le 1$ holds for every $\pi \in \Irr(\wt{G})$ if and only if the covering torus $\wt{T}$ is abelian. The latter condition rarely holds for a general $\wt{G}$, and thus the multiplicity-one property \eqref{E:M1} fails in general. For a fixed $\wt{G}$, we expect an upper bound (depending on $\wt{G}$) for the set $\set{\dim \Wh_\psi(\pi): \pi\in \Irr(\wt{G})}$, see Conjecture \ref{C:uniUB}. However, such upper bounds for varying $\wt{G}$ are not uniformly bounded above, unlike the linear algebraic case where \eqref{E:M1} holds for all $G$. This clearly suggests the necessity of new methods to attack problems from studying $\Irr(\wt{G})$. Indeed, the difficulty that arises is already observed by Kubota \cite{Kub} in his work on theta series, and also elucidated by Shimura \cite{Shim} in his correspondence on half-integral weight modular forms: the Fourier coefficients of a modular form on $\wt{G}$ involve interesting arithmetic information, yet only part of which is accessible by the method of Hecke operators. The work of Waldspurger \cite{Wal1} further illustrates on the subtleties on the parametrization of $\Irr (\wt{\SL}_2^{(2)})$, even though $\wt{\SL}_2^{(2)}$ actually has the multiplicity-one property. Instead of giving an elaborate discussion of the high multiplicity of $\Wh_\psi(\pi)$ for $\pi \in \Irr(\wt{G})$, we refer the reader to the work in \cite{Del1, KP, Ga2}.

Nonetheless, we have $\dim \Wh_\psi(\pi) < \infty$ for every $\pi \in \Irr(\wt{G})$ (see \cite{Pate}), and there is the natural  map
$$T(w, \sigma)^*:  \Wh_\psi(I({}^w \sigma) )  \to \Wh_\psi( I(\sigma)  )$$
arising from dualizing the standard intertwining operator
$$T(w, \sigma): I(\sigma) \to I({}^w \sigma)$$
 between two parabolically induced representations. There is also a natural isomorphism
 $$\Wh_\psi(I(\sigma))\simeq \Wh_\psi(I({}^w \sigma)) $$
 of vector spaces which, by pre-composing with $T(w, \sigma)^*$,  gives an endomorphism (see \S \ref{SS:lcm})
$$ \mca{T}(w, \sigma)^*: \Wh_\psi( I(\sigma)  ) \to \Wh_\psi( I(\sigma)  ).$$
It is then a natural question to ask:
\begin{enumerate}
\item[$\bullet$]  what arithmetic information regarding $\sigma$ is encoded in the characteristic polynomial
$$\bigwedge^{\rm top} (X\cdot \text{id}- \mca{T}(w, \sigma)^*)=\det(X \cdot \text{id}- \mca{T}(w, \sigma)^*)?$$
\end{enumerate}
Our investigation in this paper is motivated by this question.

If $n=1$, i.e., $\wt{G}=G$, then $\mca{T}(w, \sigma)^*$ is scalar valued and is in fact the \emph{reciprocal} of the local coefficients studied by the second-named author for linear algebraic groups, which shows that it encodes essential information of the representation $\sigma$. Indeed, the theories of $\gamma$-factors and thus $L$-functions are both developed from $\mca{T}(w, \sigma)^*$, and they are fundamental ingredients in the Langlands-Shahidi method of studying representations of $G$ and its global counterpart.

For general $\wt{G}$, we believe that, as an indispensable step towards a theory of $\gamma$-factors or $L$-functions for covering groups, it is important to understand the above characteristic polynomial $\det(X \cdot \text{id}- \mca{T}(w, \sigma)^*)$.  In particular, it is natural to consider the two invariants (i.e., as coefficients of the polynomial)
$$\Tr(\mca{T}(w, \sigma)^*) \text{ and } \det (\mca{T}(w, \sigma)^*).$$
However, there are fundamental obstacles to understanding  these invariants for general $\sigma$ and $w$. Indeed, it seems to us that it is a difficult task at the moment to determine completely $\dim \Wh_\psi(\sigma)$ for a general $\sigma\in \Irr(\wt{M})$, where $\wt{M} \subset \wt{G}$ is a covering Levi subgroup. This very first fundamental obstacle on determining $\dim \Wh_\psi(\sigma)$ and thus the size of any matrix representing $\mca{T}(w, \sigma)^*$ already poses much challenge in understanding the two invariants. It is even harder to parametrize the space $\Wh_\psi(\sigma)$ itself. One exceptional case is when $\wt{M}=\wt{T}$, in which case
$\Wh_\psi(\sigma) \simeq \sigma^\vee$ as vector spaces. For general $\wt{M}$, a widely studied family  contains the so-called theta representations $\Theta(\wt{M}, \chi)$ (see \cite{KP, Suz1, Suz2, Suz3, Kap002, Kap003, Kap004, Kap005, Ga2}), which are just characters of $M$ if $n=1$. There are also other work when $\sigma$ is a depth-zero supercuspidal representation (\cite{Blo, GW}).

For the above reasons, we specialize eventually to the case where $\sigma \in \Irr(\wt{T})$ is a genuine representation of the covering torus, and thus $I(\sigma)$ is a genuine principal series representation. In particular, we carry out an extensive study of $\det (\mca{T}(w, \sigma)^*)$ when $\sigma$ is an unramified representation of $\wt{T}$ (and thus $I(\sigma)$ is an unramified principal series). In particular, it will recover our previous result on covers of $\wt{\SL}_2$ (cf. \cite{GSS1}) and the work by Budden on $\wt{\GL}_2$  \cite{Bud}. We also study  $\Tr(\mca{T}(w, \sigma)^*)$ and $\det(\mca{T}(w, \sigma)^*)$ for $\wt{\SL}_2$ and $\wt{\GL}_2$ extensively for general $\sigma$, not necessarily the unramified ones.

\subsection{Outline and main results}
In \S \ref{S:cov}, we provide the basic set-up on covering groups and give a brief review of some results to be used in later sections.

In \S \ref{S:2mat}--\S \ref{S:eg}, we concentrate mostly on unramified principal series of a general covering group $\wt{G}$. More precisely, specializing to the case where $\sigma=i(\chi)$ is an unramified representation of $\wt{T}$, the Whittaker space $\Wh_\psi(I(i(\chi)))$ is parametrized by a certain ``moduli space" $\msc{X}_{Q,n}$, which is a finite abelian quotient of the cocharacter lattice $Y$ of $\mbf{G}$ by a sublattice $Y_{Q,n}$:
$$f_\msc{X}: Y\onto \msc{X}_{Q,n}:=Y/Y_{Q,n}.$$
In particular, one has
$$\dim \Wh_\psi(I(i(\chi))) = \val{\msc{X}_{Q,n}}.$$
Here $\msc{X}_{Q,n}$ is naturally endowed with a \emph{twisted} Weyl group action which we denote by $\w[-]$.
The results in \S \ref{S:2mat}--\S \ref{S:eg}, which we will explain below, rely crucially on certain properties of the twisted action $\w[-]$.

In \S \ref{S:RES}--\S \ref{S:LCM-G}, on the other hand, we treat general (not necessarily unramified) genuine  principal series of $\wt{\GL}_2$ and $\wt{\SL}_2$. As we deal with ramified principal series as well, we take the viewpoint of \cite{Szp6} involving partial zeta-integrals. The relation between entries of the local coefficients matrix and gamma factor or its metaplectic analogue will be exploited to establish our results.

\vskip 10pt

The results in this paper are of multi-fold in a logical order, as we elaborate now.

\subsubsection{{}} \label{SS:(i)}
First, for unramified $i(\chi)$ and a simple reflection $w_\alpha$, we determine explicitly  the invariant $\det(\mca{T}(w_\alpha,i(\chi))^*)$ in terms of the Plancherel measure $\mu(w_\alpha, i(\chi))$ and certain gamma factor $\gamma(w_\alpha, i(\chi))$ or metaplectic gamma factor $\tilde{\gamma}(w_\alpha, i(\chi))$ (see \S \ref{SS:Pg} for the definition of these factors):
\begin{thm}[{Theorem \ref{T:M1}}] \label{T:01}
Loosely speaking, we have either
$$\bigwedge^{ \val{ \msc{X}_{Q,n} } } \mca{T}(w_\alpha,i(\chi))^*  \approx  \mu(w_\alpha, i(\chi))^{-a_\alpha}  \cdot \gamma(w_\alpha, i(\chi))^{-b_\alpha}$$
or
$$\bigwedge^{ \val{ \msc{X}_{Q,n} } } \mca{T}(w_\alpha,i(\chi))^*   \approx  \mu(w_\alpha, i(\chi))^{-a_\alpha}  \cdot \tilde{\gamma}(w_\alpha, i(\chi))^{-b_\alpha}.$$
\end{thm}
Here $\approx$ indicates an equality modulo some other factors, which actually embody part of the information of the group structure of $\wt{G}$. For the exact formula, see Theorem \ref{T:M1}.

To facilitate the computation, we define a local coefficients matrix to be the matrix
$$\mca{M}_\mfr{B}(w_\alpha, i(\chi))$$
 representing $\mca{T}(w_\alpha,i(\chi))^*$ with respect to an ordered basis $\mfr{B}$ of $\Wh_\psi(I(i(\chi)))$, see Definition \ref{D:LCM}; thus,
$$\bigwedge^{ \val{ \msc{X}_{Q,n} } } \mca{T}(w_\alpha,i(\chi))^* = \det( \mca{M}_\mfr{B}(w_\alpha, i(\chi))  ). $$
The main difficulty with computing $\det(\mca{M}_\mfr{B}(w_\alpha, i(\chi)))$ is that there is no preferred natural choice of $\mfr{B}$ such that the matrix $\mca{M}_\mfr{B}(w_\alpha, i(\chi))$ takes a simple form. (See however Theorem \ref{T:SL-c3} where $\det(\mca{M}_\mfr{B}(w_\alpha, i(\chi)))$ is computed for ramified $i(\chi)$ for the group $\wt{\SL}_2$ in a certain case.)

To overcome this, we relate the local coefficients matrix to another matrix (see Definition \ref{D:Sca})
$$\mca{S}_\mfr{R}(w_\alpha, i(\chi); r_w)$$
 which appears more frequently in literature: the so-called scattering matrix. Here $\mfr{R} \subset Y$ is a set of representatives for the space $\msc{X}_{Q,n}$.
The matrix $\mca{S}_\mfr{R}(w_\alpha, i(\chi); r_w)$ was first studied in \cite{KP} and then further investigated extensively in \cite{KP, Pat, Suz1, Suz2, Suz3, Suz4, Mc2}. Here $r_w$ is a choice of isomorphism between the two isomorphic representations ${}^w i(\chi)$ and $i({}^\w \chi)$ of $\wt{T}$. For general $\chi$ (not necessarily unramified), there is no natural choice of $r_w$. However, in the unramified setting, one has a canonical choice $r_w^{\rm un}$ as given in \eqref{r-w}. In any case, the disadvantage of $\mca{S}_\mfr{R}(w_\alpha, i(\chi); r_w^{\rm un})$ is that it is not a good substitute for the local coefficients matrix in pursuit of arithmetic invariants. In particular, the characteristic polynomial
$$\det(X \cdot I_{ \val{ \msc{X}_{Q,n} } }-\mca{S}_\mfr{R}(w_\alpha, i(\chi); r_w^{\rm un}) )$$
 depends sensitively on the choice $\mfr{R}$ for general $\chi$.

 Nevertheless, the advantageous side of $\mca{S}_\mfr{R}(w_\alpha, i(\chi); r_w^{\rm un})$ is that it is ``essentially" a diagonal-block matrix with size either two or one. The two possible sizes here correspond to the free orbits or trivial orbits of the group
 $$W_\alpha=\set{\text{id}, \w_\alpha}$$
 acting on the space $\msc{X}_{Q,n}$ with respect to the twisted Weyl-action alluded to above. In fact, the number of free $W_\alpha$-orbits is the exponent $a_\alpha$ of the Plancherel measure which appears in 
 $$\det( \mca{S}_\mfr{R}(w_\alpha, i(\chi); r_w^{\rm un})  ),$$
 while the number of trivial $W_\alpha$-orbits equals the exponent $b_\alpha$ of the factor $\gamma(w_\alpha, i(\chi))$ or $\tilde{\gamma}(w_\alpha, i(\chi))$ (but not both). Moreover, for any representative set $\mfr{R} \subset Y$, we could have a natural choice of basis $\mfr{B}$ and thus the local coefficients matrix $\mca{M}_\mfr{B}(w_\alpha, i(\chi))$, which is then closely related to $\mca{S}_\mfr{R}(w_\alpha, i(\chi); r_w^{\rm un})$: their difference is given by an invertible matrix $\mca{C}(\mfr{B}_{{}^\w\chi}, \mfr{B}_\chi; r_w^{\rm un})$ arising from a certain change of basis (see Lemma \ref{L:comp}):
$$ \mca{M}_{\mfr{B}}(w_\alpha, i(\chi))=  \mca{S}_\mfr{R}(w_\alpha, i(\chi); r_{w_\alpha}^{\rm un}) \circ \mca{C}( \mfr{B}_{ {}^{\w_\alpha} \chi}, \mfr{B}_\chi; r_{w_\alpha}^{\rm un}).$$
 This is the key property that we exploit in \S \ref{S:2mat}--\S \ref{S:PG} to prove the above formula for $\det(\mca{M}_\mfr{B}(w_\alpha, i(\chi)))$.

In fact, at the end of \S \ref{S:PG}, we obtain an explicit description of entries of the local coefficients matrix $\mca{M}_{\mfr{B}}(w_\alpha, i(\chi))$, see Theorem \ref{T:M2}. As a consequence, we analyze for $\wt{\SL}_2^{(n)}$ the invariant $\text{Tr}(\mca{T}(w_\alpha, i(\chi))$, which is shown to be an average of the reciprocals of unramified gamma-factors. See Proposition \ref{P:trace}. This result agrees with \cite{Szp6} which treats the cases for general $\chi$ (not necessarily unramified) when $4\nmid n$; it also includes the case $4|n$ which is already implicit in \cite{GoSz, Szp6}. In fact, in Theorem \ref{trcodd24} and Corollary \ref{exptrun}, we will give another proof of Proposition \ref{P:trace} from the perspective of partial zeta-integral and show that one has a natural expression of $\text{Tr}(\mca{T}(w_\alpha, i(\chi))$ in terms of averaged sum of gamma or metaplectic factors over both ramified and unramified characters; however, the partial sum over ramified characters vanishes, and thus we recover Proposition \ref{P:trace}.

 \subsubsection{} \label{SS:(ii)}
 Second, motivated by the discussion in \S \ref{S:2mat}--\S \ref{S:PG}, it is natural to consider the set
$$(\msc{X}_{Q,n})^W.$$
We show in \S \ref{S:CS} that there exists an exceptional subset $Y_n^{\rm exc} \subset f^{-1}_\msc{X}( \msc{X}_{Q,n}^W )$ (possibly empty) giving rise to certain exceptional subspaces of $\Wh_\psi (I(\chi))$, where a coarse theory of gamma-factors and $L$-functions for covering groups inhabits. There are several important consequences of this we highlight below.

The usage of elements in $Y_n^{\rm exc}$ enables us to obtain the following analogue of the Casselman-Shalika formula for covering groups.
\begin{thm}[{Theorem \ref{T:WV} and Theorem \ref{T:CS}}]
Let $z\in Y_n^{\rm exc}$. One has a dominant element $\s_{-z} \in \wt{T}$ and
\begin{equation} \label{E:CS0}
\mca{W}_{\s_z}(\s_{-z}) = \delta_B^{1/2}(\s_{-z}) \cdot \prod_{\alpha>0} (1-q^{-1} \chi_\alpha);
\end{equation}
or in another formulation,
\begin{equation} \label{E:CS00}
\mca{W}(\s_{y+(-z)}) =  \delta(\s_{y}) \cdot {\rm Tr}( {}^L\bbrho(\varpi); \pi_y) \cdot \mca{W}(\s_{-z}),
\end{equation}
where $y\in Y_{Q,n}^+$ is a dominant element and $\pi_y$ is the highest-weight-$y$ representation of the dual group $\wt{G}_{Q,n}^\vee$ of $\wt{G}$.
\end{thm}

We remark that if $\wt{G}$ is a cover of a semisimple $G$ not of metaplectic type (see Definition \ref{D:metaG}), then $$Y_n^{\rm exc} = \set{\rho - \rho_{Q,n}} \cap Y.$$
In particular, in the linear algebraic case when $n=1$, we have $\rho_{Q,n} =\rho$ and thus
$$Y_n^{\rm exc}= \set{0}.$$ The formula \eqref{E:CS0}  above then recovers the Casselman-Shalika formula for split linear algebraic groups \cite{CS}.  Note that formula \eqref{E:CS00} only explains the Whittaker value for $\s_{y+(-z)} \in \wt{T}$ with $\s_y\in Z(\wt{T})$. This is not surprising, as the value $\mca{W}(\s_y)$ for general $y\in Y$ is a delicate issue and involves the deep theory of Weyl group multiple Dirichlet series (see \cite{BBCFH, BBF1, BBFH1}).
We also note that \eqref{E:CS0} is implicit in the work \cite{CO, Mc2, PP1}, see \S \ref{SS:other-wk}. For Kazhdan-Patterson coverings of $\GL_r$, such a formula is explicated in the work of Suzuki on Bump-Hoffstein conjecture (see \cite[Proposition 3.4]{Suz2}). The works of Suzuki \cite{Suz1, Suz2, Suz3, Suz4}  in fact provide many insights on the unramified Whittaker function for covering groups. In any case, the formula \eqref{E:CS0} shows clearly the obstacle of studying global $L$-functions from computing Whittaker-Fourier coefficients, as $z \ne 0$ for general covering groups.

Another application of considering exceptional points $Y_n^{\rm exc}$ (or just $f_\msc{X}^{-1}(\msc{X}_{Q,n}^W)$) is to give a crude form of the theory of $\gamma$-factors for general parabolic induction where the inducing representation on the Levi subgroup is a full unramified principal series representation. See Theorem \ref{T:naiveLS} for details. Again, this gives another illustration on the remarkable property of the set $Y_n^{\rm exc}$ of exceptional points and also $\msc{X}_{Q,n}^W$.

\subsubsection{} \label{SS:(iii)}
Third, in \S \ref{S:Poinc} we study the set $(\msc{X}_{Q,n})^W$ and also its size
$$b_{W,n}:= \val{(\msc{X}_{Q,n})^W}$$
qualitatively.

\begin{thm}[{Theorem \ref{T:pi1} and Theorem \ref{T:bP}}]  \label{T:1-3}
Let $\mbf{G}$  be a semisimple group.
\begin{enumerate}
\item[(i)] One has $b_{W,n} \le \val{\pi_1(\wt{G}_{Q,n}^\vee)}$, where $\pi_1(\wt{G}_{Q,n}^\vee)$ is the fundamental group  of the dual group $\wt{G}_{Q,n}^\vee$ of $\wt{G}$.
\item[(ii)] If $\mbf{G}$ is simply-connected, then $b_{W,n} \le 1$ for every $n$.
\item[(iii)] If $\mbf{G}= \mbf{G}_{ad}$ is of adjoint type, then $b_{W,n}= \val{\pi_1(\wt{G}_{Q,n}^\vee) }$ for all $n$.
\item[(iv)] For fixed $(\mbf{G}, Q)$ with $\mbf{G}$ almost-simple, $b_{W,n}$ is periodic as a function of $n\in \N$. Such periodicity is compatible with that of the dual group $\wt{G}_{Q,n}^\vee$ for varying $n$, as proved in Proposition \ref{P:DG}.
\end{enumerate}
\end{thm}

Such relation between $b_{W,n}$ and $\pi_1(\wt{G}_{Q,n}^\vee)$ appears to be a new phenomenon which only manifests in the context of general covering groups. Indeed, for linear algebraic groups (i.e., $n=1$) we always have $b_{W,1}=1$. However, for covering groups, it is possible to have $b_{W,n}=0$ or  $b_{W,n}>1$. In any case, for an almost-simple $G$, it follows from Theorem \ref{T:1-3} (iv) that the Poincar\'e series
$$\mca{P}_W(T):= \sum_{n\ge 1}^\infty b_{W,n} T^n$$
associated to $\set{b_{W,n}}_{n\in \N}$ is a rational function, and similarly for the Poincar\'e series $\mca{P}_{\rm exc}(T)$ associated with $\val{ f_\msc{X}(Y_n^{\rm exc})}$, see Corollary \ref{C:rat}.   

In  \S \ref{S:eg}, we consider  covers of almost-simple simply-connected groups, the orthogonal adjoint group and certain Kazhdan-Patterson cover of $\GL_2$ and Savin covers of $\GL_r$. We work out explicitly the two Poincar\'e series $\mca{P}_W(T)$ and $\mca{P}_{\rm exc}(T)$. We conjecture the rationality of $\mca{P}_W(T)$ and $\mca{P}_{\rm exc}(T)$ for covers of linear reductive groups in general. Such rationality of the Poincare series $\mca{P}_W(T)$ corresponds to a certain recurrence relation on the $b_{W,n}$'s, see \S \ref{SS:rat}.

\subsubsection{} \label{SS:(iv)}
 Fourth, in \S \ref{S:RES}, we consider the restriction of principal series representation of $\wt{\GL}_2$ to $\wt{\SL}_2$. For a general covering group $\wt{G}$ with derived group $\wt{G}_\text{der}$ and a genuine representation $\pi$ of $\wt{G}$, it follows from the definition that the space $\Wh_\psi(\pi)$ is a priori determined by the Whittaker space
 $$\Wh_\psi(\pi|_{\wt{G}_\text{der}}),$$
 as $\wt{G}$ and $\wt{G}_\text{der}$ have the same unipotent subgroup $U$. Hence the local coefficients matrix for $\pi$ is naturally determined by the local coefficients matrix arising from considering $\pi|_{\wt{G}_\text{der}}$. In the linear algebraic case the situation is relatively simple, since we have uniqueness of Whittaker models and thus there exists exactly one generic summand in this restriction. In the covering case, we expect to have more than one generic summand and the restriction problem becomes more subtle.

This motivates the study of local coefficients matrices of $\wt{\GL}_2$ and $\wt{\SL}_2$ and their relations in detail. We carry this out in \S \ref{S:RES}--\S \ref{S:LCM-G}.  The main result in \S \ref{S:RES} is the following, which describes the decomposition of the restriction of a genuine principal series $I(\sigma)$ of $\wt{\GL}_2$ to $\wt{\SL}_2$.

\begin{thm}[{Theorem \ref{decopropcor}}]
Denote by $\wt{T}$ (resp. $\wt{T}_o$) the covering torus of $\wt{\GL}_2$ (resp. $\wt{\SL}_2$). Let $\sigma \in \Irr(\wt{T})$.
\begin{enumerate}
  \item[$\bullet$] Suppose that $n$ is odd. Let $\sigma_o \in \Irr(\wt{T}_o)$ be the genuine smooth irreducible  representation  of $\wt{T}_o$ determined by the relation $\chi_{\sigma_o}=\chi_\sigma|_{Z(\wt{T}_o)}.$ Then
  $$I(\sigma)\mid_{\wt{\SL}_2} \simeq  n_c \val{n_c}^{-1/2} \cdot I(\sigma_o).$$
  \item[$\bullet$] Suppose that $n$ is even. Fix a $ \sigma_o \in  \Irr(\wt{T}_o)$ such that $\chi_{\sigma_o}$ agrees with  $\chi_{\sigma}$ on $Z(\wt{T})\cap \wt{T}_o$. Then we have
      $$I(\sigma)|_{\wt{\SL}_2} \simeq \bigoplus_{x\in F^\times/ F^{\times 2}}  d_c \val{d_c}^{-1/2}\cdot  I_o(\eta_{x,(n)} \otimes\sigma_o),$$
      where $\eta_{x,(n)}$ is a certain non-genuine character of $\wt{T}_o$. Furthermore, if $n \equiv 2 \ (\text{mod } 4)$, then
      $$I(\sigma)|_{\wt{\SL}_2} \simeq  \bigoplus_{x\in F^\times/F^{\times 2}} d_c \val{d_c}^{-1/2} \cdot I_o(\eta_{x,(2)} \otimes\sigma_o).$$
\end{enumerate}
\end{thm}

We highlight that if $n$ is even and $I(\sigma)$ is an unramified genuine principal series representation, there are always ramified genuine principal series representations appearing as direct summands in $I(\sigma)|_{\wt{\SL}_2}$. Hence, one inevitably has to study local coefficients matrix for certain \emph{ramified} genuine principal series of $\wt{\SL}_2$ as well, in order to unveil the behaviour of such a matrix under the restriction. We refer the reader to the beginning of \S \ref{S:RES} for a more detailed discussion.

\subsubsection{} \label{SS:(v)}
Fifth, for $\wt{\SL}_2$, we study in the general case (without assuming $\gcd(n, p)=1$ whenever possible) a local coefficients matrix $\mca{M}(w, \sigma_o, s, \psi)$ and also the two invariants
$$\Tr(\mca{M}(w, \sigma_o, s, \psi)) \text{ and }  \det(\mca{M}(w, \sigma_o, s, \psi)),$$
for the latter of which we concentrate on $\sigma_o$ arising from the restriction of an unramified genuine principal series of $\wt{\GL}_2$.  The computation of such a matrix (especially in the ramified case) seems to be a delicate issue, as we actually express the entries of such a matrix in terms of certain $\gamma$-factors or $\tilde{\gamma}$-factors, see Proposition \ref{matnot4} and Proposition \ref{P:M-3C}. The above two invariants are expressed in terms of (metaplectic-)gamma factors  and Plancherel measures, and this constitutes the main results in \S \ref{S:LCM-Go}.

 \begin{thm}[Theorems \ref{trcodd24}, \ref{detunramnot4} and \ref{T:SL-c3}]
Regarding the trace, we have
$$
\begin{aligned}
&  \Tr(\mca{M}(w, \sigma_o, s,\psi)) \\
= &  (\dim \sigma_o)^{-1} \cdot
 \begin{cases}
 \sum_{\eta \in \widehat{F^\times/ F^{\times d}}} \gamma(1-s,\chi^{-1}\eta,\psi)  & \text{ if }  n \not \equiv 2 \, (\text{mod }4) ; \\
 \sum_{\eta \in \widehat{F^\times/ F^{\times d}}} \tilde{\gamma}(1-s,\chi^{-1}\eta,\psi) & \text{ if }   n \equiv 2 \, (\text{mod }4). \end{cases}
 \end{aligned}
$$
 On the other hand, for determinant:
 \begin{enumerate}
\item[$\bullet$] Assume $n \not \equiv 0 \ (\text{mod }4)$,  $\gcd(n,p)=1$, the character $\psi$ is normalized and that $\sigma_o$ occurs in an unramified element of $\Irr(\wt{T})$. Then,
$$ \det(\mca{M}(w, \sigma_o,s, \psi))=\mu(\sigma_o,s)^{\frac{1-d}{2}} \cdot
\begin{cases}
\gamma(1-ds,\chi^{-d},\psi)  &  \text{ if $n$ is odd}; \\
\tilde{\gamma}(1-ds,\chi^{-d},\psi) &   \text{ if } n \equiv 2 \ (\text{mod }4) .
\end{cases}$$
 \item[$\bullet$] Assume that $n \equiv 0 \, (\text{mod }4)$,  $\gcd(n,p)=1$,  $\mfr{f}(\psi)=0$ and $\sigma_o$ is unramified. Then,
$$
\det(\mca{M}(w, \sigma_o,s, \psi))=\bigl(\eta_{u,(n)}\chi^{-1}(\varpi)q^s \bigr)^d  \cdot \mu^{-\frac d 2}(\sigma_o,s).
$$
 \end{enumerate}
 \end{thm}

Now for a genuine principal series $I(\sigma)$ of $\wt{\GL}_2$, by using the explicit decomposition $I(\sigma)_{\wt{\SL}_2}$ described in Theorem \ref{decopropcor}, we compute the two invariants
$$\Tr(\mca{M}(w, \sigma, s, \psi)) \text{ and }  \det(\mca{M}(w, \sigma, s, \psi))$$
for $\wt{\GL}_2$ to obtain the following.

\begin{thm}[Theorems \ref{trgl} and \ref{detformgl}]
Fix $\sigma \in \Irr(\wt{T})$. Let $\chi$ be a linear character of $F^\times$ associated with $\sigma$.
\begin{enumerate}
\item[$\bullet$] To consider the trace, if $n  \equiv 0 \, (\text{mod }4)$, then we assume that $\gcd(p,n)=1$. One has
$$\Tr( \mca{M}(w, \sigma,s,\psi))=\frac{\val{\gcd(n,4c+1)}^{\half}}{\gcd(n,4c+1)} \cdot \sum_{\eta \in \widehat{F^\times/F^{\times n}} } \gamma(1-s,\chi^{-1}\eta,\psi). $$
\item[$\bullet$]
To consider the determinant, we assume that $\gcd(p,n)=1$ and $\sigma$ is unramified. Also assume $\mfr{f}(\psi)=0$. Then
$$\det(\mca{M}(w, \sigma, s, \psi))=\tau(n) \cdot \mu(\sigma, s)^{\frac{(1-n)n_c}{2}}  \cdot \gamma(1-ns, \chi^{-n}, \psi)^{n_c},$$
where $$
\tau(n)=
\begin{cases}
(-1,\varpi)_2 &  \text{ if } n \equiv 2 \ (\text{mod }4); \\
1  &   \text{otherwise}.
\end{cases}$$
\end{enumerate}
\end{thm}
With reference to the previous notation, for $\wt{\SL}_2$ one has $\sigma_o = i(\chi_o)$ and thus
$$\mca{M}(w, \sigma_o, 0, \psi) = \mca{M}_\mfr{B}(w_\alpha, i(\chi_o));$$
similarly $$\mca{M}(w, \sigma, 0, \psi) = \mca{M}_\mfr{B}(w_\alpha, i(\chi))$$
  for $\wt{\GL}_2$. Also, the computation of $\det(\mca{M}(w, \sigma_o, s, \psi))$ and $\det(\mca{M}(w, \sigma, s, \psi))$ above is compatible with Theorem \ref{T:01} in the unramified case.

One application of our investigation is to understand the contrast between $\det(\mca{M}(w_\alpha, i(\chi)))$ and $\det(\mca{M}(w_\alpha, i(\chi_o)))$ for $\wt{\GL}_2$ and $\wt{\SL}_2$ respectively. More precisely, from our previous work \cite{GSS1}, it is evident that one has a trichotomy for the ``pattern" of $\det(\mca{M}(w_\alpha, i(\chi_o)))$ in terms of the (meta-)gamma factor and Plancherel measure, depending on whether $n$ is odd, $n \equiv 2, 0 \mod 4$. On the other hand, there is a uniform description of $\det(\mca{M}(w_\alpha, i(\chi)))$, which was first observed in \cite{Bud} for a special family of unramified principal series and also follows from our Theorem \ref{T:M1} discussed in \S \ref{SS:(i)} above. Now in the proof of Theorem \ref{detformgl}, it is visibly clear how the ``trichotomy" phenomenon for the expression of $\det(\mca{M}(w_\alpha, i(\chi_o)))$ is dissolved in the computation of $\det(\mca{M}(w_\alpha, i(\chi)))$: for example, if $n\equiv 2 \ (\text{mod } 4)$, then it accounts to a certain identity relating $\tilde{\gamma}$-factors and $\gamma$-factors as shown in Corollary \ref{verybicelem}, which appears to be a mystery to us at the moment despite its proven truth.

%\subsubsection{} \label{SS:(vi)}
%Besides the above results, we note that it is a folklore that the Knapp-Stein dimension theorem for parabolic induction from a discrete series on the Levi subgroup continues to hold for covering groups (as announced by W.-W. Li in \cite{Li3} and also appeared in \cite{Szp4}). Since we failed to find a proof in literature, we include one by C.-H. Luo as Appendix \S \ref{S:App} in the setting of finite-degree central covering groups.
%
%\begin{thm}
%For a general parabolic induction $I_{\wt{P}}^{\wt{G}} (\sigma)$, where $\sigma$ is a discrete series representation of the covering Levi subgroup of $\wt{P}$, one has
%$$\dim {\rm End}_{\wt{G}}(I^{\wt{G}}_{\wt{P}}(\sigma))=|R(\sigma)|,$$
%where $R(\sigma)$ denotes the $R$-group for $I_{\wt{P}}^{\wt{G}}(\sigma)$. Therefore, there exists a finite central extension $Z\into \wt{R} \onto R(\sigma)$ of $R(\sigma)$ by a finite group $Z$ such that
%$$I_{\wt{P}}^{\wt{G}}(\sigma) = \bigoplus_{\rho \in \Pi_{\chi_{\sigma}}( \wt{R} )} \rho \otimes \pi_\rho,$$
%a decomposition as $\wt{R} \times \wt{G}$-modules. Here $\Pi_{\chi_{\sigma}}( \wt{R} )$ is a certain subset of irreducible representations of $\wt{R}$.
%\end{thm}
%
%
%
%As an example, the $R$-group structure for $\Mp_{2r}$ is elaborated based on Gan-Savin's work on theta correspondence for the dual pair $(\Mp_{2r},\SO_{2r+1})$.
%One application of the result in this paper is Corollary \ref{irrres} on the reducibility of genuine principal series of $\wt{\SL}_2$.

\subsection{Several remarks}
Since we only treat \emph{unramified} principal series here for a general reductive group $\wt{G}$, for general principal series representations, more work remains to be done. Indeed, for a general $I(i(\chi))$, the space $\Wh_\psi( I(i(\chi)) )$ is parametrized by a discrete space $\Ftn(i(\chi))$ of size greater than $\val{ \msc{X}_{Q,n}  }$ in general. There is also a natural action of the Weyl group on $\Ftn(i(\chi))$. Presumably the analysis could be carried out in a similar way as in the unramified case, and we believe that an analogue of Theorem \ref{T:M1} also holds; that is, $\det( \mca{T}(w_\alpha, i(\chi))^* )$ is essentially a product of Plancherel measures and gamma (or metaplectic gamma) factors. However, understanding the exponents of such factors requires more delicate analysis, and one would have yet to find more evidence for supporting any conjectural formula. We would like to leave this to a future work.

As mentioned in \cite{GSS1}, it is a crucial issue that our local results on Whittaker functions could not be globalised to study global $L$-functions, as discussed in \S \ref{SS:(ii)} above (see also Remark \ref{R:glob}). It seems to us that one needs some new ideas to overcome such difficulties. On the other hand, for classical groups the extension of the doubling method by Cai, Friedberg, Ginzburg and Kaplan \cite{CFGK1, CFK1} represents an important advance in handling global $L$-functions, see especially the recent work \cite{Kap01} on coverings of the symplectic groups.

Nevertheless, we hope that this paper could help clarifying the relation between the local coefficients matrix $\mca{M}_{\mfr{B}_\chi}(w, i(\chi))$ and the scattering matrix $\mca{S}_\mfr{R}(w, i(\chi); r_w)$ which appear in literature. Indeed, since these two matrices differ by an invertible matrix $\mca{C}(\mfr{B}_{{}^\w\chi}, \mfr{B}_\chi; r_w)$, thus their ranks are equal. That is, it is immaterial to consider either $\mca{M}_{\mfr{B}_\chi}(w, i(\chi))$ or $\mca{S}_\mfr{R}(w, i(\chi); r_w)$ for the purpose of determining the rank; in fact in literature the scattering matrix $\mca{S}_\mfr{R}(w, i(\chi); r_w^{\rm un})$ in the unramified setting is a preferred choice since its entries can be explained by incorporating the twisted Weyl group action mentioned above and discussed in \S \ref{SS:(i)}. This is also the practice in the existing work on theta representations, see \cite{KP, Ga2}.

Despite their subtle difference, we hope to emphasize via the investigation in this paper that \emph{both} $\mca{M}_{\mfr{B}_\chi}(w, i(\chi))$ and $\mca{S}_\mfr{R}(w, i(\chi); r_w)$ are important objects for studying representation of $\wt{G}$, especially when there is a natural choice of $r_w$.  The fact that each of these two matrices has its own advantage suggests the necessity of studying them simultaneously, which is actually the main novelty in the starting point of our paper. We believe that a tale of these two matrices together (rather than any of them alone) will encode a fascinating story of the representation theory for covering groups, especially that pertinent to a Langlands-Shahidi theory of covering groups.

Lastly, we remark that in the context of  covering groups, we strive to work with the minimal and necessary assumption that $\bbmu_n \subset F^\times$, though for computational simplicity we may assume the stronger inclusion $\bbmu_{2n} \subset F^\times$ at various places. However, the latter compromising assumption will be explicated whenever a theorem is proved based on it, as there is non-negligible arithmetic arising from the difference between the two assumptions, for example, see Remark \ref{R:n-2n}.

\subsection{Acknowledgement} The project was initiated when F. Gao held a Golomb postdoctorate fellowship at Purdue university. He would like to thank Professors Shahidi, C.-P. Mok and the faculty in the mathematics department in general for hospitality and encouragement during his stay. The paper was largely completed when the authors were attending the Singapore conference ``On the Langlands Program: Endoscopy and Beyond" in 2019 January. F. Shahidi and D. Szpruch would like to thank the conference organizers, the NUS mathematics department and IMS for support and hospitality. The authors would also like to thank the referees for very insightful and helpful comments on an earlier version of the paper. Improvements have been made at various places in the paper following the referees' suggestions.

F. Shahidi is partially supported by the NSF grant DMS-1801273.

%%%
\section{Covering groups} \label{S:cov}
\subsection{Notation}
Let $F$ be a finite extension of $\Q_p$. Denote by $O_F$ the ring of integers and $\mfr{p} \subset O_F$ its maximal ideal. Let 
$$q:=\val{O_F/\mfr{p}}$$
be the size of the residual field. We fix once and for all a uniformizer $\varpi$ of $F^\times$. We normalize the absolute value on $F$ such that $\val{\varpi} =q^{-1}$.

Let $\psi$ be a non-trivial character of $F$. For $a \in F^\times$, let $\psi_a$ be the character of $F$ given by
$$\psi_a(x)=\psi(ax).$$
We define $\mfr{f}(\psi)$, the conductor of $\psi$, to be the smallest positive integer $k$ such that $\psi$ is trivial on $\mfr{p}^k$. We say that $\psi$ is normalized if
$$\mfr{f}(\psi)=0.$$
For a ramified character $\chi$ of $F^\times$ we define $\mfr{f}(\chi)$, the conductor of $\chi$, to be the smallest integer $k$ such that $\chi$ is trivial on $1+\mfr{p}^k$. For an unramified character $\chi$ of $F^\times$ (i.e., trivial on $O_F^\times$), we set $\mfr{f}(\chi)=0$.

Let $G$ be a group; we denote its center by $Z(G)$.

\subsection{$\mbf{K}_2$-extension} \label{S:K2-extn}
Let $\mbf{G}$ be a split connected linear reductive group over $F$ with root datum:
$$(X, \Phi, \Delta; Y, \Phi^\vee, \Delta^\vee).$$
Here $X$ and $Y$ denote the character and cocharacter lattices of a fixed maximal split torus $\mbf{T} \subset \mbf{G}$. The sets $\Phi$ and $\Phi^\vee$ denote the roots and coroots respectively. We fix a set of simple roots $\Delta \subset \Phi$, and an associated Borel subgroup $\mbf{B}=\mbf{T} \mbf{U}$. Let
$$Y^{\rm sc} \subset Y$$
be the sublattice generated by all coroots. Denote by 
$$W=N(\mbf{T})/\mbf{T}$$
the Weyl group for $(\mbf{G}, \mbf{T})$. We identify $W$ with the group of reflections of $Y\otimes \Q$ generated by $\w_{\alpha}$ for all $\alpha^\vee \in \Phi^\vee$. Fix a Chevalley-Steinberg system of pinnings:
$$\set{e_\alpha: \mbf{G}_a \to \mbf{U}_\alpha:  \alpha \in \Phi},$$
where $\mbf{U}_\alpha$ is the root subgroup associated to $\alpha$. Let $\mbf{G}_{\rm der}:=[\mbf{G}, \mbf{G}]$ be the derived subgroup of $\mbf{G}$ and let
$$
\mbf{G}^{\rm sc} \onto \mbf{G}_{\rm der}
$$
be the simply-connected cover. We have the groups of rational points
$$G, B=TU, G^{\rm sc}, G_{\rm der}$$
of $\mbf{G}, \mbf{B}, \mbf{G}^{\rm sc}$ and $\mbf{G}_{\rm der}$ respectively. By abuse of language, we call $G$ reductive, semisimple or almost-simple whenever $\mbf{G}$ is reductive, semisimple or almost-simple respectively.

% By abuse of notation, denote by
%$$f: G^{\rm sc} \to G_{\rm der}$$ the arising map, whose image is the derived subgroup $\msc{D}(G_{\rm der}):=[G_{\rm der}, G_{\rm der}]$ of $G_{\rm der}$. We have $G^{\rm sc}=\msc{D}(G^{\rm sc})$, but $\msc{D}(G_{\rm der})$ is only a subgroup of $G_{\rm der}$ in general. Both $G^{\rm sc}$ and $\msc{D}(G_{\rm der})$ are Chevalley groups in the sense of \cite{Ste16}.

Consider the following data:
\begin{enumerate}
\item[$\bullet$] an integer-valued bilinear form  $D: Y\times Y \to \Z$ such that
$$Q(y):=D(y,y)$$
is an integer-valued Weyl-invariant quadratic form on $Y$;
\item[$\bullet$] a homomorphism
$$\eta: Y^{\rm sc} \to F^\times$$
of $Y^{\rm sc}$, which is completely determined by the finite set
$$\set{\eta(\alpha^\vee): \alpha \in \Delta}.$$
\end{enumerate}
For any integral Weyl-invariant quadratic form $Q$, there exists $D$ satisfying the above equality; in this case, we call $D$ a bisector associated to $Q$. Then
$$B(x, y):=D(x, y) + D(y, x)$$
is an integer-valued Weyl-invariant bilinear form on $Y$.

It is shown by Brylinski and Deligne (under an equivalent formulation given in \cite{We3, GG}) that there is an equivalence of Picard categories:
$${\rm CExt}(\mbf{G}, \mbf{K}_2) \to \set{(D, \eta)},$$
where ${\rm CExt}(\mbf{G}, \mbf{K}_2)$ denotes the category of $\mbf{K}_2$-extensions $\wm{G}$ of $\mbf{G}$. There are natural homomorphisms on the two categories, see \cite{GG} for details. Assume that $F^\times$ contains the full group $\bbmu_n$ of $n$-th roots of unity. An $n$-fold cover, in the sense of \cite{We6}, is just $(\wm{G}, n)$. We note that if $\mbf{G}_{\rm der}$ is simply-connected, then one may assume $\eta=\mbm{1}$ without loss of generality on the isomorphism class of $\wm{G}$.

\subsection{Hilbert symbol and Lagrangian decomposition}
We identify $\bbmu_n$ as a subgroup of $\C^\times$ via a fixed embedding $\bbmu_n \into \C^\times$. Denote by
$$(-, -)_n: F^\times \times F^\times \to \bbmu_n$$
the $n$-th power Hilbert symbol.  Observe that if $n$ is odd then $\bbmu_{2n}  \subseteq F^\times$ as $\bbmu_n$ and $-1$ generate $\bbmu_{2n}$.

In the rest of this subsection, we will recall some technical results on the Hilbert symbol and the Lagrangian decomposition, which will be mainly used later in \S \ref{S:RES} to construct genuine principal series for covers of $\SL_2$ and $\GL_2$. The reader may skip to \S \ref{S:top-c} to avoid this digression.

In general, for every $m|n$ one has the $m$-th Hilbert symbol $(-,-)_m$, which gives a perfect pairing between $F^\times/F^{\times m}$ and $F^\times/F^{\times m}$, see \cite[page 260]{WeilB}. In particular, it identifies $F^\times / F^{\times m}$ with its dual, $\widehat{F^\times / F^{\times m}}$, which may also  be identified with the group of  characters of $F^\times$ whose order divides $m$. One has (see \cite[page 48]{Lang})
\begin{equation} \label{index}
[F^\times: F^{\times m}]=m^2 \cdot \val{m}^{-1}.
\end{equation}
Moreover, for an integer $c' \in \Z$, if we write $m=m_1m_2$ with $m_1=\text{gcd}(m,c')$, then (see \cite[Lemma 1]{CO})
$$F^{\times m_2}=\{x \in F^\times \mid x^{c'} \in F^{\times m}\}.$$

For every   $x \in F^\times$,
\begin{equation} \label{x with x}
(x,x)_{m}=(-1,x)_{m}.
\end{equation}
In particular, if $\bbmu_{2m} \subset F^\times$,
then
\begin{equation} \label{x with x is 1}
(x,x)_{m}=(-1,x)_{m}=1.
\end{equation}
Also, if $mm'$ divides $n$ then
\begin{equation} \label{FV fact}
(x,y)_{mm'}^{m'}=(x,y)_m.
\end{equation}
In particular, $(x, y)_m= (x,y)_n^{n/m}$.

For every $x\in F^\times$, we have the character
$$\eta_{x,(m)}: F^\times \to \C^\times \text{ given by } \eta_{x, (m)}(a):= (x,a)_m.$$
Note that  $\eta_{x,(m)}$ is determined by the image of $x$ in $F^\times / F^{\times m}$. In particular, it is trivial if and only if $x \in F^{\times m}$. We view $\eta_{x,(m)}$ as a character of $F^\times/ {F^\times}^m$. Conversely, every character of $F^\times/ F^{\times m}$ is of the form $\eta_{x,(m)}$. It follows from the orthogonality of characters that for $x \in F^\times$, one has
\begin{equation} \label{dualhilbert}
\sum_{a \in F^\times/F^{\times m}} \eta_{a,(m)}(x)=
\begin{cases}
[F^\times:F^{\times m}]  &  \text{ if } x \in F^{\times m}; \\
0  &   \text{ otherwise}.
\end{cases}
\end{equation}

\begin{lm} \label{oldlem}
Suppose that  $m |n$. For an integer $c$, if we set $m=m_1m_2$ where $m_1=\gcd(m,c)$, then
$${\rm Ker} (\eta_{x,(m)}^{c} ) ={\rm Ker }(\eta_{x,(m_2)}).$$
\end{lm}
\begin{proof}
Write $c=m_1c_2$. By \eqref {FV fact} we have
$${\rm Ker}(\eta_{x,(m)}^{c})={\rm Ker}(\eta_{x,(m_1m_2)}^{m_1c_2}) = {\rm Ker}(\eta_{x,(m_2)}^{c_2}).$$
Since ${\rm Im}(\eta_{x,(m_2)} ) \subseteq  \bbmu_{m_2}$ and $\gcd(m_2,c_2)=1$, we are done.
\end{proof}

\begin{lm} \label{dualcenter}
Let $m,l$ be two integers such that $m$ and $ml$ both divide $n$.
\begin{enumerate}
\item[(i)] The map
$x \mapsto \eta_{x,(ml)}\mid_{{F^\times}^{m}}$ gives an isomorphism from $F^\times/F^{\times l}$ to $\widehat{ F^{\times m}/F^{\times ml} }$.
\item[(ii)] The map $x \mapsto \eta_{x,(ml)}$ gives an isomorphism from $F^{\times m}/F^{\times ml}$ to   $\widehat{ F^\times/ F^{\times l} }$.
\item[(iii)] Let $L$ be a set of representatives of $F^\times/F^{\times l}$. The map $(a,b) \mapsto \eta_{a,(ml)} \cdot \eta_{b,(m)}$ gives a bijection from the set $L \times F^\times/{F^\times}^m$ to  $\widehat{{F^\times}/{F^\times}^{ml}}$.
\end{enumerate}
\end{lm}
\begin{proof}
Every character of $F^{\times m}$ whose kernel contains $F^{\times ml}$ is of the form
$$\eta_{x,(ml)}|_{ F^{\times m}}.$$
For $x,y \in {F^\times}$, we have
$$(x,y^m)_{ml}=(x,y)_l.$$
Thus, $\eta_{x,(ml)}|_{ F^{\times m}}$ is the trivial character if and only if $x \in F^{\times l}$. This proves (i).

It follows from \eqref{FV fact} that every character of $F^\times$ which is trivial on $F^{\times l}$ is of the form $\eta_{x, (ml)}$, where $x \in F^{\times m}$. Moreover,  $\eta_{x, (ml)}$ is the trivial character if and only if  $x \in F^{\times ml}$. This gives us (ii).

To prove (iii), it is sufficient to show that the map $(a,b) \mapsto \eta_{a,(ml)} \cdot \eta_{b,(m)}$ is an injection. Thus, we assume
$$\eta_{a,(ml)} \cdot \eta_{b,(m)}=\eta_{c,(ml)} \cdot \eta_{d,(m)}.$$
Then for all $x \in F^\times$, we have
$$(ac^{-1},x)_{ml}  \cdot (bd^{-1},x)_{m}=1.$$
In particular, for all $x \in F^{\times m}$, one has
$$(ac^{-1},x)_{ml}=1;$$
equivalently, for all $z \in F^\times$,
$$(ac^{-1},z^m)_{ml}=1.$$
This implies that $\eta_{ac^{-1},(l)}$ is the trivial character. Thus $ac^{-1} \in F^{\times l}$. Since $a,c \in L$, we deduce that $a=c$. This shows that  $\eta_{b,(m)}=\eta_{d,(m)}$. The proof is now completed.
\end{proof}

For a subgroup $J \subset F^\times$ (or a subgroup of $F^\times/F^{\times m}$), we define
$$J^{\perp}_{(m)}=\bigcap_{x \in J} {\rm Ker}(\eta_{x,(m)}).$$
If $J\subset F^\times$, then we view $J_{(m)}^\perp$ as a subgroup of $F^\times$. On the other hand, if $J\subset F^\times/F^{\times m}$, then we consider $J_{(m)}^\perp \subset F^\times/F^{\times m}$.

\begin{lm} \label{nicesat1}
Let $J \subset F^\times$ be a subgroup containing $F^{\times m}$. Then map
$$\phi: {F^\times}/J^{\perp}_{(m)}\rightarrow \widehat{J/{F^\times}^m}$$
given by
$$\phi \bigl(x\cdot J^{\perp}_{(m)} \bigr)=\eta_{x,(m)}|_J$$
is a well-defined isomorphism. In particular,
$$[ {F^\times}:J^{\perp}_{(m)}]=[J:F^{\times m}].$$
\end{lm}
\begin{proof} Every character of $J/F^{\times m}$ is of the form $\eta_{x,(m)}|_J$ for some $x \in F^\times$. By definition, $\eta_{x,(m)}|_J$ is trivial if and only if $x\in J^{\perp}_{(m)}$.
\end{proof}

 \begin{dfn} A subgroup $J$ of $F^\times/ F^{\times m}$ (resp. of $F^\times$) is called an $m$-Lagrangian subgroup if ${J} =J^{\perp}_{(m)} \subset  F^\times/ F^{\times m}$ (resp. as subgroups in $F^\times$).
\end{dfn}

If $J \subseteq F^\times$ is an $m$-Lagrangian subgroup, then it follows from \eqref{index} that $\val{m}^{-\half} \in \N$ and moreover
$$[F^\times: J]=[J: F^{\times m}]=\sqrt{[F^\times: F^{\times m}]}= m \cdot \val{m}^{-\half}.$$
Note that $J \subseteq  F^\times$ is an $m$-Lagrangian subgroup if an only if its image in $F^\times/F^{\times m}$ is an $m$-Lagrangian subgroup of $F^\times/F^{\times m}$.

\begin{dfn} Let $J$ and $K$ be two Lagrangian subgroups of $F^\times/F^{\times m}$. The pair $(J, K)$ is called a Lagrangian decomposition  of $F^\times/F^{\times m}$ if $F^\times/ F^{\times m}= J \times K$ and that the map $b \mapsto {\eta_{b,(m)}}|_J$ gives an isomorphism from $K$ to $\widehat{J}$.
\end{dfn}
Note that if $(J, K)$  is a Lagrangian decomposition of $F^\times/F^{\times m}$, then $J \simeq K$ and both are of size $m\cdot \val{m}^{-\half}$.

\begin{lm}[{\cite[Lemma 2.3]{Szp6}}] \label{lang decomp}
If $\bbmu_{2m} \subseteq F^\times$, then a Lagrangian decomposition of $F^\times/F^{\times m}$ exists. Furthermore, if $(J, K)$ and $(J', K')$ are two Lagrangian decompositions  of  $F^\times/F^{\times m}$, then there exists a group automorphism $\theta$ of $F^\times/F^{\times m}$ preserving $(-,-)_m$ such that $\theta(J)=J'$ and $\theta(K)=K'$.
\end{lm}

\begin{lm} \label{Jaut}
Let $J \subset F^\times$ be an $m$-Lagrangian subgroup and $l\in \N$. If $\gcd(l,m)=1$, then the map $x \mapsto x^l$ gives rise to an automorphism of $F^\times/J$ and also of $J/F^{\times m}$.
\end{lm}

\begin{lm} \label{extlag} Suppose that $m'|m$ and $m|n$. If $\bbmu_{2m'} \subseteq F^\times$, then every  $m$-Lagrangian subgroup of $F^\times$ is contained in an $m'$-Lagrangian subgroup.
\end{lm}

\begin{proof}
Let $J \subseteq F^\times$ be an $m$-Lagrangian subgroup. Write $m=m't$. We note that for any $x,y \in J$,
$$1=(x,y)^t_m=(x,y)_{m'}.$$
Thus, $J \subseteq J^{\perp}_{(m')}$. We now describe a finite inductive process, which will give us the desired result.

Set $J_0=J$. If $J_0$ is an $m'$-Lagrangian subgroup, then we are done. Otherwise, there exists $x\in F^\times$ which lies outside $J_0$ such that $(x,j)_{m'}=1$ for all $j \in J$. Since $\bbmu_{2m'} \subseteq F^\times$,  $(x,x)_{m'}=1.$ Define $J_1$ to be the group generated by $x$ and $J$. Clearly ,
$$J \varsubsetneq J_1  \subseteq {J_1}^{\perp}_{(m')}.$$
If $J_1={J_1}^{\perp}_{(m')}$, we stop. Otherwise we continue the argument in a similar way as above. This process must end at one point since $F^\times/{F^\times}^{m'}$ is finite. This completes the proof.
\end{proof}

\begin{lm} \label{nicesat2}
Let $m,l$ be two integers such that $(ml)|n$. Let $J \subseteq F^\times$ be an $m$-Lagrangian subgroup.
Then
$$J^{\perp}_{(ml)}=J^l, \text{ where } J^l:=\{x^l \mid x\in J \}.$$
In particular, $\widehat{J/F^{\times ml}} \simeq F^\times/J^l$ and
$$[F^\times: J^l]= m \val{m}^{-1/2} \cdot l^2 \val{l}^{-1}.$$
\end{lm}
\begin{proof} For all $x,y \in J$ we have $$(x^l,y)_{ml}=(x,y)^{l}_{ml}=(x,y)_m=1.$$
This shows  that $J^l  \subseteq J^{\perp}_{(ml)}  $. On the other hand, if  $y \not \in J$, then there exists  $x \in J$ such that $(x,y)_{ml} \neq 1$. Indeed, this follows from the equality $(x,y)_{ml}^l=(x,y)_m$ along with the fact that $J$ is an $m$-Lagrangian subgroup. Thus we have shown
$$J^l \subseteq J^{\perp}_{(ml)} \subseteq J.$$

We now show that if  $y\in J$ lies outside $J^l$, then there exists $x \in J$ such that $(x,y)_{ml} \neq 1$. From the assumption on $y$, it follows that there exists a character $\chi$ of $J/J^l$ such that $\chi(y) \neq 1$. Since $J \supseteq {F^\times}^m$, we deduce that  $J^l \supseteq {F^\times}^{ml}$. Thus, the pullback of $\chi$ to $J$ must be of the form $\eta_{x,(ml)}|_J$ for some $x\in F^\times$. Since $\eta_{x,(ml)}|_{J^l}$ is trivial, we have for all $z \in J$
$$1=(x,z^l)_{ml}=(x,z)_m.$$
Thus, since $J$ is an $m$-Lagrangian subgroup, we deduce that $x \in J$. We have proven  $J^{\perp}_{(ml)}=J^l$. Note that
\begin{equation}
\begin{aligned}\label{Mindex}
[J:{F^\times}^{ml}]  & =[J:{F^\times}^{m}] \cdot [{F^\times}^{m}:{F^\times}^{ml}] =[J:{F^\times}^{m}] \cdot  \frac{ [{F^\times}:{F^\times}^{ml}]}{[{F^\times}:{F^\times}^{m}]} \\
 & = m \val{m}^{-1/2} \frac {(ml)^2 \val{ml}^{-1}}{m^2 \val{m}^{-1}}=  m \val{m}^{-1/2} \cdot l^2 \val{l}^{-1}
\end{aligned}
\end{equation}
The rest follows from Lemma \ref{nicesat1}.
\end{proof}

\begin{lm} \label{pn1} Assume that  $\gcd(p,n)=1$.
\begin{enumerate}
\item[(i)] Let $K_m=O_F^\times F^{\times m} $. Let
$$K=O_F^\times F^{\times m} /F^{\times m},\  J=\varpi^{\Z} F^{\times m}/ F^{\times m}.$$
Then $K_m$ is an $m$-Lagrangian subgroup of $F^\times$. Both $J$ and $K$ are isomorphic to the cyclic group with $m$ elements. If $\bbmu_{2m} \subseteq F^\times$, then  $(J, K)$ is a Lagrangian decomposition of  $F^\times / F^{\times m}$.
\item[(ii)] If $n$ is even, then a set of representatives for $F^\times/F^{\times 2}$ is given by
$\{1, u, \varpi^{-1}, u\varpi^{-1} \}$ where $\eta_{u, (n)}$ is unramified with $\eta_{u,(n)}^{n/2}(\varpi)=-1$ and  $\mfr{f}(\eta_{\varpi,(n)})=1$.
\end{enumerate}
\end{lm}
\begin{proof}
Under the condition $p\nmid n$, we have $[F^\times : F^{\times m}]=m^2$. Also $(-,-)_m$ is trivial on $O_F^\times \times O_F^\times$. The assertion (i) now follows from the non-degeneracy of the Hilbert symbol along with \eqref{x with x is 1}. For the proof of (ii),  we simply note that $1+\mfr{p} \subset F^{\times n}$ (see \cite[page 43]{Lang}).
\end{proof}

\subsection{Topological covers} \label{S:top-c}
In the setting of \S \ref{S:K2-extn}, every $\mbf{K}_2$-extension $\wm{G}$ gives rise to an $n$-fold central covering
$$\begin{tikzcd}
\bbmu_n \ar[r, hook, "\iota"] & \wt{G} \ar[r, two heads, "\p"] & G
\end{tikzcd}$$
by pushing out the short exact sequence
$$\begin{tikzcd}
\mbf{K}_2(F) \ar[r, hook] & \wm{G}(F) \ar[r, two heads, "\p"] & \mbf{G}(F)
\end{tikzcd}$$
via the $n$-th Hilbert symbol 
$$(-, -)_n: \mbf{K}_2(F) \to \bbmu_n.$$
The above extension $\wt{G}$ of $G$ by $\bbmu_n$ is a central extension of locally compact topological groups with $\bbmu_n$ a finite and discrete subgroup. The embedding $\iota$ and the projection $\p$ are continuous, and moreover $\p$ induces an isomorphism 
$$\wt{G}/ \iota(\bbmu_n) \simeq G$$
 of topological groups.

A representation of $\wt{G}$ is called genuine if $\bbmu_n$ acts via the fixed embedding $\bbmu_n \into \C^\times$. In this paper, we consider genuine representations of $\wt{G}$ (and its subgroups containing $\bbmu_n$), and denote the isomorphism class of such genuine representations by $\Irr(\wt{G})$. Occasionally, we also consider anti-genuine representations, such that the action of $\bbmu_n$ is given by the inverse of the embedding $\bbmu_n \subset \C^\times$.

The structure of $\wt{G}$ is described elegantly in terms of generators and relations in the early work of Steinberg \cite{Ste1}, Matsumoto \cite{Mat1} and Moore \cite{Mo1} etc, mostly for covers of semisimple groups. The work of Brylinski and Deligne \cite{BD} allows an extension of such description to covers of general reductive groups. See also the work \cite{We2, We3, We6, GG}  following this description.

The covering group $\wt{G}$ splits over unipotent subgroups canonically and $G$-equivariantly, where the action of $G$ on $\wt{G}$ is by conjugation. Namely, if $U \subset G$ is a unipotent subgroup, then there is a unique splitting $s_U: U \to \wt{G}$ satisfying
$$s_U(g u g^{-1}) = \wt{g} s_U(u) \wt{g}^{-1},$$
where the right hand side is independent of the choice of lifting $\wt{g} \in \wt{G}$ of $g \in G$.

 Denote by $\wt{e}_\alpha(F)$ the splitting of $e_\alpha(F)$ in $\wt{G}$. For any $\alpha \in \Phi$ and $x\in F^\times$, define
\begin{equation} \label{E:w}
\wt{w}_\alpha(x):=\wt{e}_\alpha(x) \wt{e}_{-\alpha}(-x^{-1}) \wt{e}_\alpha(x)
\end{equation}
and
\begin{equation} \label{E:h}
\wt{h}_\alpha(x):=\wt{w}_\alpha (x) \wt{w}_{\alpha}(-1).
\end{equation}
For any $\alpha\in \Phi$, for simplicity we write
$$\wt{w}_\alpha:=\wt{w}_\alpha(1).$$
When we consider the case $n=1$ (i.e. $\wt{G}=G$), we will use the notations $e_\alpha, w_\alpha, h_\alpha$ for $\wt{e}_\alpha, \wt{w}_\alpha, \wt{h}_\alpha$ respectively.

Let $W' \subset \wt{G}$ be the subgroup generated by $\wt{w}_\alpha$ for all $\alpha \in \Phi$. Then the map $\wt{w}_\alpha \mapsto \w_\alpha$ gives a surjective homomorphism
$$W' \onto W$$
with kernel being a finite group (see \cite{Sav1}). For any $\w=\w_{\alpha_k} ... \w_{\alpha_2} \w_{\alpha_1} \in W$ in a minimal decomposition, we let
$$\wt{w}:=\wt{w}_{\alpha_k} ... \wt{w}_{\alpha_2} \wt{w}_{\alpha_1} \in W'$$
be its representative, which is independent of the minimal decomposition (see \cite[Lemma 83 (b)]{Ste16} and \cite[Page 141]{BLS}). In particular, we denote by $\wt{w}_G \in \wt{G}$ the above representative of the longest Weyl element $\w_G$ of $W$. Note that we also have the natural representative
$$w:=w_{\alpha_k} ... w_{\alpha_2} w_{\alpha_1} \in G$$
 of the above $\w\in W$. In particular, one has the representative $w_G \in G$ for $\w_G$, which is the image of $\wt{w}_G$ in $G$.

The group $\wt{G}$, which is associated to the data $(D, \eta)$ of $\wm{G}$, is generated by $\set{\wt{e}_\alpha(F): \alpha \in \Phi}$ and $\set{y(a): y\in Y, a\in F^\times}$, where $Y$ is the cocharacter lattice of $G$. Relations among the generators include the following:
\begin{enumerate}
\item[(A)] $\wt{e}_\alpha(x) $ is additive in $x$.
\item[(B)] If $\alpha$ and $\beta$ are roots with $\alpha + \beta \ne 0$, then the commutator
$$[\wt{e}_\alpha(x), \wt{e}_\beta(y)]=\prod \wt{e}_{i\alpha + j\beta}(c_{i,j} x^i y^j),$$
where $i$ and $j$ are positive integers and $c_{i,j}$'s are certain integers.
\item[(B)'] For any $\alpha\in \Phi$ and $x\in F^\times$:
$$\wt{w}_\alpha(x) \wt{e}_\alpha(u) \wt{w}_\alpha(x)^{-1} = \wt{e}_{-\alpha}(-x^2 u).$$
\item[(C)] There exists a section $\s$ of $\wt{T}$ over $T$ such that
$$\s(y_1(a)) \cdot \s(y_2(b)) = \s(y_1(a) \cdot y_2(b)) \cdot (a, b)_n^{D(y_1, y_2)}$$
for any $y_1, y_2\in Y$ and $a, b\in F^\times$. For any $\alpha \in \Delta$ and $x\in F^\times$:
$$\wt{h}_\alpha(x)= \s(\alpha^\vee(x)) \cdot (\eta(\alpha^\vee), x)_n^{Q(\alpha^\vee)}.$$
\item[(D)] For $\wt{t} \in \wt{T}$ whose image in $T$ is denoted by $t$, one has
$$\wt{w}_\alpha \cdot \wt{t}  \cdot \wt{w}_\alpha^{-1} = \wt{t} \cdot \wt{h}_\alpha(\alpha(t)^{-1}).$$
\end{enumerate}
Here the relations (A), (B), (B)' follow from \cite[Page 249, Theorem 2]{We2}, the first equality relation in (C) follows from \cite[\S 3.9]{BD} (see also \cite[Page 100]{We3}) and the second relation in (C) follows from \cite[Proposition 2.5]{GG}. Lastly, the relation (D) is given in \cite[Proposition 11.9]{BD}.

By (C) above, the commutator $T \times T \to \bbmu_n$ is given by
$$[y_1(a), y_2(b)]=(a, b)_n^{B(y_1, y_2)},$$
where $y_i \in Y$ and $a, b\in F^\times$.   Fix a uniformizer $\varpi \in F$. For any $y\in Y$, we write
$$\varepsilon:=(-1, \varpi)_n \in \bbmu_n, \quad \s_y:=\s(y(\varpi)) \in \wt{T}.$$
Moreover, in this paper, we assume that the composition
$$\eta_n: Y^{sc} \to F^\times \to F^\times/(F^\times)^n$$
 of $\eta$ with the obvious quotient is trivial. Thus, it follows  immediately from this assumption and (C) above that 
 $$\wt{h}_\alpha(\varpi^k)=\s_{k\alpha^\vee}$$
  for every $k\in \Z$. Moreover, this assumption on $\eta_n$ entails many other consequences  including on the existence of an isomorphism
 $${}^L \wt{G} \simeq \wt{G}^\vee \times W_F$$
  for the $L$-group of $\wt{G}$ (to be discussed in \S \ref{SS:L-g}), and on the existence of splitting of $\wt{G}$ over $K$ (as we will discuss in \S \ref{SS:unrep}). For more detailed discussion on this, we refer the reader to \cite{GG}. For convenience, we may also write
  $$w(\wt{t}):=w \wt{t} w^{-1}$$ for $\wt{t} \in \wt{T}$.

\begin{lm} \label{L:W-act}
For $y\in Y$ and $\alpha\in \Delta$, one has
$$w_\alpha (\s_y)=  \varepsilon^{\angb{y}{\alpha}D(y,\alpha^\vee)} \cdot \s_{\w_\alpha(y)}.$$
If $\bbmu_{2n} \subset F^\times$, then the map $\s: Y \to \wt{T}$ given by $y\mapsto \s_y$ is a homomorphism; moreover, in this case, $w(\s_y)= \s_{\w(y)}$ for every $\w \in \W$ and $y\in Y$.
\end{lm}
\begin{proof}
The first equality follows from (C) and (D) above. The rest of the statement is clear, since $\varepsilon=1$ if $\bbmu_{2n} \subset F^\times$.
\end{proof}

\subsection{Dual group}

For a cover $(\wm{G}, n)$ associated to $(D, \eta)$, with $Q$ and $B_Q$ arising from $D$, we define
\begin{equation} \label{YQn}
Y_{Q,n}:= Y\cap nY^*,
\end{equation}
where $Y^* \subset Y\otimes \Q$ is the dual lattice of Y with respect to $B_Q$; more explicitly,
$$Y_{Q,n}= \set{y\in Y: B_Q(y, y')\in n\Z \text{ for all } y'\in Y} \subset Y.$$
For every $\alpha^\vee\in \Phi^\vee$, denote
$$n_\alpha:= \frac{n}{\text{gcd}(n, Q(\alpha^\vee))}$$
and
$$ \alpha_{Q,n}^\vee=n_\alpha \alpha^\vee, \quad \alpha_{Q,n}=\frac{\alpha}{n_\alpha} .$$
Let
$$Y_{Q,n}^{sc} \subset Y_{Q,n}$$
be the sublattice generated by $\Phi_{Q,n}^\vee=\{\alpha_{Q,n}^\vee: \alpha^\vee \in \Phi^\vee \}$.  Denote $X_{Q,n}=\text{Hom}_\Z(Y_{Q,n}, \Z)$ and $\Phi_{Q,n}=\set{\alpha_{Q,n}: \alpha \in \Phi }$. We also write
$$\Delta_{Q,n}^\vee=\{ \alpha_{Q,n}^\vee: \alpha^\vee \in \Delta^\vee \} \text{ and } \Delta_{Q,n}=\set{\alpha_{Q,n}: \alpha\in \Delta}.$$
Then
$$\big( Y_{Q,n}, \ \Phi_{Q,n}^\vee, \ \Delta_{Q,n}^\vee;\  X_{Q,n},\  \Phi_{Q,n}^\vee, \Delta_{Q,n} \big)$$
forms a root datum with a choice of simple roots $\Delta_{Q,n}$. It gives a unique (up to unique isomorphism) pinned reductive group $\wm{G}_{Q,n}^\vee$ over $\Z$, called the dual group of $(\wm{G}, n)$. In particular, $Y_{Q,n}$ is the character lattice for $\wt{G}_{Q,n}^\vee$ and $\Delta_{Q,n}^\vee$ the set of simple roots. Let
$$\wt{G}_{Q,n}^\vee:=\wm{G}_{Q,n}^\vee(\C)$$
be the associated complex dual group. We may write $\wt{G}^\vee$ for $\wt{G}_{Q,n}^\vee$ whenever no confusion arises. The center of $\wt{G}_{Q,n}^\vee$ is
$$Z(\wt{G}_{Q,n}^\vee) = \Hom(Y_{Q,n}/Y_{Q,n}^{sc}, \C^\times).$$
Let $\mbf{H}$ be the pinned split reductive group over $F$ such that
$$\mbf{H}^\vee \simeq \wm{G}_{Q,n}^\vee,$$
where $\mbf{H}^\vee$ is the Langlands dual group of $\mbf{H}$. Then $\mbf{H}$ is the principal endoscopic group of $(\wm{G},n)$, and clearly $\mbf{H} = \mbf{G}$ if $n=1$.

For fixed $(\mbf{G}, Q)$ with $\mbf{G}$ a semisimple group, we want to show the periodicity of $\wt{G}_{Q,n}^\vee$, as $n$ changes.
On one hand, this is not very surprising as there are only finitely many semisimple groups associated with a fixed type of root system. Thus, the set $\set{\wt{G}_{Q,n}^\vee: n\in \N}$ can be partitioned into finitely classes, each containing only isomorphic groups. On the other hand, the periodicity we will prove shows  that there exists $c\in \N$ and a partition such that each class can be represented by $\set{\wt{G}_{Q,n_0 + c k}^\vee: k\in \N}$ for some $n_o$.

To proceed, we define for $a \in \N$ and $E\in \Z$ the number
$$a_E:= \frac{a}{\text{gcd}(a, E)} \in \N,$$
and prove a general lemma which will be used in a later section in the paper as well.

\begin{lm} \label{L:key}
Let $V$ be a finite-dimensional $\Q$-vector space and let $v \in V$ be a fixed element. Let $L_1, L_2 \subset V$ be lattices which are commensurable, i.e., there exists a lattice $L \subset V$ such that $[L_i: L] < \infty$ for $i=1, 2$. Let $E \in \Z \backslash \set{0}$ be a fixed nonzero integer. Consider
$$S(n):= (n_E^{-1} (L_1-v) ) \cap L_2 \subset V.$$
Then, there exists $c:=c(L_1, L_2) \in \N$ depending only on $L_1$ and $L_2$ such that
$$S(n + c) = S(n)$$
for all $n$. Similarly, if we define $S'(n):= (n_E^{-1} (L_1-v) ) \cap (nn_E^{-1} \cdot L_2) \subset V$, then there exists $c'=c'(L_1, L_2)$ such that $S'(n+c')=S'(n)$ for all $n$.
\end{lm}
\begin{proof}
Let $\set{e_1, e_2, ..., e_r}$ be a basis for $L_2$. Let ${k_1, k_2, ..., k_r} \subset \N$ be such that $k_i e_i$ lies in $L_1$ for all $i$. Denote $k=\text{lcm}\set{k_i: 1\le i\le r}$.

Let $c:=\val{E} \cdot k$. Note that
$$(n+c)_E= \frac{n+c}{ \text{gcd}(n+c, E) } = \frac{n+ c}{ \text{gcd}(n, E)  }= n_E + k \cdot \val{E}_n.$$
One has
$$\begin{aligned}
& y \in (n_E^{-1} \cdot (L_1 -v)) \cap L_2 \\
\iff &  y=\sum_i a_i e_i \text{ with } a_i \in \Z \text{ such that } n_E \cdot (\sum_i a_i e_i) + v\in L_1 \\
\iff &  y=\sum_i a_i e_i \text{ with }  a_i \in \Z \text{ such that } (n + c)_E \cdot (\sum_i a_i e_i) + v\in L_1 \\
\iff & y \in ((n+c)_E^{-1} \cdot (L_1 -v)) \cap L_2.
\end{aligned}$$
This completes the proof for $S(n)$. The above consideration applies to $S'(n)$ with an easy modification, by noting that for the $c$ above one has $(n+c)\cdot (n+c)_E^{-1} =n \cdot n_E^{-1}$.
\end{proof}

\begin{prop} \label{P:DG}
Let $\mbf{G}$ be a semisimple group and $Q$ a fixed Weyl-invariant quadratic form on $Y$. Then there exists a number $c:=c(\mbf{G}, Q)\in \N$ such that
$$\wt{G}^\vee_{Q,n} \simeq \wt{G}^\vee_{Q, n+c}$$
for all $n$.
\end{prop}
\begin{proof}
If $Q=0$, then $\wt{G}_{Q,n}^\vee= G^\vee$ (the Langlands dual group of $\mbf{G}$) for all $n$; in this case we simply take $c=1$. Now in the proof below we assume that $Q\ne 0$.

There are two cases we discuss, depending on whether there are two lengths of the co-roots in $\Phi^\vee$ of $\mbf{G}$.

First, assume that the coroot system $\Phi^\vee$ is simply-laced.  Then $\wt{G}_{Q,n}^\vee$ is a semisimple group with coroot system the same type as $\mbf{G}$. Therefore, it is determined uniquely by the center
$Z(\wt{G}_{Q,n}^\vee)=\text{Hom}(Y_{Q,n}/Y_{Q,n}^{sc}, \C^\times)$ of $\wt{G}^\vee_{Q,n}$; or equivalently, by
$Y_{Q,n}/Y_{Q,n}^{sc}$. Here, $Y_{Q,n}^{sc} = n_\alpha Y^{sc}$, where $n_\alpha$ is independent of the coroot $\alpha^\vee$. Thus,
$$Y_{Q,n}/Y_{Q,n}^{sc} \simeq (n_\alpha^{-1} Y_{Q,n})/ Y^{sc}.$$
Consider $$M= \set{y\in Y\otimes \Q: B(y, z) \in Z \text{ for all } z\in Y}.$$
Now it follows that
$$\begin{aligned}
n_\alpha^{-1} Y_{Q,n} & = \set{y\in Y\otimes \Q: n_\alpha y\in Y \text{ and } n|B(n_\alpha y, z) \text{ for all } z\in Y} \\
&= (n_\alpha^{-1} Y) \cap (n n_\alpha^{-1} M).
\end{aligned}$$
Applying Lemma \ref{L:key} with $(L_1, L_2, E, v)= (Y, M, Q(\alpha^\vee), 0)$, we see that there exists $c:=c(\mbf{G}, Q)$ such that  $Z(\wt{G}_{Q,n}^\vee) = Z(\wt{G}_{Q, n+ c}^\vee)$.

Second, we consider the case where $\mbf{G}$ is not simply-laced. Let $\alpha^\vee, \beta^\vee \in \Delta^\vee$ be two adjacent nodes in the Dynkin diagram such that $\alpha^\vee$ is long and $\beta^\vee$ is short.

Let $l:=\val{ \angb{\beta}{\alpha^\vee} }$ be the ratio of lengths of $\alpha^\vee$ over $\beta^\vee$. We have $Q(\alpha^\vee)= l\cdot Q(\beta^\vee)$. Define
$$Y_\flat^{sc}= \langle g_i \alpha_i^\vee: 1\le i\le r \rangle,$$
where
$$g_i =\begin{cases}
1 & \text{ if } \alpha_i^\vee \text{ is long , and } \\
l & \text{ if } \alpha_i^\vee \text{ is short.}
\end{cases}$$
Define
$$\N_{\mbf{G}, Q}:=\set{n: n_{\alpha_i} = n_{\alpha_j} \text{ for all } i, j }=\set{n: l \nmid n_{\beta} }.$$
We have
$$Y_{Q,n}^{sc} =\begin{cases}
n_\alpha \cdot Y^{sc} & \text{ if } n \in  \N_{\mbf{G}, Q} \\
n_\alpha \cdot Y_\flat^{sc} & \text{ if }  n \in \N- \N_{\mbf{G}, Q}.
\end{cases}$$

For any $n\in \N_{\mbf{G}, Q}$ we have $Y_{Q,n}/Y_{Q,n}^{sc} \simeq (n_\alpha^{-1} Y_{Q,n}) / Y^{sc}$, where
$$n_\alpha^{-1} Y_{Q,n}= (n_\alpha^{-1} Y) \cap (n n_\alpha^{-1} M)$$
for any long coroot $\alpha$. Again, by Lemma \ref{L:key}, there exists $c_1:=c_1(\mbf{G}, Q)$ such that
$n_\alpha^{-1} Y_{Q,n} = (n+c_1)_\alpha^{-1} Y_{Q,n + c_1}$. By replacing $c_1$ by a multiplier of it, we may assume that $c_1$ is divisible by $Q(\alpha^\vee)$. We get
$$(n+c_1)_\beta= \frac{n+c_1}{\text{gcd}(n+ c_1, Q(\beta^\vee))}= \frac{ n+c_1 }{ \text{gcd}(n, Q(\beta^\vee)) }= n_\beta + l \cdot (c_1)_\beta.$$
Thus $n+c_1 \in \N_{\mbf{G}, Q}$ for any $n\in \N_{\mbf{G}, Q}$.

On the other hand, for $n \in \N - \N_{\mbf{G}, Q}$, one has
$$\frac{ Y_{Q,n} }{ Y_{Q,n}^{sc} } = \frac{ (n_\alpha^{-1} Y) \cap (n n_\alpha^{-1} M)  }{ Y_\flat^{sc} }.$$
A similar argument shows that there exists $c_2:=c_2(\mbf{G}, Q)$ such that for all $n \in \N - \N_{\mbf{G}, Q}$ we have $n + c_2 \in \N - \N_{\mbf{G}, Q}$ and $Y_{Q,n}/Y_{Q,n}^{sc} \simeq Y_{Q, n+c_2} / Y_{Q, n+ c_2}^{sc}$.

Taking $c=\text{lcm}\set{c_1, c_2}$ completes the proof.
\end{proof}

\subsection{Weyl orbits} \label{SS:Worb}
Let $$\rho:= \frac{1}{2} \sum_{\alpha^\vee >0} \alpha^\vee$$
be the half sum of all positive coroots of $\mbf{G}$. Denote by $\w(y)$ the natural Weyl group action on $Y$ and $Y\otimes \Q$ generated by the reflections $\w_\alpha$. We consider the twisted Weyl-action
$$\w[y]:=\w(y-\rho)+ \rho.$$
Clearly $Y$ is stable under this twisted action. Throughout the paper, we denote
$$y_\rho:=y-\rho \in Y\otimes \Q$$
for any $y\in Y$, and thus 
$$\w[y]-y=\w(y_\rho) - y_\rho.$$
One important reason motivating the consideration of $\w[\cdot]$ is that it is involved in describing the entries of the scattering matrices and local coefficients matrices, the main objects of interest in our paper. Such observation was initially made by Kazhdan-Patterson \cite{KP} for covers of $\GL_r$.

Define
$$\msc{X}_{Q,n}^{sc}= Y/Y_{Q,n}^{sc}, \quad \msc{X}_{Q,n}= Y/Y_{Q,n}.$$
Since $Y_{Q,n}$ and $Y_{Q,n}^{sc}$ are both Weyl-invariant with respect to the action $\w(-)$, both the usual action $\w(-)$ and $\w[-]$ descend to give well-defined actions on $\msc{X}_{Q,n}^{sc}$ and $\msc{X}_{Q,n}$. 
From now on, by Weyl orbits in $Y$ or $Y\otimes \Q$  we always refer to the ones with respect to the action $\w[y]$, unless specified otherwise. 
For any subgroup $W' \subset W$, denote by 
$$\mca{O}_{W'}(\msc{X}_{Q,n})$$ the set of twisted $W'$-orbits in $\msc{X}_{Q,n}$. if $W' =W$, we simply write $\mca{O}(\msc{X}_{Q,n})$ instead. We also denote by
$$f_\msc{X}: Y \onto \msc{X}_{Q,n},$$
the natural quotient map, which is $W$-equivariant with respect to the twisted Weyl action. Define analogously
$$\rho_{Q,n}:=\frac{1}{2} \sum_{\alpha^\vee >0} \alpha_{Q,n}^\vee,$$
the half sum of all positive coroots of the principal endoscopic linear group $\mbf{H}$ of $(\wm{G}, n)$.

\begin{eg}
We consider the case of $\SL_2$ where $Y = Y^{sc} = \Z\alpha^\vee$. For every $n\in \N$, we define
$$d= \frac{n}{{\rm gcd}(2, n)}.$$
Let $Q: Y \to \Z$ be the quadratic for such that $Q(\alpha^\vee) = -1$. One has
\begin{enumerate}
\item[$\bullet$] if $n$ is odd, then $Y_{Q,n} = Y_{Q,n}^{sc} = \Z(n\alpha^\vee)$ and $\msc{X}_{Q,n} \simeq \Z/n\Z$;
\item[$\bullet$] if $n=2d$ is even, then $Y_{Q,n} = \Z(d\alpha^\vee)$ and $Y_{Q,n}^{sc}= \Z(n\alpha^\vee)$, and in this case $\msc{X}_{Q,n} \simeq \Z/d\Z$.
\end{enumerate}
In any case, the twisted action of $\w_\alpha$ on $Y$ is given by
$$\w_\alpha[k\alpha^\vee] = (1- k)\alpha^\vee.$$
The same formula describes the twisted action on $\msc{X}_{Q,n}$ and $\msc{X}_{Q,n}^{sc}$.
\end{eg}

\subsection{L-group} \label{SS:L-g}
In \cite{We3, We6}, Weissman constructed the global $L$-group as well as the  local $L$-group extension
 $$\begin{tikzcd}
\wt{G}^\vee_{Q,n} \ar[r, hook] & {}^L\wt{G} \ar[r, two heads] & W_{F},
\end{tikzcd}$$
which is compatible with the global $L$-group. Here $W_F$ is the Weil group of $F$. His construction of $L$-group is functorial, and in particular it behaves well with respect to the restriction of $\wm{G}$ to parabolic subgroups. More precisely, let $\mbf{M} \subset \mbf{G}$ be a Levi subgroup. By restriction, one has the $n$-cover $\wt{M}$ of $M$. Then the $L$-groups ${}^L\wt{M}$ and ${}^L\wt{G}$ are compatible, i.e., there are natural homomorphisms of extensions:
$$\begin{tikzcd}
\wt{G}^\vee_{Q,n} \ar[r, hook] & {}^L\wt{G} \ar[r, two heads] & W_{F} \\
\wt{M}^\vee_{Q,n} \ar[r, hook] \ar[u, hook] & {}^L\wt{M} \ar[u, hook] \ar[r, two heads] & W_{F}  \ar[u, equal] .
\end{tikzcd}$$
For details on the construction and some properties regarding the $L$-group, we refer the reader to \cite{We3, We6, GG}.

For convenience of reference, we recall the following result shown in \cite{GG}. Under the assumption that $\eta_n=\mbm{1}$, there exists a so-called distinguished genuine character 
$$\chi_\psi: Z(\wt{T}) \to \C^\times,$$
depending on a nontrivial additive character $\psi$ of $F$, such that the following properties hold:
\begin{enumerate}
\item[$\bullet$] the character $\chi_\psi$ takes values in $\bbmu_4\subseteq \C^\times$ and is Weyl-invariant, i.e., $\chi_\psi(w^{-1} \cdot \wt{t} \cdot w) = \chi_\psi(\wt{t})$ for all $\wt{t} \in Z(\wt{T})$ and $w$;
\item[$\bullet$] $\chi_\psi$ gives rise to a splitting of ${}^L\wt{G}$ over $W_F$, with respect to which one has an isomorphism ${}^L\wt{G} \simeq_{\chi_\psi} \wt{G}^\vee \times W_F$.
\end{enumerate}

%%%
\section{Local coefficients matrix and scattering matrix} \label{S:2mat}

\subsection{Whittaker functionals}
Let $\psi: F \to \C^\times$ be a nontrivial character. By abuse of notation, denote by
$$\psi: U \to \C^\times$$
the unique character such that
$$\psi(e_\alpha(x)) = \psi(x)$$
for every $\alpha\in \Delta $ and $x\in F$.  We identify $U$ as a subgroup of $\wt{G}$ via the canonical splitting given by $e_\alpha(x) \mapsto \wt{e}_\alpha(x)$.

Let $(\pi, V_\pi) \in \Irr(\wt{G})$ be a genuine irreducible representation of $\wt{G}$. A functional $\lambda: V_\pi \to \C$ is called a $\psi$-Whittaker functional if
$$\lambda( \pi(u) v) = \psi(u) \cdot \lambda(v)$$
for all $u\in U$ and $v \in V_\pi$. Denote by $\text{Wh}_\psi(\pi)$ the space of $\psi$-Whittaker functionals for $\pi$.
It follows from the work of Patel \cite{Pate}, which generalizes \cite{MW1} by Moeglin-Waldspurger, that
$$\dim \Wh_\psi(\pi) < \infty  \text{ for every } \pi \in \Irr(\wt{G}).$$

\subsection{Local coefficients matrix} \label{SS:lcm}
Let
$$\mbf{P}=\mbf{M} \mbf{N} \subset \mbf{G}$$ be a parabolic subgroup of $\mbf{G}$ associated with $\theta \subset \Delta$. Let $(\sigma, V_\sigma) \in \Irr(\wt{M})$, where $V_\sigma$ is a space of realization of $\sigma$.
% Let $\Wh_\psi(V_\sigma)$ be the space of $\psi$-Whittaker functionals of $V_\sigma$.

Consider the  normalized induced representation
$$I_{\wt{P}}^{\wt{G}} (\sigma)=\Ind_{\wt{P}}^{\wt{G}}(\sigma)$$
of $\wt{G}$. Let $\mbf{P}'= \mbf{M}' \mbf{N}'$ be the parabolic subgroup corresponding to $\theta'\subset \Delta$.
Define
$$W^{\theta, \theta}=\set{ \w \in W:  \w(\theta) =\theta' }.$$
Call $\mbf{P}$ and $\mbf{P}'$ associated if $W^{\theta, \theta} \ne \emptyset$. Let $\w \in W^{\theta, \theta'}$. Then $\w \mbf{M} \w^{-1} = \mbf{M}'$. Let $w \in G$ the natural representative given in \S \ref{S:top-c}.  We obtain a representation ${}^w\sigma$ of $\wt{M}'$  given by
$${}^w\sigma(m')(v):= \sigma(w^{-1} m w) (v)$$
for any $m'\in \wt{M}'$ and $v\in V_\sigma$. In particular, for the  underlying vector space, we have $V_{{}^w\sigma}=V_\sigma$.

%Let $\wt{w} \in \wt{G}$ be the representative in the covering group.

Consider the intertwining operator
$$ T(w,\sigma): \ I_{\wt{P}}^{\wt{G}}(\sigma) \to  I_{\wt{P}'}^{\wt{G}}( {}^w \sigma) $$
given by the meromorphic continuation of the integral
\begin{equation} \label{T(w)}
T(w, \sigma)(f)(g)=\int_{N_w} f(\wt{w}^{-1} n g) dn
\end{equation}
where $N_w=U\cap (w N^- w^{-1})$ with $N^-$ the unipotent opposite $N$. Let
$$\mfr{a}_{M, \C}^* = X(M) \otimes \C.$$
There is a natural action of $\chi \in \mfr{a}_{M, \C}^*$ on $\sigma$ by twisting $\chi\otimes \sigma$. We write $\mca{O}_{\sigma, \C}$ for the orbit $\mfr{a}_{M, \C}^* \cdot \sigma$, which has a complex analytic structure. It is known that $T(w, -)$ defines a meromorphic function on  $\mca{O}_{\sigma,\C}$, see \cite[Page 323, Theorem 2.2.2]{Sha2} and \cite[Thoereme IV.1.1]{Wal2}. In most part of this paper, we only consider the generic case where $T(w, \sigma)$ is holomorphic at $\sigma$. (However, in Corollary \ref{irrres} we also consider the case of singular $T(w, \sigma)$ pertaining to reducibility of representation.)

By dualizing, the intertwining operator $T(w,\sigma)$, whenever holomorphic at $\sigma$, gives a linear map between finite dimensional vector spaces:
$$\begin{tikzcd}
T(w,\sigma)^*: \  \Wh_\psi(I_{\wt{P}'}^{\wt{G}}({}^w\sigma)) \ar[r]  & \Wh_\psi( I_{\wt{P}}^{\wt{G}}(\sigma)).
\end{tikzcd}$$

Let $\psi_M$ be the restriction of $\psi$ to the unipotent radical $U_M$ of the Borel subgroup $TU_M$ of $M$. Here $\psi_M$ and $w$ are compatible (in the sense of \cite[Page 51]{Sha4}). We have a natural isomorphism (see \cite{CS, Rod1}, \cite[\S 3]{Sha2} or \cite[\S 3.3]{Sha4})
$$J(\sigma): \Wh_{\psi_M}(\sigma) \to \Wh_\psi(I_{\wt{P}}^{\wt{G}}(\sigma))$$
given by
$$J(\sigma)(\lambda)(f) = \int_{N'} \lambda( f(\wt{w}_0^{-1} n')  ) \cdot \psi(n') dn',$$
where $N'=w_0^{-1} N^- w_0^{-1}$ with $w_0^{-1} = w_l \cdot w_{l,M}$. Here $w_l$ (resp. $w_{l,M}$) is the longest
Weyl element in the Weyl group $W(T, G)$ (resp. $W(T, M)$).  Note that the injectivity of $J(\sigma)$ follows from \cite[Proposition 3.1]{Sha2}. Rodier's heredity asserting  $\dim \Wh_{\psi_M}(\sigma) = \dim \Wh_\psi(I_{\wt{P}}^{\wt{G}}(\sigma))$, which was proved in \cite{Ban1} for covers of $\GL_r$ but extends naturally to general $\wt{G}$, enforces $J(\sigma)$ to be an isomorphism of vector spaces.

Similarly, one has
$$J({}^w \sigma): \Wh_{\psi_{M'}}({}^w\sigma) \to \Wh_\psi(I_{\wt{P}'}^{\wt{G}}(\sigma)).$$
At the same time, since $V_{{}^w\sigma} =V_\sigma$, we also have a canonical identification
$$\iota: \Wh_{\psi_{M}} (\sigma)=\Wh_{\psi_{M'}} ({}^w\sigma) .$$
To verify this identification, we note that $l \in \Wh_{\psi_{M}} (\sigma)$ if and only if $l(\sigma(u)(v)) = \psi_M(u) \cdot v \text{ for all } u\in U_M, v\in V_\sigma$. Now for every $u' \in U_{M'}$ one has $w^{-1} u' w \in U_M$ and
$$\psi_{M'}(u') = \psi(u') = \psi(w^{-1} u' w) = \psi_M(w^{-1} u' w),$$
where the second equality follows from the fact that $\psi(e_\alpha(x))= \psi(x)$ for all simple root $\alpha$, and also $w^{-1} e_\alpha(x) w = e_{w^{-1}(\alpha)}(x)$ if $w^{-1}(\alpha) \in \Delta$ (see \cite[Proposition 9.3.5]{SprB}). Thus we get
$$l({}^w\sigma (u') (v)) = l(\sigma(w^{-1} u' w)(v)) = \psi_M(w^{-1} u' w)\cdot v = \psi_{M'}(u')\cdot v$$
for all $u'\in U_{M'}$ and $v\in V_{{}^w\sigma} = V_\sigma$. This shows that $l \in \Wh_{\psi_{M'}}( {}^w \sigma)$ and gives the canonical $\iota$.

One thus has a natural isomorphism of vector spaces
$$\begin{tikzcd}
\mca{C}=J({}^w\sigma) \circ \iota \circ J(\sigma)^{-1} : \  \Wh_\psi( I_{\wt{P}}^{\wt{G}}(\sigma))  \ar[r]  & \Wh_\psi(I_{\wt{P}'}^{\wt{G}}({}^w\sigma)).
\end{tikzcd}$$
We therefore obtain an endomorphism
\begin{equation} \label{End-T}
\begin{tikzcd}
\mca{T}(w,\sigma)^*= T(w,\sigma)^* \circ \mca{C} : \  \Wh_{\psi}( I_{\wt{P}}^{\wt{G}} (\sigma) ) \ar[r]  & \Wh_{\psi}( I_{\wt{P}}^{\wt{G}}(\sigma) )
\end{tikzcd}
\end{equation}
of the finite dimensional vector space $\Wh_{\psi}( I_{\wt{P}}^{\wt{G}} (\sigma) )  \simeq \Wh_{\psi_M}( \sigma )$.

\begin{dfn}[{\cite[\S 4]{Szp6}}] \label{D:LCM}
A local coefficients matrix associated to $(\wt{P}, w, \sigma)$ is the matrix $\mca{M}_\mfr{B}(w, \sigma)$ representing $\mca{T}(w, \sigma)^*$ with respect to an ordered basis $\mfr{B} \subset \Wh_{\psi}( I_{\wt{P}}^{\wt{G}} (\sigma) )$.
\end{dfn}

\begin{rmk} \label{R:LCM}
An equivalent way of obtaining $\mca{M}_\mfr{B}(w, \sigma)$ is as follows. The ordered basis $\mfr{B}$ gives rise to an ordered basis $\mca{C}(\mfr{B}) \subset \Wh_\psi (I_{\wt{P}'}^{\wt{G}}({}^w\sigma))$. Then $\mca{M}_\mfr{B}(w, \sigma)$ is just the matrix representing $T(w,\sigma)^*$ with respect to the two ordered bases $\mca{C}(\mfr{B})$ and $\mfr{B}$.
\end{rmk}

If $n=1$, i.e. $\wt{G}=G$, then $\mca{T}(w, \sigma)^*$ is scalar valued and is the \emph{reciprocal} of the local coefficients studied by the second author for linear algebraic groups. It is one of the starting point of the Langlands-Shahidi theory of $L$-functions for linear algebraic groups. For general covering $\wt{G}$, the characteristic polynomial
$$\bigwedge^{\rm top} (X\cdot \text{id}- \mca{T}(w, \sigma)^*)=\det(X \cdot \text{id}- \mca{M}_\mfr{B}(w, \sigma))$$
 of $\mca{T}(w, \sigma)^*$ can be computed by considering any local coefficient matrix $\mca{M}_\mfr{B}(w, \sigma)$, and it is independent of chosen basis $\mfr{B}$. As mentioned in the Introduction, our goal in this paper is to investigate the two invariants (i.e. as coefficients of the above characteristic polynomial) $\Tr(\mca{M}_\mfr{B}(w, \sigma))$ and $\det (\mca{M}_\mfr{B}(w, \sigma))$.

%%%
\subsection{Genuine principal series representation} \label{SS:gps}
Considering the Borel subgroup $\mbf{B}=\mbf{T} \mbf{U}$, since $U$ splits canonically in $\wt{G}$, we have $\wt{B}=\wt{T}\ltimes U$. The covering torus $\wt{T}$ is a Heisenberg group with center $Z(\wt{T})$. Moreover,
$$Z(\wt{T})=\phi^{-1}(\text{Im}(i_{Q,n})),$$
where
$$i_{Q,n}: Y_{Q,n}\otimes F^\times \to T$$
is the isogeny induced from the embedding $Y_{Q,n} \subset Y$.

Let $\chi \in \Irr (Z(\wt{T}))$ be a genuine character of $Z(\wt{T})$. Let $\wt{A} \subset \wt{T}$ be a maximal abelian subgroup. By the Stone-von Neumann theorem (see \cite[Theorem 3.1]{We1}), the construction
$$\chi \mapsto i(\chi):=\text{Ind}_{\wt{A}}^{\wt{T}}(\chi')$$
gives a bijection between the isomorphism classes of genuine representations of $Z(\wt{T})$ and $\wt{T}$, where $\chi'$ is any extension of $\chi$ to $\wt{A}$. In fact, the isomorphism class is independent of the choice of $\wt{A}$ as well. As a consequence of the non-degeneracy of the commutator map on $\wt{T}/Z(\wt{T}) \times \wt{T}/Z(\wt{T})$, there is a Lagrangian decomposition of $\wt{T}/Z(\wt{T})$, see \cite[\S 2]{We5}. In particular, $[\wt{T}: Z(\wt{T})] \in \N$ is a square and one has
$$\dim i(\chi) =[\wt{T}: \wt{A}] = \sqrt{[\wt{T}: Z(\wt{T})]}.$$
We note that the group $W$ does not act on $i(\chi)$, but only on its isomorphism class, see \cite[Page 310]{Mc1}. On the other hand, we have a well-defined action of $W$ on $\chi$ given by
$$({}^\w \chi)(\wt{t}):= \chi(w^{-1} \wt{t} w).$$
The two representations ${}^wi(\chi)$ and $i({}^\w \chi)$ are isomorphic since they have the same central character. If we fix an isomorphism
\begin{equation} \label{r-w}
r_w: {}^w i(\chi) \to i({}^\w \chi),
\end{equation}
then it induces an isomorphism also denoted by
$$r_w: I({}^w i(\chi)) \to I(i({}^\w \chi)).$$
For simplicity of notation, we may denote
$$I(\chi): =I(i(\chi))$$
for any $\chi \in \Irr (Z(\wt{T}))$. Every element in $I(\chi)$ is viewed as an $i(\chi)$-valued function.

\subsection{Parametrization of $\Wh_\psi(I(\chi))$} \label{SS:para}
We would like to have a concrete parametrization of the space of $\Wh_\psi (I(\chi))$, which will enable a natural basis $\mfr{B}$. The parametrization arises essentially from composing the Jacquet integral with functionals of $i(\chi)$.

Let $\Ftn(i(\chi))$ be the vector space of  functions $\cc$ on $\wt{T}$  satisfying
$$\cc(\wt{t} \cdot \wt{z}) =  \cc(\wt{t}) \cdot \chi(\wt{z}), \quad \wt{t} \in \wt{T} \text{ and } \wt{z} \in \wt{A}.$$
The support of $\cc \in \Ftn(i(\chi))$ is a disjoint union of cosets in $\wt{T}/\wt{A}$.
For every $\gamma \in \wt{T}$, let $\cc_\gamma \in \Ftn(i(\chi))$ be the unique element satisfying
$$\text{supp}(\cc_{\gamma})=\gamma \cdot \wt{A} \text{ and } \cc_{\gamma}(\gamma)=1.$$
Clearly, $\cc_{\gamma \cdot a} = \chi(a)^{-1} \cdot \cc_\gamma$ for every $a\in \wt{A}$. If $\set{\gamma_i}\subset \wt{T}$ is a chosen set of representatives of $\wt{T}/\wt{A}$, then $\set{\cc_{\gamma_i}}$ forms a basis for $\Ftn(i(\chi))$. Let $i(\chi)^\vee$ be the vector space of functionals of $i(\chi)$, which affords the contragredient representation of $i(\chi)$.
 The set $\set{\gamma_i}$ gives rise to linear functionals $l_{\gamma_i} \in i(\chi)^\vee$ such that $l_{\gamma_i}(\phi_{\gamma_j})=\delta_{ij}$, where $\phi_{\gamma_j}\in i(\chi)$ is the unique element such that
$$\text{supp}(\phi_{\gamma_j})=\wt{A}\cdot \gamma_j^{-1} \text{ and  } \phi_{\gamma_j}(\gamma_j^{-1})=1.$$
See also \cite{CO}. It is easy to see that for every $\gamma\in \wt{T}$ and $a\in \wt{A}$, one has
$$\phi_{\gamma a}= \chi(a)\cdot \phi_\gamma, \quad  l_{\gamma a} = \chi(a)^{-1} \cdot l_\gamma.$$
Moreover, there is a natural isomorphism of vector spaces
$$\Ftn(i(\chi)) \simeq i(\chi)^\vee$$
 given by
$$ \cc  \mapsto   l_\cc:= \sum_{\gamma_i \in \wt{T}/\wt{A}} \cc(\gamma_i) \cdot l_{\gamma_i}.
$$
It can be checked easily that this isomorphism does not depend on the choice of representatives for $\wt{T}/\wt{A}$.

Furthermore, there is an isomorphism between $i(\chi)^\vee$ and the space $\Wh_\psi (I(\chi))$ of $\psi$-Whittaker functionals  on $I(\chi)$  given by
$$l \mapsto \lambda_l$$
with
$$\lambda_l:  I(\chi) \to \C, \quad f \mapsto l \circ J(f) \text{ where } J(f)=\left( \int_{U} f(\wt{w}_G^{-1}u) \psi(u)^{-1} du \right) \in i(\chi).$$
Here $f\in I(\chi)$ is an $i(\chi)$-valued function on $\wt{G}$. For any $\cc\in \Ftn(i(\chi))$, write $\lambda_\cc \in \Wh_\psi(I(\chi))$ for the $\psi$-Whittaker functional of $I(\chi)$ associated to $l_\cc$. Therefore, $\cc \mapsto \lambda_\cc$ gives an isomorphism between $\Ftn(i(\chi))$ and $\text{Wh}_\psi(I(\chi))$. For $\gamma \in \wt{T}$, we will write
$$\lambda_\gamma:=\lambda_{\cc_\gamma}.$$
To avoid confusion, we may write $\lambda^{\chi}$ instead of $\lambda \in \text{Wh}_\psi(I(\chi))$ to emphasize the underlying representation $I(\chi)$ involved.  It is clear that
\begin{equation} \label{psWd}
\dim \Wh_\psi(I(\chi)) = \dim i(\chi)^\vee =  \sqrt{[\wt{T}: Z(\wt{T})]}.
\end{equation}

In fact, we believe that the following holds.

\begin{conj} \label{C:uniUB}
Let $\wt{G}$ be an $n$-fold Brylinski-Deligne cover of a connected reductive group $G$. Then
$$\dim \Wh_\psi( \pi ) \le \sqrt{[\wt{T}: Z(\wt{T})]}$$
for every $\pi \in \Irr (\wt{G})$.
\end{conj}
The above conjecture is motivated from the fact that for linear algebraic group, the generic element in an $L$-packet has the minimal Bernstein support (see \cite{Han18}). This we expect to be true for covering groups as well, and that the full principal series has the highest dimension for the Whittaker space. We also note that \cite[Corollary 2.8]{GSS1} follows from the above conjecture.

\subsection{Unramified $I(\chi)$} \label{SS:unrep}
If $|n|_F=1$, let $K \subset G$ be the hyperspecial maximal compact subgroup generated by $\mbf{T}(O)$ and $e_\alpha(O)$ for all root $\alpha$. With our assumption that $\eta_n$ is trivial, the group $\wt{G}$ splits over $K$ (see \cite{GG}) and we fix such a splitting
$$s_K: K \into \wt{G}.$$
We call such $\wt{G}$ (together with the fixed splitting over $K$)  an unramified group. If no confusion arises, we will simply write 
$$K\subset \wt{G}$$
and omit $s_K$. Moreover, a representation $\pi \in \Irr(\wt{G})$ for an unramified $\wt{G}$ is called unramified if $\pi^K \ne 0$, in which case we have $\dim \pi^K =1$. In particular, $I(\chi)$ is unramified if and only if $\chi$ is unramified, i.e., $\chi$ is trivial on $Z(\wt{T}) \cap K$.

If $\wt{G}$ is an unramified group, then
$$\wt{A}:=Z(\wt{T})\cdot (K\cap T)$$
is a maximal abelian subgroup of $\wt{T}$ with
$$\msc{X}_{Q,n} \simeq \wt{T}/\wt{A}.$$
In this case,
$$\dim \Wh_\psi(I(\chi)) = \dim i(\chi) = \val{ \msc{X}_{Q,n} }.$$
Since
$$\wt{A}/(T\cap K) \simeq Z(\wt{T})/Z(\wt{T}) \cap K,$$
which we denote by $\wt{Y}_{Q,n}$, there is a natural abelian extension
\begin{equation} \label{Ext1}
\begin{tikzcd}
\bbmu_n \ar[r, hook] & \wt{Y}_{Q,n} \ar[r, two heads, "\varphi"] & Y_{Q,n}.
\end{tikzcd}
\end{equation}
Unramified genuine characters of both $Z(\wt{T})$ and $\wt{A}$ correspond to a genuine character of $\wt{Y}_{Q,n}$. In particular, any unramified $\chi$ of $Z(\wt{T})$ has a canonical unramified extension to $\wt{A}$, which we also denote by $\chi$. 
Let $f_0 \in I(\chi)$ be the normalized unramified vector such that $f_0(1_{\wt{G}}) \in i(\chi)$ is the unramified vector with $f_0(1_{\wt{G}})(1_{\wt{T}})=1$.

\begin{dfn}
The $i(\chi)$-valued unramified Whittaker function $\mca{W}$ is given by
$$\mca{W}(g)=J\left( I(\chi)(g)(f_0) \right) \text{ for every } g\in \wt{G},$$
while  $\mca{W}_\cc(g):=l_\cc \left( \mca{W}(g) \right)$ is a scalar-valued unramified Whittaker function for every $\cc \in \Ftn(i(\chi))$.
\end{dfn}
For convenience, for $\gamma\in \wt{T}$, we may also denote
$$\mca{W}_\gamma:= \mca{W}_{\cc_\gamma}$$
for the scalar valued Whittaker function associated to $\cc_\gamma$.

\subsection{Scattering matrix} \label{SS:Sca}
Let $\chi$ be a genuine character of $Z(\wt{T})$. Let 
$$\mfr{R}=\set{\gamma_i: i\in I} \subset \wt{T}$$
 be an ordered set of representatives of the quotient $\wt{T}/\wt{A}$. We denote by
$$\mfr{B}_\chi=\set{\lambda_{\gamma_i}^\chi: \gamma_i\in \mfr{R}  }  \subset \Wh_\psi (I(\chi))$$
the ordered basis for the Whittaker functionals arising from $\set{l_{\gamma_i}  } \subset i(\chi)^\vee$. Note that the restriction of the dual representation $i(\chi)^\vee$ of $\wt{T}$ to $\wt{A}$ is given by
$$i(\chi)^\vee = \bigoplus_{\gamma_ i \in \mfr{R}}  \C \cdot l_{\gamma_i},$$
an isotypic sum with the action of $\wt{A}$ on $\C \cdot l_{\gamma_i}$ given by $\chi^{-1}$. Thus, the basis $\mfr{B}_\chi$ is a natural choice.

Let
$$r_w: {}^wi(\chi) \to i({}^\w \chi)$$
be an isomorphism. We use $T(w, \chi; r_w)$ to denote the composite
$$T(w, \chi; r_w):=r_w \circ T(w, i(\chi)): I(i(\chi)) \to I({}^w i(\chi) ) \to I(i({}^\w \chi))$$
of the intertwining map $T(w, i(\chi))$ (see \eqref{T(w)}) with the isomorphism between $I({}^wi(\chi))$ and $I(i({}^\w\chi))$ induced from $r_w$ (see \eqref{r-w}). It induces a homomorphism
$$T(w, \chi; r_w)^* = T(w, i(\chi))^* \circ r_w^* : \Wh_\psi(I({}^\w \chi)) \to \Wh_\psi (I(\chi)),$$
where
$$r_w^*:  \Wh_\psi(I({}^\w \chi)) \to \Wh_\psi (I{}^w(\chi))$$
is the isomorphism induced from $r_w$. The set $\mfr{R}$ also gives rise to an ordered basis
$$\mfr{B}_{ {}^\w\chi }=\set{\lambda_{\gamma_i}^{{}^\w\chi}: \gamma_i\in \mfr{R}  }  \subset \Wh_\psi (I({}^\w\chi)).$$
That both $\mfr{B}_\chi$ and $\mfr{B}_{ {}^\w \chi}$ are defined using the same $\mfr{R}$ motivates the consideration of $T(w, \chi; r_w)$ instead of $T(w, \chi)$. In fact, we have the following

\begin{dfn} \label{D:Sca}
A scattering matrix $\mca{S}_\mfr{R}(w, i(\chi); r_w)$ is the matrix representing the map $T(w, \chi; r_w)^*$ with respect to the ordered bases $\mfr{B}_{ {}^\w\chi }$ and $\mfr{B}_\chi$, both of which depend on $\mfr{R}$.
\end{dfn}
The matrix $\mca{S}_\mfr{R}(w, i(\chi); r_w)$ was first investigated in \cite{KP} and then studied extensively in \cite{KP, Pat, Suz1, Suz2, Suz3, Suz4, Mc2} along the same line. In another direction, we refer to \cite{BBB, BBBF} and reference therein for the recent work investigating such matrix or its analogue (coined as R-matrices there) in the representation theory of quantum groups.

The natural isomorphism
$$ \mca{C}:  \Wh_\psi(I(\chi)) \to \Wh_\psi( I({}^wi(\chi))$$
from \S \ref{SS:lcm} gives rise to an isomorphism
\begin{equation} \label{E:C}
\mca{C}_{r_w}:=(r_w^{-1})^* \circ \mca{C}:  \Wh_\psi(I(\chi)) \to \Wh_\psi( I(i({}^\w\chi)).
\end{equation}
We denote by
$$\mca{C}( \mfr{B}_{ {}^\w \chi}, \mfr{B}_\chi; r_w)$$
the matrix representing $\mca{C}_{r_w}$ with respect to the two ordered bases  $\mfr{B}_\chi$ and $\mfr{B}_{ {}^\w\chi }$.
We have the following commutative diagram
\begin{equation} \label{Strat}
\begin{tikzcd}
\Wh_\psi(I( i({}^\w \chi) ))  \ar[r, "{r_w^*}"]  \ar[rr, bend left=25, "{T(w, \chi; r_w)^* }"]     &  \Wh_\psi(I( {}^wi(\chi) ))   \ar[r, "{ T(w, i(\chi))^* }"]  &   \Wh_\psi(I( i(\chi) ))   \\
&    \Wh_\psi( I(i(\chi)))  \ar[lu, "{\mca{C}_{r_w}}"]   \ar[u, "{\mca{C}}"]  \ar[ru, "{ \mca{T}(w, i(\chi))^* }"'] ,
\end{tikzcd}
\end{equation}
where $\mca{T}(w, i(\chi))^*$ is given in \eqref{End-T}. It immediately gives:

\begin{lm} \label{L:comp}
Retain the above notations, we have
$$\mca{M}_{\mfr{B}_\chi}(w, i(\chi))=  \mca{S}_\mfr{R}(w, i(\chi); r_w) \circ \mca{C}( \mfr{B}_{ {}^\w \chi}, \mfr{B}_\chi; r_w).$$
\end{lm}
%\begin{proof}
%This follows from the identity
%$$\mca{T}(w, i(\chi))^* = T(w, i(\chi))^* \circ \mca{C} = T(w, \chi; r_w)^* \circ \mca{C}_{r_w}$$
%and the definition of the matrices involved.
%\end{proof}

We would like to discuss the matrix $\mca{S}_\mfr{R}(w, i(\chi); r_w)$ in more detail. Following the notation in \cite{KP}, we write
$$\mca{S}_\mfr{R}(w, i(\chi); r_w)=[\tau(w, \chi, \gamma, \gamma')]_{\gamma, \gamma' \in \mfr{R}}$$
such that
$$T(w, \chi; r_w)^*(\lambda_{\gamma}^{^\w \chi}) = \sum_{\gamma'\in \mfr{R}} \tau(w, \chi, \gamma, \gamma') \cdot \lambda_{\gamma'}^{\chi}.$$
In particular, $\tau(w, \chi, -, \gamma')$ denotes the ``rows" of the matrix $\mca{S}_\mfr{R}(w, i(\chi); r_w)$ and $\tau(w, \chi, \gamma, -)$ the ``columns". In fact, we will consider slightly a more general pseudo-matrix
$$[\tau(w, \chi, \gamma, \gamma')]_{\gamma \in \mfr{R}, \gamma' \in \mfr{R}'},$$
where $\mfr{R}, \mfr{R}' \subset \wt{T}$ are both ordered sets of representatives of $\wt{T}/\wt{A}$, such that  for every $\gamma \in \mfr{R}$,
$$T(w, \chi; r_w)^*(\lambda_{\gamma}^{^\w \chi}) = \sum_{\gamma'\in \mfr{R}'} \tau(w, \chi, \gamma, \gamma') \cdot \lambda_{\gamma'}^{\chi}.$$

In general, there is no canonical choice of $r_w$, and thus the scattering matrix depends sensitively on such a choice. In particular, different choices of $r_w$ may give non-conjugate scattering matrices. Nevertheless, in the unramified setting, we have a preferred choice given as follows. For this, we resume the setting of \S \ref{SS:unrep}. In particular, $\chi$ is unramified and that $\psi$ is of conductor $O_F$. With the choice of $\wt{A}$ and canonical unramified extension of $\chi$ to $\wt{A}$ (and thus the realization $\Ind_{\wt{A}}^{\wt{T}} \chi$ for $i(\chi)$) as in \S \ref{SS:unrep}, we have a natural isomorphism
\begin{equation} \label{r-w}
r_w^{\rm un}: {}^w i(\chi) \to i({}^\w \chi) \text{ given by } f \mapsto r_w^{\rm un}(f)(t):= f(w^{-1} t w),
\end{equation}
where $w\in K \subset \wt{G}$. Here $r_w^{\rm un}$ sends the normalized unramified vector of ${}^w i(\chi)$ to that of $i({}^\w \chi)$.

We will describe the matrix $[\tau(w, \chi, \gamma, \gamma')]$ representing $T(w, \chi; r_w^{\rm un})^*$. A remark is in order on the dependence of the matrix $[\tau( w, \chi,  \gamma, \gamma')]$ on the measures. Note that there are a priori three choices of measures: a measure $du$ of $N_w$ in the definition of $T(w, i(\chi))$ in \eqref{T(w)}, a measure $du'$ of $U$ in the definition of $\lambda_{\gamma_i}^\chi$, and a measure $du''$ of $U$ in the definition of $\lambda_{\gamma_i}^{{}^\w \chi }$.
We conveniently assume that $du' = du''$ and thus the local coefficients matrix $[\tau( w, \chi,  \gamma, \gamma')]$ depends only on the measure $du$ of $N_w$.

We first show that $T(w, \chi; r_w^{\rm un})^*$ satisfies the cocycle relation. For this purpose, we denote $r_{w, \chi}^{\rm un}:  {}^w i(\chi) \to  i({}^\w \chi)$ for the isomorphism $r_w^{\rm un}$ given above.

\begin{lm} \label{L:Tcr}
Let $\w, \w' \in W$ be such that $l(\w' \w) = l(\w') + l(\w)$. Then one has
$$T(w' w, \chi; r_{w' w, \chi}^{\rm un}) =T(w', {}^\w \chi; r_{w', {}^\w \chi}^{\rm un})  \circ T(w, \chi; r_{w, \chi}^{\rm un}).$$
\end{lm}
\begin{proof}
The claimed result follows from the following commutative diagram:
$$\begin{tikzcd}
I(i(\chi))  \ar[r, "{T(w, i(\chi))}"]   \ar[rd, "{T(w'w, i(\chi))}"']   &  I({}^w i(\chi))     \ar[r, "{r_{w, \chi}^{\rm un}}"]   \ar[d, "{T(w', {}^w i(\chi))}"]   & I(i({}^\w \chi))   \ar[d, "{T(w',i({}^\w \chi))}"]    \\
                          &  I({}^{w' w} i(\chi))    \ar[r, "{r_{w, \chi}^{\rm un}}"]  \ar[rd, "{r_{w' w, i(\chi)}^{\rm un}}"']                 & I({}^{w'}i({}^\w \chi))   \ar[d, "{r_{w', {}^\w \chi}^{\rm un}}"]   \\
                          & &   I(i({}^{\w' \w} \chi)).
\end{tikzcd} $$
The commutativity of the upper triangle follows as in the linear case, i.e., the cocycle relation for $T(w, \sigma)$. The lower triangle commutes by the defining formula of $r_{w, \chi}^{\rm un}$ in \eqref{r-w}. It is also easy to check that the square commutes by considering \eqref{T(w)} and \eqref{r-w}. This completes the proof since $T(w, \chi; r_{w, \chi}^{\rm un}) = r_{w, \chi}^{\rm un} \circ T(w, i(\chi))$ by definition.
\end{proof}

The matrix $[\tau(w, \chi, \gamma, \gamma')]_{\gamma \in \mfr{R}, \gamma' \in \mfr{R}'}$ representing $T(w, \chi; r_w^{\rm un})^*$ satisfies some immediate properties:
\begin{enumerate}
\item[$\bullet$] For $\w\in W$ and $\wt{z}, \wt{z}'\in \wt{A}$, the identity
\begin{equation} \label{SLCM1}
\tau(w, \chi, \gamma \cdot \wt{z}, \gamma' \cdot \wt{z}')=({}^{w}\chi)^{-1}(\wt{z}) \cdot \tau(w, \chi, \gamma, \gamma') \cdot \chi(\wt{z}')
\end{equation}
holds.
\item[$\bullet$] For $\w_1, \w_2 \in W$ such that $l(\w_2\w_1)=l(\w_2) + l(\w_1)$, one has
\begin{equation} \label{SLCM2}
\tau( w_2 w_1, \chi,  \gamma, \gamma')=\sum_{\gamma''\in \wt{T}/\wt{A}} \tau(w_2, {}^{w_1}\chi, \gamma, \gamma'') \cdot \tau(w_1, \chi, \gamma'', \gamma'),
\end{equation}
where the sum is well-defined, independent of any chosen representative $\gamma''$. The above equality follows from Lemma \ref{L:Tcr}, and is also referred to as the cocycle relation.
\end{enumerate}
In view of the cocycle relation in (\ref{SLCM2}), the understanding  of $\tau(w, \chi, \gamma, \gamma')$ in principle is reduced to the case where $w=w_\alpha$ for some $\alpha\in \Delta$.

To explicate the entries of $[\tau(w_\alpha, \chi, \gamma, \gamma')]$, we will take the Haar measure $du$ on $F$ such that $du(O_F)=1$, and thus $du(O_F^\times)=1-1/q$. The measure of each unipotent subgroup $U_\alpha$ involved in the intertwining operator $T(w, \chi; r_w^{\rm un})$ and in the definition of Jacquet integral $J(i(\chi))$ (see \S \ref{SS:lcm}) is the push-out of the measure $du$ via $e_\alpha: F \to U_\alpha$.

The Gauss sum is given by
$$G_\psi(a, b)=\int_{O^\times_F} (u, \varpi)_n^a \cdot \psi(\varpi^b u) du, \quad a, b\in \Z.$$
It is known that
\begin{equation*}
G_\psi(a, b)=
\begin{cases}
0 & \text{ if } b<-1, \\
1-1/q & \text{ if } n| a, b\ge 0,\\
0 &\text{ if } n\nmid a, b\ge 0, \\
-1/q &\text{ if } n|a, b=-1,\\
G_\psi(a, -1) \text{ with } |G_\psi(a,-1)|=q^{-1/2} &\text{ if } n\nmid a, b=-1.
\end{cases}
\end{equation*}
Recalling $\varepsilon=(-1,\varpi)_n \in \bbmu_n$, one has
$$\overline{G_\psi(a, b)}=\varepsilon^a \cdot G_\psi(-a, b).$$
For any $k\in \Z$, we write
$$\mathbf{g}_{\psi}(k):=G_{\psi}(k, -1).$$
The following was first proved by Kazhdan-Patterson \cite{KP} for covers of $\wt{\GL}_r$ and later generalized to general $\wt{G}$ by McNamara \cite{Mc2} (with some refinement from \cite{Ga2}).

\begin{thm} \label{T:Sca}
Let $\gamma=\s_{y_1}$ and $\gamma'=\s_y$ with $y_1, y\in Y$. Then we can write $\tau(w_\alpha, \chi, \gamma, \gamma')=\tau^1(w_\alpha, \chi, \gamma, \gamma') + \tau^2(w_\alpha, \chi, \gamma, \gamma')$ with the following properties:
\begin{enumerate}
\item[$\bullet$] $\tau^i(w_\alpha, \chi,\gamma \cdot \wt{z}, \gamma' \cdot \wt{z}')=({}^{\w_\alpha} \chi)^{-1}(\wt{z}) \cdot \tau^i(w_\alpha, \chi, \gamma, \gamma') \cdot \chi(\wt{z}'), \quad \wt{z}, \wt{z}'\in \wt{A}$;
\item[$\bullet$] $\tau^1(w_\alpha, \chi, \gamma, \gamma')=0$  unless  $y_1 \equiv y \mod Y_{Q,n}$;
\item[$\bullet$] $\tau^2(w_\alpha, \chi, \gamma, \gamma')=0$   unless $y_1 \equiv \w_\alpha[y] \mod Y_{Q,n}$.
\end{enumerate}
Moreover,
\begin{enumerate}
\item[$\bullet$] if $y_1= y$, then
$$\tau^1(w_\alpha, \chi, \gamma, \gamma')=(1-q^{-1}) \frac{\chi (\wt{h}_\alpha(\varpi^{n_\alpha}))^{k_{y,\alpha}}}{1-\chi (\wt{h}_\alpha(\varpi^{n_\alpha}))}, \text{ where } k_{y,\alpha}=\ceil{\frac{\angb{y}{\alpha}}{n_\alpha}};$$
\item[$\bullet$] if $y_1=\w_\alpha[y]$, then
$$\tau^2(w_\alpha, \chi, \gamma, \gamma') = \varepsilon^{ \angb{y_\rho}{\alpha} \cdot D(y, \alpha^\vee) } \cdot \g(\angb{y_\rho}{\alpha}Q(\alpha^\vee)).$$
\end{enumerate}
\end{thm}

Now, for $\w=\w_{\alpha_l} ... \w_{\alpha_1} \in W$ written in a minimal decomposition with $\alpha_i\in \Delta$, let
$$\Delta_\w=\set{\alpha_i: 1\le i\le l}.$$
It is known that $\Delta_\w \subset \Delta$ is independent of the minimal decomposition chosen (see \cite[Page 12]{Bou}). Denote by $W(\Delta_\w) \subset W$ the parabolic subgroup generated by $\Delta_\w$. In particular, $W=W(\Delta_{\w_G})$.

Take an order set of representatives of $\wt{T}/\wt{A}$ as
$$\mfr{R}=\set{\s_{y_i}:  1\le i \le \val{\msc{X}_{Q,n}}},$$
where $\set{y_i} \subset Y$ is an ordered set of representatives of $\msc{X}_{Q,n}$.
Let $W'$ be a subgroup such that
$$W(\Delta_\w) \subset W' \subset W.$$
Define an equivalence relation on $\mfr{R}$ by requiring that $y_i$ is equivalent to $y_j$ if their images $\hat{y_i}$ and $\hat{y_j}$ in $\msc{X}_{Q,n}$ lies in the same $W'$-orbit; that is, if
$$y_i \equiv \w[y_j] \mod Y_{Q,n} \text{ for some } \w\in W'.$$
Then we could reorder elements in  $\mfr{R}$ such that $\mfr{R}$ is decomposed into an ordered equivalent classes:
\begin{equation} \label{R:dec}
\mfr{R}= \bigsqcup_k \mfr{R}_k,
\end{equation}
and there are $\val{\mca{O}_{W'}(\msc{X}_{Q,n})}$-many equivalence classes. In particular, if $W'$ acts transitively on $\msc{X}_{Q,n}$, then $k=1$ with $\mfr{R}_1 =\mfr{R}$.

In general, by abuse of notation, we denote
$$\mca{S}_{\mfr{R}_k}(w, i(\chi); r_w^{\rm un}):=[\tau(w, \chi, \gamma, \gamma')]_{\gamma, \gamma'\in \mfr{R}_k}.$$

\begin{cor} \label{C:diag}
With notations as above, we have
$$\mca{S}_{\mfr{R}}(w, i(\chi); r_w^{\rm un}) = \bigoplus_k  \mca{S}_{\mfr{R}_k}(w, i(\chi); r_w^{\rm un}),$$
where the right hand side denotes the diagonal block matrix with each block arising from $\mfr{R}_k$.
\end{cor}
\begin{proof}
The decomposition was already observed in \cite[\S 3.2]{Suz2}. More precisely, the cocycle relation \eqref{SLCM2} coupled with Theorem \ref{T:Sca} shows that  $\tau(w, \chi, \s_y, \s_z)=0$, if $\hat{y}$ and $\hat{z}$ lie in different $W(\Delta_\w)$-orbits in $\msc{X}_{Q,n}$. Since $W' \supset W(\Delta_\w)$, the statement follows.
\end{proof}

Clearly, one can consider the case $W'=W$, where the decomposition \eqref{R:dec} is always independent of $\w$.

\section{Plancherel measure and gamma factor} \label{S:PG}

The goal of this section is to determine  completely
$$\det(\mca{M}_{\mfr{B}_\chi}(w_\alpha, i(\chi))  ) \text{ for } \alpha\in \Delta$$
 in the unramified setting, in terms of the Plancherel measure and gamma or metaplectic gamma factor, where $\mfr{B}_\chi$ depends on a chosen $\mfr{R}$ as in the preceding section. It is not visibly clear how to directly compute  $\det(\mca{M}_{\mfr{B}_\chi}(w_\alpha, i(\chi))  ) $ with respect to a chosen ordered set $\mfr{R}$.

However, by Lemma \ref{L:comp}, the local coefficients matrix  $\mca{M}_{\mfr{B}_\chi}(w_\alpha, i(\chi))$ is closely related to the scattering matrix $S_\mfr{R}(w, i(\chi); r_w^{\rm un})$ which, by Corollary \ref{C:diag} (applied to $W'=W(\Delta_{\w_\alpha})$), is essentially a diagonal block matrix. Moreover, since
$$W(\Delta_{\w_\alpha})= \set{\text{id}, \w_\alpha}$$ is an order two group, each equivalence class $\mfr{R}_i$ in \eqref{R:dec} has size 2 or 1, corresponding to the two sizes of $W(\Delta_{\w_\alpha})$-orbits in $\msc{X}_{Q,n}$. That is, $\mca{S}_{\mfr{R}_i} (w_\alpha, i(\chi); r_w^{\rm un})$ is either a two by two matrix or a scalar. This is the key observation in our strategy of computing $\det(\mca{M}_{\mfr{B}_\chi}(w_\alpha, i(\chi))  )$. Therefore, we will first analyze and compute $$\det(\mca{S}_\mfr{R}(w, i(\chi); r_w^{\rm un}))$$ and then
$$\det( \mca{C}(\mfr{B}_{{}^\w\chi}, \mfr{B}_\chi; r_w^{\rm un}) ),$$
 and combine the results. In fact, as a byproduct of analysing $\mca{C}(\mfr{B}_{{}^\w\chi}, \mfr{B}_\chi; r_w^{\rm un})$, we also give an explicit description of $\mca{M}_\mfr{B}(w_\alpha, i(\chi))$.

We remark that in our first step of computing $\det(S_\mfr{R}(w, i(\chi); r_w^{\rm un}))$, we assume only the necessary condition that $\bbmu_n \subset F^\times$. However, in a later stage we impose the stronger assumption that $\bbmu_{2n} \subset F^\times$ while dealing with $\det( \mca{C}(\mfr{B}_{{}^\w\chi}, \mfr{B}_\chi; r_w^{\rm un}) )$. Thus, our formula for $\det(\mca{M}_{\mfr{B}_\chi}(w_\alpha, i(\chi))  )$ is proven under this stronger condition. However, our description of the matrix $\mca{M}_\mfr{B}(w_\alpha, i(\chi))$ per se relies only on $\bbmu_n \subset F^\times$.

\subsection{Weil index} \label{wp}
We denote by $d_\psi x$ the Haar measure on $F$ which is self dual with respect to $\psi$ and we set
$$d^\times_\psi x=\frac{ d_\psi x}{\val{x}},$$
which is a Haar measure on $F^\times$. Let
\begin{equation} \label{weilint}
\omega(\psi)= \lim_{r \rightarrow \infty} \int_{\mfr{p}^{-r}} \psi(x^2) d_{\psi_2} x \in \bbmu_8
\end{equation}
be the unnormalized  Weil index defined in \cite{Weil}, where $\psi_a(x) = \psi(ax)$. Since the sequence in the above definition of $\omega(\psi)$ stabilizes, one has
\begin{equation} \label{weilinverse}
\omega^{-1}(\psi)= \omega(\psi_{-1}).
\end{equation}
Also, \eqref{weilint} implies that  $\omega(\psi_{a})=\omega(\psi)$ for all $a\in F^{\times 2}$. Equivalently, the map $a \mapsto \omega(\psi_a)$ is a well defined function on $F^\times/F^{\times 2}$. By \cite[\S 14]{Weil}, for all $a,b \in F^\times$
\begin{equation} \label{weil} (a,b)_2=\omega(\psi)\omega(\psi_{-a})\omega(\psi_{-b})\omega(\psi_{ab}).\end{equation}

Now, the normalized Weil index $\omega_\psi$ is defined as
\begin{equation} \label{norweildef}
\omega_\psi(a)=\frac{\omega(\psi_a)}{\omega(\psi)}.
\end{equation}
In particular,
$$\omega_\psi (F^{\times 2})=1.$$
The equality \eqref{weil} implies that the map $a \mapsto \omega_\psi(a)$ splits the Hilbert symbol, i.e.,
\begin{equation} \label{weilweil} 
\omega_\psi(ab)=\omega_\psi(a)\omega_\psi(b)(a,b)_2. 
\end{equation}
In particular, $\omega_\psi(a)^2=(a,a)_2=(a,-1)_2$. This implies that  $\omega_\psi(F^\times) \subseteq \mu_4$ and that if $-1 \in F^{\times 2}$, then  $\omega_\psi(F^\times) \subseteq \bbmu_2$. It follows from \eqref{weil} that
\begin{equation} \label{weilchangepsi}
\omega_{\psi_a}(x)=\eta_{a,(2)}(x) \cdot \omega_\psi(x).
\end{equation}
In particular, it follows from \eqref{weilinverse} that
\begin{equation} \label{weilchangepsiinv}
\omega_{\psi}(x)^{-1}=(-1,x)_2 \cdot \omega_\psi(x).
\end{equation}

\begin{lm} \label{for0mod4}
Assume that $p$ is odd and $\mfr{f}(\psi)$ is even. Fix $u\in O_F^\times$. Then $\omega_{\psi_u}(O_F^\times)=\omega_\psi(O_F^\times)=1$. If in addition $-1 \in F^{\times 2}$ and $u$ is not a square, then $\{\omega_\psi(\varpi), \omega_{\psi_u}(\varpi)\} =\bbmu_2$.
\end{lm}
\begin{proof}
The first assertion is well known, for a proof see \cite[Lemma 3.4]{SzpT}. To prove the second assertion, note that  $\omega_{\psi_u}(\varpi)=\omega_{\psi}(\varpi)(u,\varpi)_2$; since $(O_F^\times,O_F^\times)_2=1$, it follows from the non-degeneracy of the Hilbert symbol that $(u,\varpi)_2=-1$.
\end{proof}

\subsection{Tate $\gamma$-factor}
Let 
$$\uchi: F^\times \to \C^\times$$
be a character. Depending on a nontrivial character $\psi$, Tate \cite{Tat}
defined a gamma factor $\gamma(s, \uchi, \psi), s\in \C$, which is essentially the ratio of the two integrals involving a test function and its Fourier transform. We have
$$\gamma(s, \uchi, \psi)= \varepsilon(s, \uchi, \psi) \cdot \frac{L(1-s, \uchi^{-1})}{L(s, \uchi)},$$
where $L(s, \uchi)$ is the $L$-function of $\uchi$.

In fact,  ${\gamma}(1-s,\uchi^{-1}, \psi)$ is given by the  meromorphic continuation of
\begin{equation} \label{tateintegral}
\lim_{r \rightarrow \infty}\int_{\mfr{p}^{-r}}\uchi_{s}(x)\psi(x) \, d_\psi^\times x,
\end{equation}
where $\uchi_s = \uchi  \val{\cdot}^s$, see \cite[Page 278]{Bum}. If $\uchi$ is unitary, then this limit exists for ${\rm Re}(s)>0$. Since $\uchi=\uchi_u \val{\cdot}^{s_o}$ for some unitary character $\uchi_u$ and $s_o \in \C$, it follows that the limit above exists for ${\rm Re}(s) \gg 0$. In fact, the sequence in \eqref{tateintegral} stabilizes as $r$ increases.

It is well-known that $\varepsilon(s,\uchi,\psi)$ is a monomial function in $q^{-s}$ satisfying (see \cite{Tat1} or \cite[\S 1]{Sch1}):
\begin{eqnarray}
\label{Tate gamma} \varepsilon(1-s,\uchi^{-1},\psi) &=& \uchi(-1) \cdot \varepsilon(s,\uchi,\psi)^{-1}, \\
\label{changepsi}\varepsilon(s,\uchi,\psi_a) &=& \uchi(a)\val{a}^{s-\half} \cdot \varepsilon(s,\uchi,\psi), \\
\label{epsilon old twist} \varepsilon(s+t,\uchi,\psi) &=& q^{\mfr{f}(\psi)-\mfr{f}(\chi)t} \cdot \varepsilon(s,\uchi,\psi),\\
\label{epsilon twist} \varepsilon(s,\uchi \xi,\psi) &=& \xi(\varpi)^{\mfr{f}(\uchi)-\mfr{f}(\psi)}\varepsilon(s,\uchi,\psi),
\end{eqnarray}
where $\xi$ is any unramified character of $F^\times$. The following is proven in \cite{Kah1} and \cite{Kah2}:
\begin{eqnarray}
\label{epsilon weil}  \omega_\psi(a) &=& \varepsilon(1/2,\eta_{a,(2)},\psi_{-1}) \\
\label{epsilon weil exp} \omega(\psi) &=& [F^\times: F^{\times 2}]^{-\half} \sum_{c \in F^\times /F^{\times 2}}\varepsilon(1/2,\eta_{c,(2)},\psi).
\end{eqnarray}
(See also \cite{Szp5-2} for proofs using notations which are closer to the ones used here.) Now \eqref{epsilon weil} combined with \eqref{weilweil} imply that
\begin{equation} \label{ep prod}
\varepsilon(1/2,\eta_{a,(2)},\psi)  \cdot  \varepsilon(1/2,\eta_{b,(2)},\psi) =(a,b)_2 \cdot \varepsilon(1/2,\eta_{ab,(2)},\psi).
\end{equation}

Assume $n$ is even, $\gcd(n,p)=1$ and that $\psi$ is normalized and $\uchi$ unramified. Let $\xi$ be any character of $F^\times$. If the restriction of $\xi$ to $F^{\times n/2}$ is trivial, then (see \cite[Lemma 3.13]{GSS1})
\begin{equation}
\label{epsilonprod} \varepsilon(s,\uchi\xi,\psi) \cdot \varepsilon(s,\uchi\xi^{-1},\psi)=q^{1-2s}\uchi^2(\varpi).
\end{equation}

We note here that \eqref{epsilonprod} holds also under the slightly weaker assumption that the restriction of $\xi$ to $F^{\times n}$ is trivial. Indeed, the proof of \cite[Lemma 3.13]{GSS1} is still valid under this assumption. In fact, the following minor generalization of \eqref{epsilonprod} will be used in the proof of Theorem \ref{T:SL-c3} later.

\begin{lm}  \label{gen epsilon}
Assume that $n$ is even, $\gcd(n,p)=1$ and $\mfr{f}(\psi)=0$. Let $\uchi$ and $\xi$ be two characters of $F^\times$. Assume that $\uchi^2$ is unramified and that the restriction of $\xi$ to $F^{\times n}$ is trivial. Then,  equality \eqref{epsilonprod} holds.
\end{lm}
\begin{proof} By assumption $\uchi=\uchi_o\nu$ where $\uchi_o$ is an unramified character of $F^\times$ and $\nu$ is a quadratic character of $F^\times$. We have
$$\varepsilon(s,\uchi \xi,\psi) \cdot \varepsilon(s,\uchi\xi^{-1},\psi)=\varepsilon(s,\uchi_o (\nu \xi),\psi) \cdot \varepsilon(s,\uchi_o (\nu \xi)^{-1},\psi).$$
Since $ \nu \xi$ is trivial on $F^{\times n}$, the lemma now follows from the remark preceding it.
\end{proof}

If $\uchi$ is unramified and $\psi$ is normalized, then $\varepsilon(s, \uchi, \psi)=1$ and 
$$L(s, \uchi)= (1-q^{-s} \uchi(\varpi))^{-1}.$$
In this case we write
$$\gamma(s, \uchi):= \gamma(s, \uchi, \psi)= \frac{1-q^{-s} \chi(\varpi)}{ 1- q^{-1+ s} \chi(\varpi)^{-1}   }$$
with the $\psi$ omitted.

\subsection{Metaplectic $\tilde{\gamma}$-factor} \label{SS:meta-ga}
Let $\tilde{\gamma}(s,\uchi,\psi)$ be the factor defined and studied in  \cite{Szp3}. Recall that $\tilde{\gamma}(1-s,\uchi^{-1},\psi)$ is the meromorphic continuation of
\begin{equation} \label{sweetintegral}
\lim_{r \rightarrow \infty}\int_{\mfr{p}^{-r}}\uchi_s(x)  \omega_\psi(x)^{-1} \psi(x) \, d_\psi^\times x.
\end{equation}
This limit exists for ${\rm Re}(s) \gg 0$. We first state a result on $\tilde{\gamma}(s,\uchi,\psi)$ which will be used later in \S \ref{4remark} and \S \ref{glinv}.

\begin{lm} \label{sumlemma}
The following equality holds:
$$\sum_{a \in F^\times /F^{\times 2} } \tilde{\gamma}(s, \uchi \eta_{a,(2)},\psi)=\sum_{a \in F^\times /F^{\times 2}} \gamma(s,\uchi\eta_{a,(2)},\psi).$$
\end{lm}
\begin{proof}
It is sufficient to assume ${\rm Re}(s) \gg 0$, in which case all the limits below exist. We have
$$\begin{aligned}
& \sum_{a \in F^\times /F^{\times 2} } \tilde{\gamma}(1-s,\uchi^{-1}\eta_{a,(2)},\psi) \\
= & \sum_{a \in F^\times /F^{\times 2} } \lim_{r \rightarrow \infty}\int_{\mfr{p}^{-r}} \eta_{a,(2)}(x)\uchi_{s}(x)  \omega_\psi(x)^{-1} \psi(x) \,  d_\psi^\times x \\
=&  \lim_{r \rightarrow \infty}  \int_{\mfr{p}^{-r}} \Big(\sum_{a \in F^\times/  F^{\times 2}} \eta_{a,(2)}(x) \Big) \uchi_{s}(x)  \omega_\psi(x)^{-1} \psi(x) \, d_\psi^\times x .
\end{aligned}$$
It follows form \eqref{dualhilbert} that
$$\sum_{a \in F^\times /F^{\times 2} } \tilde{\gamma}(1-s,\uchi^{-1} \eta_{a,(2)},\psi)=[F^\times: F^{\times 2}] \cdot \lim_{r \rightarrow \infty}  \int_{\mfr{p}^{-r} \cap F^{\times 2}}  \uchi_{s}(x)  \omega_\psi(x)^{-1} \psi(x) \, d_\psi^\times x.$$
Since $\omega_\psi(x)^{-1}=1$ for all $x \in F^{\times 2}$,  we have shown that
$$\sum_{a \in F^\times /F^{\times 2} } \tilde{\gamma}(1-s,\uchi^{-1} \eta_{a,(2)},\psi)=[F^\times: F^{\times 2}] \cdot \lim_{r \rightarrow \infty}  \int_{\mfr{p}^{-r} \cap F^{\times 2}}  \uchi_{s}(x)  \psi(x) \, d_\psi^\times x.$$
In view of \eqref{tateintegral}, a similar argument also gives that
$$\sum_{a \in F^\times /F^{\times 2}}  \gamma(1-s,\uchi^{-1}\eta_{a,(2)},\psi)=[F^\times: F^{\times 2}] \cdot \lim_{r \rightarrow \infty}  \int_{\mfr{p}^{-r} \cap F^{\times 2}}  \uchi_{s}(x)  \psi(x) \, d_\psi^\times x.$$
This completes the proof.
\end{proof}

The computation of the principal value integral in \eqref{sweetintegral} is contained in the unpublished notes \cite{Swe} by W. Jay Sweet,  and was reproduced in the appendix of \cite{GoSz}.

\begin{thm}[{\cite[Theorem A.1]{GoSz}}] \label{sweetthm}
One has the following equality
\begin{equation} \label{meta gama formula}
\tilde{\gamma}(1-s,\uchi^{-1},\psi)=\omega(\psi) \uchi(-1) \cdot \frac{ \gamma(s+\half,\uchi,\psi)}{ \gamma(2s,\uchi^{2},{\psi_2})}. \end{equation}
\end{thm}

\begin{cor}  \label{verybicelem}
Suppose that $p$ is odd and that $\psi$ is normalized. Let $\uchi$ be an unramified character of $F^\times$. Then,
\begin{equation} \label{verifygao}
\prod_{a \in F^\times /F^{\times 2} } \tilde{\gamma}(1-s,\uchi^{-1}\eta_{a,(2)},\psi)=(\varpi,-1)_2 \cdot \gamma(1-2s,\uchi^{-2},\psi)^2 \frac{L (2s,\uchi^2 )L (-2s,\uchi^{-2}\bigr)}{L (1-2s,\uchi^{-2} )L (1+2s,\uchi^2)}. \end{equation}
\end{cor}
\begin{proof} Since $p$ is odd, we may apply Lemma \ref{pn1} to the case $n=2$. A set of representatives for $F^\times / F^{\times 2}$ is given by
$\{1, u, \varpi, u\varpi \}$. We have that $\eta_{u, (2)}$ is unramified, $\eta_{u,(2)}(\varpi)=-1$ and  $\mfr{f}(\eta_{\varpi, (2)})=1$. In particular,
$$\prod_{a \in F^\times /F^{\times 2}} \eta_{a,(2)}(-1)=1.$$
Since $p$ is odd, $2 \in O_F^\times.$ It follows from \eqref{Tate gamma} and \eqref{changepsi} that
$$\gamma(2s,\uchi^{2},{\psi_2})^{-1}=\gamma(1-2s,\uchi^{-2},\psi).$$
Therefore, by \eqref{meta gama formula} we have
$$\prod_{a \in F^\times /F^{\times 2} } \tilde{\gamma}(1-s,\uchi\eta_{a,(2)},\psi)=\gamma(1-2s,\uchi^{-2},{\psi})^4 \cdot \prod_{a \in F^\times /F^{\times 2} } \gamma(s+\half,\uchi\eta_{a,(2)},\psi)$$
Note that the right hand side of \eqref{verifygao} equals
$$(\varpi,-1)_2 \cdot \gamma(1+2s,\uchi^{2},\psi)  \cdot \gamma(1-2s,\uchi^{-2},\psi)^3.$$
Thus, the equality \eqref{verifygao} is equivalent to
\begin{equation} \label{verifygao2}
\prod_{a \in F^\times /F^{\times 2} } \gamma(s+\half,\uchi\eta_{a,(2)},\psi)= (\varpi,-1)_2 \cdot \gamma(2s,\uchi^{2},\psi) \cdot \gamma(1+2s,\uchi^{2},\psi).
 \end{equation}
A direct computation gives
$$\begin{aligned}
& \prod_{a \in F^\times /F^{\times 2} } \gamma(s+\half,\uchi \eta_{a,(2)},\psi) \\
= & \gamma(s+\half,\uchi,\psi) \cdot \gamma(s+\half,\uchi\eta_{u,(2)},\psi) \cdot \gamma(s+\half,\uchi\eta_{\varpi,(2)},\psi)\cdot \gamma(s+\half,\uchi\eta_{\varpi u,(2)},\psi) \\
= & \frac{\bigl(1-q^{-(s+\half)}\uchi(\varpi)\bigr)}{\bigl(1-q^{s-\half}\uchi^{-1}(\varpi)\bigr)}\frac{\bigl(1+q^{-(s+\half)}\uchi(\varpi)\bigr)}{\bigl(1+q^{s-\half}\uchi^{-1}(\varpi)\bigr)} \cdot
\varepsilon(s+\half,\uchi\eta_{\varpi, (2)},\psi) \cdot \epsilon(s+\half,\uchi\eta_{\varpi u, (2)},\psi).
\end{aligned} $$
By \eqref{epsilon old twist} and \eqref{epsilon twist} we have
$$\prod_{a \in F^\times /F^{\times 2} } \gamma(s+\half,\uchi\eta_a,\psi)= \frac{1-q^{-2s-1}\uchi^2(\varpi)}{1-q^{2s-1}\uchi^{-2}(\varpi)}q^{-2s}\uchi^2(\varpi) \cdot \varepsilon(\half,\eta_{\varpi, (2)},\psi) \varepsilon(\half,\eta_{\varpi u, (2)},\psi),$$
where by \eqref{ep prod},
$$\varepsilon(\half,\eta_{\varpi,(2)},\psi) \cdot \varepsilon(\half,\eta_{\varpi u,(2)},\psi)=(\varpi,u\varpi)_2 \cdot \varepsilon(\half,\eta_{\varpi^2 u,(2)},\psi)=-(\varpi,\varpi)\cdot \varepsilon(\half,\eta_{ u},\psi).$$
Since $(\varpi,\varpi)_2=(-1,\varpi)_2$ and $\varepsilon(\half,\eta_{ u},\psi)=1$,  we deduce that
$$\prod_{a \in F^\times /F^{\times 2}} \gamma(s+\half,\uchi\eta_a,\psi)=-(\varpi,-1)_2 \cdot \frac{L(1-2s,\uchi^{-2})}{L(1+2s,\uchi^2)} q^{-2s} \uchi^2(\varpi).$$
Equation \eqref{verifygao2} is thus finally reduced to
$$-q^{-2s}\uchi^2(\varpi)=\frac{L(-2s,\uchi^{-2})}{L(2s,\uchi^2)},$$
which clearly holds. This completes the proof.
\end{proof}

%Following \cite{Szp3}, we also define a metaplectic gamma factor (in this unramified case) by
%$$\tilde{\gamma}(s, \uchi)= \frac{ \gamma(s, \uchi^{2})  }{\gamma((s+1)/2, \uchi)}.$$
%Note that we have chosen a different normalization here compared to \cite{Szp, GSS1}.

\subsection{Plancherel measure, $\gamma$- and $\tilde{\gamma}$- factor} \label{SS:Pg}
Let $\chi: Z(\wt{T}) \to \C^\times$ be an unramified central character. Let $T(w,i(\chi)): I(i(\chi)) \to I({}^w i(\chi))$ be the intertwining operator given in \eqref{T(w)} with measure
$$\prod_{\alpha>0, U_\alpha \subset U_w} du_\alpha$$
on $U_w$ normalized such that $du_\alpha$ is the measure on $U_\alpha = e_\alpha(F)$ satisfying $du_\alpha(e_\alpha(O)) =1$. Let $f_0\in I(i(\chi))$ and $f_0' \in I({}^wi( \chi))$ be the normalized unramified vectors. Denote
$$\Phi_\w= \set{\alpha>0: \w(\alpha) <0}.$$
We have
$$T(w,i(\chi))(f_0)= \gk(w, i(\chi)) \cdot f_0',$$
where
$$\gk(w, i(\chi))= \prod_{\alpha\in \Phi_\w}  \frac{ 1-q^{-1} \chi(\wt{h}_\alpha (\varpi^{n_\alpha})) }{ 1-\chi(\wt{h}_\alpha (\varpi^{n_\alpha}))   }$$
is the Gindikin-Karpelevich  coefficient  (see \cite{Cas1, Mc0, Mc2, Ga1}). The Plancherel measure $\mu(w, i(\chi))$ associated to $T(w, i(\chi))$ is the meromorphic function on the variety $\mca{O}_{i(\chi), \C}$ (see \S \ref{SS:lcm}, and also \cite[\S 3.1]{CO} for an explicit parametrization of $\mca{O}_{i(\chi), \C}$ for Kazhdan-Patterson covers of $\GL_r$) such that
$$T(w^{-1}, {}^w i(\chi)) \circ T(w, i(\chi))= \mu(w, i(\chi))^{-1} \cdot {\rm id}.$$
More explicitly,  if $i(\chi)$ is unramified, then
$$\mu(w, i(\chi))^{-1}=  \gk(w^{-1}, {}^wi(\chi)) \cdot \gk(w, i(\chi)).$$
In \S \ref{plansec}, we will also discuss $\mu(w, i(\chi))^{-1}$ for ramified $i(\chi)$.

For $w=w_\alpha$ with $\alpha\in \Delta$, we define the gamma factor $\gamma(w_\alpha, i(\chi))$ to be such that
\begin{equation} \label{g-inv}
\gamma(w_\alpha, i(\chi))^{-1}= \frac{ 1-q^{-1} \chi(\wt{h}_\alpha (\varpi^{n_\alpha}))^{-1} }{ 1-\chi(\wt{h}_\alpha (\varpi^{n_\alpha}))   }.
\end{equation}
Then clearly,
$$\mu(w_\alpha, i(\chi))= \gamma(w_\alpha,  i(\chi)) \cdot \gamma(w_\alpha^{-1}, {}^{w_\alpha}  i(\chi)).$$

We note that for $\alpha \in \Phi$ and unramified $\chi$, the linear character
$$\uchi_\alpha: F^\times \to \C^\times, \text{ given by } \uchi_\alpha(a):= \chi(\wt{h}_\alpha(a^{n_\alpha}) )$$
is an unramified character of $F^\times$. Moreover,
$$\gamma(w_\alpha,  i(\chi))^{-1} = \gamma(0, \uchi_\alpha )^{-1}= \frac{1-q^{-1} \uchi_\alpha(\varpi)^{-1}  }{ 1- \uchi_\alpha(\varpi) },$$
where $\gamma(s, \uchi_\alpha)$ is the Tate gamma factor.

To consider the metaplectic gamma factor $\tilde{\gamma}(w_\alpha,  i(\chi))$, we assume that $n_\alpha=2m_\alpha$ with $m_\alpha$ odd and $m_\alpha \alpha^\vee \in Y_{Q,n}$. Define
$$\uchi_\alpha^\natural: F^\times \to \C^\times$$
given by
$$\uchi_\alpha^\natural(a):= \chi(\wt{h}_\alpha(a^{m_\alpha}) ) \cdot \omega_\psi(a)^{-1},$$
where $\omega_\psi(-): F^\times \to \bbmu_4$ is the Weil index. We have:

\begin{lm}
With notations as above, $\uchi_\alpha^\natural$ is a character of $F^\times$, and moreover
$$\uchi_\alpha= \big( \uchi_\alpha^\natural \big)^2.$$
\end{lm}
\begin{proof}
We write $c:=\text{gcd}(n, Q(\alpha^\vee))$ and $Q_\alpha= Q(\alpha^\vee)/c$. Then,
$$\begin{aligned}
& \uchi_\alpha^\natural(a) \cdot  \uchi_\alpha^\natural(b) \\
 = &   \chi(\wt{h}_\alpha(a^{m_\alpha}) ) \cdot \chi(\wt{h}_\alpha(b^{m_\alpha}) ) \cdot \omega_\psi(b)^{-1}  \cdot \omega_\psi(a)^{-1}  \\
 = &  \chi(\wt{h}_\alpha((ab)^{m_\alpha}) ) \cdot (a, b)_n^{m_\alpha^2 Q(\alpha^\vee)} \cdot \omega_\psi(a)^{-1}  \cdot \omega_\psi(b)^{-1} \text{ by relation (C) in \S \ref{S:top-c}} \\
 =&  \chi(\wt{h}_\alpha((ab)^{m_\alpha}) ) \cdot (a, b)_n^{m_\alpha^2 Q(\alpha^\vee)} \cdot \omega_\psi(ab)^{-1}  \cdot (a, b)_n^{m_\alpha c} \\
 = &(a, b)_n^{m_\alpha^2 Q_\alpha c+ m_\alpha c} \cdot  \uchi_\alpha^\natural(ab).
 \end{aligned}
$$
Now since $m_\alpha$ and $Q_\alpha$ are both odd, we see $n|(m_\alpha c(m_\alpha Q_\alpha + 1))$. The first statement follows, and the proof also shows that $\big( \uchi_\alpha^\natural \big)^2= \uchi_\alpha$.
\end{proof}

With the above assumption on $n_\alpha$, we define the metaplectic gamma-factor $\tilde{\gamma}(w_\alpha, i(\chi))$ to be
\begin{equation} \label{meta-g}
\tilde{\gamma}(w_\alpha,  i(\chi)):=\tilde{\gamma}(1, \uchi_\alpha^\natural)^{-1},
\end{equation}
the reciprocal of the $\tilde{\gamma}$-factor associated to $\uchi_\alpha^\natural$ in \S \ref{SS:meta-ga}, which by  Lemma \ref{for0mod4} and Theorem \ref{sweetthm} is equal to
$$\frac{ \gamma(0, \uchi_\alpha)   }{ \gamma(1/2, \uchi_\alpha^\natural)  }.$$

\subsection{The matrix $\mca{S}_\mfr{R}(w_\alpha, i(\chi); r_w^{\rm un})$}
To proceed, we first introduce some notations. For $\alpha^\vee \in \Delta^\vee$, let $W_\alpha \subset W$ be the order two subgroup generated by $\w_\alpha$.
The surjection 
$$\msc{X}_{Q,n}^{sc} \onto \msc{X}_{Q,n}$$
 induces a natural map of sets
$$\phi_\alpha: (\msc{X}_{Q,n}^{sc})^{W_\alpha} \to (\msc{X}_{Q,n})^{W_\alpha},$$
which is not surjective in general. The property that $\phi_\alpha$ is surjective is equivalent to the following: if $y+Y_{Q,n}$ lies in $(\msc{X}_{Q,n})^{W_\alpha}$, then $y+Y_{Q,n}^{sc} \in (\msc{X}_{Q,n}^{sc})^{W_\alpha}$.
Recall the natural surjection 
$$f_\msc{X}: Y \onto \msc{X}_{Q,n}.$$
For $y \in Y$, denote $\hat{y}:=f_\msc{X}(y) \in \msc{X}_{Q,n}$.

\begin{dfn} \label{D:3cases}
A point $y\in Y$ is called:
\begin{enumerate}
\item[$\bullet$] $\alpha$-free if $\hat{y} \notin  (\msc{X}_{Q,n})^{W_\alpha}$;
\item[$\bullet$] $\alpha$-normal if $\hat{y}$ lies in the image of $\phi_\alpha$;
\item[$\bullet$] $\alpha$-special if $\hat{y} \in (\msc{X}_{Q,n})^{W_\alpha} - \text{Im}(\phi_\alpha)$, i.e., if it does not lie in the image of $\phi_\alpha$.
\end{enumerate}
\end{dfn}

Denote by 
$$Y_\alpha^{\rm fre}, Y_\alpha^{\rm nor} \text{ and } Y_\alpha^{\rm spe} \subset Y$$
  the set of $\alpha$-free, $\alpha$-normal and $\alpha$-special points respectively. We have a decomposition
$$Y=Y_\alpha^{\rm fre} \sqcup Y_\alpha^{\rm nor} \sqcup Y_\alpha^{\rm spe}.$$
Clearly, $y$ is $\alpha$-free if and only if the $W_\alpha$-orbit of $\hat{y}$ in $\msc{X}_{Q,n}$ is free; also,
$$Y_\alpha^{\rm nor} \sqcup Y_\alpha^{\rm spe} = f_\msc{X}^{-1}( (\msc{X}_{Q,n})^{W_\alpha}  ).$$
In general, it is quite rare to have an $\alpha$-special point, as shown by the following.

\begin{prop} \label{P:dich}
For a covering group $\wt{G}$ and fixed $\alpha$, if $Y_\alpha^{\rm spe} \neq \emptyset$, then necessarily the following hold:
\begin{enumerate}
\item[(i)] the root system of $\mbf{G}$ is of type $C_r, r\ge 1$;
\item[(ii)] $\alpha^\vee$ is the unique short simple coroot, and $n_\alpha \equiv 2 \mod 4$;
\item[(iii)] $(\msc{X}_{Q,n}^{sc})^{W_\alpha}=\emptyset$, i.e., in this case $Y_\alpha^{\rm nor}= \emptyset$.\end{enumerate}
In particular, if $\mbf{G}$ is almost-simple and $Y_\alpha^{\rm spe} \ne \emptyset$, then necessarily $\wt{G}=\wt{\Sp}_{2r}^{(n)}$ with $n_{\alpha} \equiv 2 \mod 4$ where $\alpha^\vee$ is the unique short simple coroot.
\end{prop}
\begin{proof}
By assumption, let $y\in Y$ be an $\alpha$-special point. Then $y-\w_\alpha[y] \in Y_{Q,n} - Y_{Q,n}^{sc}$; that is,
$$\angb{y_\rho}{\alpha} \alpha^\vee \in Y_{Q,n} - Y_{Q,n}^{sc}.$$
This gives that $n_\alpha\nmid \angb{y_\rho}{\alpha}$, but
\begin{equation} \label{KF}
B(\angb{y_\rho}{\alpha} \alpha^\vee, z) =\angb{y_\rho}{\alpha}\cdot \angb{z}{\alpha} Q(\alpha^\vee) \in n\Z \text{ for every } z\in Y.
\end{equation}
In particular, $n_\alpha|(2\angb{y_\rho}{\alpha})$, by applying $z\in \alpha^\vee$.

Moreover, we have necessarily that $2|\angb{z}{\alpha}$ for all $z\in Y$; otherwise there exists $z_0$ such that $\angb{z_0}{\alpha}=1$, which enforces $n_\alpha|\angb{y_\rho}{\alpha}$. In particular,
$$2|\angb{\beta^\vee}{\alpha} \text{ for all } \beta^\vee \in \Delta^\vee.$$
This shows that the root system of $\mbf{G}$ is of type $C_r$ and $\alpha$ is the unique long simple root.

We also see from the above that $n_\alpha=2m_\alpha$ is an even number for some $m_\alpha$. We claim that $m_\alpha$ has to be odd. From above, we have $\angb{y_\rho}{\alpha}=k\cdot m_\alpha$ with $k$ an odd number.
By \eqref{KF} we have
$$n_\alpha| \angb{y_\rho}{\alpha} \cdot \angb{y}{\alpha},$$
that is, $n_\alpha$ divides $(km_\alpha)(km_\alpha + 1)$, which implies that $m_\alpha$ is odd. This proves the second assertion.

Lastly, we show that in this case, there is no $\alpha$-normal point. If $y'$ is $\alpha$-normal, then $y'-\w_\alpha[y] \in Y_{Q,n}^{sc}$; that is,
$$n_\alpha| \angb{(y')_\rho}{\alpha}.$$
However, since both $n_\alpha$ and $\angb{y'}{\alpha}$ are even, this is not possible. This completes the proof.
\end{proof}

\begin{dfn} \label{D:metaG}
Let $(\wm{G}, n)$ be a cover of a connected reductive group $\mbf{G}$ associated with $(D, \eta)$. It is called of metaplectic type if $Y_\alpha^{\rm spe} \ne \emptyset$ for some $\alpha \in \Delta$.
\end{dfn}
We say that $\wt{G}$ is of metaplectic type, if it arises from a metaplectic $(\wm{G}, n)$. By Proposition \ref{P:dich}, if  $\wt{G}$ is a cover of an almost-simple $G$, then $\wt{G}$ is of metaplectic type if and only if $\wt{G}= \wt{\Sp}_{2r}^{(n)}$, which arises from $(\wm{Sp}_{2r}, n)$ such that $n_{\alpha} \equiv 2 \mod 4$ for the unique simple short-coroot $\alpha^\vee$ of $\mbf{Sp}_{2r}$. In particular, the classical degree-two metaplectic cover of $\Sp_{2r}$ is one such example.

\begin{cor} \label{C:dich}
For each $\alpha$, we have either $Y_\alpha^{\rm nor} = \emptyset$ or  $Y_\alpha^{\rm spe} = \emptyset$. Also, there are exactly three mutually exclusive cases for the map $\phi_\alpha$:
\begin{enumerate}
\item[(i)]  $ \msc{X}_{Q,n}^{W_\alpha}= \emptyset$ (and thus necessarily   $(\msc{X}_{Q,n}^{sc})^{W_\alpha} = \emptyset$);
\item[(ii)]  $(\msc{X}_{Q,n}^{sc})^{W_\alpha}= \emptyset$ but $\msc{X}_{Q,n}^{W_\alpha} \neq \emptyset$;
\item[(iii)]  if $(\msc{X}_{Q,n}^{sc})^{W_\alpha} \neq \emptyset$ and $\msc{X}_{Q,n}^{W_\alpha} \neq \emptyset$, then $\phi_\alpha$ is surjective with each fiber being a torsor over $Z(\wt{G}^\vee_{Q,n})$.
\end{enumerate}
\end{cor}
%\begin{proof}
%It suffices to show the assertion in the second statement. More precisely, it is sufficient to show that if $y\in (\msc{X}_{Q,n}^{sc})^{W_\alpha}$ and $z\in Y_{Q,n}$, then $y + z$ lies in $(\msc{X}_{Q,n}^{sc})^{W_\alpha}$. For this purpose, note
%$$\w_\alpha[y+z] - (y+z)= \w_\alpha[y] - y + (\w_\alpha(z) - z)= \w_\alpha[y] - y - \angb{z}{\alpha} \alpha^\vee.$$
%Since $z\in Y_{Q,n}$, we have $B(z, \alpha^\vee) =Q(\alpha^\vee) \cdot \angb{z}{\alpha} \in n\Z$, and thus $n_\alpha| \angb{z}{\alpha}$. That is, $\angb{z}{\alpha} \alpha^\vee \in Y_{Q,n}^{sc}$.
%
%It follows that $\w_\alpha[y+z] - (y+z) \in Y_{Q,n}^{sc}$. This completes the proof.
%\end{proof}

Now, recall the decomposition of $\mfr{R}$ in \eqref{R:dec} when $W'=W_\alpha$, then for every equivalence class $\mfr{R}_k \subset \mfr{R}$, we have either
$$\mfr{R}_k=\set{y_i, y_j} \text{ or } \mfr{R}_k=\set{y_i}.$$
In the first case, $y_i$ is an $\alpha$-free point and $y_i- \w_\alpha[y_j] \in Y_{Q,n}$ (thus, $y_j$ is also $\alpha$-free). In the second case, $y_i - \w_\alpha[y_i] \in Y_{Q,n}$, and $y_i$ is either $\alpha$-normal or $\alpha$-special.
For convenience, we will write
$$\tau(w, \chi, y_i, y_j):=\tau(w, \chi, \s_{y_i}, \s_{y_j}).$$
\subsubsection{For $\alpha$-free points}
Let $\mfr{R}=\set{y_i, y_j}$ with $y_i$ an $\alpha$-free point. We compute $\det( \mca{S}_\mfr{R}(w_\alpha, i(\chi); r_w^{\rm un})  )$, where
$$  \mca{S}_\mfr{R}(w_\alpha, i(\chi); r_w^{\rm un})=
\left[ \begin{array}{cccc}
 \tau(w_\alpha, \chi, y_i, y_i) &  \tau(w_\alpha, \chi, y_j, y_i) \\
\tau(w_\alpha, \chi, y_i, y_j) &  \tau(w_\alpha, \chi, y_j, y_j)
\end{array} \right].
$$

\begin{prop} \label{P:abs}
One has
$$\det \big( \mca{S}_\mfr{R}(w_\alpha, i(\chi); r_w^{\rm un}) \big) =(-1) \cdot \chi (\wt{h}_\alpha(\varpi^{n_\alpha}))^{ \frac{ \angb{(y_i)_\rho + (y_j)_\rho}{\alpha} }{n_\alpha} } \cdot \mu(w_\alpha,  i(\chi))^{-1},$$
where $\mu(w_\alpha,  i(\chi))$ is the Plancherel measure associated to $w_\alpha$ and $i(\chi)$.
\end{prop}
\begin{proof}
Recall in this case, $y_i - y_j \notin Y_{Q,n}$ and $y_i \equiv \w_\alpha[y_j]$. Suppose
$$y_i = \w_\alpha[y_j] + z, \text{ with } z \in Y_{Q,n}. $$
Since $y_i - y_j \notin Y_{Q,n}$, one has
$$  \mca{S}_\mfr{R}(w_\alpha, i(\chi); r_w^{\rm un})=
\left[ \begin{array}{cccc}
 \tau^1(w_\alpha, \chi, y_i, y_i) &  \tau^2(w_\alpha, \chi, y_j, y_i) \\
\tau^2(w_\alpha, \chi, y_i, y_j) &  \tau^1(w_\alpha, \chi, y_j, y_j)
\end{array} \right].
$$
We compute each entry explicitly. First,
$$\begin{aligned}
& \tau^2(w_\alpha, \chi, y_i, y_j) \\
=& \tau^2(w_\alpha, \chi, \w_\alpha[y_j] + z, y_j) \\
=& \chi(\varpi^z) \cdot \varepsilon^{D(\w_\alpha[y_j], z)} \cdot \tau^2(w_\alpha, \chi, \w_\alpha[y_j], y_j) \\
=& \chi(\varpi^z) \cdot  \varepsilon^{D(\w_\alpha[y_j], z)} \cdot \varepsilon^{ \angb{(y_j)_\rho}{\alpha} \cdot D(y_j, \alpha^\vee) } \cdot \g(\angb{(y_j)_\rho}{\alpha}Q(\alpha^\vee)).
\end{aligned} $$
Similarly, since $y_j=\w_\alpha[y_i] - \w_\alpha(z)$, it gives that
$$\tau^2(w_\alpha, \chi, y_j, y_i)=\chi(\varpi^{-\w_\alpha(z)}) \cdot \varepsilon^{D(\w_\alpha[y_i], -\w_\alpha(z))}  \cdot \varepsilon^{ \angb{(y_i)_\rho}{\alpha} \cdot D(y_i, \alpha^\vee) } \cdot \g(\angb{(y_i)_\rho}{\alpha}Q(\alpha^\vee)).$$
It follows that
$$\tau^2(w_\alpha, \chi, y_j, y_i) \cdot \tau^2(w_\alpha, \chi, y_i, y_j) = \chi(\varpi^{ \angb{z}{\alpha} \alpha^\vee })  \cdot q^{-1} \cdot \varepsilon^{P(i,j)},$$
where
$$\begin{aligned}
P(i,j)= & D(\w_\alpha[y_j], z) + \angb{(y_j)_\rho}{\alpha} \cdot D(y_j, \alpha^\vee) \\
 & + D(\w_\alpha[y_i], -\w_\alpha(z)) + \angb{(y_i)_\rho}{\alpha} \cdot D(y_i, \alpha^\vee) \\
 & + D(z, -\w_\alpha(z)) + \angb{(y_i)_\rho}{\alpha}Q(\alpha^\vee).
\end{aligned}$$
By noting that $\angb{y_\rho}{\alpha} \alpha^\vee = y- \w_\alpha[y]$ for every $y\in Y$, and computing the residue of $P(i,j)$  modulo 2, we see that
$$\begin{aligned}
P(i,j)= & D(y_i, z)  +  Q(z) + D(y_j, y_j-y_i) + D(y_j, z) \\
 & + D(y_j, \w_\alpha(z)) + Q(\w_\alpha(z)) +  D(y_i, y_i -y_j) + D(y_i, \w_\alpha(z)) \\
 & + D(z, \w_\alpha(z)) + D(\alpha^\vee, y_j-y_i) + D(\alpha^\vee, z) \\
 =& D(y_i+ y_j, z) + D(y_i + y_j, \w_\alpha(z)) + Q(y_i - y_j) \\
 & + D(z, \w_\alpha(z)) + D(\alpha^\vee, y_j - y_i) + D(\alpha^\vee, z) \\
=&  \angb{z}{\alpha} \cdot D(\angb{(y_j)_\rho}{\alpha}\alpha^\vee + z, \alpha^\vee) + Q(\angb{(y_j)_\rho}{\alpha}\alpha^\vee + z) \\
& + D(z, \w_\alpha(z)) + D(\alpha^\vee, \angb{(y_j)_\rho}{\alpha}\alpha^\vee) \\
=& \angb{(y_j)_\rho}{\alpha} B(\alpha^\vee, z) + \angb{z}{\alpha} D(z, \alpha^\vee) \\
& + \angb{(y_j)_\rho}{\alpha}^2 Q(\alpha^\vee) + \angb{(y_j)_\rho}{\alpha} B(\alpha^\vee, \alpha) + Q(z) \\
& + D(z, \w_\alpha(z)) + \angb{(y_j)_\rho}{\alpha} Q(\alpha^\vee) \\
=& \angb{z}{\alpha} D(z, \alpha^\vee) + Q(z) + D(z, \w_\alpha(z)) \\
=& 0 \in \Z/2\Z.
\end{aligned}$$
That is, we have $\tau^2(w_\alpha, \chi, y_j, y_i) \cdot \tau^2(w_\alpha, \chi, y_i, y_j) = \chi(\varpi^{ \angb{z}{\alpha} \alpha^\vee })  \cdot q^{-1}$. Hence,
$$\begin{aligned}
\det(\mca{S}_\mfr{R}(w_\alpha, i(\chi); r_w^{\rm un})) & = \left( \frac{1-q^{-1}}{ 1-\chi (\wt{h}_\alpha(\varpi^{n_\alpha})) } \right)^2 \chi (\wt{h}_\alpha(\varpi^{n_\alpha}))^{k_{y_i,\alpha} + k_{y_j,\alpha}} - \chi(\varpi^{ \angb{z}{\alpha} \alpha^\vee })  \cdot q^{-1},
\end{aligned}$$
where
$$k_{y_i, \alpha} + k_{y_j, \alpha} =\ceil{\frac{\angb{y_i}{\alpha}}{n_\alpha}} + \ceil{\frac{\angb{y_j}{\alpha}}{n_\alpha}}= \ceil{\frac{2- \angb{y_i}{\alpha}}{n_\alpha}} + \ceil{\frac{\angb{y_j}{\alpha}}{n_\alpha}} + \frac{ \angb{z}{\alpha} }{n_\alpha}.$$
By the proof of \cite[Lemma 3.9]{Ga2}, we get $k_{y_i, \alpha} + k_{y_j, \alpha}= 1+  (\angb{z}{\alpha}/n_\alpha)$. Therefore
$$\begin{aligned}
& \det(\mca{S}_\mfr{R}(w_\alpha, i(\chi); r_w^{\rm un})) \\
 = & \left( \frac{1-q^{-1}}{ 1-\chi (\wt{h}_\alpha(\varpi^{n_\alpha})) } \right)^2 \chi (\wt{h}_\alpha(\varpi^{n_\alpha}))^{1+ \frac{ \angb{z}{\alpha} }{n_\alpha} } - \chi (\wt{h}_\alpha(\varpi^{n_\alpha}))^{ \frac{ \angb{z}{\alpha} }{n_\alpha} }  \cdot q^{-1} \\
 = & (-1) \cdot \chi (\wt{h}_\alpha(\varpi^{n_\alpha}))^{ \frac{ \angb{z}{\alpha} }{n_\alpha} } \cdot \frac{\left(1-q^{-1} \chi (\wt{h}_\alpha(\varpi^{n_\alpha}))\right) \cdot \left(1-q^{-1} ({}^{w_\alpha}\chi) (\wt{h}_\alpha(\varpi^{n_\alpha}))\right)}{\left( 1- \chi (\wt{h}_\alpha(\varpi^{n_\alpha})) \right) \cdot \left( 1- ({}^{w_\alpha}\chi) (\wt{h}_\alpha(\varpi^{n_\alpha})) \right)} \\
 =& (-1) \cdot \chi (\wt{h}_\alpha(\varpi^{n_\alpha}))^{ \frac{ \angb{z}{\alpha} }{n_\alpha} } \cdot \mu(w_\alpha,  i(\chi))^{-1}.
\end{aligned}$$
Lastly, we check easily that $\angb{z}{\alpha}= \angb{(y_i)_\rho + (y_j)_\rho}{\alpha}$. This completes the proof.
\end{proof}

\subsubsection{For $\alpha$-normal points}
Let $\mfr{R}=\set{y_i}$ be a singleton with $y_i \in Y_\alpha^{\rm nor}$.  We would like to compute the scalar
$$\det(\mca{S}_\mfr{R}(w_\alpha, i(\chi); r_w^{\rm un}))= \tau(w_\alpha, \chi, y_i, y_i).$$
By assumption,
$$y_i=\w_\alpha[y_i] + z_i$$
for some $z_i \in Y_{Q,n}^{sc}$; that is $n_\alpha| \angb{(y_i)_\rho}{\alpha}$ in this case.

%\textcolor{red}{STOP HERE}

\begin{prop} \label{P:nor}
If $y_i$ is an $\alpha$-normal point, then
$$\tau(w_\alpha, \chi, y_i, y_i)=\chi (\wt{h}_\alpha(\varpi^{n_\alpha}))^{ \frac{\angb{(y_i)_\rho}{\alpha}}{n_\alpha}  + 1} \cdot \gamma(w_\alpha,  i(\chi))^{-1}.$$
\end{prop}
\begin{proof}
We have
$$\begin{aligned}
& \tau(w_\alpha, \chi, y_i, y_i) \\
=& \tau^1(w_\alpha, \chi, y_i, y_i) + \tau^2(w_\alpha, \chi, y_i, y_i) \\
=& \frac{1-q^{-1} }{1-\chi (\wt{h}_\alpha(\varpi^{n_\alpha}))}  \chi (\wt{h}_\alpha(\varpi^{n_\alpha}))^{\ceil{\frac{\angb{y}{\alpha}}{n_\alpha}}} +  \tau^2(w_\alpha, \chi, \w_\alpha[y_i] + z_i, y_i) \\
=&  \frac{1-q^{-1} }{1-\chi (\wt{h}_\alpha(\varpi^{n_\alpha}))}  \chi (\wt{h}_\alpha(\varpi^{n_\alpha}))^{\ceil{\frac{\angb{y}{\alpha}}{n_\alpha}}} \\
& + \chi(\s_z) \cdot  \varepsilon^{D(\w_\alpha[y_i], z_i)} \cdot \varepsilon^{ \angb{(y_i)_\rho}{\alpha} \cdot D(y_i, \alpha^\vee) } \cdot \g(\angb{(y_i)_\rho}{\alpha}Q(\alpha^\vee)) \\
\end{aligned} $$
Note that by assumption $z= \angb{(y_i)_\rho}{\alpha} \alpha^\vee \in Y_{Q,n}^{sc}$, and it follows from a straightforward checking that the exponent of $\varepsilon$ is even and $\g(\angb{(y_i)_\rho}{\alpha}Q(\alpha^\vee))=-q^{-1}$. Thus,
$$\begin{aligned}
& \tau(w_\alpha, \chi, y_i, y_i)  \\
=& \frac{1-q^{-1} }{1-\chi (\wt{h}_\alpha(\varpi^{n_\alpha}))}  \chi (\wt{h}_\alpha(\varpi^{n_\alpha}))^{ \frac{\angb{(y_i)_\rho}{\alpha}}{n_\alpha}  + 1} - q^{-1} \cdot \chi (\wt{h}_\alpha(\varpi^{n_\alpha}))^{ \frac{\angb{(y_i)_\rho}{\alpha}}{n_\alpha} } \\
=& \chi (\wt{h}_\alpha(\varpi^{n_\alpha}))^{ \frac{\angb{(y_i)_\rho}{\alpha}}{n_\alpha}  + 1} \cdot \frac{1-q^{-1} ({}^{w_\alpha}\chi) (\wt{h}_\alpha(\varpi^{n_\alpha}))}{1-\chi (\wt{h}_\alpha(\varpi^{n_\alpha}))} \\
=& \chi (\wt{h}_\alpha(\varpi^{n_\alpha}))^{ \frac{\angb{(y_i)_\rho}{\alpha}}{n_\alpha}  + 1} \cdot \gamma(w_\alpha,  i(\chi))^{-1}.
\end{aligned}$$
This completes the proof.
\end{proof}

\subsubsection{When $y_i$ is an $\alpha$-special point} Let $\mfr{R}=\set{y_i}$ be a singleton with $y_i \in Y_\alpha^{\rm spe}$.  Note that this occurs only when $n_\alpha=2m_\alpha$ is even. Again, let $z_i$ be such that
$$y_i=\w_\alpha[y_i] + z_i .$$
In this case,
$$z_i= \angb{(y_i)_\rho}{\alpha} \alpha^\vee = k m_\alpha \alpha^\vee \in Y_{Q,n} - Y_{Q,n}^{sc},$$
where $k$ is an odd number; thus, $\angb{(y_i)_\rho}{\alpha} + m_\alpha = (k+1)m_\alpha$ is divisible by $n_\alpha$.

We would like to compute the scalar
$$\det(\mca{S}_\mfr{R}(w_\alpha, i(\chi); r_w^{\rm un}))= \tau(w_\alpha, \chi, y_i, y_i).$$

\begin{prop} \label{P:spe}
If $y_i$ is an  $\alpha$-special point, then
$$\tau(w_\alpha, \chi, y_i, y_i)=\chi (\wt{h}_\alpha(\varpi^{n_\alpha}))^{\frac{\angb{(y_i)_\rho}{\alpha}  + m_\alpha }{n_\alpha}} \cdot \tilde{\gamma}(w_\alpha,  i(\chi))^{-1}.$$
\end{prop}
\begin{proof}
We have
$$\begin{aligned}
& \tau(w_\alpha, \chi, y_i, y_i) \\
=& \tau^1(w_\alpha, \chi, y_i, y_i) + \tau^2(w_\alpha, \chi, y_i, y_i) \\
=&  \frac{1-q^{-1} }{1-\chi (\wt{h}_\alpha(\varpi^{n_\alpha}))}  \chi (\wt{h}_\alpha(\varpi^{n_\alpha}))^{\ceil{\frac{\angb{y}{\alpha}}{n_\alpha}}} \\
& + ({}^{w_\alpha}\chi)(\s_z)^{-1} \cdot  \varepsilon^{D(\w_\alpha[y_i], z_i)} \cdot \varepsilon^{ \angb{(y_i)_\rho}{\alpha} \cdot D(y_i, \alpha^\vee) } \cdot \g(\angb{(y_i)_\rho}{\alpha}Q(\alpha^\vee)) \\
=&  \frac{1-q^{-1} }{1-\chi (\wt{h}_\alpha(\varpi^{n_\alpha}))}  \chi (\wt{h}_\alpha(\varpi^{n_\alpha}))^{\ceil{\frac{\angb{y_i}{\alpha}}{n_\alpha}}} + ({}^{w_\alpha}\chi)(\s_z)^{-1} \cdot  \varepsilon^{Q(z_i)} \cdot \g(m_\alpha \cdot Q(\alpha^\vee)) \\
=&  \frac{1-q^{-1} }{1-\chi (\wt{h}_\alpha(\varpi^{n_\alpha}))}  \chi (\wt{h}_\alpha(\varpi^{n_\alpha}))^{\ceil{\frac{\angb{y_i}{\alpha}}{n_\alpha}}} + \chi(\s_z) \cdot \g(m_\alpha \cdot Q(\alpha^\vee)).
\end{aligned} $$
Note that by assumption $z= \angb{(y_i)_\rho}{\alpha} \alpha^\vee =km_\alpha \alpha^\vee $ with $k$ odd. It follows that
$$\ceil{\frac{\angb{y_i}{\alpha}}{n_\alpha}}=\ceil{ \frac{km_\alpha + 1}{2m_\alpha} }= \frac{k+1}{2} = \ceil{\frac{\angb{(y_i)_\rho}{\alpha}}{n_\alpha}}= \frac{\angb{(y_i)_\rho}{\alpha}  + m_\alpha }{n_\alpha}$$
and
$$\chi(\s_z)=\chi (\wt{h}_\alpha(\varpi^{n_\alpha}))^{\frac{k+1}{2}} \cdot \chi( \wt{h}_\alpha(\varpi^{-m_\alpha}) ).$$
By \cite[Lemma 4.3]{Ga2}, one has $\g(m_\alpha \cdot Q(\alpha^\vee))=q^{-1/2} \cdot \omega_\psi(\varpi)^{-1}$ (see also \cite{Szp5-2}),  and therefore it gives that
$$\begin{aligned}
& \tau(w_\alpha, \chi, y_i, y_i) \\
=&\chi (\wt{h}_\alpha(\varpi^{n_\alpha}))^{\frac{k+1}{2}} \cdot
\left(
\frac{1-q^{-1} }{1-\chi (\wt{h}_\alpha(\varpi^{n_\alpha}))}   + \chi( \wt{h}_\alpha(\varpi^{-m_\alpha}) ) \cdot \g(m_\alpha \cdot Q(\alpha^\vee))
\right) \\
=&\chi (\wt{h}_\alpha(\varpi^{n_\alpha}))^{\frac{k+1}{2} } \cdot
\left(
\frac{1-q^{-1} }{1-\chi (\wt{h}_\alpha(\varpi^{n_\alpha}))}   + \chi( \wt{h}_\alpha(\varpi^{-m_\alpha}) ) \cdot   q^{-1/2} \omega_\psi(\varpi)^{-1}
\right) \\
%=& \chi (\wt{h}_\alpha(\varpi^{n_\alpha}))^{\frac{k+1}{2} } \cdot
%\left(
%\frac{1-q^{-1} }{1-\chi (\wt{h}_\alpha(\varpi^{n_\alpha}))}   + \chi( \wt{h}_\alpha(\varpi^{-m_\alpha}) ) \cdot q^{-1/2} \omega_\psi(\varpi)
%\right) \\
=& \chi (\wt{h}_\alpha(\varpi^{n_\alpha}))^{\frac{k+1}{2} } \cdot
\frac{( 1+ q^{-1/2} \chi( \wt{h}_\alpha(\varpi^{-m_\alpha}) ) \omega_\psi(\varpi)^{-1}) \cdot ( 1- q^{-1/2} \chi( \wt{h}_\alpha(\varpi^{m_\alpha}) ) \omega_\psi(\varpi)^{-1} )}{1-\chi (\wt{h}_\alpha(\varpi^{n_\alpha}))}     \\
=& \chi (\wt{h}_\alpha(\varpi^{n_\alpha}))^{\frac{k+1}{2} } \cdot
\frac{( 1- q^{-1} \chi( \wt{h}_\alpha(\varpi^{-n_\alpha}) ) ) \cdot ( 1- q^{-1/2} \chi( \wt{h}_\alpha(\varpi^{m_\alpha}) ) \omega_\psi(\varpi)^{-1} )}{\big( 1-\chi (\wt{h}_\alpha(\varpi^{n_\alpha})) \big) \cdot \big(1- q^{-1/2} \chi( \wt{h}_\alpha(\varpi^{-m_\alpha}) ) \omega_\psi(\varpi)^{-1} \big)}     \\
=& \chi (\wt{h}_\alpha(\varpi^{n_\alpha}))^{\frac{k+1}{2} } \cdot
\tilde{\gamma}(w_\alpha,  i(\chi))^{-1}.
\end{aligned} $$
This completes the proof.
\end{proof}

\subsection{The matrix $\mca{C}(\mfr{B}_{{}^\w\chi}, \mfr{B}_\chi; r_w^{\rm un})$}
Recall the matrix $\mca{C}( \mfr{B}_{ {}^\w \chi}, \mfr{B}_\chi; r_w^{\rm un})$ representing the homomorphism
$$\mca{C}_{r_w^{\rm un}}: \Wh_\psi (I(\chi))  \to \Wh_\psi (I({}^\w \chi))$$
as in \eqref{E:C} with chosen bases $\mfr{B}_\chi$ and $\mfr{B}_{{}^\w\chi}$ on the two sides respectively. Here $\mfr{B}_\chi$ and $\mfr{B}_{{}^\w\chi}$ depend on the ordered set
$$\mfr{R}=\set{\s_{y_i}:  1\le i \le \val{\msc{X}_{Q,n}}},$$
where
$$\set{y_i: 1\le i\le \val{\msc{X}_{Q,n}}} \subset Y$$
 is an ordered set of representatives of $\msc{X}_{Q,n}$.

In order to compute entries of $\mca{C}( \mfr{B}_{ {}^\w \chi}, \mfr{B}_\chi; r_w^{\rm un})$, we do a further analysis of it as follows. First, note that $\mca{C}_{r_w^{\rm un}}$ is induced (see \S \ref{SS:lcm}) from the natural isomorphism of vector spaces
$$\mca{C}_0: i(\chi)^\vee \to ({}^w i(\chi))^\vee  \to (i({}^\w \chi))^\vee.$$
Here for the underlying vector spaces, we have $i(\chi)= {}^w i(\chi)$, which gives the first identity. The second map arises from dualizing the isomorphism $r_w^{\rm un}: {}^w i(\chi) \simeq i({}^\w \chi)$ given in \eqref{r-w}.
For clarity of notation, for any $\gamma\in  \wt{T}$ we write
$$l(\gamma; \chi):=l_\gamma^\chi \in i(\chi)^\vee$$ for the element introduced in \S \ref{SS:para}; similarly, we have $$l(\gamma; {}^\w \chi) \in i({}^\w \chi)^\vee.$$
The set $\mfr{R}$ gives an ordered basis
$$\mfr{B}_\chi^0=\set{l(\s_y; \chi):  \s_y \in \mfr{R}} \subset i(\chi)^\vee$$
for $i(\chi)^\vee$; similarly, we have an ordered basis
$$\mfr{B}_{{}^\w\chi}^0= \set{l(\s_y; {}^\w \chi):  \s_y \in \mfr{R}} \subset i({}^\w \chi)^\vee.$$
Let $$\mca{C}_0(\mfr{B}_{{}^\w\chi}^0, \mfr{B}_\chi^0; r_w^{\rm un})$$ be the matrix representing $\mca{C}_0$ with respect to these two bases.
Transporting $\mfr{B}_\chi^0$ to the right hand side of the isomorphism
$$i(\chi)^\vee \simeq \Wh_\psi(I(\chi))$$
in \S \ref{SS:para} gives the ordered basis $\mfr{B}_\chi$ of the right hand side. Similarly, one obtains the ordered basis $\mfr{B}_{{}^\w\chi}$ of $\Wh_\psi (I({}^\w \chi))$. In view of this, we have
$$\mca{C}_0(\mfr{B}_{{}^\w\chi}^0, \mfr{B}_\chi^0; r_w^{\rm un} )= \mca{C}(\mfr{B}_{{}^\w\chi}, \mfr{B}_\chi; r_w^{\rm un}).$$
Recall the action $w(\gamma):=w \gamma w^{-1}$ of the representative $w$ on any $\gamma \in \wt{T}$.

\begin{prop} \label{P:C0}
For every $\gamma\in \wt{T}$, one has
$$\mca{C}_0\left( l(\gamma; \chi) \right)= l(w(\gamma); {}^\w \chi).$$
\end{prop}
\begin{proof}
We will show $l(w^{-1}(\gamma); \chi) = \mca{C}_0^{-1} \left(l(\gamma; {}^\w \chi)\right)$ instead, which is clearly equivalent to the assertion. For $\gamma \in \wt{T}$ let $\phi_\gamma \in i(\chi)$ be the unique element such that
$$\text{supp}(\phi_\gamma)= \wt{A} \cdot \gamma^{-1}, \text{ and } \phi_\gamma( \gamma^{-1}  )=1.$$
We have
$$(r_w^{\rm un}(\phi_\gamma))( (w(\gamma))^{-1})= (r_w^{\rm un}(\phi_\gamma))(w \gamma^{-1} w^{-1})= \phi_\gamma(\gamma^{-1})=1.$$
It is also easy to check that the support of  $(r_w(\phi_\gamma))$ is $\wt{A}\cdot w(\gamma)^{-1}$.
This shows that
$$ \mca{C}_0^{-1} \left(l(\gamma; {}^\w \chi)\right)(\phi_{\gamma'})  = l(\gamma; {}^\w \chi) \circ r_w(\phi_{\gamma'}) =
\begin{cases}
1 & \text{ if } \gamma=w(\gamma') ;\\
0 & \text{ if } \gamma \notin w(\gamma') \cdot \wt{A}.
\end{cases}
$$
This implies that $\mca{C}_0^{-1} \left(l(\gamma; {}^\w \chi)\right)=l(w^{-1}(\gamma); \chi) $ and the proof is completed.
\end{proof}

\begin{cor} \label{C:C0}
For $\w_\alpha$ with $\alpha \in \Delta$, one has
$$ \mca{C}_0\left( l(\s_y; \chi) \right)=\varepsilon^{\angb{y}{\alpha} D(y,\alpha^\vee)} \cdot l(\s_{\w_\alpha(y)}; {}^{\w_\alpha} \chi)$$
If $\bbmu_{2n} \subset F^\times$, then for every $\w \in W$ one has
$$\mca{C}_0\left( l(\s_y; \chi) \right)= l(\s_{\w(y)}; {}^\w \chi).$$
\end{cor}
\begin{proof}
The result follows from combining Lemma \ref{L:W-act} and Proposition \ref{P:C0}.
\end{proof}

Denote by
$$\w_\alpha^\msc{P}: \msc{X}_{Q,n} \to \msc{X}_{Q,n}$$
the map given by
$$\w_\alpha^\msc{P}(y+Y_{Q,n}) = \w_\alpha(y) + Y_{Q,n}.$$
Let $s_\alpha \in \text{Perm}(\msc{X}_{Q,n})$ be the element the permutation group of the set $[1, \val{\msc{X}_{Q,n}}] \cap \N$ such that
$$\w_\alpha^\msc{P}(y_i+Y_{Q,n}) = y_{s_\alpha(i)} + Y_{Q,n}$$
for all $\s_{y_i} \in \mfr{R}$. Denote by
$$\text{sgn}(\w_\alpha^\msc{P}):=\text{sgn}(s_\alpha)$$ its sign, which is independent of the labelling of elements in $\msc{X}_{Q,n}$.

\begin{prop} \label{P:det-C}
Assume $\bbmu_{2n} \subset F^\times$. Then,
$$\det ( \mca{C}(\mfr{B}_{{}^{\w_\alpha} \chi}, \mfr{B}_\chi; r_w^{\rm un}) ) = {\rm sgn}(\w_\alpha^\msc{P})  \cdot \chi(\wt{h}_\alpha(\varpi^{-\angb{\sum_{i=1}^d y_i}{\alpha}} ).$$
\end{prop}
\begin{proof}
As noted, it suffices to compute the determinant of $\mca{C}_0(\mfr{B}_{{}^\w\chi}^0, \mfr{B}_\chi^0; r_w^{\rm un} )$.  For any $y_i$ with $\s_{y_i} \in \mfr{R}$, let $z_i \in Y_{Q,n}$ be such that
$$\w_\alpha(y_i) = y_{s_\alpha(i)} + z_i.$$
It follows from Corollary \ref{C:C0} that for every $i$,
$$\mca{C}_0(l(\s_{y_i}; \chi))=  l(\s_{\w_\alpha(y_i)}; {}^{\w_\alpha} \chi) = ({}^{\w_\alpha}\chi )(\s_{z_i})^{-1} \cdot  l(\s_{y_{s_\alpha(i)}}; {}^{\w_\alpha} \chi).$$
That is, the $(y_{s_\alpha(i)}, y_i)$-entry of the matrix $\mca{C}(\mfr{B}_{{}^{\w_\alpha} \chi}, \mfr{B}_\chi; r_w^{\rm un})$ is $({}^{\w_\alpha}\chi )(\s_{z_i})^{-1}$. Write $d:=\val{\msc{X}_{Q,n}}$. This gives that
$$\begin{aligned}
\det ( \mca{C}(\mfr{B}_{{}^{\w_\alpha} \chi}, \mfr{B}_\chi; r_w^{\rm un}) )  = & \text{sgn}(\w_\alpha^\msc{P}) \cdot \prod_{i=1}^d  ({}^{\w_\alpha} \chi)(\s_{z_i})^{-1}  \\
= & \text{sgn}(\w_\alpha^\msc{P})  \cdot \chi\left( \varpi^{\sum_{i=1}^d \w_\alpha(z_i)        } \right)^{-1} \\
=&  \text{sgn}(\w_\alpha^\msc{P})  \cdot \chi\left( \varpi^{\sum_{i=1}^d (y_i -\w_\alpha(y_i)        } \right)^{-1} \\
=& \text{sgn}(\w_\alpha^\msc{P})  \cdot \chi(\varpi^{-\angb{\sum_{i=1}^d y_i}{\alpha} \cdot \alpha^\vee} ) \\
= & {\rm sgn}(\w_\alpha^\msc{P})  \cdot \chi(\wt{h}_\alpha(\varpi^{-\angb{\sum_{i=1}^d y_i}{\alpha}} ).
\end{aligned}$$
This completes the proof.
\end{proof}

\subsection{Summary on $\det(\mca{M}_\mfr{B}(w_\alpha, i(\chi)))$}
For a simple root $\alpha \in \Delta$, denote
\begin{equation} \label{dba}
d=\val{\msc{X}_{Q,n}}, \quad b_\alpha= \val{(\msc{X}_{Q,n})^{W_\alpha}}, \quad a_\alpha= \frac{ d - b_\alpha }{2},
\end{equation}
where $W_\alpha=\set{\text{id}, \w_\alpha}$.

\begin{thm} \label{T:M1}
Assume $\bbmu_{2n} \subset F^\times$. Let $\mca{M}_\mfr{B}(w_\alpha, i(\chi))$ be any local coefficients matrix representing the endomorphism $\mca{T}(w_\alpha, i(\chi)) \in {\rm End}(\Wh(I(\chi)))$. Then,
$$\begin{aligned}
& \det(\mca{M}_\mfr{B}(w_\alpha, i(\chi))) \\
= & {\rm sgn}(\w_\alpha^\msc{P}) \cdot (-1)^{a_\alpha} \cdot \mu(w_\alpha,  i(\chi))^{-a_\alpha}  \cdot
\begin{cases}
\chi(\wt{h}_\alpha(\varpi^{b_\alpha n_\alpha - d}))  \cdot \gamma(w_\alpha,  i(\chi))^{-b_\alpha} & \text{ if } Y_\alpha^{\rm spe} =\emptyset , \\
 \chi(\wt{h}_\alpha(\varpi^{b_\alpha m_\alpha - d}))  \cdot \tilde{\gamma}(w_\alpha,  i(\chi))^{-b_\alpha} & \text{ if } Y_\alpha^{\rm nor} =\emptyset,
 \end{cases}
\end{aligned} $$
where $\mu(w_\alpha,  i(\chi))$ is the Plancherel measure associated with $w_\alpha$ and $i(\chi)$, and $\gamma(w_\alpha,  i(\chi))$ (resp. $\tilde{\gamma}(w_\alpha,  i(\chi))$) the gamma-factor (resp. metaplectic-gamma factor) as in \eqref{g-inv}  and \eqref{meta-g}. In particular, if $Y_\alpha^{\rm spe} = Y_\alpha^{\rm nor}=\emptyset$, then
$$\det(\mca{M}_\mfr{B}(w_\alpha, i(\chi)))= {\rm sgn}(\w_\alpha^\msc{P}) \cdot (-1)^{d/2} \cdot \chi(\wt{h}_\alpha(\varpi^{- d}))  \cdot \mu(w_\alpha,  i(\chi))^{-d/2}.$$
\end{thm}
\begin{proof}
Note that, a priori, $\det(\mca{M}_\mfr{B}(w, i(\chi)))$ is independent of $\mfr{B}$ chosen. Now the result follows from combining Propositions \ref{P:abs}, \ref{P:nor}, \ref{P:spe} and \ref{P:det-C} as follows.

We choose an ordered $\mfr{R}=\set{\s_{y_i}: 1\le i\le d}$ such that $\set{y_i: 1\le i \le d}$ is a set of representatives of $\msc{X}_{Q,n}$.

From Propositions \ref{P:abs} and \ref{P:nor}, we get that if $Y_\alpha^{\rm spe} =\emptyset$, then
$$\begin{aligned}
& \det( \mca{S}_\mfr{R}(w_\alpha, i(\chi); r_w^{\rm un})  ) \\
= & (-1)^{a_\alpha} \cdot \chi(\wt{h}_\alpha(\varpi^{n_\alpha}))^{ \frac{\angb{\sum_i y_i}{\alpha}- d}{n_\alpha} + b_\alpha  } \cdot \mu(w_\alpha,  i(\chi))^{- a_\alpha} \cdot \gamma(w_\alpha,  i(\chi))^{- b_\alpha} \\
= & (-1)^{a_\alpha} \cdot \chi(\wt{h}_\alpha(\varpi^{b_\alpha n_\alpha - d + \angb{\sum_i y_i}{\alpha}})) \cdot \mu(w_\alpha,  i(\chi))^{- a_\alpha} \cdot \gamma(w_\alpha,  i(\chi))^{- b_\alpha} .
\end{aligned}$$
On the other hand, if $Y_\alpha^{\rm nor} =\emptyset$ (with necessarily $n_\alpha=2m_\alpha$, $m_\alpha$ odd), then Proposition \ref{P:abs} and \ref{P:spe} give that
$$\begin{aligned}
& \det( \mca{S}_\mfr{R}(w_\alpha, i(\chi); r_w^{\rm un})  ) \\
= & (-1)^{a_\alpha} \cdot \chi(\wt{h}_\alpha(\varpi^{n_\alpha}))^{ \frac{\angb{\sum_i y_i}{\alpha}- d + b_\alpha m_\alpha}{n_\alpha} } \cdot \mu(w_\alpha,  i(\chi))^{- a_\alpha} \cdot \tilde{\gamma}(w_\alpha,  i(\chi))^{- b_\alpha} \\
=& (-1)^{a_\alpha} \cdot \chi(\wt{h}_\alpha(\varpi^{b_\alpha m_\alpha - d + \angb{\sum_i y_i}{\alpha}})) \cdot \mu(w_\alpha,  i(\chi))^{- a_\alpha} \cdot \tilde{\gamma}(w_\alpha,  i(\chi))^{ - b_\alpha} .
\end{aligned} $$
In either case, using Proposition \ref{P:det-C}, the result follows from a simplification, since $\det( \mca{M}_\mfr{B}(w_\alpha, i(\chi)) )= \det(\mca{S}_\mfr{R}(w_\alpha, i(\chi); r_w^{\rm un}) ) \cdot \det(  \mca{C}( \mfr{B}_{ {}^{\w} \chi }, \mfr{B}_\chi ; r_w^{\rm un}) )$ in view of Lemma \ref{L:comp}.
\end{proof}

\begin{eg} \label{eg-SL}
Let $\wt{\SL}_2^{(n)}$ be the degree $n$ covering arising from $Q(\alpha^\vee)=1$. This is the group studied in our previous paper \cite{GSS1}. Note $n_\alpha=n$ in this case. Let
$$d=\frac{n}{\text{gcd}(n, 2)}.$$
We have $\msc{X}_{Q,n}\simeq \Z/d\Z$. The action $\w_\alpha[-]$ is given by
$$\w_\alpha[y]= 1-y \in \Z/d\Z.$$
It is easy to compute:
$$
b_\alpha=
\begin{cases}
1 & \text{ if $d$ is odd}; \\
0 & \text{ otherwise}
\end{cases}.
$$
Also, $Y_\alpha^{\rm spe} \neq \emptyset$ if and only if $n=2d$ with $d$ odd. Theorem \ref{T:M1}  then gives that (under the assumption $\bbmu_{2n} \subset F^\times$)
$$
 \det(\mca{M}_\mfr{B}(w, i(\chi))) =
 \begin{cases}
 \mu(w_\alpha, \chi)^{-(n-1)/2} \cdot \gamma(w_\alpha,  i(\chi))^{-1}  & \text{ if  $n$ is odd}; \\
  \mu(w_\alpha, \chi)^{-(d-1)/2} \cdot \tilde{\gamma}(w_\alpha,  i(\chi))^{-1}  & \text{ if  $n=2d$ with $d$ odd}; \\
(-1)\cdot  \chi(\wt{h}_\alpha(\varpi^{- d})) \cdot \mu(w_\alpha,  i(\chi))^{-d/2} & \text{ if $n=2d$ with $d$ even}.
 \end{cases}
$$
This is the content of \cite[Theorem 3.14 and \S 3.6]{GSS1}.
\end{eg}

\vskip 5pt

The case of covers of $\wt{\GL}_2$, which is considered in \cite{Bud} for some special principal series representations, is already in contrast with $\wt{\SL}_2$, as we show below.

\begin{eg}  \label{eg-GL}
We follow the notation in \cite[\S 2.1]{Ga3}, see also \cite[Example 3.16]{Ga2}. In particular, let $B_Q$ be the unique Weyl-invariant bilinear form of $Y$, which is generated by $e_1$ and $e_2$, given by
$$B_Q(e_i, e_i)=2\mbf{p},  \quad B_Q(e_1, e_2)=\mbf{q} \in \Z.$$
For simplicity, we consider the case $\mbf{p}=0$ and $\mbf{q}=1$. Such $B_Q$ gives rise to a degree $n$ cover $\wt{\GL}_2^{(n)}$, which is the untwisted cover in the sense of Kazhdan-Patterson (i.e. $c=0$ in their notation). It is also the one studied in \cite{Bud}.

In this case, we have $Y_{Q,n}=nY$ and thus $\msc{X}_{Q,n}= \Z/n\Z \oplus \Z/n\Z$ with the action given by
$$\w_\alpha[(s, t)]= (t+1, s-1) \in \msc{X}_{Q,n}$$
for every $(s, t) \in \msc{X}_{Q,n}$. It is easy to see that
$$d=n^2, b_\alpha=n, a_\alpha=\frac{n^2-n}{2} \text{ and } {\rm sgn}(\w_\alpha)^\msc{P}=(-1)^{a_\alpha}.$$
Also, $Y_\alpha^{\rm spe}=\emptyset$. It follows from Theorem \ref{T:M1} that
$$ \det(\mca{M}_\mfr{B}(w, i(\chi))) = \mu(w_\alpha,  i(\chi))^{-(n^2-n)/2} \cdot \gamma(w_\alpha,  i(\chi))^{-n},$$
which recovers and refines the results obtained by Budden \cite[Theorem 4.1]{Bud} for $\wt{\GL}_2$. (See the work of Budden-Goehle \cite{BG}  for $\wt{\GL}_r$.) As noted, the uniformity in $\det(\mca{M}_\mfr{B}(w, i(\chi)))$ for $\wt{\GL}_2^{(n)}$ is in contrast with the trichotomy for $\wt{\SL}_2^{(n)}$ which depends on the $2$-exponent factor in $n$. In \S \ref{S:RES}--\S \ref{S:LCM-G},  We will investigate this contrast from the perspective of restriction of principal series on $\wt{\GL}_2$ to $\wt{\SL}_2$.
\end{eg}

Theorem \ref{T:M1} entails a ``pattern" in the description of $\det(\mca{M}_\mfr{B}(w_\alpha, i(\chi)) )$ in terms of Plancherel measure, gamma (or metaplectic) gamma factors, and some other factors. As a simple consequence of Theorem \ref{T:M1}, we have
\begin{cor}
If $\alpha^\vee, \beta^\vee \in \Delta^\vee$ are of the same length, then $b_\alpha=b_\beta$ and ${\rm sgn}(\w_\alpha^\msc{P})= {\rm sgn}(\w_\beta^\msc{P})$. Also, in this case, there are natural bijections $Y_\alpha^{\rm spe} \to Y_\beta^{\rm spe} $ and $Y_\alpha^{\rm nor} \to Y_\beta^{\rm nor} $.  Therefore, for simply-laced groups, there is a uniform description of $\det( \mca{M}_\mfr{B}(w_\alpha, i(\chi))  )$, i.e., $b_\alpha = b_\beta$ and  ${\rm sgn}(\w_\alpha^\msc{P})= {\rm sgn}(\w_\beta^\msc{P})$.
\end{cor}

\begin{proof}
By assumption, there exists $\w\in W$ such that $\w(\alpha^\vee)= \beta^\vee$. We have
$$\w_\beta= \w \cdot \w_\alpha \cdot \w^{-1}.$$
We have a commutative diagram
$$\begin{tikzcd}
Y \ar[d, "{f_{\alpha, \beta} }"]  \ar[r, two heads] & \msc{X}_{Q,n}^{sc}  \ar[d, "{\hat{f}_{\alpha, \beta} }"]  \ar[r, two heads] & \msc{X}_{Q,n} \ar[d, "{\hat{f}_{\alpha, \beta} }"]     \\
Y   \ar[r, two heads] & \msc{X}_{Q,n}^{sc}   \ar[r, two heads] & \msc{X}_{Q,n}
\end{tikzcd} $$
given by
$$f_{\alpha, \beta}(y)= \w(y),$$
and $\hat{f}_{\alpha, \beta}$ is the naturally induced map. Then it is easy to see that both $f_{\alpha, \beta}$ and $\hat{f}_{\alpha,\beta}$ are equivariant with respect to the $\w_\alpha[\cdot]$-action on their domains and $\w_\beta[\cdot]$ on their codomains.

From this, it is easy to see that there are natural bijections between $Y_\alpha^{\rm spe}$ and $Y_\beta^{\rm spe}$,  $Y_\alpha^{\rm nor}$ and $Y_\beta^{\rm nor}$. One also has a bijection between $(\msc{X}_{Q,n})^{W_\alpha}$ and $(\msc{X}_{Q,n})^{W_\beta}$; hence, $b_\alpha = b_\beta$. Lastly, we also have ${\rm sgn}(\w_\alpha^\msc{P})= {\rm sgn}(\w_\beta^\msc{P})$. This concludes the proof.
\end{proof}

The above corollary answers affirmatively the pertinent question raised in \cite[\S 4.2]{GSS1}.

\begin{rmk} In the rank-one quasi-split case where $\wt{G}=\wt{\rm SU}_3^{(n)}$, entries of the scattering matrix $\mca{S}_\mfr{R}(w_\alpha, i(\chi); r_w^{\rm un})$ for unramified $i(\chi)$ are already computed in \cite[Theorem 14.1]{Mc2}. It would be desirable to compute the determinant of the local coefficients matrix $\mca{M}_\mfr{B}(w_\alpha, i(\chi))$. For this purpose, it might be necessary to refine the formula in \cite{Mc2} (as we did for the split case in \cite{Ga2}) and to express the entries naturally in terms of the structural data of a Brylinski-Deligne covering group. Moreover, it would be interesting to consider the case when $i(\chi)$ is ramified. 
%We will leave such an investigation to a future work.

The formula for $\det(\mca{M}_\mfr{B}(w_\alpha, i(\chi)))$ does not generalize easily to the ramified case. Indeed, consider  a general genuine character $\chi$.  Let $(L_1, L_2)$ be a Langragian decomposition of $\wt{T}/Z(\wt{T})$ (a finite group), then we have from \eqref{psWd}
$$\dim \Wh_\psi (I(\chi)) = \sqrt{[\wt{T}: Z(\wt{T})]  } = \val{L_1} =\val{L_2}.$$
However, in general $\val{L_1}\ge \val{\msc{X}_{Q,n}}$ and its size depends on the arithmetic of $F$. The case of ramified $\chi$ for $\wt{\SL}_2^{(n)}$ has been investigated in \cite{Szp6}, some results of which are recalled and also extended in \S \ref{S:LCM-Go}.
\end{rmk}

\subsection{An explicit description of $\mca{M}_\mfr{B}(w_\alpha, i(\chi))$}
Suppose we have chosen an ordered set of representatives $\set{y_i}$ of $\msc{X}_{Q,n}$ and thus $\mfr{R}$ and $\mfr{B}_\chi$. In this subsection, we give an explicit description of entries of the matrix
$$\mca{M}_\mfr{B}(w_\alpha, i(\chi))=\left[ \mca{M}(\s_{y_i}, \s_{y_j})  \right],$$
where $\mca{M}(y_i, y_j)$ denotes a generic entry. We also denote matrix
$$\mca{C}(\mfr{B}_{{}^{\w_\alpha}\chi}, \mfr{B}_\chi; r_w^{\rm un})= \left[ \mca{C}(\s_{y_i}, \s_{y_j})  \right],$$
Again, for simplicity, we also write  $\mca{S}_{\mfr{R}}(w_\alpha, i(\chi); r_w^{\rm un})= [\tau(\s_{y_i}, \s_{y_j})]$, where
$$\tau(\s_{y_i}, \s_{y_j}):=\tau(w_\alpha, \chi, \s_{y_i}, \s_{y_j}).$$
As noted in \S \ref{SS:Sca}, the labelling in $\tau(\s_{y_i}, -)$ denotes the ``columns" while $\tau(-, \s_{y_j})$ denotes the ``rows". Thus one has
$$\mca{M}(\s_{y_i}, \s_{y_j})= \sum_{\s_{y_k} \in \mfr{R} } \tau(\s_{y_k}, \s_{y_i}) \cdot \mca{C}(\s_{y_k}, \s_{y_j}).$$
Since $\tau(\s_{y_i}, \s_{y_k})= \tau^1(\s_{y_i}, \s_{y_k}) + \tau^2(\s_{y_i}, \s_{y_k})$, it gives a decomposition of
$$\mca{M}(\s_{y_i}, \s_{y_j})= \mca{M}^1(\s_{y_i}, \s_{y_j})  +  \mca{M}^2(\s_{y_i}, \s_{y_j})$$
where for $\nu=1, 2$,
$$\mca{M}^\nu(\s_{y_i}, \s_{y_j})= \sum_{y_k} \tau^\nu(\s_{y_k}, \s_{y_i}) \cdot \mca{C}(\s_{y_k}, \s_{y_j}).$$

\begin{thm} \label{T:M2}
The matrix $\mca{M}(w_\alpha, i(\chi))=[ \mca{M}(\s_{y_i}, \s_{y_j})  ]$ satisfies the following properties:
\begin{enumerate}
\item[(i)] $\mca{M}(\s_{y_i}, \s_{y_j})= \mca{M}^1(\s_{y_i}, \s_{y_j})  +  \mca{M}^2(\s_{y_i}, \s_{y_j})$ and  for every $z, z' \in Y_{Q,n}$,
$$\mca{M}^\nu(\s_{y_i} \cdot  \s_z, \s_{y_j} \cdot \s_{z'})= \chi(\s_z) \cdot  \mca{M}^\nu(\s_{y_i} , \s_{y_j}) \cdot  \chi^{-1}(\s_{z'});$$
\item[(ii)]  $\mca{M}^1(\s_{y_i}, \s_{y_j})=0$ unless $y_i \equiv \w_\alpha(y_j) \mod Y_{Q,n}$, and $\mca{M}^2(\s_{y_i}, \s_{y_j})=0$ unless $y_i \equiv y_j + \alpha^\vee \mod Y_{Q,n}$;
\item[(iii)]  one has
$$\mca{M}^1(\s_{y}, \s_{\w_\alpha(y)})=\varepsilon^{ \angb{y}{\alpha} D(\alpha^\vee, y)  } \cdot (1-q^{-1}) \frac{\chi (\wt{h}_\alpha(\varpi^{n_\alpha}))^{k_{y,\alpha}}}{1-\chi (\wt{h}_\alpha(\varpi^{n_\alpha}))} $$
where $k_{y,\alpha}=\ceil{\frac{\angb{y}{\alpha}}{n_\alpha}}$,  and
$$\mca{M}^2(\s_{y}, \s_{y-\alpha^\vee})= \varepsilon^{D(\alpha^\vee, y)} \cdot \g(\angb{y_\rho}{\alpha}Q(\alpha^\vee)).$$
\end{enumerate}
\end{thm}
\begin{proof}
First, (i) and (ii) follow from Theorem \ref{T:Sca} and Corollary \ref{C:C0}. For (iii), we see from Corollary \ref{C:C0} that  $\mca{C}(\s_{\w_\alpha(y)}, \s_y)= \varepsilon^{ \angb{y}{\alpha} D(y,\alpha^\vee)  }$, which gives that
$$\mca{M}^1(\s_y, \s_{\w_\alpha(y)})= \tau^1( \s_y, \s_y) \cdot \mca{C}(\s_y, \s_{\w_\alpha(y)}).$$
Then the first equality in (iii) follows from an easy simplification of the exponent of $\varepsilon$ in $\mca{C}(\s_y, \s_{\w_\alpha(y)})$.

On the other hand, noting that $\w_\alpha[y]= \w_\alpha(y-\alpha^\vee)$, we have
$$\begin{aligned}
& \mca{M}^2(\s_{y}, \s_{y-\alpha^\vee}) \\
= &\ \tau^2( \s_{\w_\alpha[y]}, \s_y) \cdot \mca{C}(\s_{\w_\alpha[y]}, \s_{y-\alpha^\vee}) \\
 = & \ \varepsilon^{ \angb{y_\rho}{\alpha} \cdot D(y, \alpha^\vee) } \cdot \g(\angb{y_\rho}{\alpha}Q(\alpha^\vee)) \cdot \varepsilon^{ \angb{y-\alpha^\vee}{ \alpha } D(y-\alpha^\vee, \alpha^\vee)  }  \\
 = & \ \varepsilon^{D(\alpha^\vee, y)} \cdot \g(\angb{y_\rho}{\alpha}Q(\alpha^\vee)).
\end{aligned}$$
This completes the proof.
\end{proof}

\begin{eg} \label{SL2-3}
To illustrate, we compute for $\wt{\SL}_2^{(3)}$ the three matrices explicitly. Keep notation as from the previous Example \ref{eg-SL}. Let $\set{0, \alpha^\vee, -\alpha^\vee} \subset Y$ be a set of ordered representatives of $\msc{X}_{Q,3} \simeq \Z/3\Z$. We have
$$\mfr{R} = \set{1, \s_{\alpha^\vee}, \s_{-\alpha^\vee}}$$
which decomposes into $\mfr{R}=\mfr{R}' \sqcup \mfr{R}''$ with
$$\mfr{R}'=\set{1,  \s_{\alpha^\vee}} \text{ and } \mfr{R}''=\set{\s_{-\alpha^\vee}}.$$
For simplicity of notation, we write $\chi_\alpha:=\chi( \wt{h}_\alpha(\varpi^3)  )$, since $n_\alpha=3$ in this case.  We have automatically $\bbmu_{2n} \subset F^\times$ and thus $\varepsilon=1$. Hence Theorem \ref{T:Sca} (applied to $\mfr{R}$) gives that
$$\mca{S}_\mfr{R}(w_\alpha, i(\chi); r_w^{\rm un})=
\begin{bmatrix}
\frac{1-q^{-1}}{1-\chi_\alpha  } &  \g(-1) &  0  \\
 \g(1) &  \frac{1-q^{-1}}{ 1-\chi_\alpha  } \chi_\alpha  & 0 \\
0 & 0  & \frac{ 1-q^{-1} \chi_\alpha^{-1}  }{ 1- \chi_\alpha }
\end{bmatrix}.
$$
On the other hand, by Corollary \ref{C:C0}, we have
$$\mca{C}(\mfr{B}_{{}^{\w_\alpha} \chi}, \mfr{B}_\chi; r_w^{\rm un})=
\begin{bmatrix}
1 &  0  &  0  \\
0 &  0  & 1 \\
0 &  1 & 0
\end{bmatrix}.
$$
We get that the local coefficients matrix associated with $\mfr{R}$ and therefore $\mfr{B}_\chi$ is
$$\mca{M}_{\mfr{B}_\chi}(w_\alpha, i(\chi))=
\begin{bmatrix}
\frac{1-q^{-1}}{1-\chi_\alpha  } &  0 & \g(-1)  \\
 \g(1) &  0 & \frac{1-q^{-1}}{ 1-\chi_\alpha  } \chi_\alpha   \\
0 & \frac{ 1-q^{-1} \chi_\alpha^{-1}  }{ 1- \chi_\alpha } & 0
\end{bmatrix}.
$$
\end{eg}

\begin{rmk} \label{R:n-2n}
Retain the notations from Example \ref{eg-GL}. Let $\p, \bq\in \Z$ be such that  $2\p- \bq=-1$ and $Q(e_i)=\p$ and $B(e_i, e_j)= \bq$. Let $\wt{\GL}_2^{(2)}$ be the associated double cover. Then it follows from Theorem \ref{T:M2} that
$$ \det(\mca{M}_\mfr{B}(w, i(\chi))) = \varepsilon\cdot \mu(w_\alpha,  i(\chi))^{-1} \cdot \gamma(w_\alpha,  i(\chi))^{-2}.$$
This example already shows that if we assume only $\bbmu_n \subset F^\times$ instead of $\bbmu_{2n} \subset F^\times$, then there might be factors involving $\varepsilon=(-1, \varpi)_n$ in the formula for $\det(\mca{M}_\mfr{B}(w, i(\chi)))$.
\end{rmk}

\subsection{The trace $\text{Tr}( \mca{T}(w_\alpha, i(\chi)) )$  for $\wt{\SL}_2^{(n)}$}
Besides the determinant of $\mca{T}(w_\alpha, i(\chi))$, another important invariant associated to it is the trace $\text{Tr}( \mca{T}(w_\alpha, i(\chi)) )$. Presumably, $\text{Tr}( \mca{T}(w_\alpha, i(\chi)) )$ should also encode interesting arithmetic information for $i(\chi)$.

To illustrate, we consider in this subsection the covering group $\wt{\SL}_2^{(n)}$ and compute explicitly the trace of the endomorphism map $\mca{T}(w_\alpha, i(\chi))$. We include the case $n = 0 \mod 4$, which thus supplements \cite[Corollary 4.2]{Szp6}. In \S \ref{trace}, we will investigate $\Tr( \mca{T}(w_\alpha, i(\chi)) )$ from the different perspective of partial zeta integrals and recover the same results, see Corollary \ref{exptrun}.

To give an interpretation, we use the following result whose proof is straightforward.

\begin{lm} \label{L:aux}
Let $n\ge 2 \in \N$ and $\zeta$ be a primitive $n$-th roots of unity. Then
$$\frac{1-q^{-1}}{1-T^n}=\frac{1}{n} \cdot \sum_{i=0}^{n-1} \frac{1-q^{-1} (T \zeta^i)^{-1}  }{ 1- T \zeta^i } \in \Q(\zeta)(T)$$
\end{lm}

Now for any $n\in \N$ and $\uchi: F^\times \to \C^\times$ a linear unramified character, we define
$${\rm RT}(\uchi) = \set{ \uchi' \in \text{Hom}(F^\times, \C^\times): \ \uchi' \text{ is unramified and }   (\uchi')^n= \uchi   }.$$
Clearly, $\text{RT}(\uchi)$ corresponds to the $n$-th roots of $\uchi(\varpi)$. Recall the unramified linear character $\uchi_\alpha$ defined in \S \ref{SS:Pg}. We have

\begin{prop} \label{P:trace}
Let $n\ge 1\in \N$. Then for $\Tr(\mca{T}(w_\alpha, i(\chi)))$ we have the following cases:
\begin{enumerate}
\item[$\bullet$] If $n=1$, then $\Tr(\mca{T}(w_\alpha, i(\chi)))= \gamma(w_\alpha, \chi)^{-1}$; if $n=2$, then $\Tr(\mca{T}(w_\alpha, i(\chi)))= \tilde{\gamma}(w_\alpha, \chi)^{-1}$.
\item[$\bullet$] If $n\ge 3$ and $4\nmid n$, then
$$\Tr(\mca{T}(w_\alpha, i(\chi))) = \frac{1}{n} \cdot \sum_{ \uchi  \in { \rm RT}(\uchi_\alpha)}  \gamma(0, \uchi)^{-1}. $$
\item[$\bullet$] If $n=2m$ with $m$ even, then
$$\Tr(\mca{T}(w_\alpha, i(\chi))) = \frac{1}{m} \cdot \sum_{ \uchi  \in { \rm RT}(\uchi_\alpha^\dag)}  \gamma(0, \uchi)^{-1},$$
where $\uchi_\alpha^\dag(a):=\chi(\wt{h}_\alpha(a^{m}))$.
\end{enumerate}
\end{prop}
\begin{proof}
The case $n=1$ is the known case for linear algebraic groups. For $n=2$, it either follows from \cite{Szp6}, or a direct computation with the matrix in Theorem \ref{T:M2} also gives the desired result.

Now we assume $n\ge 3$ and $4\nmid n$. Then $Y_{Q,n}$ is generated by $(n/\text{gcd}(n,2))\cdot \alpha^\vee$. In this case, since $\alpha^\vee \notin Y_{Q,n}$, we see that $\mca{M}(\s_y, \s_y) \ne 0$ unless $y\equiv \w_\alpha(y) \mod Y_{Q,n}$; that is, $y\equiv 0 \mod Y_{Q,n}$. Thus,
$$\Tr(\mca{T}(w_\alpha, i(\chi))) =  \mca{M}^1(\s_0, \s_0) = \frac{ 1-q^{-1}  }{1-\chi (\wt{h}_\alpha(\varpi^n))}.$$
The result then follows from Lemma \ref{L:aux}.

We consider the last case $n=2m$ with $m$ even. Similar discussion as above shows that $\mca{M}(\s_y, \s_y) \ne 0$ unless $y=0 \text{ or } m\alpha^\vee/2 \mod Y_{Q,n}$. This shows that
$$\begin{aligned}
& \Tr(\mca{T}(w_\alpha, i(\chi))) \\
= &   \mca{M}^1(\s_0, \s_0) +  \mca{M}^1(\s_{m\alpha^\vee/2}, \s_{m\alpha^\vee/2})  \\
=& \frac{ 1-q^{-1}  }{1-\chi (\wt{h}_\alpha(\varpi^n))} + \mca{M}^1(\s_{m\alpha^\vee/2}, \s_{- m\alpha^\vee/2})  \cdot \chi(\s_{m\alpha^\vee})^{-1} \cdot \varepsilon^{D(m\alpha^\vee/2, m\alpha^\vee)} \\
=& \frac{ 1-q^{-1}  }{1-\chi (\wt{h}_\alpha(\varpi^n))} + \varepsilon^{\angb{m\alpha^\vee/2}{\alpha} D(\alpha^\vee, m\alpha^\vee/2)} \cdot  \frac{ (1-q^{-1})  \chi (\wt{h}_\alpha(\varpi^n))    }{1-\chi (\wt{h}_\alpha(\varpi^n))}     \cdot \chi(\s_{m\alpha^\vee})^{-1} \\
= & \frac{ 1-q^{-1}     }{  1-\chi (\wt{h}_\alpha(\varpi^n))  } (1+ \chi (\wt{h}_\alpha(\varpi^m)) ) \\
=& \frac{ 1-q^{-1}     }{  1-\chi (\wt{h}_\alpha(\varpi^m))  }.
\end{aligned} $$
Again, the result follows from Lemma \ref{L:aux}.
\end{proof}

% We see that the above proposition agrees with \cite[Corollary 4.2]{Szp}, where the case $n\nequiv 0 \mod 4$ is treated.

%\textcolor{red}{can say anything about trace in general??}

%%

\section{Exceptional point and the Casselman-Shalika formula} \label{S:CS}

From the discussion in the previous section, we see that points $y$ such that
$$\hat{y} \in (\msc{X}_{Q,n})^{W_\alpha}$$ play a pivotal role: they contribute $\gamma$ or $\tilde{\gamma}$-factors to $\det(\mca{M}_\mfr{B}(w_\alpha, i(\chi)))$, by Proposition \ref{P:nor} and \ref{P:spe}.

To generalize the consideration, let $W' \subset W$ be a subgroup of the Weyl group. Consider the set
$$(\msc{X}_{Q,n})^{W'},$$
the size of which is denoted by $b_{W',n}$. Clearly $b_{\alpha}= b_{W_\alpha, n}$. In this section, we study the set $\msc{X}_{Q,n}^W$ and some applications.  For the first application, we show that there is a natural exceptional set
$$Y_n^{\rm exc} \subset f_\msc{X}^{-1}( \msc{X}_{Q,n}^W )$$
 such that every $y \in Y_n^{\rm exc}$ serves as the basis point for the Whittaker function, with respect to which one has a natural generalization in the covering setting of the classical Casselman-Shalika formula, see Theorem \ref{T:WV}. The second application is basically a direct consequence of Proposition \ref{P:nor} and \ref{P:spe}, and gives a very restrictive form of the theory of $\gamma$-factors for parabolic induction from full unramified principal series, see Theorem \ref{T:naiveLS}.

We will study some properties of the number $b_{W,n}$ in the next section, including the periodicity of $b_{W_n}$ (as $n$ varies) for almost-simple groups.

In this section, we continue assume that $\psi$ is normalized and the genuine character $\chi$ of $Z(\wt{T})$ is unramified.

\subsection{Exceptional points} It is more convenient to work with some $y\in Y$ which satisfies not only that $\hat{y}\in \msc{X}_{Q,n}^W$ but also the following stronger condition.

\begin{dfn} \label{D:exc}
A point $y\in \bigcap_{\alpha\in \Delta} (Y_\alpha^{\rm nor} \sqcup Y_\alpha^{\rm spe})$ is called exceptional (for $\mbf{G}, Q, n$) if for every $\alpha \in \Delta$ it satisfies
$$ \angb{y_\rho}{\alpha} =
\begin{cases}
-n_\alpha & \text{ if  } y\in Y_\alpha^{\rm nor}; \\
- n_\alpha/2 & \text{ if  } y\in Y_\alpha^{\rm spe}.
\end{cases}$$
\end{dfn}
Denote by $Y^\exc_n \subset Y$ the set of exceptional points of $Y$.  Clearly,
$$Y_n^{\rm exc} \subseteq f_\msc{X}^{-1}(\msc{X}_{Q,n}^W),$$
where $f_\msc{X}: Y\onto  \msc{X}_{Q,n}^W$
is the natural quotient. However, $Y_n^{\rm exc}$ may not be equal to  $ f^{-1}(\msc{X}_{Q,n}^W)$ in general.

\begin{lm} \label{L:exc}
Let $\wt{G}$ be an $n$-fold cover of  a semisimple group $G$. An exceptional point, if it exists, is unique; that is, $\val{ Y^\exc_n }\le 1 $. If $\wt{G}$ is not of metaplectic type, then
$$Y^\exc_n=\set{\rho- \rho_{Q,n}} \text{ if } \rho- \rho_{Q,n} \in Y.$$
On the other hand, if $\wt{G}=\wt{\Sp}_{2r}^{(n)}$ is of metaplectic type, then
$$Y^\exc_n=\set{\rho-\rho_{Q,n} + \frac{n_{\alpha_r} }{2} \omega_{\alpha_r}  } \subset Y,$$
where $\omega_{\alpha_r}$ is the fundamental coweight for $\alpha_r \in \Delta$ (the unique long simple root).
\end{lm}
\begin{proof}
If $G$ is semisimple, then the uniqueness of exceptional point is clear, since the pairing $\angb{-}{-}$ between $Y\otimes \Q$ and $X\otimes \Q$ is perfect, and $\Delta$ is a basis for the latter vector space.
Now we check easily that if $\wt{G}$ is not of metaplectic type, then
$$\angb{(\rho-\rho_{Q,n})_\rho}{\alpha}= -n_\alpha \text{ for all } \alpha\in \Delta,$$
and if $\wt{G}$ is of metaplectic type, then
$$\angb{(\rho-\rho_{Q,n} + \frac{n_{\alpha_r}}{2} \omega_{\alpha_r})_\rho}{\alpha}= - \frac{n_\alpha}{2} \text{ for all } \alpha\in \Delta.$$
Now, it suffices to show that in the second case, the element belongs to $Y$. However, in this case, we have
$$\rho-\rho_{Q,n} + \frac{n_{\alpha_r}}{2} \omega_{\alpha_2} \equiv \omega_{\alpha_r} -  n_{\alpha_r} \omega_{\alpha_r} + \frac{n_{\alpha_r}}{2} \omega_{\alpha_r} \equiv (n_{\alpha_r}/2 + 1) \omega_{\alpha_r}  \equiv 0 \mod Y,$$
where the last equivalence follows from the fact that $n_{\alpha_r} \equiv 2 \mod 4$. This completes the proof.
\end{proof}
We will have a more extensive study of $Y_n^\exc$ in \S \ref{S:eg}. However, the above lemma immediately entails concrete examples of $Y_n^\exc$. For instance, if $n$ is odd, then for $\wt{\SL}_2^{(n)}$ in Example \ref{eg-SL} the unique exceptional point is $(1- n)\alpha^\vee/2$.

%%%
 \subsection{Exceptional subspaces of $\Wh_\psi(I(\chi))$}  \label{SS:el}
We show an immediate consequence of the existence of an exceptional point, the results of which will be used to prove Theorem \ref{T:naiveLS} later.

Let $y \in Y_n^\exc$. Recall the Whittaker functional
$$\lambda_{\s_y}^\chi \in \Wh_\psi (I(\chi))$$
given in \S \ref{SS:para}. Denote by
 $$E(y; \chi):=\C \cdot \lambda_{\s_y}^\chi \subset \Wh(I(\chi))$$
 the one dimensional space which we call the exceptional subspace associated to $y$.

Recall the map $T(w, \chi; r_w^{\rm un})^*: \Wh (I({}^\w \chi) ) \to \Wh (I(\chi)) $ induced from $T(w, \chi; r_w^{\rm un})$. Write
 $$\w=\w_l \w_{l-1} ... \w_2 \w_1$$
 into a minimal decomposition with $\w_i=\w_{\alpha_i}$ for some $\alpha_i\in \Delta$.

 \begin{prop} \label{P:prod-ga}
 Fix $y\in Y_n^{\rm exe}$. Then the restriction $T(w, \chi; r_w^{\rm un})^*: E(y, {}^\w \chi) \to E(y, \chi)$
 is well-defined, i.e., the image of $E(y, {}^\w \chi)$ via $T(w, \chi; r_w^{\rm un})^*$ lies in $E(y, \chi)$. Moreover,
 $$T(w, \chi; r_w^{\rm un})^*(\lambda_{\s_y}^{ {}^\w \chi  }) = \lambda_{\s_y} \cdot \prod_{i=1}^l \gamma^\natural(w_i, i({}^{\w_{i-1} ... \w_1} \chi))^{-1},$$
 where
 $$\gamma^\natural(w_\alpha, i(\chi))
 =
 \begin{cases}
 \gamma(w_\alpha, i(\chi)) & \text{ if  } y\in Y_\alpha^{\rm nor};\\
\tilde{\gamma }(w_\alpha, i(\chi)) & \text{ if }  y\in Y_\alpha^{\rm spe}.
 \end{cases}
 $$
 \end{prop}
 \begin{proof}
 In view of the cocycle relation in  \eqref{SLCM2}, it suffices to show the rank one $\w=\w_\alpha$ case. However, the statement for $\w=\w_\alpha$ then follows from Propositions \ref{P:nor} and \ref{P:spe}, coupled with the defining property of an exceptional point $y$. More precisely, since $y$ is exceptional, if it is $\alpha$-normal, then Proposition \ref{P:nor} implies that
$$T(w_\alpha, \chi; r_{w_\alpha}^{\rm un})^*(\lambda_{\s_y}^{ {}^{\w_\alpha} \chi  }) =\gamma(w_\alpha, i(\chi)) \cdot \lambda_{\s_y}.$$
The case when $y$ is $\alpha$-special is similar, by invoking Proposition \ref{P:spe} instead.
\end{proof}

As a simple example, we see that in the context and notations of Example \ref{SL2-3} for $\wt{\SL}_2^{(3)}$, the point $ -\alpha^\vee \in Y$ is exceptional. The $(-\alpha^\vee, -\alpha^\vee)$-entry of the diagonal-block matrix $S_\mfr{R}(w_\alpha, i(\chi); r_{w_\alpha}^{\rm un})$, which represents $T(w_\alpha, \chi; r_{w_\alpha}^{\rm un})^*$, is exactly $\gamma(w_\alpha, i(\chi))$.

%%%
\subsection{The Casselman-Shalika formula}
In this subsection, we assume that all exceptional points for a Brylinski-Deligne cover $\wt{G}$, if exist, are $\alpha$-normal for all $\alpha\in \Delta$; that is,
$$Y_n^{\rm exc} \subseteq \bigcap_{\alpha\in \Delta} Y_\alpha^{\rm nor}.$$
For cover $\wt{G}$ of a semisimple group $G$, this assumption is equivalent to the assumption that $\wt{G}$ is not of metaplectic type. For computational simplicity, we also assume that $\bbmu_{2n} \subset F^\times$. Moreover, we denote
$$\chi_\alpha:= \uchi_\alpha(\varpi)=\chi(\wt{h}_\alpha(\varpi^{n_\alpha})) \in \C,$$
this notation is used already in Example \ref{SL2-3}.
%Note that this shorthand notation is not to be confused with the linear character $\uchi_\alpha$.

\begin{lm} \label{L:CS}
Let $z\in Y^\exc_n$ be an exceptional point. Then $\tau(\w, \chi, \s_y, \s_z)=0$ unless $y\in z+ Y_{Q,n}$. Moreover,
$$\tau(\w, {}^{\w^{-1}}\chi, \s_z, \s_z)=(-1)^{l(\w)} \gk(\w^{-1}, i(\chi)) \cdot \prod_{\alpha \in \Phi_{\w^{-1}}} \chi_\alpha.$$
\end{lm}

\begin{proof}
By the defining property of $z$ and Theorem \ref{T:Sca},
$$\tau(\w_\alpha, \chi, \s_z, \s_z)=\gamma(\w_\alpha, i(\chi))^{-1}= \frac{1-q^{-1} \chi_\alpha^{-1} }{ 1-\chi_\alpha }$$
for every $\alpha\in \Delta$. It then follows from Proposition \ref{P:prod-ga} that
$$\tau(\w, {}^{\w^{-1}}\chi, \s_z, \s_z)=\prod_{\alpha\in \Phi_\w} \frac{ 1-q^{-1}  ({}^{\w^{-1}}\chi)^{-1}_\alpha}{ 1-({}^{\w^{-1}}\chi)_\alpha },$$
where $\Phi_\w=\set{\alpha\in \Phi: \alpha>0, \w(\alpha)<0}$.
%More explicitly, if $\w=\w_k \w_{k-1} ... \w_2 \w_1$ is a minimal decomposition with $\w_i:=\w_{\alpha_i}, \alpha_i \in \Delta$, then
%$$\Phi_\w=\set{\alpha_1, \w_1(\alpha_2), ..., \w_1 \w_2... \w_{k-1}(\alpha_k)}.$$
It is easy to see
$$\w(\Phi_\w) = -\Phi_{\w^{-1}}$$
and therefore
\begin{align*}
& \tau(\w, {}^{\w^{-1}}\chi, \s_z, \s_z) \\
=& \prod_{\alpha\in \Phi_{\w^{-1}}} \frac{ 1-q^{-1} \chi_\alpha}{ 1-(\chi^{-1})_\alpha } \\
=& \prod_{\alpha\in \Phi_{\w^{-1}}} (-\chi_\alpha ) \frac{ 1-q^{-1} \chi_\alpha}{ 1-\chi_\alpha } \\
=& (-1)^{l(\w)} \gk(\w^{-1}, i(\chi)) \cdot \prod_{\alpha \in \Phi_{\w^{-1}}} \chi_\alpha,
\end{align*}
as desired.
\end{proof}

An element $\wt{t} \in \wt{T}$ is called dominant if
$$\wt{t} \cdot (U\cap K) \cdot \wt{t}^{-1} \subset K.$$
Let
$$Y^+=\set{y\in Y: \angb{y}{\alpha}\ge 0 \text{ for all } \alpha\in \Delta}.$$
Then an element $\s_y\in \wt{T}$ is dominant if and only if $y\in Y^+$. We also write
$$Y_{Q,n}^+:=Y^+ \cap Y_{Q,n}.$$

Let $I(\chi)$ be an unramified principal series  of $\wt{G}$ and $\gamma \in \wt{T}$. Let $\mca{W}_\gamma$ be the unramified scalar-valued Whittaker function associated to $f^0$ and $\gamma$. Then $\mca{W}_\gamma(\wt{t})=0$ unless $\wt{t} \in \wt{T}$ is dominant. Moreover, for dominant $\wt{t}$, one has (see \cite{KP, Pat, Ga5})
\begin{equation} \label{P:Pat}
\mca{W}_\gamma(\wt{t})=\delta_B^{1/2}( \wt{t} ) \cdot \sum_{\w \in W} \gk(\w_G \w^{-1}, i(\chi))  \cdot \tau(\w, {}^{\w^{-1}}\chi, \gamma, w_G \cdot \wt{t} \cdot w_G^{-1}),
\end{equation}
where $\delta_B$ is the modular character of $B$. Before we proceed, we pause to remark that various forms for the Whittaker value $\mca{W}_\gamma(\wt{t})$  have been obtained in works such as \cite{BBCFH, BBF1, BBFH1, BBBF, BBB, Mc2}, to mention a few. However, for our purpose in this paper, we rely on \eqref{P:Pat}.

\begin{prop} \label{P:valW}
Let $z\in Y_n^\exc$, and let $\mca{W}_{\s_z} \in \Wh(I(\chi))$ be the unramified Whittaker function associated to $\lambda_{\s_z}$. Then  one has
$$\mca{W}_{\s_z}(\s_{y+(-z)})= \delta_B^{1/2}(\s_{y-z}) \cdot \gk(\w_G, i(\chi)) \cdot (-1)^{l(\w_G)} \cdot \sum_{\w\in W}   (-1)^{l(\w)} \chi(\s_{\w(y+\rho_{Q,n}) + \rho_{Q,n}} )
$$
for every $y\in Y_{Q,n}^+$
\end{prop}
\begin{proof}
Note that $y+(-z)$ is dominant. Also,  $w \s_y  w^{-1}= \s_{\w(y)}$ by our assumption $\bbmu_{2n} \subset F^\times$. Moreover, it is easy to see that $z=\w_G(-z)$ for any $z\in Y_n^\exc$. Now it follows from \eqref{P:Pat} that
$$\begin{aligned}
& \mca{W}_{\s_z}(\s_{y+(-z)}) \\
= & \delta_B^{1/2}(\s_{y-z}) \cdot \sum_{\w\in W} \gk(\w_G \w^{-1}, i(\chi)) \cdot \tau(\w, {}^{\w^{-1}} \chi, \s_z, w_G \cdot \s_{y-z} \cdot w_G^{-1}) \\
= & \delta_B^{1/2}(\s_{y-z}) \cdot \sum_{\w\in W} \gk(\w_G \w^{-1}, i(\chi)) \cdot \tau(\w, {}^{ \w^{-1}} \chi, \s_z, \s_{\w_G(-z)} \cdot \s_{\w_G(y)}) \\
= & \delta_B^{1/2}(\s_{y-z}) \cdot \sum_{\w\in W} \gk(\w_G \w^{-1}, i(\chi)) \cdot \tau(\w, {}^{ \w^{-1}} \chi, \s_z, \s_z ) \cdot \chi(\s_{\w \w_G(y)} ) \\
= &  \delta_B^{1/2}(\s_{y-z}) \cdot \sum_{\w\in W}  \gk(\w_G \w^{-1}, i(\chi)) \gk(\w^{-1}, i(\chi))  \cdot  (-1)^{l(\w)} \chi(\s_{\w \w_G(y)} ) \cdot \prod_{\alpha \in \Phi_{\w^{-1}}} \chi_\alpha,
\end{aligned}$$
where the last equality follows from Lemma \ref{L:CS}. We further simplify to get
$$\begin{aligned}
& \mca{W}_{\s_z}(\s_{y+(-z)}) \\
=& \delta_B^{1/2}(\s_{y-z}) \cdot \gk(\w_G, i(\chi)) \cdot \sum_{\w\in W}   (-1)^{l(\w)} \chi(\s_{\w \w_G(y)} ) \cdot \prod_{\alpha \in \Phi_{\w^{-1}}} \chi_\alpha \\
=& \delta_B^{1/2}(\s_{y-z}) \cdot \gk(\w_G, i(\chi)) \cdot \sum_{\w\in W}   (-1)^{l(\w \w_G)} \chi(\s_{\w(y)} ) \cdot \prod_{\alpha \in \Phi_{\w_G \w^{-1}}} \chi_\alpha  \\
=& \delta_B^{1/2}(\s_{y-z}) \cdot \gk(\w_G, i(\chi)) \cdot (-1)^{l(\w_G)} \cdot \sum_{\w\in W}   (-1)^{l(\w)} \chi(\s_{\w(y)} ) \cdot \prod_{\alpha \in \Phi_{\w_G \w^{-1}}} \chi_\alpha.
\end{aligned} $$
Note that $\Phi_{\w_G \w^{-1}}= \set{\alpha > 0: \w^{-1}(\alpha) > 0}$. Moreover,
$$\w(\rho_{Q,n}) + \rho_{Q,n}= \sum_{\alpha\in \Phi_{\w_G \w^{-1}}} \alpha^\vee_{Q,n}.$$
Thus,
$$\prod_{\alpha \in \Phi_{\w_G \w^{-1}}} \chi_\alpha=\chi(\s_{\w(\rho_{Q,n}) + \rho_{Q,n}}),$$
which gives the desired result by a simplification.
\end{proof}

The following result is then immediate from the above. It gives a natural analogue of the Casselman-Shalika formula \cite{CS} in the covering setting.

\begin{thm} \label{T:WV}
Let $\wt{G}$ be a Brylinski-Deligne cover whose exceptional points are $\alpha$-normal for all $\alpha\in \Delta$.
Let $z\in Y_n^{\rm exc}$ be an exceptional point. Assume $\bbmu_{2n} \subset F^\times$.  Then
$$\mca{W}_{\s_z}(\s_{-z}) = \delta_B^{1/2}(\s_{-z}) \cdot \prod_{\alpha>0} (1-q^{-1} \chi_\alpha).$$
\end{thm}
\begin{proof}
The result follows from the expansion of Weyl denominator, i.e.,
$$\prod_{\alpha >0} (1- \chi_\alpha)= \sum_{\w \in W} (-1)^{l(\w)} \prod_{\alpha\in \Phi_\w} \chi_\alpha.$$
Indeed, it follows from the proof of the Proposition \ref{P:valW} that
$$\begin{aligned}
\mca{W}_{\s_z}(\s_{-z}) & =  \delta_B^{1/2}(\s_{-z}) \cdot \gk(\w_G, i(\chi)) \cdot \sum_{\w\in W}   (-1)^{l(\w)}  \prod_{\alpha \in \Phi_{\w} } \chi_\alpha \\
& = \delta_B^{1/2}(\s_{-z}) \cdot \prod_{\alpha>0} (1-q^{-1} \chi_\alpha).
\end{aligned} $$
This completes the proof.
\end{proof}

It follows from Lemma \ref{L:exc} that there are many covering groups $\wt{G}$ where Theorem \ref{T:WV} is applicable. We also note that for Kazhdan-Patterson coverings of $\GL_r$, Theorem \ref{T:WV} recovers \cite[Proposition 3.4 (2)]{Suz2}.

\begin{rmk} \label{R:glob}
Theorem \ref{T:WV} implies that the $L$-functions appear as the Whittaker value only for a specific Whittaker function $\mca{W}_{\s_z}$  evaluated at a specific dominant torus element $\s_{-z}$, both of which are associated to the exceptional point $z$ of $\wt{G}$. For linear algebraic groups, one can take $z=0$ and thus Theorem \ref{T:WV} is nothing but the Casselman-Shalika formula for split $G$. However, for global problems pertaining to Whittaker functions of a global covering group, one must study the local Whittaker function evaluated at the identity $1_{\wt{G}}$. Hence, Theorem \ref{T:WV} does not yield any global assertions.
\end{rmk}

\begin{rmk} \label{R:whit}
One could consider a general
$$z\in f^{-1}_\msc{X}( \msc{X}_{Q,n}^W  ) \cap (-(Y^+)).$$
However, for such $z$ it is not clear that $\mca{W}_{\s_z}(\s_{-z})$ is related to $L$-functions. Indeed, for $z$ with $\hat{z} \in \msc{X}_{Q,n}^W - f_\msc{X}(Y_n^{\rm exe})$, one does not necessarily have Lemma \ref{L:CS}. This shows the difference between $f_\msc{X}(Y_n^{\rm exc})$ and $\msc{X}_{Q,n}^W$ in a potential generalization of the Casselman-Shalika formula to covering groups.
\end{rmk}

By the local Langlands correspondence for covering torus (see \cite{We6, GG}), $I(\chi)$ gives rise to a splitting ${}^L\bbrho:= {}^L\bbrho_\chi $ of the $L$-group ${}^L\wt{G}$ over $W_F$. Recall that (see \S \ref{SS:L-g}) by choosing a distinguished genuine character $\chi_\psi$, we have an isomorphism
$${}^L\wt{G} \simeq_{\chi_\psi} \wt{G}^\vee \times W_F.$$
As $\psi$ is normalized, the genuine character $\chi_\psi$ is unramified. We  use $\bbrho^\vee_\psi$ to denote the composition
$$\begin{tikzcd}
\bbrho_\psi^\vee: W_F \ar[r, "{ {}^L\bbrho  }"]  & {}^L\wt{G} \ar[r, two heads] & \wt{G}^\vee,
\end{tikzcd}$$
where the second map of projection depends on the choice of a distinguished genuine character $\chi_\psi$. The element $\bbrho_\psi^\vee(\varpi) \in \wt{T}^\vee$ is the Satake parameter for $I(\chi)$ relative to $\chi_\psi$.

In the remaining of this subsection, we assume that $G$ is semisimple. Together with the assumption (at the beginning of this subsection) that $\wt{G}$ is not of metaplectic type, one has
$$Y_n^{\rm exc} = \set{\rho - \rho_{Q,n}}  \cap Y$$
by Lemma \ref{L:exc}.

For any $y\in Y_{Q,n}^+$, let
$$\pi_y: \wt{G}^\vee  \to \GL(V_y)$$
be the irreducible representation of $\wt{G}^\vee$ with highest weight $y$.

\begin{prop} \label{P:WL}
For $y\in Y_{Q,n}^+$, we have
$$\chi_\psi(\s_y)  \cdot \Tr\left(\bbrho_\psi^\vee(\varpi); \pi_y \right)= \frac{ \sum_{\w\in W} (-1)^{l(\w)} \cdot \chi (\s_{\w(y+\rho_{Q,n}) + \rho_{Q,n}}) }{ \sum_{\w\in W} (-1)^{l(\w)} \cdot \chi(\s_{\w(\rho_{Q,n}) + \rho_{Q,n}}) },$$
which in particular is independent of the choice of $\chi_\psi$.
\end{prop}
\begin{proof}
Recall that the Weyl character formula (see \cite[Page 154]{Bou-Lie3}) for $(\pi_y, V_y)$ reads
$${\sf ch}(\pi_y)=\frac{ \sum_{\w\in W} (-1)^{l(\w)} e^{\w(y+\rho_{Q,n})} }{ \sum_{\w\in W} (-1)^{l(\w)} e^{\w(\rho_{Q,n})} }= \frac{ \sum_{\w\in W} (-1)^{l(\w)} e^{\w(y+\rho_{Q,n}) + \rho_{Q,n}} }{ \sum_{\w\in W} (-1)^{l(\w)} e^{\w(\rho_{Q,n}) + \rho_{Q,n}} }.$$
We have $\wt{T}^\vee=\Hom(Y_{Q,n}, \C^\times)$. By the local Langlands correspondence, the Satake parameter $\bbrho_\psi^\vee(\varpi) \in \wt{T}^\vee$ relative to $\chi_\psi$ is given by
$$\bbrho_\psi^\vee(\varpi)(y)=\left(\frac{\chi}{\chi_\psi}\right)(\s_y), \quad y\in Y_{Q,n},$$
where $\chi\cdot \chi_\psi^{-1}$ is a linear character of $Z(\wt{T})/\bbmu_n$, i.e., the image of $Z(\wt{T})$ in $T$ with respect to the natural projection $\wt{T} \onto T$. Since $\chi_\psi$ is Weyl-invariant, it follows that for every $\w\in W$,
$$\left(\chi\cdot \chi_\psi^{-1}\right)(\s_{\w(y+\rho_{Q,n}) + \rho_{Q,n}})= \chi_\psi^{-1}(\s_y) \cdot \chi(\s_{\w(y+\rho_{Q,n}) + \rho_{Q,n}}).$$
Therefore,
$$\begin{aligned}
& \Tr(\bbrho_\psi^\vee(\varpi); \pi_y)\\
= & \ \frac{ \sum_{\w\in W} (-1)^{l(\w)} \cdot (\chi \chi_\psi^{-1}) (\s_{\w(y+\rho_{Q,n}) + \rho_{Q,n}}) }{ \sum_{\w\in W} (-1)^{l(\w)} \cdot (\chi \chi_\psi^{-1})(\s_{\w(\rho_{Q,n}) + \rho_{Q,n}}) } \\
= &\ \chi_\psi^{-1}(\s_y) \cdot \frac{ \sum_{\w\in W} (-1)^{l(\w)} \cdot \chi (\s_{\w(y+\rho_{Q,n}) + \rho_{Q,n}}) }{ \sum_{\w\in W} (-1)^{l(\w)} \cdot \chi(\s_{\w(\rho_{Q,n}) + \rho_{Q,n}}) }.
\end{aligned} $$
This completes the proof.
\end{proof}

We denote:
\begin{equation} \label{E:Tr}
\Tr\left({}^L\bbrho(\varpi); \pi_y \right):=\chi_\psi(\s_y)  \cdot  \Tr\left(\bbrho_\psi^\vee(\varpi); \pi_y \right).
\end{equation}

\begin{thm} \label{T:CS}
Assume $\bbmu_{2n} \subseteq F^\times$ and also that $z:=\rho- \rho_{Q,n}$ lies in $Y$. Let $\mca{W}: \wt{G} \to i(\chi)$ be the unramified $i(\chi)$-valued Whittaker function associated to $I(\chi)$. Then for every $y\in Y_{Q,n}^+$, we have
\begin{equation} \label{E:CS}
\mca{W}(\s_{y+(-z)}) =
\delta_B^{1/2}(\s_{y}) \cdot \Tr\left({}^L\bbrho(\varpi); \pi_y \right) \cdot \mca{W}(\s_{-z}).
\end{equation}
\end{thm}
\begin{proof}
It suffices to show that the equality holds for $\mca{W}_{z'}$ for every $z'\in Y$. First, it follows from \eqref{P:Pat} that
$$\mca{W}_{z'}(\s_{y+(-z)})=0$$
unless $z'\in z+Y_{Q,n}$. Thus, the two sides of \eqref{E:CS} are both equal to zero if $z'\notin z+Y_{Q,n}$. Thus, it suffices to consider $\mca{W}_z$, for which the result follows from Proposition \ref{P:valW} and Proposition \ref{P:WL}.
\end{proof}

If $n=1$, then Theorem \ref{T:CS} recovers the classical reinterpretation of Casselman-Shalika formula using the theory of highest weight vectors, see \cite[Proposition 1]{Tam1}. Indeed,  we can take $z=0$ in this case, as mentioned in Remark \ref{R:glob}.

\begin{rmk}
It would be desirable to have a more intrinsic definition of $\Tr\left({}^L\bbrho(\varpi); \pi_y \right)$ instead of the seemingly ad hoc \eqref{E:Tr} above, provided that one could work with a proper analogous highest weight $y$ module for the ${}^L\wt{G}$ instead of just $\wt{G}^\vee$. However, we see that in general there is no canonical extension of $\pi_y$ to be a representation of ${}^L\wt{G}$.
More precisely, recall that by definition
$${}^L\wt{G}=(\wt{G}^\vee \times E)/ \{(z, z^{-1}): z\in Z(\wt{G}^\vee)\}  ,$$
where $Z(\wt{G}^\vee) \to E \to W_F$ is a certain extension. Any extension of $\pi_y$ takes the form $\pi_y\boxtimes \pi' $ where $\pi'$ is a representation of $E$ such that $\pi_y$ and $\pi'$ agree on $Z(\wt{G}^\vee)$. A distinguished character $\chi_\psi$ provides a splitting of $E$ over $W_F$ and thus an isomorphism $E\simeq_{\chi_\psi} Z(\wt{G}^\vee) \times W_F$ (which further gives rise to an isomorphism ${}^L\wt{G} \simeq_{\chi_\psi} \wt{G}^\vee \times W_F$); it gives a representation $\pi_\psi$ of $E$ relative to $\chi_\psi$ by considering the composition
$$E \to Z(\wt{G}^\vee) \to GL(V_y)$$
where the second map is $\pi_y$. It is easy to see that
$$\Tr({}^L\rho(\varpi); \pi_y\boxtimes \pi_\psi)= \Tr(\rho_\psi^\vee(\varpi); \pi_y),$$
which is dependent on $\chi_\psi$. On the other hand, the value of $\mca{W}$ is a priori independent of any chosen distinguished character $\chi_\psi$, although it does depend on $\chi$ and $\psi$. Thus, the appearance of $\chi_\psi(\s_y)$ in \eqref{E:Tr} seems to be both expected (as to cancel the dependence of $\Tr(\rho_\psi^\vee(\varpi); \pi_y)$ on $\chi_\psi$), but also at the same time intriguing, as we do not have a natural interpretation of this value.
\end{rmk}

\subsection{Relation with other formulation} \label{SS:other-wk}
 The above Casselman-Shalika formula could be deduced from other formulations \cite{CO, Mc2, PP1} of the Whittaker function, which are essentially the same as ours. The twisted Weyl action $\w[-]$ in our paper is incarnated by the Chinta-Gunnells action $\w \star e^y$ of $W$ on $\C[Y]$ (see \cite{CG, PP1}), and thus the set $Y_n^\exc$ should play equally important role in their formulation.  Let $z \in Y_n^\exc$ be an exceptional point. The key observation is that, using notation of \cite{PP1}, one has
$$\w \star e^{-z} = e^{-z} \in \C[Y]$$
for all $\w\in W$, which follows from a straightforward computation using the formula in \cite[(4.7)]{PP1}. Now for every $y\in Y_{Q,n}^+$, by \cite[Lemma 4.3]{PP1}, we have
\begin{equation} \label{E:star}
\w \star (e^{y+ (-z)})=e^{\w(y)} \cdot e^{ -z }.
\end{equation}
Let $\mca{W}(\pi^{y})$ be the $\C[Y]$-valued universal Whittaker function given in \cite{PP1}. Then it follows from \eqref{E:star} and \cite[Corollary 5.6]{PP1} that an analogue of Theorem \ref{T:CS} holds for the universal Whittaker function $\mca{W}(\pi^y)$.

\subsection{Induction from general parabolic subgroup}

In this subsection, we carry out some preliminary study on induction for a general parabolic subgroup. We only consider the case when the inducing representation of the Levi subgroup is a full unramified principal series.

Recalling the notation in \S \ref{SS:lcm}, let $\wt{P}=\wt{M} N$ (resp. $\wt{P}'= \wt{M}' N'$) be the parabolic subgroup associated with the subset $\theta \subset \Delta$ (resp. $\theta'\subset \Delta$). We have the set
$$W^{\theta,\theta'} =\set{\w\in W:  \w(\theta)= \theta'} \subset W.$$
We say that $\wt{P}$ and $\wt{P}'$ are associated if $W^{\theta, \theta'} \ne \emptyset$. Note that for every $\theta$, $W^{\theta, \theta} \subset W$ is a subgroup. Call $\wt{P}$ self-associated if $W^{\theta, \theta} \ne \set{\text{id}}$.
If $P$ is associated to $P'$ by $\w\in W$, then $\wt{M}' = w \wt{M} w^{-1}$ and  one could consider the intertwining operator
$$T(w, \sigma): I_{\wt{P}}^{\wt{G}}(\sigma)\to  I_{\wt{P}'}^{\wt{G}} ({}^w\sigma)$$
 of parabolically induced representations from $P$ and $P'$. If $\sigma= \Ind_{\wt{B}_M}^{\wt{M}} i(\chi)$ is a full principal series representation, then by induction in stages, we have the following commutative diagram
 \begin{equation} \label{CD:is}
 \begin{tikzcd}
 I(i(\chi))  \ar[r, "{T(w, i(\chi))}"]   & I({}^w i(\chi))  \ar[r, "{r_w}"]  &  I( i({}^\w \chi)) \\
I_{\wt{P}}^{\wt{G}}(\sigma)  \ar[u, equal]   \ar[r, "{T(w, \sigma)}"]    &  I_{\wt{P}'}^{\wt{G}} ({}^w\sigma)   \ar[u, "{j_\sigma}"] .
 \end{tikzcd}
 \end{equation}
 Here $j_\sigma$ is the isomorphism induced from the isomorphism
 $${}^w \big(\Ind_{\wt{B}_M}^{\wt{M}} i(\chi)  \big) \to \Ind_{\wt{B}_{M'}}^{\wt{M}'} {}^w i(\chi)$$
 given by $f \mapsto \tilde{f}(g) =  f(w^{-1} g w)$ for all $g\in \wt{M}'$.

Furthermore, in view of the factorization of the $T(w,\sigma)$ into rank-one intertwining map (cf. \cite[Theorem 2.1.1]{Sha2}), which corresponds to the factorization of $T(w, i(\chi))$, one needs to consider the group
$$W(\theta, \theta'):=\left\langle \w_\alpha:  \alpha \in \Delta(\w) \text{ for some } \w\in W^{\theta, \theta'}  \right\rangle.$$
In the covering groups setting, we would like to consider the set
$$(\msc{X}_{Q,n})^{W(\theta, \theta')};$$
every element $y$ in the set will give rise to an exceptional subspace
$$E(y, {}^w\sigma):=(r_w \circ j_\sigma)^*(E(y, {}^w\sigma)) \subset \Wh_\psi(I({}^w \sigma)),$$
where $E(y, {}^w\sigma) \subseteq \Wh(I(i({}^\w \chi)))$ is the exceptional space discussed in \S \ref{SS:el}. The space $E(y, {}^w\sigma)$ has a distinguished basis vector given by
$$\lambda(\s_y, {}^w \sigma): = (r_w \circ j_\sigma)^*(\lambda_{\s_y}^{ {}^\w \chi  }),$$
where $\lambda_{\s_y}^{ {}^\w \chi  } \in \Wh(I({}^\w \chi))$ is as in \S \ref{SS:el}.

A priori, $W(\theta, \theta') \subset W$ might be just a proper subgroup and thus there might be more exceptional subspaces for a general parabolic group than the Borel subgroup. However, we show  that  space $(\msc{X}_{Q,n})^{W(\theta, \theta')}$ is always either $\msc{X}_{Q,n}$ or $\msc{X}_{Q,n}^W$. For simplicity, we assume that $\wt{P}$ is self-associated.

\begin{lm} \label{P:genP}
Assume $\wt{P}$ is self-associated. For every $\theta'\subset \Delta$, we have either $W(\theta, \theta')= \set{ {\rm id} }$ or $W(\theta, \theta')= W$.
\end{lm}
\begin{proof}
If $\theta = \Delta$, then $W(\theta, \theta')=\set{\text{id}}$ for every $\theta' \subset \Delta$.   It suffices to prove the case where $\theta \ne \Delta$, and thus we make this assumption. Let $\w_G$ and $\w_M$ be the longest element in the Weyl group for $(\mbf{G}, \mbf{T})$ and $(\mbf{M}, \mbf{T})$ respectively. Denote
$$\w_l= \w_G \w_M.$$
Since $P$ is self-associated, $\w_l \in W^{\theta, \theta}$ and it follows from \cite[Page 126 Proposition 2]{Bou} and \cite[Page 151 Ex. 2(b)]{Bou} that every $\w_\alpha, \alpha\in \Delta$ appears in $\w_l$. Thus in this case, $W(\theta, \theta)=W$.

For general $\theta'\subset \Delta$, if $W^{\theta, \theta'}= \emptyset$, then $W(\theta, \theta')= \set{ {\rm id} }$. Now suppose $W^{\theta,\theta'} \ne \emptyset$, then by picking $\w_0 \in W^{\theta, \theta'}$, the map of sets
$$W^{\theta, \theta} \to W^{\theta, \theta'} \text{ given by } \w \mapsto \w_0 \cdot \w,$$
endows $W^{\theta, \theta'}$ with the structure as a $W^{\theta, \theta}$-torsor. For every $\w\in W^{\theta, \theta}$, one has
$$\Delta(\w)= \Delta(\w_0^{-1} \cdot \w_0 \w) \subseteq \Delta(\w_0^{-1}) \cup \Delta(\w_0 \w) = \Delta(\w_0) \cup \Delta(\w_0 \w),$$
and therefore
$$W(\theta, \theta) \subseteq W(\theta, \theta').$$
Thus, $W(\theta, \theta')= W$ in this case. This completes the proof.
\end{proof}

In any case, we consider a general $\wt{P}= \wt{M} N$. Let $\sigma=I_{\wt{B}_M}^{\wt{M}}(\chi)$ be a genuine principal series on $\wt{M}$. Let $\w \in W^{\theta, \theta'}$ and
$$T(w,\sigma)^*:  \Wh_\psi( I_{\wt{P}'}^{\wt{G}} ({}^w \sigma) ) \to \Wh_\psi( I_{\wt{P}}^{\wt{G}} ( \sigma) ) $$
be the map induced from $T(w, \sigma)$. Suppose $Y_n^{\rm exc}$ is nonempty and we choose $y\in Y_n^{\rm exc}$.  As discussed above, we have one-dimensional exceptional subspaces
$$E(y; \sigma)=\C \cdot \lambda(\s_y, \sigma)  \subset \Wh_\psi(I_{\wt{P}}^{ \wt{G} } (\sigma))$$
and
$$E(y; {}^w \sigma)=\C \cdot \lambda(\s_y, {}^w\sigma)  \subset \Wh_\psi(I_{\wt{P}'}^{ \wt{G} } ({}^w\sigma)  ).$$

The restriction 
$$T(w,\sigma)^*: E(y, {}^w \sigma) \to E(y, \sigma)$$
 is well-defined and is determined by a scalar with respect to the two basis $\lambda(\s_y, \sigma)$ and $\lambda(\s_y, {}^w\sigma)$, which we will interpret as a product of gamma-factors. For simplicity, we consider the case where $\wt{P}$ is a maximal parabolic subgroup. In this case there exists a unique $\w_0 \in W$ such that $\theta':=\w_0(\theta) \subset \Delta$ and $\w_0(\Delta\backslash \theta) \subset \Phi^-$.

In order to give an interpretation of the result relating to $L$-functions, we consider the adjoint action
$$Ad_{\wt{\mfr{n}}^\vee}: {}^L\wt{M} \to GL(\wt{\mfr{n}}^\vee),$$
where $\wt{\mfr{n}}^\vee$ is the Lie algebra of $\wt{N}^\vee$ such that $\wt{M}^\vee \wt{N}^\vee$ is a parabolic subgroup of $\wt{G}^\vee$. It factors through $Ad^\C_{\wt{\mfr{n}}^\vee}$:
$$\begin{tikzcd}
 {}^L\wt{M}  \rar["{Ad_{\wt{\mfr{n}}^\vee}}"] \dar    & GL(\wt{\mfr{n}}^\vee) \\
 \wt{M}^\vee/Z(\wt{G}^\vee) \ar[ru, "{Ad^\C_{\wt{\mfr{n}}^\vee}}"'],
\end{tikzcd}$$
as $Z(\wt{G}^\vee)$ acts trivially on $\wt{\mfr{n}}^\vee$.

Therefore, irreducible subspaces of $\wt{\mfr{n}}^\vee$ for $Ad_{\wt{\mfr{n}}^\vee}$ are in correspondence with irreducible subspaces with respect to $Ad^\C_{\wt{\mfr{n}}^\vee}$, which are familiar (see \cite{Lds} or \cite[\S 4]{Sha88}). More precisely, we consider the decomposition of $Ad_{\wt{\mfr{n}}^\vee}$ into irreducible representations
$Ad_{\wt{\mfr{n}}^\vee}=\bigoplus_{i=1}^m Ad_i$, where $(Ad_i, V_i) \subseteq \wt{\mfr{n}}^\vee$ is an irreducible space. Here, as observed by Langlands, $V_i$ is given by
\begin{equation} \label{E:Vi}
V_i=\bigoplus_{\langle \omega_P/n_\beta, \alpha_{Q,n}^\vee \rangle=i } \C \cdot E_{\alpha_{Q,n}^\vee},
\end{equation}
where $E_{\alpha_{Q,n}^\vee} \in \wt{\mfr{n}}^\vee$ is the pinning associated to the root $\alpha_{Q,n}^\vee$ of $\wt{G}^\vee$. Moreover the following equality holds:
\begin{equation} \label{Phi-w}
\Phi_{\w_0}=\bigsqcup_{i=1}^m \set{ \alpha\in \Phi^+: \angb{\omega_P/n_\beta}{\alpha^\vee_{Q,n}}=i }.
\end{equation}

Recall that  $L(s, \sigma, Ad_i)$ is by definition equal to the local Artin $L$-function $L(s, Ad_i \circ {}^L\bbrho_\chi)$ associated with $Ad_i \circ {}^L\bbrho_\chi$, i.e.,
$$L(s, \sigma, Ad_i):=\text{det}\big(1-q^{-s} Ad_i \circ {}^L\bbrho_\chi (\text{Frob})|_{V_i^I} \big)^{-1},$$
where we also write ${}^L\bbrho_\chi$ for the composition $\W_F \to {}^L\wt{T} \to {}^L\wt{M}$. For unramified $\sigma$ (i.e. equivalently, unramified $\chi$), the inertia group $I$ acts trivially on $V_i$. It follows
$$L(s, \sigma, Ad_i)=\text{det}\big(1-q^{-s} Ad_i \circ {}^L\bbrho_\chi(\varpi)|_{V_i}\big)^{-1}.$$
We also define
$$\gamma(s, \sigma, Ad_i):= \frac{ L(1-s, \sigma, Ad_i^\vee) }{ L(s, \sigma, Ad_i)  },$$
where $Ad_i^\vee$ is the contragredient representation of $Ad_i$.

\begin{thm} \label{T:naiveLS}
Retain notations as above. Assume that $Y_n^{\rm exc}\ne \emptyset$, and $Y_\alpha^{\rm spe} = \emptyset$ for all $\alpha\in \Delta$. Fix any $y\in Y_n^{\rm exc}$. Then,
 $$T(w_0,\sigma)^*( \lambda(\s_y, {}^{w_0} \sigma) ) = \lambda(\s_y, \sigma) \cdot \prod_{i=1}^m \gamma(0, \sigma, Ad_i)^{-1}.$$
 \end{thm}
\begin{proof}
Let $\w_0=\w_l... \w_2 \w_1$ be a minimal decomposition of $\w_0$, where $\w_i$ is associated with $\alpha_i\in \Delta$. By our assumption that $Y_\alpha^{\rm spe}=\emptyset$ for all $\alpha\in \Delta$ and Proposition \ref{P:prod-ga},
$$\begin{aligned}
& T(w_0,\sigma)^*( \lambda(\s_y, {}^{w_0} \sigma) ) \\
= & \lambda(\s_y, \sigma) \cdot \prod_{i=1}^l \gamma(\w_i, i({}^{\w_{i-1} ... \w_1} \chi))^{-1} \\
=& \lambda(\s_y, \sigma) \cdot \prod_{\alpha \in \Phi_{\w_0}} \gamma(0, \uchi_\alpha )^{-1} ,
\end{aligned}
$$
where $$\gamma(0, \uchi_\alpha)^{-1} = \frac{ 1-q^{-1}  \chi( \wt{h}_\alpha(\varpi^{n_\alpha})  )^{-1} }{1- \chi( \wt{h}_\alpha(\varpi^{n_\alpha})  ) }.$$
It follows from \cite[Theorem 7.8]{Ga1} that
$${}^L\bbrho_\chi \circ Ad(\varpi) (E_{\alpha_{Q,n}^\vee}) = \chi(\wt{h}_\alpha(\varpi^{n_\alpha})) \cdot E_{\alpha_{Q,n}^\vee}.$$
Thus the result follows in view of \eqref{E:Vi} and \eqref{Phi-w}.
\end{proof}

\begin{rmk} \label{R:gamma}
As noted in Remark \ref{R:whit}, one could consider more generally an element $y \in f^{-1}_\msc{X}( \msc{X}_{Q,n}^W  )$. In this case, an analogous formula as in Theorem \ref{T:naiveLS} holds:
$$T(w_0,\sigma)^*( \lambda(\s_y, {}^{w_0} \sigma) ) = \lambda(\s_y, \sigma) \cdot J(\sigma, y, \w_0) \cdot  \prod_{i=1}^m \gamma(0, \sigma, Ad_i)^{-1},$$
where however $J(\sigma, y, \w_0) \in \C^\times$ is not necessarily equal to 1. Such complication arises from the exponents of $\chi(\wt{h}_\alpha(\varpi^{n_\alpha}))$ in Proposition \ref{P:nor}.
\end{rmk}

\section{The fundamental group $\pi_1(\wt{G}_{Q,n}^\vee)$ and periodicity of $b_{W,n}$} \label{S:Poinc}

In this section, we consider  an $n$-fold cover $\wt{G}$ of a general linear reductive group $G$, unless specified otherwise.  In view of the discussion from the preceding section, we want to study the number $$b_{W,n}:=\val{\msc{X}_{Q,n}^W}.$$
We will give upper bounds for $b_{W,n}$, and prove the periodicity of $b_{W,n}$ as a function of $n$, if $G$ is semisimple.

\subsection{Bounds and periodicity of $b_{W,n}$}
Recall the map
$$\phi_{W,n}:  (\msc{X}_{Q,n}^{sc})^{W} \to (\msc{X}_{Q,n})^{W}$$
induced from the natural quotient
$$\phi: \msc{X}_{Q,n}^{sc} \onto \msc{X}_{Q,n}.$$
Denote by $Y_\rho$ the affine lattice
$$Y_\rho:= Y - \rho \subset  Y\otimes \Q.$$
If $(\msc{X}_{Q,n})^W=\emptyset$, then necessarily $(\msc{X}_{Q,n}^{sc})^W=\emptyset$.

\begin{lm} \label{L:fob-t}
If $(\msc{X}_{Q,n})^W \neq \emptyset$ and $(\msc{X}_{Q,n}^{sc})^W \neq \emptyset$, then $\phi_{W,n}$ is surjective with fiber being a torsor over $Z(\wt{G}^\vee_{Q,n})$.
\end{lm}
\begin{proof}
Recall $Z(\wt{G}_{Q,n}^\vee) = \Hom(Y_{Q,n}/Y_{Q,n}^{sc}, \C^\times)$, which is (non-canonically) isomorphic to ${\rm Ker}(\phi) = Y_{Q,n}/Y_{Q,n}^{sc}$. Also, for every $z\in {\rm Ker}(\phi)$ and $y\in \msc{X}_{Q,n}^{sc}$, one has
$$\w[y+ z] = \w[y] + \w(z) = \w[y] + z.$$
In view of this, it suffices to prove the surjectivity.

Now for every $\alpha\in \Delta$, one has the following commutative diagram of sets
$$\begin{tikzcd}
( \msc{X}_{Q,n}^{sc} )^W  \ar[r, "{\phi_{W,n}}"]    \ar[d, hook]   &       (\msc{X}_{Q,n})^W   \ar[d, hook] \\
( \msc{X}_{Q,n}^{sc} )^{W_\alpha}  \ar[r, "{\phi_\alpha}"]     &       (\msc{X}_{Q,n})^{W_\alpha}.
\end{tikzcd}$$
The assumption implies that the domain and codomain of $\phi_\alpha$ are both nonempty. Thus, by Corollary \ref{C:dich}, $\phi_\alpha$ is surjective. That is, for any  $y+ Y_{Q,n} \in  (\msc{X}_{Q,n})^{W_\alpha}$, one has  $y+ Y_{Q,n}^{sc} \in (\msc{X}_{Q,n}^{sc})^{W_\alpha}$ as well.

Since the above holds for all $\alpha\in \Delta$, it follows that if $y+ Y_{Q,n} \in  (\msc{X}_{Q,n})^W$, one has  $y+ Y_{Q,n}^{sc} \in (\msc{X}_{Q,n}^{sc})^W$. This proves the desired surjectivity.
\end{proof}

Let $L\subset Y\otimes \Q$ be a lattice and $S \subset Y\otimes \Q$ a subset. We denote by
$$S/L \subset (Y\otimes \Q)/L$$
the set of left cosets of $S$ in $(Y\otimes \Q)/L$. Denote by
$$P(Y_{Q,n}^{sc}) \subset Y_{Q,n}^{sc} \otimes_\Z \Q$$
 the dual of the lattice $X_{Q,n}^{sc}$ with respect to the pairing $\angb{-}{-}$ on $(X\otimes \Q) \times (Y\otimes \Q)$.
Thus, $P(Y_{Q,n}^{sc})$ is the weight lattice associated to the root lattice $Y_{Q,n}^{sc}$, as $\angb{-}{-}$ gives the natural $\Z$-valued pairing between $X_{Q,n}^{sc}$ and $Y_{Q,n}^{sc}$.

\begin{prop} \label{P:period}
Let $\mbf{G}$ be a semisimple group.  Assume $(\msc{X}_{Q,n})^W \neq \emptyset$ (see an elaboration on this condition in \S \ref{S:eg}).
\begin{enumerate}
\item[(i)] If $\phi_{W,n}$ is surjective, then there is a bijection between $(\msc{X}_{Q,n}^{sc})^{W}$ and
$(Y_\rho \cap P(Y_{Q,n}^{sc}))/Y_{Q,n}^{sc}$. In this case,
$$b_{W,n}= \val{ \frac{ Y_\rho \cap P(Y_{Q,n}^{sc}) }{ Y_{Q,n}  }   }.$$
%where $S/Y_{Q,n} \subset (Y\otimes \Q)/Y_{Q,n}$ denotes the cosets for any $S\subset Y\otimes \Q$.
\item[(ii)] If  $\mbf{G}$ is almost-simple and ${\rm domain}(\phi_{W,n}) = \emptyset$, then $\wt{G} = \wt{\Sp}_{2r}^{(n)}$ and is of metaplectic type; in this case,
$$b_{W, n}=1.$$
\end{enumerate}
\end{prop}
\begin{proof}
%We first consider the first statement, in this case,
%$$b_{W,n}= \val{ (\msc{X}_{Q,n}^{sc})^W } \cdot \val{Y_{Q,n}/Y_{Q,n}^{sc}}^{-1}.$$
First, we show (i). The map $y\mapsto y-\rho$ gives a bijection between two cosets in $(Y\otimes \Q)/ Y_{Q,n}^{sc}$:
$$ \msc{X}_{Q,n}^{sc} \leftrightarrow  Y_\rho /Y_{Q,n}^{sc}.$$
This bijection is equivariant with respect to the $\w[\cdot]$-action on the left and the usual $\w(\cdot)$-action on the right.

Now
\begin{equation} \label{key-arg}
\begin{aligned}
%& y \text{ is such that } y+Y_{Q,n}^{sc} \text{ lies in } (\msc{X}_{Q,n}^{sc})^{W} \\
      & \w(y_\rho) \equiv y_\rho \mod Y_{Q,n}^{sc} \text{ for all } \w \in W \\
\iff & y_\rho \in Y\otimes \Q  \text{ is a special point for the affine Weyl group } W_\text{aff} = W \ltimes Y_{Q,n}^{sc} \\
\iff & y_\rho \in P(Y_{Q,n}^{sc}) \text{ by \cite[Page 187, Proposition 3]{Bou} } .
%\iff & y \in Y\cap (P(Y_{Q,n}^{sc}) + \rho).
\end{aligned}
\end{equation}
Therefore, we have a bijection
$$ (\msc{X}_{Q,n}^{sc})^W \leftrightarrow  \frac{ Y_\rho \cap P(Y_{Q,n}^{sc}) }{Y_{Q,n}^{sc}}.$$
By Lemma \ref{L:fob-t}, it gives a bijection
$$(\msc{X}_{Q,n})^W \leftrightarrow  \frac{ Y_\rho \cap P(Y_{Q,n}^{sc}) }{Y_{Q,n}}.$$
This proves the first statement.

Now we show (ii). By assumption, we pick $y$ such that $y+Y_{Q,n} \in (\msc{X}_{Q,n})^W$.
Since $(\msc{X}_{Q,n}^{sc})^W=\emptyset$, there exists $\alpha\in \Delta$ such that  $y+ Y_{Q,n}^{sc} \notin (\msc{X}_{Q,n}^{sc})^W$. Thus the first statement in (ii) follows from Proposition \ref{P:dich}, we have $\wt{G}= \wt{\Sp}_{2r}^{(n)}$ and it is of metaplectic-type. In particular, $n_\alpha =2m$ with $m$ odd for the unique short simple coroot $\alpha$ of $\Sp_{2r}$.

In this case, a simple computation gives that
$$Y_{Q,n}= m\cdot Y^{sc} ,$$
which, by the above argument for \eqref{key-arg} (relying on the fact that $mY^{sc}$ is a coroot lattice), gives bijections
$$(\msc{X}_{Q,n})^W \leftrightarrow \frac{Y^{sc}_\rho \cap P( m Y^{sc})  }{ m Y^{sc} } \leftrightarrow   \frac{ (m^{-1} \cdot Y^{sc}_\rho) \cap P(Y^{sc})  }{ Y^{sc} } .$$
We observe that
$$\rho \in (m^{-1} \cdot Y^{sc}_\rho) \cap P(Y^{sc});$$
that is, the set $(\msc{X}_{Q,n})^W$ is nonempty and $b_{W,n}\ge 1$. Moreover, we will show that if for $i=1, 2$ we have
$$y_i \in (m^{-1} \cdot Y^{sc}_\rho) \cap P(Y^{sc}),$$
then $y_1 - y_2 \in Y^{sc}$. For this, we have $y_i \in P(Y^{sc})$ and $m y_i + \rho \in Y^{sc}$; it gives that
$$\frac{m-1}{2} \cdot 2(y_1 - y_2) + (y_1 -y_2) \in Y^{sc}$$
and therefore $y_1 - y_2 \in Y^{sc}$. This shows that $b_{W,n}=1$ and the proof is completed.
\end{proof}

Assume that $\mbf{G}$ is semisimple. Let $\pi_1(\wt{G}_{Q,n}^\vee)$ be the fundamental group of the algebraic dual group $\wm{G}_{Q,n}^\vee$ (see \cite[Definition 18.7]{Mil-B}). We have
$$\pi_1(\wt{G}_{Q,n}^\vee) \simeq P(Y_{Q,n}^{sc})/Y_{Q,n}.$$

\begin{thm} \label{T:pi1}
Let $\mbf{G}$  be a semisimple group. Then the following holds.
\begin{enumerate}
\item[(i)] For every semisimple $\mbf{G}$, one has $b_{W,n} \le \val{\pi_1(\wt{G}_{Q,n}^\vee)}$.
\item[(ii)] If $\mbf{G}$ is simply-connected, then $b_{W,n} \le 1$ for all $n$.
\item[(iii)] If $\rho\in Y$,  then for every $n$ such that $X \subset X_{Q,n}^{sc}$ we have $b_{W,n}= \val{ \pi_1(\wt{G}_{Q,n}^\vee) }$. In particular, if $\mbf{G}= \mbf{G}_{ad}$, then we always have $b_{W,n}= \val{\pi_1(\wt{G}_{Q,n}^\vee) }$ for every $n$.
\end{enumerate}
\end{thm}
\begin{proof}
For (i), if the domain of $\phi_{W,n}$ is empty, then by Proposition \ref{P:period}, $b_{W,n}=1$.
Now assume $\phi_{W,n}$ is surjective, again by Proposition \ref{P:period},
$$b_{W,n}\le \val{ \frac{P(Y_{Q,n}^{sc}) }{Y_{Q,n}} } = \val{\pi_1(\wt{G}_{Q,n}^\vee)}.$$

For (ii), assume $\mbf{G}$ is simply-connected. It suffices to show that if
$$y_i \in Y_\rho \cap P(Y_{Q,n}^{sc}), i=1, 2,$$
then $y_1-y_2 \in Y_{Q,n}$.
Since $y_1 - y_2 \in Y\cap P(Y_{Q,n}^{sc})$, for every $\alpha^\vee$ we have
$$B(y_1- y_2, \alpha^\vee) = Q(\alpha^\vee) \cdot \angb{y_1 - y_2}{\alpha}= Q(\alpha^\vee) n_\alpha \cdot \angb{y_1 - y_2}{n_\alpha^{-1}\alpha} \in n\Z,$$
where the last step follows from the fact that $y_1-y_2\in P(Y_{Q,n}^{sc})$. Since the coroots span $Y=Y^{sc}$, we see $y_1- y_2 \in Y_{Q,n}$. This completes the proof for (ii).

We show the last assertion (iii). Assume $\rho \in Y$, then $(\wm{G}, n)$ is not of metaplectic-type for all $n$. Therefore, $\phi_{W,n}$ is surjective for every $n$. Moreover, if $X \subset X_{Q,n}^{sc}$, then $P(Y_{Q,n}^{sc}) \subset Y$. By the first part of Proposition \ref{P:period}, one has
$$b_{W,n}=\val{ \frac{Y\cap P(Y_{Q,n}^{sc})}{ Y_{Q,n}} }= \val{ P(Y_{Q,n}^{sc})/Y_{Q,n} }= \val{ \pi_1(\wt{G}_{Q,n}^\vee) }.$$
If $\mbf{G}$ is adjoint, then $X=X^{sc} \subset X_{Q,n}^{sc}$ for every $n$. This completes the proof of (iii).
\end{proof}

\begin{thm} \label{T:bP}
Let $\mbf{G}$ be almost-simple and $Q$ a Weyl invariant quadratic form on $Y$. Then there exists a natural number $c:=c(\mbf{G}, Q)$ depending only on $\mbf{G}$ and $Q$ such that
$$b_{W,n} = b_{W, n+c}$$
for every $n\in \N$.
\end{thm}
\begin{proof} If $Q=0$, then $b_{W,n}=1$ for all $n$. Clearly the statement holds, by taking $c=1$. Now we assume that $Q\ne 0$, i.e., $Q(\alpha^\vee) \ne 0$ for every coroot $\alpha^\vee$. We consider two cases  of $\mbf{G}$, depending on whether it is equal to $\mbf{Sp}_{2r}$ or not.

The first case, $\mbf{G} \ne \mbf{Sp}_{2r}$, and we fix the quadratic form $Q$. In this case, $\phi_{W,n}$ is surjective for every $n$. From the proof of Proposition \ref{P:period}, we see that
$$b_{W,n}= \val{ \frac{ Y_\rho \cap P(Y_{Q,n}^{sc}) }{ Y_{Q,n}^{sc}  }   } \cdot \val{Y_{Q,n}/Y_{Q,n}^{sc}}^{-1}.$$
We note that $\val{ Z(\wt{G}_{Q,n}^\vee) }= \val{Y_{Q,n}/Y_{Q,n}^{sc}}$, which by Proposition \ref{P:DG}  is periodic with respect to some $c_1:=c_1(\mbf{G},Q)$ . On the other hand, as in the proof of Proposition \ref{P:DG}, we see that
$$Y_{Q,n}^{sc}=
\begin{cases}
n_\alpha \cdot Y^{sc}  & \text{ if } n\in \N_{\mbf{G}, Q}, \text{ and } \\
n_\alpha \cdot Y^{sc}_\flat & \text{ if } n\in \N-\N_{\mbf{G}, Q};
\end{cases}$$
where $\alpha^\vee$ is any long simple coroot. (If $\mbf{G}$ is simply-laced, we have $\N_{\mbf{G}, Q}= \N$.)  We have
$$S(n):=\frac{ Y_\rho \cap P(Y_{Q,n}^{sc}) }{ Y_{Q,n}^{sc}  }=
\begin{cases}
\frac{ (n_\alpha^{-1} \cdot Y_\rho) \cap P(Y^{sc})  }{Y^{sc}} & \text{ if } n \in \N_{\mbf{G}, Q}, \\
\frac{ (n_\alpha^{-1} \cdot Y_\rho) \cap P(Y^{sc}_\flat)  }{Y^{sc}_\flat} & \text{ if } n \in \N-\N_{\mbf{G}, Q}.
\end{cases}$$
The same argument in Proposition \ref{P:DG} coupled with Lemma \ref{L:key}  shows that there is $c_2:=c_2(\mbf{G}, \Q)$ such that $S(n+ c_2)= S(n)$ for all $n$. Taking $c={\rm lcm}(c_1, c_2)$ gives that $b_{W, n}= b_{W, n+c}$ for all $n$.

Now, we treat the case $\mbf{G}=\mbf{Sp}_{2r}, r\ge 1$ with  a fixed quadratic form $Q$ on $Y=Y^{sc}$. Let $\alpha_r^\vee$ be the unique short simple coroot; define
$$\N_Q= \set{n: 2| n_{\alpha_r} \text{ but } 4 \nmid n_{\alpha_r} } \subset \N.$$
We see that ${\rm domain}(\phi_{W,n})= \emptyset$ if and only if $n\in \N_Q$. By Proposition, $b_{W,n}=1$ for all $n\in \N_Q$.

Define $c':=4 \cdot Q(\alpha_r^\vee)$. We claim that $n+ c' \in \N_Q$ for all $n \in \N$. For this, we compute
$$(n + c')_{\alpha_r}=\frac{ n+ c'  }{ \text{gcd}(n+ c', Q(\alpha_r^\vee))} = \frac{ n+ 4Q(\alpha_r^\vee)  }{ \text{gcd}(n, Q(\alpha_r^\vee))} = n_{\alpha_r} + 4 \cdot \frac{Q(\alpha_r^\vee)}{ \text{gcd}(n, Q(\alpha_r^\vee)) } .$$
Since $Q(\alpha_r^\vee) \cdot \text{gcd}(n, Q(\alpha_r^\vee))^{-1}$ is odd, it follows that $(n+c')_{\alpha_r} \in \N_Q$. Therefore, $b_{W,n}= b_{W,n+c'}$ for all $n\in \N_Q$.

If $n\in \N- \N_Q$, then $\phi_{W,n}$ is surjective. The same argument as in the $\mbf{G} \ne \mbf{Sp}_{2r}$ case above shows that there exists $c'':=c''(\mbf{Sp}_{2r}, Q)$ such that $n + c'' \in \N- \N_Q$ and $b_{W, n}= b_{W, n + c''}$ for all $n \in \N-\N_Q$.

We take $c=\text{lcm}\set{c', c''}$; this completes the proof for $\mbf{G}= \mbf{Sp}_{2r}$.
\end{proof}

In view of (iii) of Corollary \ref{T:pi1}, we see that Proposition \ref{P:DG} and Theorem \ref{T:bP} are compatible, when $\wt{G}$ is a cover of an adjoint group $G$.

\subsection{Poincar\'e series for $b_{W,n}$} Let $\mbf{G}$ be a split connected reductive group and $(\wm{G}, n)$ a cover. The Poincar\'e series attached to $b_{W,n}$ is by definition
$$\mca{P}_{W}(T):= \sum_{n\ge 1} b_{W, n} T^n \in \Z[\![T]\!].$$
Analogously, define
$$\mca{P}_{\rm exc}(T):=\sum_{n=1} \val{ f_\msc{X}(Y_n^{\rm exc}) } \cdot T^n.$$
Note that if $G$ is a semisimple group, then
$$\val{ Y_{n}^{\rm exc} } = \val{ f_\msc{X}(Y_n^{\rm exc}) }.$$

\begin{cor} \label{C:rat}
The following two statements are equivalent:
\begin{enumerate}
\item[$\bullet$] There exists a natural number $c \ge 1$ (depending on $\mbf{G}, Q$ only) such that
$$b_{W, n} = b_{W, n+c}$$ for all $n\ge 1$.
\item[$\bullet$] There exists $k\in \N$ and $f(T) \in \Z[T]$ such that
 $$\mca{P}_{W}(T)=\frac{f(T)}{1-T^k};$$
 in particular, $\mca{P}_{W}(T)$ lies in $\Q(T)$ in this case.
\end{enumerate}
Moreover, if $\mbf{G}$ is an almost-simple group, then any of the above two statements (and thus both) holds for $\mca{P}_{W}(T)$; the analogous assertion also holds for $\mca{P}_{\rm exc}(T)$.
\end{cor}
\begin{proof}
The equivalence of the two statements is straightforward. The assertion on $\mca{P}_W(T)$ for almost-simple $\mbf{G}$ follows from Theorem \ref{T:bP}. The assertion on $\mca{P}_{\rm exc}(T)$ follows from Lemma \ref{L:exc}. More precisely, it suffices to show that  there exists $c=c(\mbf{G}, Q)$ such that for all $n$, one has
\begin{equation} \label{P-y}
\val{Y_n^{\rm exc}}  = \val{Y_{n+c}^{\rm exc}}
\end{equation}
for every $n$. However, it is easy to see that there exists $c'=c'(\mbf{G},Q)$ such that
$$\rho_{Q,n} - \rho_{Q, n + c'} \in Y$$
for every $n$. If $\mbf{G}= \mbf{Sp}_{2r}$, there also exists $c''=c''(\mbf{G}, Q)$ such that
$$(\rho_{Q,n} - \frac{ n_{\alpha} }{2} \omega_\alpha) - (\rho_{Q,n+c''} - \frac{ (n+c'')_{\alpha} }{2} \omega_\alpha) \in Y,$$
where $\alpha \in \Delta$ is the unique simple long root. In any case, in view of Lemma \ref{L:exc}, we see that there exists $c=c(\mbf{G}, Q)$ such that \eqref{P-y} holds and thus $\val{ Y_n^{\rm exc} }$ is periodic. This completes the proof.
\end{proof}

%%%
% \section{Restriction of principal series}

%%%
\section{Two Poincar\'e series} \label{S:eg}

In this section, we work out explicitly the $b_{W,n}$ and $Y_n^{\rm exc}$ for several family of covering groups. First note that
$$\val{ f_\msc{X}(Y_n^{\rm exc} ) } \le b_{W,n},$$
since $Y_n^{\rm exc} \subset f^{-1}_\msc{X}( \msc{X}_{Q,n}^W  )$. The set $\val{Y_n^{\rm exc}}$ measure the size of the exceptional points, which provide natural basis points for the Casselman-Shalika formula in view of Theorem \ref{T:WV}. Indeed, if $\wt{G}^{(n)}$ is such that $b_{W,n}=0$ (or just $Y_n^{\rm exc} = \emptyset$), we are not aware of a natural analogue of Theorems \ref{T:WV} and \ref{T:CS}. For instance, for the $4$-fold cover $\wt{\SL}_2$, one has $b_{W,4}=0$.

We will compute explicitly the two Poincar\'e series $\mca{P}_W(T)$ and $\mca{P}_{\rm exc}(T)$. For all the examples considered in this section, one has
$$\mca{P}_{\rm exc}^{\rm}(T), \mca{P}_W(T) \in \Q(T),$$
including certain covers of $\GL_r$. We conjecture that this is the case for covers of a general reductive group $G$.

\subsection{Almost-simple simply-connected groups} In this subsection, let $\mbf{G}$ be an almost-simple simply-connected group. We assume that
$$Q(\alpha^\vee)=1$$
 for every short coroot $\alpha^\vee$. This gives rise to a $\mbf{K}_2$-extension $\wm{G}$.

Let $\wt{G}^{(n)}$ be the $n$-fold cover arising from $(\wm{G}, n)$. Then the necessary and sufficient conditions for $b_{W,n}>0$ and $Y_n^{\rm exe} \ne \emptyset$, which is equivalent to $b_{W,n}=1$ and $\val{Y^\exc_n}=1$ respectively by Corollary \ref{T:pi1},  are tabulated as follows. For $\mbf{G}$ a classical group (resp. exceptional group), we list the conditions in Table 1 (resp. Table 2).
In the tables $r\equiv a$ represents $r \equiv a \mod 4$. Also, for every $n$ we denote by $J(n)$ the highest exponent of $2$ dividing $n$.

\begin{table}[!htbp]  \label{T1}
\caption{For classical groups }
\vskip 5pt

\begin{tabular}{|c|c|c|c|c|}
\hline
 & $\wt{A}_r^{(n)}, r\ge 2$  &  $\wt{C}_r^{(n)}, r\ge 1$ & $\wt{B}_r^{(n)}, r\ge 3$  & $\wt{D}_r^{(n)}, r\ge 3$ \\
\hline
$b_{W,n}=1$ & all $n$ if $r$ is even; &  $4\nmid n$ &  all $n$ if $r\equiv 0 $; & all $n$ if $r\equiv 0 $;  \\
 & $2^{J(n)}| \frac{(r+1)}{2}$ if $r$ is odd & & $4\nmid n$ if $r\equiv 1$; &  all $n$ if $r\equiv 1 $; \\
&&& odd $n$ if $r\equiv 2$; & odd $n$ if $r\equiv 2 $; \\
&&& all $n$ if $r\equiv 3$ & $4\nmid n$ if $r\equiv 3 $ \\
\hline
$\val{Y^\exc_n}=1 $ & all $n$ if $r$ is even; & $4\nmid n$ & all $n$ if $n\equiv 0$; & all $n$ if $r\equiv 0 $;   \\
 & $n$ is odd if $r$ is odd & & $4\nmid n$ if $r\equiv 1$ ;  & all $n$ if $r\equiv 1 $ ;   \\
&&& odd $n$ if $r\equiv 2$; & odd $n$ if $r\equiv 2 $ ;   \\
&&& odd $n$ or $4|n$ if $r\equiv 3$ & odd $n$ if $r\equiv 3 $   \\
\hline
\end{tabular}
\end{table}

\begin{table}[!htbp]  \label{T2}
\caption{For exceptional groups}
\vskip 5pt
\begin{tabular}{|c|c|c|c|c|c|}
\hline
 & $\wt{E}_6^{(n)}$  &  $\wt{E}_7^{(n)}$ & $\wt{E}_8^{(n)}$  & $\wt{F}_4^{(n)}$ &  $\wt{G}_2^{(n)}$ \\
\hline
$b_{W,n}=1$ & all $n$ & odd $n$  & all $n$ & all $n$  & all $n$ \\
 &  & & & & \\
\hline
$ \val{Y^\exc_n}=1 $ & all $n$  & odd $n$  & all $n$ & all $n$ & all $n$ \\
 & & & & & \\\hline
\end{tabular}
\end{table}

To facilitate the computation, we first show two useful results.

\begin{lm} \label{L:e1}
Let $\mbf{G}$ be a semisimple simply-connected group. If $\rho\in Y$ and the coroot system $\Phi_{Q,n}^\vee=\set{\alpha_{Q,n}^\vee}$ is the same type as $\Phi^\vee$, then automatically $\rho_{Q,n} \in Y_{Q,n}^{sc}$ and therefore $Y_n^{sc}=\set{\rho- \rho_{Q,n}}$; in particular, $\val{Y_n^{sc}}= b_{W,n}=1$  in this case.
\end{lm}
\begin{proof}
Under the assumption one has $\rho - \rho_{Q,n} \in Y$. Thus the statement is a direct consequence of Lemma \ref{L:exc}.
\end{proof}

\begin{lm} \label{L:e2}
Let $P(Y^{sc}) \subset Y^{sc}\otimes \Q$ be the weight lattice associated to the coroot lattice $Y^{sc}$. If $n\equiv 1 \mod [P(Y^{sc}): Y^{sc}]$, then
$$\val{ \frac{ (n^{-1}Y^{sc}_\rho) \cap P(Y^{sc}) }{ Y^{sc} }  } \le 1.$$
\end{lm}
\begin{proof}
Let $y_i \in (n^{-1} Y_\rho) \cap P(Y^{sc})$ for $i=1, 2$. Then $n(y_1-y_2) \in Y^{sc} \cap P(Y^{sc})$, and thus in $P(Y^{sc})$,
$$(y_1 - y_2) \equiv n(y_1 - y_2) \mod Y^{sc},$$
which follows from the assumption on $n$.
\end{proof}
\subsubsection{Type $A_r, r\ge 2$}

Consider the Dynkin diagram for the simple coroots of $A_r$:

$$\qquad
\begin{picture}(4.7,0.2)(0,0)
\put(1,0){\circle{0.08}}
\put(1.5,0){\circle{0.08}}
\put(2,0){\circle{0.08}}
\put(2.5,0){\circle{0.08}}
\put(3,0){\circle{0.08}}
\put(1.04,0){\line(1,0){0.42}}
\multiput(1.55,0)(0.05,0){9}{\circle{0.02}}
\put(2.04,0){\line(1,0){0.42}}
\put(2.54,0){\line(1,0){0.42}}
\put(1,0.1){\footnotesize $\alpha_{1}^\vee$}
\put(1.5,0.1){\footnotesize $\alpha_{2}^\vee$}
\put(2,0.1){\footnotesize $\alpha_{r-2}^\vee$}
\put(2.5,0.1){\footnotesize $\alpha_{r-1}^\vee$}
\put(3,0.1){\footnotesize $\alpha_r^\vee$}
\end{picture}
$$
\vskip 10pt

By Lemma \ref{L:exc}, $Y^\exc_n \ne \emptyset$ if and only if $\rho -\rho_{Q,n} =(1-n)\rho$ lies in $Y$. We discuss the two cases depending on the parity of $r$.

First, if $r\ge 2$ is even, then $\rho\in Y$. Thus $\rho- \rho_{Q,n} \in Y$ and it is the only exceptional point in $Y^\exc_n$. In this case
$$\val{Y_n^\exc}= 1 = b_{W,n}.$$

Now we assume that $r\ge 3$ is odd. Then $\rho \notin Y$, thus $Y_n^{sc} \ne \emptyset $ if and only if $n$ is odd.
To consider $b_{W,n}$, we follow notations in \cite[\S 4]{Ga2} and have
$$Y_{Q,n}=\set{\sum_{i=1}^{r+1} k_i e_i:  \sum_i k_i=0 \text{ and } k_1 \equiv k_2 \equiv ... \equiv k_{r+1} \mod n  },$$
where $\set{e_i}$ is the standard coordinate used as in Bourbaki \cite{Bou} and in particular $\alpha_i^\vee=e_i - e_{i+1}$.
Let $y=\sum_{i=1}^{r+1} k_i e_i \in Y$ with $\sum_{i} k_i=0$. Then
$$y+Y_{Q,n} \in (\msc{X}_{Q,n})^W$$
 if and only if
$$\angb{y_\rho}{\alpha_i} \in n\Z \text{ for all } 1\le i\le r;$$
that is,
\begin{equation} \label{E:A}
k_{i+1} + 1 - k_i \in n\Z \text{ for all } 1\le i \le r.
\end{equation}
This gives that
$$\sum_{i=1}^r i\cdot (k_{i+1} + 1-k_i) = \frac{r+1}{2}(2k_{r+1} + r) \in n\Z.$$
Since $r$ is odd, this implies that
$$2^{J(n)}| \left(\frac{r+1}{2} \right),$$
where $2^{J(n)}$ denotes the 2-exponent in a natural number $n$.

Conversely, if $2^{J(n)}| \frac{r+1}{2}$, we take
$k_{r+1} \in \Z$ be to any number such that
$$(r+1)k_{r+1} + r(r+1)/2 = m\cdot n \text{ for some } m.$$
Also take
$$k_i= k_{i+1} + 1 \text{ for all } 2\le i \le r-1,$$
and $k_1= k_2 + 1 - mn$. Clearly this element $\sum_i k_i e_i$ lies in $Y$ and thus $b_{W,n} > 0$ in this case.

The above discussion for $b_{W,n}$ and $Y_n^{\rm exc}$ can be summarized as follows:

\begin{enumerate}
\item[$\bullet$] If $r$ is even, then
$$\mca{P}_W(T)= \mca{P}_{\rm exc}(T) = \frac{T}{1-T}.$$
\item[$\bullet$] For $r$ is odd, let $J(r)$ be such that $2^{J(r)}$ is the 2-exponent appearing in $(r+1)/2$.
Then we have
$$\mca{P}_W(T)= \sum_{i=0}^{J(r)} \frac{ T^{2^i} }{ 1- (T^{2^i})^2   }$$
and
$$\mca{P}_{\rm exc}(T) = \frac{T}{1- T^2}.$$
\end{enumerate}

\subsubsection{Type $C_r, r\ge 1$}
Consider the Dynkin diagram for the simple coroots for $C_r$:

$$ \qquad
\begin{picture}(4.7,0.2)(0,0)
\put(1,0){\circle{0.08}}
\put(1.5,0){\circle{0.08}}
\put(2,0){\circle{0.08}}
\put(2.5,0){\circle{0.08}}
\put(3,0){\circle{0.08}}
\put(1.04,0){\line(1,0){0.42}}
\multiput(1.55,0)(0.05,0){9}{\circle*{0.02}}
\put(2.04,0){\line(1,0){0.42}}
\put(2.54,0.015){\line(1,0){0.42}}
\put(2.54,-0.015){\line(1,0){0.42}}
\put(2.74,-0.04){$>$}
\put(1,0.1){\footnotesize $\alpha_1^\vee$}
\put(1.5,0.1){\footnotesize $\alpha_2^\vee$}
\put(2,0.1){\footnotesize $\alpha_{r-2}^\vee$}
\put(2.5,0.1){\footnotesize $\alpha_{r-1}^\vee$}
\put(3,0.1){\footnotesize $\alpha_r^\vee$}
\end{picture}
$$
\vskip 10pt

Note
$$\rho=\sum_{i=1}^r \frac{i(2r-i)}{2} \alpha_i^\vee.$$
We consider this case in three steps.
\begin{enumerate}
\item[$\bullet$] First, if $n$ is odd, then $n_{\alpha}=n$ for all $\alpha^\vee$ and
$$\rho-\rho_{Q,n}= (1-n)\cdot \rho \in Y.$$
Thus, $Y^\exc_n= \set{ \rho-\rho_{Q,n} }$ and
$$\val{Y^\exc_n} = 1 = b_{W,n}.$$
\item[$\bullet$] Second, $n= 2 m$ with $m$ odd. In this case, it follows from Lemma \ref{L:exc} that $Y^\exc_n= \set{ \rho-\rho_{Q,n} + m \omega_{\alpha_r} }$.
\item[$\bullet$] Third, $4|n$. In this case, we show that $b_{W,n}=0$, which also implies that $Y^\exc_n = \emptyset$. For this purpose, we claim
$$Y_\rho \cap P(Y_{Q,n}^{sc}) = \emptyset.$$
However, in this case $P(Y_{Q,n}^{sc})= Y_{Q,n} \subset Y$. On the other hand, $\rho \notin Y$. Therefore the claim follows.
\end{enumerate}

The above analysis  shows that the Poincar\'e series $\mca{P}_W(T), \mca{P}_{\rm exc}(T)$ for $\Sp_{2r}$ and the fixed $Q$ are
$$\mca{P}_W(T)= \mca{P}_{\rm exc}(T) = \frac{T^3 + T^2 + T}{1- T^4}.$$

\subsubsection{Type $B_r, r\ge 3$}
Consider the Dynkin diagram for the simple coroots for the group $\Spin_{2r+1}$ of type $B_r$:

$$ \qquad
\begin{picture}(4.7,0.2)(0,0)
\put(1,0){\circle{0.08}}
\put(1.5,0){\circle{0.08}}
\put(2,0){\circle{0.08}}
\put(2.5,0){\circle{0.08}}
\put(3,0){\circle{0.08}}
\put(1.04,0){\line(1,0){0.42}}
\multiput(1.55,0)(0.05,0){9}{\circle*{0.02}}
\put(2.04,0){\line(1,0){0.42}}
\put(2.54,0.015){\line(1,0){0.42}}
\put(2.54,-0.015){\line(1,0){0.42}}
\put(2.74,-0.04){$<$}
\put(1,0.1){\footnotesize $\alpha_1^\vee$}
\put(1.5,0.1){\footnotesize $\alpha_2^\vee$}
\put(2,0.1){\footnotesize $\alpha_{r-2}^\vee$}
\put(2.5,0.1){\footnotesize $\alpha_{r-1}^\vee$}
\put(3,0.1){\footnotesize $\alpha_r^\vee$}
\end{picture}
$$
\vskip 10pt

Again, following notation in \cite[\S 6]{Ga2}, we may identify
$$Y=\set{\sum_{i=1}^r k_i e_i \in \bigoplus_i \Z e_i:  2|(\sum_i k_i)  }.$$
In this case $\alpha_i^\vee=e_i - e_{i+1}$ for every $1\le i\le r-1$ and also $\alpha_r^\vee= 2e_r$.

We have $y\in (\msc{X}_{Q,n})^W$ if and only if
$$ y_\rho \in Y_\rho \cap P(Y_{Q,n}^{sc});$$
that is,
$$\angb{y_\rho}{\alpha_i} \in n_\alpha \Z \text{ for all } 1\le i\le r,$$
or more explicitly,
\begin{equation} \label{E:B}
k_i- k_{i+1} -1 \in n_{\alpha_i} \Z \text{ for every } 1\le i\le r-1, \text{ and } k_r -1 \in n_{\alpha_r}\Z.
\end{equation}
We discuss the two cases according to the parity of $n$.

First, if $n$ is odd, then $n_{\alpha_i}=n$ for every $i$. In this case, we see that one can find $\set{k_i}$ such that $2| (\sum_i k_i)$ satisfying \eqref{E:B}. For example, we can take
$$k_i= r+1-i \text{ for every } 1\le i\le r-1, \text{ and } k_r=\frac{r(r+3)}{2}.$$
Also, in this case $Y_{Q,n}^{sc}= nY^{sc}$ and thus
$$\rho - \rho_{Q,n}= (1-n) \rho \in Y.$$
That is, $Y^\exc_n= \set{\rho-\rho_{Q,n}}$. Therefore we have
$$b_{W,n}=\val{Y^\exc_n  }=1.$$

Second, if $n=2m$ is even, then we note that in this case
$$n_{\alpha_i}=n \text{ for } 1\le i\le r-1, \text{ and } n_{\alpha_r}=m.$$
That is, the root system $\Phi_{Q,n}^\vee$ is of type $B_r$. We have
$$\rho -\rho_{Q,n} \equiv \frac{ r(rn + r +1)  }{ 4} \alpha_r^\vee  \mod Y.$$
A straightforward checking shows that $\rho - \rho_{Q,n} \in Y^{sc}$ if and only if we are in one of the following three cases:
\begin{enumerate}
\item[(I)] $r\equiv 0 \text{ mod } 4$;
\item[(II)] $r\equiv 1 \text{ mod } 4, \text{ and } m \text{ is odd}$;
\item[(III)] $r\equiv 3 \text{ mod } 4, \text{ and } m \text{ is even}$.
\end{enumerate}
%$$(I)\ r\equiv 0 \text{ mod } 4; ({\rm II})\ r\equiv 1 \text{ mod } 4, \text{ and } m \text{ is odd}; ({\rm III})\ r\equiv 3 \text{ mod } 4, \text{ and } m \text{ is even}.$$
%$$ \rho - \rho_{Q,n}
%\begin{cases}
%\in Y^{sc}  &  \text{ if } r\equiv 0 \mod 4; \\
%\in Y^{sc}  &  \text{ if } r\equiv 1 \mod 4, m \text{ is odd}; \\
%\notin Y^{sc}  &  \text{ if } r\equiv 1 \mod 4, m \text{ is odd};
%\end{cases}
%$$
At the same time, the solvability of \eqref{E:B} under the constraint $2|(\sum_ i k_i)$ is equivalent to the existence of
$t\in \Z$, where $k_r= 1 + m\cdot t$, such that
\begin{equation} \label{E:B1}
2 \text{ divides }  \left(\frac{r(r+1)}{2} + rm \cdot t \right).
\end{equation}
We discuss this in the following four cases:
\begin{enumerate}
\item[$\bullet$] $r\equiv 0$ mod 4. In this case, $Y_n^\exc=\set{\rho- \rho_{Q,n}}$ and thus
$$\val{Y_n^\exc} = b_{W,n}=1.$$
\item[$\bullet$] $r\equiv 1$ mod 4. In this case, \eqref{E:B1} holds for some $t$ if and only if $m$ is odd.
Therefore,
$$ b_{W,n}=\val{Y_n^\exc}=
\begin{cases}
0 & \text{ if  $m$ is even};\\
1 & \text{ if  $m$ is odd}.
\end{cases}
$$
\item[$\bullet$] $r\equiv 2$ mod 4. In this case, \eqref{E:B1} is always unsolvable. Thus
$$b_{W,n}=\val{Y_n^\exc}= 0.$$
\item[$\bullet$] $r\equiv 3$ mod 4. In this case, \eqref{E:B1} is always solvable and hence
$$b_{W,n}=1.$$
However, note that in this case $\val{Y_n^\exc}=1$ if and only if $m$ is even.
\end{enumerate}

In summary, we have
$$ \mca{P}_W(T)= \mca{P}_{\rm exc}(T)
\begin{cases}
\frac{T}{1-T} & \text{ if $r\equiv 0$ mod 4};\\
\frac{T^3 + T^2 + T}{1-T^4} & \text{ if $r\equiv 1$ mod 4};\\
\frac{T}{1-T^2} & \text{ if $r\equiv 2$ mod 4}.
% \frac{2T^4+ T^3 + T^2 + T}{1-T^4} & \text{ if $r\equiv 3$ mod 4}.
\end{cases}$$
Moreover, if $r\equiv 3 \mod 4$, then
$$\mca{P}_W(T) = \frac{T}{1-T}, \quad \mca{P}_{\rm exc}(T) = \frac{ T+ T^3 + T^4 }{ 1-T^4 }.$$

\subsubsection{Type $D_r, r\ge 4$}
Consider the Dynkin diagram for simple coroots of $D_r$:

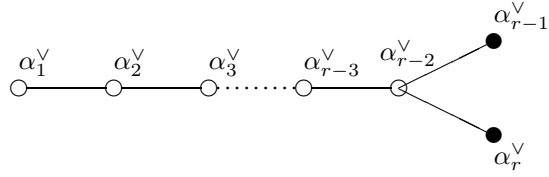
$$
\begin{picture}(4.7,0.4)(0,0)
\put(1,0){\circle{0.08}}
\put(1.5,0){\circle{0.08}}
\put(2,0){\circle{0.08}}
\put(2.5,0){\circle{0.08}}
\put(3,0){\circle{0.08}}
\put(3.5, 0.25){\circle*{0.08}}
\put(3.5, -0.25){\circle*{0.08}}
\put(1.04,0){\line(1,0){0.42}}
\put(1.54,0){\line(1,0){0.42}}
\multiput(2.05,0)(0.05,0){9}{\circle*{0.02}}
\put(2.54,0){\line(1,0){0.42}}
\put(3.00,0){\line(2,1){0.46}}
\put(3.00,0){\line(2,-1){0.46}}
\put(1,0.1){\footnotesize $\alpha_1^\vee$}
\put(1.5,0.1){\footnotesize $\alpha_2^\vee$}
\put(2,0.1){\footnotesize $\alpha_3^\vee$}
\put(2.5,0.1){\footnotesize $\alpha_{r-3}^\vee$}
\put(2.9,0.15){\footnotesize $\alpha_{r-2}^\vee$}
\put(3.5,0.35){\footnotesize $\alpha_{r-1}^\vee$}
\put(3.5,-0.4){\footnotesize $\alpha_r^\vee$}
\end{picture}
$$

\vskip 30pt

We use the Bourbaki notation in \cite[Page 220]{Bou}. In particular,
$$Y=\set{(y_1, y_2, ..., y_r) \in \bigoplus_{i=1}^r \Z e_i: 2| (\sum_{i=1}^r y_i)}.$$
Moreover,
\begin{equation*}
Y_{Q,n}=
\left\{
\begin{array}{cc}
(y_1, y_2, ..., y_r) \in \bigoplus_{i=1}^r \Z e_i: \\
\bullet \ 2| (\sum_{i=1}^r y_i), \\
\bullet \ y_1 \equiv y_2 \equiv ... \equiv y_r \text{ mod } n,\\
\bullet \ n|(2y_i) \text{ for all } i.
\end{array}\right\} ,
\quad
Y_{Q,n}^{\sct}=
\left\{
\begin{array}{cc}
(y_1, y_2, ..., y_r) \in \bigoplus_{i=1}^r \Z e_i: \\
\bullet \ 2n| (\sum_{i=1}^r y_i), \\
\bullet \ n|y_i \text{ for all } i.
\end{array}\right\}
\end{equation*}
Since $D_r$ is simply-laced, we have $n_{\alpha}=n$ for all $\alpha^\vee$. To consider the set $Y^\exc_n$ first, we note that $\rho-\rho_{Q,n}= (1-n)\rho$. By \cite[Page 221 (VII)]{Bou},
$$\rho-\rho_{Q,n} \equiv (1-n)\frac{r(r-1)}{4} (\alpha_{r-1}^\vee + \alpha_r^\vee) \mod Y^{sc}.$$
Thus it follows easily
$$ \val{Y^\exc_n}=
\begin{cases}
1 & \text{ for all $n$, if  } r \equiv 0 ,1 \mod 4;\\
1 & \text{ for odd $n$, if  } r \equiv 2 ,3 \mod 4;\\
0 & \text{ for even $n$, if  } r \equiv 2, 3 \mod 4.
\end{cases} $$

Now we consider $b_{W,n}$. Let $\sum y_i e_i \in Y$ be such that
$$\angb{y_\rho}{\alpha_i} \in n\Z.$$
We get
\begin{equation} \label{E:D}
k_i - k_{i+1} -1 \in n\Z \text{ for } 1\le i \le r-1, \text{ and } k_{r-1} + k_r -1 \in n\Z.
\end{equation}
It is equivalent to
\begin{equation} \label{E:D1}
k_i - (  k_r + (r-i) ) \in n\Z \text{ for } 1\le i\le r-1, \text{ and }  2k_r \in n\Z.
\end{equation}
Again, the existence of $y \in Y_\rho \cap P(Y_{Q,n}^{sc})$ is equivalent to the solvability of \eqref{E:D1}; it is further equivalent to saying that there exist $t_i$ such that
$$\sum_{i=1}^r k_i= r \cdot k_r + \frac{r(r-1)}{2} + n\cdot \sum_{i=1}^{r-1} t_i$$
is even, and $2k_r= n t_r$.

We discuss the situation according to the parity of $n$. If $n$ is odd, then we have
$$b_{W,n}=\val{Y_n^\exc}=1.$$
For $n=2m$ even, assume $k_r=  m t_r$. We need to look at the solvability for $t_i$'s  under the condition that
\begin{equation} \label{D:solv}
2 | \left( rm\cdot t_r + \frac{r(r-1)}{2} \right).
\end{equation}
We will divide our discussion on $b_{W,n}$ into three cases:
\begin{enumerate}
\item[$\bullet$] $r\equiv 0$ or 1 mod 4. In this case,
$b_{W,n}=\val{Y_n^\exc}=1$.
\item[$\bullet$] $r\equiv 2$ mod 4. In this case, there is no $t_r$ such that \eqref{D:solv} holds. Thus, $b_{W,n}=\val{Y_n^\exc}=0$.
\item[$\bullet$] $r\equiv 3$ mod 4. If $m$ is even, then \eqref{D:solv} is not solvable, and we have $b_{W,n}=0$. On the other hand, if $m$ is odd, then we can check that $b_{W,n}=1$; note in this case $\val{Y_n^\exc}=0$.
\end{enumerate}

We can compute the Poincar\'e series easily to obtain:
$$ \mca{P}_W(T)= \mca{P}_{\rm exc}(T) =
\begin{cases}
\frac{T}{1-T} & \text{ if $r\equiv 0$ or 1 mod 4};\\
\frac{T}{1-T^2} & \text{ if $r\equiv 2$ mod 4}.
% \frac{2T^4+ T^3 + T^2 + T}{1-T^4} & \text{ if $r\equiv 3$ mod 4}.
\end{cases}$$
Moreover, if $r \equiv 3 \mod 4$, then
$$ \mca{P}_W(T)= \frac{T^3 + T^2 + T}{1-T^4}, \quad \mca{P}_{\rm exc}(T)= \frac{T}{1-T^2}.$$

\subsubsection{Type $E_6$}
In this case $\rho\in Y$; since $E_6$ is simply-laced, we have $\rho_{Q,n} \in Y_{Q,n}^{sc}$. Therefore $Y^\exc_n = \set{\rho-\rho_{Q,n}}$ for all $n$. Thus $b_{W,n}=\val{Y_n^{sc}}=1$ for all $n$. It follows that the Poincar\'e series are
$$\mca{P}_W(T)= \mca{P}_{\rm exc}(T) = \frac{T}{1-T}.$$

\subsubsection{Type $E_7$}
Note $[P(Y^{sc}): Y^{sc}]=2$ for $E_7$. Since $E_7$ is simply-laced, one has $Y_{Q,n}=n \cdot Y^{sc}$. Therefore, by Proposition \ref{P:period} there are bijections
$$\msc{X}_{Q,n}^W \leftrightarrow  \frac{Y_\rho \cap P(Y_{Q,n}^{sc})}{Y_{Q,n}^{sc}}  \leftrightarrow   \frac{ (n^{-1}Y_\rho) \cap P(Y)  }{ Y} .$$
There are two cases:
\begin{enumerate}
\item[$\bullet$] If $n$ is odd, then $\rho-\rho_{Q,n}=(1-n)\cdot \rho \in Y^{sc}$. Thus
$$\val{ Y^\exc_n }= b_{W,n} =1.$$
\item[$\bullet$] If $n$ is even, then we see that $(n^{-1}Y_\rho) \cap P(Y)=\emptyset$ and thus
$$\val{ Y^\exc_n }= b_{W,n} =0$$
in this case.
\end{enumerate}
It follows that the two Poincar\'e series for $E_7$ are
$$\mca{P}_W(T)=\mca{P}_{\rm exc}(T)= \frac{T}{1-T^2}.$$

\subsubsection{Type $E_8, F_4$ and $G_2$}
Let $\mbf{G}$ be one of these three exceptional groups. It is both simply-connected and adjoint. Therefore, $\rho\in Y$ and $\rho_{Q,n} \in Y_{Q,n}^{sc}$. We have $Y^\exc_n=\set{\rho- \rho_{Q,n}}$ for every $n$ and therefore
$$b_{W,n}= \val{Y^\text{sc}_n } = 1$$
for all $n$. The Poincar\'e series are thus
$$\mca{P}_W(T)=\mca{P}_{\rm exc}(T)= \frac{T}{1-T}.$$

\subsection{Adjoint groups} It follows from Proposition \ref{P:period} that if $\mbf{G}$ is an adjoint group, then $b_{W,n}= \val{\pi_1(\wt{G}_{Q,n}^\vee)}$ for every $n$. We will not treat general adjoint groups, but consider the example of $\mbf{SO}_{2r+1}$.

Let $Y$ be the cocharacter lattice of $\mbf{SO}_{2r+1}$ spanned by
$$\set{e_1, e_2, ..., e_r}$$
and the coroot lattice $Y^{\rm sc}$ by the set of simple coroots
$$\set{\alpha_i^\vee:=e_i - e_{i+1} \text{ for } 1\le i\le r-1} \cup \set{\alpha_r^\vee:=2e_r}.$$
Let $Q: Y\to \Z$ be the unique Weyl-invariant quadratic form on $Y$ such that
$$Q(e_r)=-1.$$
Pick a bisector $D$ associated to $Q$. Consider $(D, \eta)$ with $\eta_n$ being the trivial map. It gives rise to the $\mbf{K}_2$ extension $\wm{SO}_{2r+1}$. In fact, let $\wm{SL}_{2r+1}$ be the extension of $\mbf{SL}_{2r+1}$ determined by $Q(\beta^\vee)=-1$ for any coroot $\beta$ of $\mbf{SL}_{2r+1}$. Then $\wm{SO}_{2r+1}$ is just its pull-back via the embedding $\mbf{SO}_{2r+1} \into \mbf{SL}_{2r+1}$.

In any case, for the dual group of $\wt{G}= \wt{\SO}_{2r+1}^{(n)}$, we have
$$ \wt{G}_{Q,n}^\vee=
\begin{cases}
\Sp_{2r} & \text{ if $4\nmid n$}; \\
\SO_{2r+1} & \text{ if $4 | n$}.
\end{cases}
$$
For $\mbf{SO}_{2r+1}$ and the above $Q$, it follows that the Poincar\'e series for $b_{W,n}$ is
$$\mca{P}_W(T)=\frac{2T^4 + T^3 + T^2 +T}{ 1-T^4 }.$$
One can check easily that $\rho_{Q,n} \in Y$ for all $n$, and thus $\val{Y_n^\exc}=1$ for all $n$.
Therefore,
$$\mca{P}_{\rm exc}(T)= \frac{T}{1-T}.$$

\subsection{Covers of $\GL_r$}
As in Example \ref{eg-GL}, we use the notation in \cite[\S 2.1]{Ga3}. Below we only illustrate the situation by considering two special families.

First, we consider  the Kazhdan-Patterson $\mbf{K}_2$-extension $\wm{GL}_2$ with
$$Q(e_i)=\p=0, i=1, 2;   \quad B_Q(e_1, e_2)=\bq=1.$$
In this case, $Y_{Q,n}=nY$ for every $n\in \N$, and thus $\msc{X}_{Q,n} \simeq (\Z/n\Z)^2$. We already obtained from Example \ref{eg-GL} that $b_{W,n} = n$. Hence,
$$\mca{P}_W(T) =\frac{T}{(1-T)^2}.$$
Now we consider the set $Y_n^\exc$. By definition, $y = y_1 e_1 + y_2 e_2 \in Y_n^\exc$ if and only if
$$y_1 - y_2 = -n + 1.$$
In particular, there are infinitely many exceptional points. However, we see that
$$\val{ f_\msc{X}( Y_n^{\rm exc} ) } = n$$
and hence
$$\mca{P}_{\rm exc}(T)= \mca{P}_W(T) = \frac{T}{(1-T)^2}.$$

Second, we consider the Savin $\mbf{K}_2$-extension $\wm{GL}_r$ with
$$Q(e_i)=\p=-1;   \quad B_Q(e_i, e_j)=\bq=0, i\ne j.$$
We have $Q(\alpha^\vee) = -2$ for every coroot $\alpha^\vee$. For $n\in \N$, denote $n_o:= n/\text{gcd}(2, n)$, which equals to $n_\alpha$ for every $\alpha$. Then $Y_{Q,n} = n_o Y$ and thus
$$\msc{X}_{Q,n} \simeq (\Z/n_o\Z)^r.$$
For every $1\le i\le r-1$, the coset $\sum_j k_j e_j +Y_{Q,n}$ lies in  $(\msc{X}_{Q,n})^{W_{\alpha_i}}$ if and only if
$$k_i- (k_{i+1} + 1) \in n_o\Z.$$
It follows that
$$b_{\alpha_i, n} = n_o^{r-1}, \text{ and } a_{\alpha_i, n} = \frac{n_o^r- n_o^{r-1}}{2}.$$
Moreover, one has $b_{W,n} = n_o$ and therefore the Poincar\'e series for the Savin-cover $\wt{\GL}_r^{(n)}$ is
$$\mca{P}_W(T)= \frac{T + T^2 + T^3}{(1-T^2)^2}.$$
To consider $Y_n^\exc$, note that $y =\sum_i y_i e_i \in Y_n^\exc$ if and only if
$$y_i - y_{i+1} = -n_o + 1 \text{ for all } 1\le i \le r-1.$$
Again, there are infinitely many exceptional points. On the other hand, since
$$\val{ f_\msc{X}( Y_n^{\rm exc} ) } = n_o= b_{W,n}$$
for all $n$, we have
$$\mca{P}_{\rm exc}(T)= \mca{P}_W(T) =  \frac{T + T^2 + T^3}{(1-T^2)^2}.$$

We note that for general Brylinski-Deligne cover $\wt{\GL}_r^{(n)}$ parametrized by $\p, \bq \in \Z$,  the number $b_{W,n}$ varies sensitively. We leave the computation of the general $\mca{P}_W(T)$ and  $\mca{P}_{\rm exc}(T)$ to the interested reader.

\subsection{Rationality of the Poincar\'e series} \label{SS:rat}
In view of Corollary \ref{C:rat} and the above examples, we have
\begin{conj} \label{C:red}
Let $\wt{G}$ be an $n$-fold cover of a connected reductive group $G$. Then both $\mca{P}_W(T)$ and $\mca{P}_{\rm exc}(T)$ are rational functions in $T$.
\end{conj}

% \textcolor{red}{A proof???}
Note that the Poincar\'e series $\mca{P}_W(T)$ in Conjecture \ref{C:red} is rational if and only if there is a recurrence relation (of order $m$) on $b_{W,n}$, i.e., there exist constants $c_1, c_2, ..., c_m$ such that
$$b_{W,i}= c_1 b_{W,i-1} + c_2 b_{W,i-2} + ... + c_m b_{W,i-m}$$
for every $i\ge N$, where $N$ is some fixed number $\ge m+1$. If $G$ is semisimple, then the periodicity of $b_{W,n}$ gives a special recurrence relation. However, for general reductive group $G$, the number $b_{W,n}$ may grow as $n$ increases: this can be seen from the covers $\wt{\GL}_r$ considered above.
The same consideration also applies to $\mca{P}_{\rm exc}(T)$.

%%%
\section{Principal series of $\wt{\GL}_2$ and $\wt{\SL}_2$} \label{S:RES}
Let $G$ be a $p$-adic linear algebraic group with derived group $G_o$. Let $\tau$ be a parabolic induction on $G$. Since the Whittaker functional is an equivariant functional with respect to a unipotent subgroup, it follows that a Shahidi local coefficient associated with  $\tau$ is determined by the restriction of $\tau$ to $G_o$. Thus, due to the uniqueness of Whittaker model, the relation between Shahidi local coefficients associated with $G$ and $G_o$ is clear. In the case of covering groups, this uniqueness fails and one needs to study carefully the restriction from a covering group to its derived group. 

In this section, we investigate the problem of restricting a genuine principal series representation $\tau$ of Kazhdan-Patterson covers $\wt{G}=\wt{\GL}_2$ to $\wt{G}_o=\wt{\SL}_2^{(n)}$.  For the exact definition of these groups see \S \ref{SS:RES1} below. After proving properties about local coefficients matrices associated with $\wt{G}_o$ in \S \ref{S:LCM-Go}, we  use in \S \ref{S:LCM-G} such results to describe the local coefficients matrices associated with $\tau$, by means of those associated with  genuine principal series of $\wt{G}_o$ which occur in the restriction of $\tau$ to $\wt{G}_o$ . Moreover, we compute the invariants of the local coefficients matrix associated with $\tau$ by using the invariants of the local coefficients matrices associated with $\wt{G}_o$. In an ongoing project we intend to conduct a more thorough study of the restriction problem for covering groups, and apply it to the study of local coefficients matrices and related invariants. We intend the detailed example provided here to serve as a first non-trivial example and as a reference.

As we show in Theorem \ref{decopropcor}, unlike the linear case, the restriction of $\tau$ to $\wt{G}_o$ is always a direct sum of principle series representations. Moreover, the restriction for most of the coverings is seldom multiplicity-free. While in the odd $n$ case all the principle series representations appearing in the restriction are isomorphic, in the even $n$ case there are $[F^\times: F^{\times 2}]$ non-isomorphic principle series representations from the restriction. If $n$ is even and $\tau$ is unramified, then one can always find ramified principle series representations of $\wt{G}_o$ in the restriction. See \cite{GPS} and \cite{GPS81} for the $n=2$ case. Similar phenomenon appears in the case of a double cover of $\GSp_{2r}$, see \cite{Szp4-1} and \cite{Szp5}.

We show that a local coefficients matrix associated with a principal series representation of  $\wt{G}$  can be chosen to be a diagonal-block matrix with each block arising from a genuine principal series representation of $\wt{G}_o$, see Theorem \ref{Smain}. This gives a strong connection between the invariants extracted from local coefficients matrix associated with $\wt{G}$ and that associated with $\wt{G}_o$.

One of the outcomes of this study is an explanation for a discrepancy in the form a local coefficients matrix takes between  $\wt{G}_o$ and $\wt{G}$. Our results indicate again that there exists a trichotomy among different degree coverings of $\wt{G}_o$: odd-fold covers, $(4k+2)$-fold covers and $(4k)$-fold covers. This trichotomy is evident in the formulas for  the determinant of unramified  local coefficients matrices, see Theorems \ref{detunramnot4} and \ref{T:SL-c3}. Remarkably, this trichotomy is dissolved if we consider $\wt{G}$, see Theorem \ref{detformgl}. This phenomenon, though observed in Example \ref{eg-SL} and Example \ref{eg-GL} already,  is better explained by nature of the restriction mentioned above. It is particularly striking that in the even fold cover case we have to deal with certain \emph{ramified} local coefficients matrices for  $\wt{G}_o$, when restricted from an unramified principal series of $\wt{G}$.

From this section onwards, the approach taken towards the local coefficients matrix and its invariants is different from the one adopted in earlier sections, where we essentially modified the scattering matrices studied in \cite{Mc2}. This is inevitable, as remarked in \S \ref{SS:Sca}, there is no canonical choice of $r_w$ in general and thus a local coefficients matrix is difficult to compute directly for ramified data. Here we use the computation from \cite{Szp6} which relies on ideas involving partial $\zeta$-integrals, partial $\gamma$-factors and partial $\widetilde{\gamma}$-factors, see \cite[\S 2]{Szp6}. One of the advantages of this approach is that it incorporates naturally $\gamma$ and  $\widetilde{\gamma}$-factors in the study of the local coefficients matrices. Another advantage is that in these computations we can treat both ramified and unramified cases, most often without assuming $\gcd(n,p)=1$. In fact, in  \S \ref{S:LCM-Go} we complement the results in  \cite{Szp6} and \cite{GSS1} in some ramified cases.

There might be a small overlap between the results given in this section and the results given in earlier sections. For example, the determinants computed in Theorems \ref{detunramnot4}, \ref{T:SL-c3} and \ref{detformgl} in the unramified case may be deduced from Theorem \ref{T:M1}, and the explicit formula for the trace in the unramified case given in Corollary \ref{exptrun} appears also in Proposition \ref{P:trace}. However, we believe that this repetition is actually important as it demonstrates the advantages of each of the approaches mentioned above.

\vskip 10pt

Let $K$ be a group and $H \subset K$ a subgroup. Let $\pi'$ be a representation of $K$ and $\pi$ a representation of $H$. We say that $\pi$ occurs in $\pi'$, or that $\pi'$ contains $\pi$, if
$$\Hom_H(\pi',\pi) \ne 0.$$
If  $K$ is finite, we denote
$$\widehat{K}:=\Hom(K, \C^\times).$$
For $k,h \in K$ we define
$$h^k=khk^{-1}.$$
If $H$ is a normal subgroup of $K$, then for $k\in K$ we define $\pi^k$ to be the representation of $H$ given by
$$\pi^k(h):=\pi(h^k).$$
Thus, $\pi \mapsto \pi^k$ defines an action of $K/H$ on the set of isomorphism classes of representations of $H$.

\subsection{The Kazhdan-Patterson covers of $\GL_2$} \label{SS:RES1}
For the rest of this paper, we consider the Kazhdan-Patterson covering $\wt{G}=\wt{\GL}_2^{(n)}$ depending on a parameter $c \in \Z$, see \cite{KP}. In fact, the discussion could be carried out for general Brylinski-Deligne covers of $\GL_2$ with proper modification. However, our restriction to the Kazhdan-Patterson covers is only for the purpose of convenience, as later we will use some results in \cite{Szp6} which are stated for such covers only.

Recall from Example \ref{eg-GL} that we have the cocharacter lattice
$Y= \Z e_1 \oplus \Z e_2$ of $\GL_2$. Consider the Weyl-invariant bilinear form $B_Q$ of $Y$ such that
$$B_Q(e_i, e_i)= 2c, \quad B_Q(e_1, e_2)= 2c+ 1.$$
This gives rise to the $n$-fold Kazhdan-Patterson covering $\wt{G}$ of $G=\GL_2$. The bisector
$D: Y\times Y \to \Z$ is chosen to be
$$ D(e_i, e_j) =
\begin{cases}
c  & \text{ if }  i= j; \\
c+ 1 & \text{ if }  i=1, j=2; \\
c   & \text{ if }  i=2, j=1.
\end{cases}
$$
One can check that the cocycle $c_D$ on $T \subset G$ determined by $D$ is given by
$$c_D(e_1(a_1) e_2(a_2),  e_1(b_1) e_2(b_2) ) = (a_1, b_2)_n \cdot (a_1 a_2, b_1 b_2)_n^c,$$
where $e_i(a)\in T$ for $a\in F^\times$. In fact, a cocycle $c$ on the whole group $G$ which entails a description of $\wt{G}$ in terms of the set-theoretic $\bbmu_n \times G$ is given by Kubota as follows:
\begin{equation}\label{rao}
c(g,g')=\bigl(x(gg')x^{-1}(g),x(gg')x^{-1}(g') \det(g)^{-1}\bigr)_n \cdot \bigl(\det(g),\det(g') \bigr)_n^c,
\end{equation}
where
$$x \begin{pmatrix} a & b \\ c & d \end{pmatrix}
=
\begin{cases}
c & \text{ if } c \neq 0; \\
d & \text{ otherwise}.
\end{cases}$$
That is, we have a section $\s: G \to \wt{G}$ such that $\wt{G} = \bbmu_n \times \s(G)$ as sets, and the group law on $\bbmu_n \times \s(G)$ is given by $c$:
$$\s(g) \cdot \s(g') =c(g, g') \cdot \s(g g').$$
It is easy to check that $c$ actually extends the cocycle $c_D$ on $T$ as above.

For every $\wt{g}, \wt{g}' \in \wt{G}$, we have the commutator
$$[\wt{g}, \wt{g}']:= \wt{g} \cdot \wt{g}' \cdot \wt{g}^{-1} \cdot \wt{g}'^{-1} \in \wt{G}.$$
As $\wt{G}$ is a central extension of $G$, the commutator map $[-,-]$ of $\wt{G} \times \wt{G}$ factors through $G\times G$, and thus we write $[g, g']=[\wt{g}, \wt{g}']$ where $g, g' \in G$ is the image of $\wt{g}$ and $\wt{g}'$ respectively. If $gg'=g'g \in G$, then one has
$$[g, g']= c(g, g') \cdot c(g', g)^{-1}.$$

Let $\alpha^\vee = e_1 - e_2$. For a fixed $n$, the derived subgroup of  $\wt{G}$ is $\wt{G}_o=\wt{\SL}_2^{(n)}$, which arises from the quadratic form 
$$Q:  \Z(\alpha^\vee) \to \Z$$
such that $Q(\alpha^\vee)= -1$. In particular, the derived subgroup of $\wt{G}$ is independent of the underlying parameter $c$. This fact also follows from \eqref{rao}. We will write
$$H_o := H\cap G_o$$
for every subgroup $H \subseteq G$.  For every lifting $\wt{g} \in \wt{G}$ of $g\in G$, we define $\det(\wt{g})=\det(g)$. For convenience, we also denote
$$\diag(a, b):=e_1(a) e_2(b) \in G.$$

\begin{lm} \label{justcomp}
The following hold.
\begin{enumerate}
\item[(i)] $\s(a1_G) \cdot \s(g) \cdot \s(a 1_G)^{-1}=\s(g) \cdot \bigl(a,\det(g) \bigr)_n^{4c+1}$.
\item[(ii)] $\s(g) \cdot \s(a 1_G) \cdot \s(g)^{-1}= \s(a1_G) \cdot \bigl(\det(g), a\bigr)_n^{4c+1}$.
\item[(iii)] For $g=\diag(x,y)$ and $g'=\diag(z,t)$, we have
\begin{equation} \label{compeq}
 [g,g']=(x,t)_n (y,z)_n (xy,zt)_n^{2c}.
\end{equation}
%In particular For $g=aI_2$, $g'=bI_2$ we have $c(g,g')c^{-1}(g',g)=(a,b)_n^{2(4c+1)}.$
\end{enumerate}
\end{lm}
\begin{proof}
It follows from a straighforward computation with \eqref{rao}.
\end{proof}

Note that $\wt{g},\wt{g}' \in \wt{G}$  commute if and only if $[g,g']=1$, and \eqref{compeq} explicates the latter condition for elements in $T$.

\subsection{Centers}
It is important to understand the centers of various groups.
\begin{lm} \label{Dcent}
$\wt{Z(G)}$ is the centralizer of $\wt{G}_o$ inside $\wt{G}$.
\end{lm}
\begin{proof}
Clearly, the centralizer of $\wt{G}_o$ inside $\wt{G}$ is contained in $\wt{Z(G)}$. On the other hand, $\wt{G}_o$ is generated by unipotent elements, which split $G$-equivariantly with respect to the $G$-conjugation action; since the action of $Z(G)$ on any unipotent elements is trivial, we see that $\wt{Z(G)}$ centralizes every element in $\wt{G}_o$. This second assertion also follows from (i) of Lemma \ref{justcomp}.
\end{proof}
Denote
$$n_c=\frac{n}{\gcd(n,4c+1)}, \quad
d=
\begin{cases}
n &  \text{ if $n$ is odd}, \\
\frac {n} {2}  &  \text{ otherwise};
\end{cases}
$$
and
$$
d_c=\frac{d}{\gcd(d,4c+1)}=
\begin{cases} n_c &  \text{ if $n$ is odd}, \\
\frac {n_c} {2}  &  \text{ otherwise}.
\end{cases}$$

\begin{lm} \label{centers} The following hold.
\begin{enumerate}
\item[(i)] $Z(\wt{T})$ is the inverse image in  $\wt{G}$ of $\set{\diag(a,ab): a\in F^{\times n_c}, b \in F^{\times n} }$. In particular, $[\wt{T}:Z(\wt{T})]=(nn_c)^2 \cdot \val{n n_c}^{-1}$.
\item[(ii)] $Z({\wt{G}})$ is the inverse image in $\wt{G}$ of $\{a1_G: a\in F^{\times n_c} \}.$
\item[(iii)] $Z\bigl(\wt{Z(G)}\bigr)$  is the inverse image in $\wt{G}$ of $\{a1_G: a\in F^{\times d_c} \}.$
\item[(iv)] $Z(\wt{T}_o)$  is the inverse image in $\wt{G}$ of $\{ \diag(a,a^{-1}):  a \in F^{\times d} \}.$
\item[(v)]  $Z(\wt{T})\cap \wt{T}_o$ is the inverse image in $\wt{G}$ of $\{ \diag(a,a^{-1}): a \in F^{\times n} \}.$
\end{enumerate}
\end{lm}
\begin{proof}
For (i), it follows from a straightforward computation that
$$Y_{Q,n} = \Z n_c(e_1 + e_2) \oplus \Z ne_2,$$
which then gives the desired result. We can also argue more explicitly as follows. By \eqref{compeq}, for $g=\diag(a,b), g'=\diag(x,x^{-1})$, we have
$$[g,g']=(ba^{-1},x)_n.$$
Since $Z(\wt{T})$ is contained in the centralizer of $\wt{T}_o$, it follows that any element in $Z(\wt{T})$ must be a lifting of an element of the form $\diag(a,at)$ where $t \in F^{\times n}$. Note that for $g=\diag(a,ay^n)$, $g'=\diag(1,x)$ we have
$$[g,g']=(a,x)_n^{4c+1}.$$
This implies that the centralizer of $\wt{\diag(1,F^\times)}$ inside the centralizer of  $\wt{T}_o$ is
$$\{ \diag(a,ab): a\in F^{\times n_c}, b \in F^{\times n} \}.$$
Since elements of the form $\diag(1,y)$ and $\diag(x,x^{-1})$ generate $T$, the proof of (i) is now complete.

For (ii), we note that since $T$ and $G_o$ generate $G$, one has $Z(\wt{G})=\wt{Z(G)} \cap Z(\wt{T})$. The assertion follows from (i). In fact, (ii) also follows from an explicit checking as in  \cite[Lemma 1]{CO}.

Now, it follows from \eqref{compeq} and Lemma  \ref{oldlem} that
$$[a1_G, b1_G]=1$$
if and only if $(a,b)_{d_c}=1$. This proves (iii). The fourth assertion (iv) is proven in a similar way. In fact, it is also contained in \cite[Lemma 3.1]{Szp6}.

To prove (v), we note that if $g=\diag(a,a^{-1})$, then $\s(g) \in Z(\wt{T})$ if and only if $\s(g) \in Z(\wt{T}_o)$ and also $\s(g)$ commutes with all the elements of the form $\s(\diag(1,x))$. The last condition here is equivalent to that $(a,x)_n=1$ for all $x \in F^\times$. Now (v) follows from (iv).
\end{proof}

Let $\xi$ be a genuine character of $Z(\wt{T})\cap \wt{T}_o$. Let $E_o(\xi)$ be the set of characters of $Z(\wt{T}_o)$ extending $\xi$. One has
$$\val{E_o(\xi)}=[Z(\wt{T}_o): Z(\wt{T})\cap \wt{T}_o] < \infty.$$
Using (iv) and (v) in Lemma \ref{centers}, we deduce that
$$ \val{E_o(\xi)}=
\begin{cases} 1 &  \text{ if $n$ is odd}; \\
 [F^\times: F^{\times 2}] &  \text{ otherwise}.
 \end{cases}$$

\begin{lm} \label{evenslcent}
Assume that $n$ is even.
\begin{enumerate}
\item[(i)] For $t_o=\s(\alpha^\vee(z))  \in Z(\wt{T}_o)$ and $t \in \wt{T}$, we set $\mu_t(t_o)=(z,\det(t))_n \in \bbmu_n$. Then
$$t t_o t^{-1}=\mu_t(t_o)) \cdot t_o$$
and the map $t \mapsto \mu_t$ gives a well-defined isomorphism from $\wt{T}/ \wt{Z(G)}\wt{T}_o$ to the Pontryagin dual of $Z(\wt{T}_o)/Z(\wt{T})\cap \wt{T}_o$.
\item[(ii)] Let $\xi$ be a genuine character of $Z(\wt{T})\cap \wt{T}_o$. Then $E_o(\xi)$ is a $\wt{T}/ \wt{Z(G)}\wt{T}_o$-torsor.
\end{enumerate}
\end{lm}
\begin{proof}
Since $n$ is even, it follows from Lemma \ref{centers} that
$$Z(\wt{T}_o)/Z(\wt{T})\cap \wt{T}_o \cong F^{\times d}/F^{\times n}.$$
Now Lemma \ref {dualcenter} gives that $F^\times/F^{\times 2}$ is isomorphic to the Pontryagin dual of $Z(\wt{T}_o)/Z(\wt{T})\cap \wt{T}_o$. The isomorphism is given by $a \mapsto \xi_a$, where $\xi_a$ is the character of $Z(\wt{T}_o)/Z(\wt{T})\cap \wt{T}_o$ given by
$$\s(\alpha^\vee(z)) \cdot \zeta \mapsto (z,a)_n.$$
To finish the proof of the first assertion, just note that for $t_o=\s(\alpha^\vee(z))  \in Z(\wt{T}_o)$ and $t=\s(\diag(x,y)) \in \wt{T}$, we have
$$ t t_o t^{-1}=(z,xy^{-1})_n \cdot \s(\alpha^\vee(z))$$
and that $xy \equiv xy^{-1}  \ \text{mod } F^{\times 2}$.

For (ii), we see from the above argument that $\wt{T}/ \wt{Z(G)} \wt{T}_o$ acts freely on the set of genuine characters of $Z(\wt{T}_o)$. The assertion now follows from the equality
$$[\wt{T}: \wt{Z(G)} \cdot \wt{T}_o]= \val{E_o(\xi)}.$$
\end{proof}

%$
\subsection{Maximal abelian subgroups}
We fix a  $d_c$-Lagrangian subgroup $J \subseteq  F^\times$, the existence of which follows from Lemma \ref{lang decomp}, as $\bbmu_{2d_c} \subset F^\times$. We also fix
$$ M=\{j1_G:  j \in J \} \subset G.$$

\begin{lm} \label{mlem}
The group $\wt{M} \subseteq \wt{Z(G)}$ is a maximal abelian subgroup of $\wt{Z(G)}.$ Also,
$$\wt{M}/Z(\wt{G}) \cong J/F^{\times n_c}.$$
In particular,
\begin{equation} \label{Dindex}
[\wt{M}: Z(\wt{G})]=n_c \val{n_c}^{-1/2} \cdot
\begin{cases}
1 &  \text{ if $n$ is odd}; \\
2 \val{2}^{-1/2} &  \text{otherwise}
\end{cases}
\end{equation}
Moreover, every maximal abelian subgroup $\wt{M}'$ of $\wt{Z(G)}$ is the inverse image in $\wt{G}$ of $\{ j1_G: j \in J' \}$ where $J'$ is a $d_c$-Lagrangian subgroup of $F^\times$.
\end{lm}
\begin{proof}
From  \eqref{compeq} combined with Lemma  \ref{oldlem} we deduce that for $g=a1_G$, $g'=b1_G$, the equality
$$[g,g']=1$$
holds if and only if $(a,b)_{d_c}=1$.
The first and fourth assertions follow from this.  To prove the second assertion, note that
$$\epsilon\cdot \s(a1_G) \mapsto a{F^\times}^{n_c}$$
 is  surjection from $\wt{M}$ to $J/F^{\times n_c}$. By the second assertion in Lemma \ref{centers}, the kernel of this surjection is $Z(\wt{G})$.
The third assertion follows from \eqref{Mindex}.
\end{proof}

 Let $\omega$ be a genuine character of $Z(\wt{G})$. Let
 $$\msc{E}(\wt{M}, \omega)=\set{\xi \in \Hom(\wt{M}, \C^\times):  \xi|_{Z(\wt{G})} = \omega}$$
 be the finite set of extensions of $\omega$ to $\wt{M}$. For $g \in G$,  let $\lambda_g$ be the linear character of $Z(G)$ given by
$$\lambda_g(a1_G) =[g, a1_G].$$
 Observe that by the second assertion  in Lemma \ref{justcomp},
 $$\lambda_g((a1_G))= (\det(g),a)_n^{4c+1}.$$
 We may view $\lambda_g$ as a character of $\wt{Z(G)}$ (or any subgroup $\wt{H}$ of it) such that $\bbmu_n$ acts trivially.

\begin{lm} \label{detstab}
Let $\chi$ be a genuine character of $\wt{M}$. For $g\in \wt{G}$ we have $\chi^g=\chi$ if and only if $\det(g) \in J^{n/d}$.
\end{lm}
\begin{proof}
Note that $\chi^g=\chi \cdot \lambda_g|_M$. It follows from Lemma \ref{oldlem} that $\lambda_g|_M$ is trivial if and only if $\det(g) \in J^{\perp}_{(n_c)}.$ The assertion follows from Lemma \ref{nicesat2}.
\end{proof}

\begin{lm} \label{crucia}
Let $\omega$ be a genuine character of $Z(\wt{G})$ .
\begin{enumerate}
\item[(i)]  If $n$ is even, then
$$g \mapsto \lambda_g$$
 is a well-defined isomorphism from $\wt{T}/ \wt{M} \wt{T}_o$ to the Pontryagin dual of $\wt{M}/Z(\wt{G})$. Moreover, in this case, $\wt{T}/ \wt{M}\wt{T}_o$ acts freely and transitively on $\msc{E}(\wt{M}, \omega)$ by $\xi \mapsto \xi^x$ for $x\in \wt{T}/ \wt{M}\wt{T}_o$.
\item[(ii)] If $n$ is odd, then $g \mapsto \lambda_g$ is a well defined isomorphism from $\wt{Z(G)}/\wt{M}$ to the Pontrayagin dual of ${\wt{M}/Z(\wt{G})}$. In this case, $\wt{Z(G)}/ \wt{M}$ acts freely and transitively on $\msc{E}(\wt{M}, \omega)$  by $\xi \mapsto \xi^x$ for $x\in \wt{Z(G)}/ \wt{M}$.
\end{enumerate}
\end{lm}
\begin{proof}
Assume first that $n$ is even. In this case,  Lemma \ref{oldlem} and Lemma \ref{nicesat2} imply that
$$x \mapsto \eta_{x,(n)}^{4c+1}$$
gives rise to an isomorphism
$$\phi_1: F^\times/J^{2} \to \Hom( J/F^{\times n_c}, \C^\times).$$
Consider the surjection
$$\phi_2: \wt{T} \to F^\times/J^{2}$$
 given by
$$\phi_2(g)=\det(g) \cdot J^2.$$
We show that the kernel of $\phi_2$ is $\wt{M}\wt{T}_o$. Indeed, $g \in \wt{T}$ lies in the kernel of $\phi_2$ if and only if its projection to $T$ is $\diag(a,a^{-1}j^2)$ where $a\in F^\times$ and $j \in J$. We have
$$\diag(a,a^{-1}j^2)=\diag(aj^{-1},a^{-1}j) \cdot \diag(j,j).$$
Thus,
$$g \mapsto (\phi_1 \circ \phi_2)(g)=\lambda_g$$
 is an isomorphism from $\wt{T}/ \wt{M}\wt{T}_o$ to the Pontryagin dual of  $\wt{M}/Z(\wt{G})$. From the above it follows that $\wt{T}/ \wt{M}\wt{T}_o$ acts freely on the set of genuine characters of $\wt{M}$. The proof of the second part of (i) is completed once we note the equality
 $$[\wt{T}: \wt{M}\wt{T}_o]= \val{\msc{E}(\wt{M}, \omega)}.$$

If $n$ is odd, then $J$ is an $n_c$-Lagrangain subgroup and thus Lemma \ref{oldlem} and Lemma \ref{nicesat2} give
an isomorphism
$$\phi_3: F^\times/J \to \Hom(J/F^{\times n_c}, \C^\times)$$
 by
$$\phi_3(x):= \eta_{x^2,(n)}^{4c+1}.$$
Also, consider the surjection
$$\phi_4: \wt{Z(G)} \to F^\times/J$$
 given by
$$\phi_4(\zeta\cdot \s(a1_G)) = a\cdot J.$$
One has ${\rm Ker}(\phi_4)=\wt{M}$. Thus, we have shown that $g \mapsto (\phi_3 \circ \phi_4)(g)=\lambda_g$ is an isomorphism from $\wt{Z(G)}/\wt{M}$ to the Pontryagin dual of $\wt{M}/Z(\wt{G})$. This proves the first part of (ii), whereas the second part follows from the same argument as we used for (i).
\end{proof}

\begin{rmk} \label{tricky}
The covering torus $\wt{T}$ acts on the set of genuine characters of $\wt{M}$ in the case $n$ is odd as well. However, an element in $\wt{T}-\wt{Z(G)}\wt{T}_o$ acts the same as some element inside $\wt{Z(G)}\wt{T}_o$. More precisely, let $x \in F^\times$ be a non-square element. Then
$$g=\s(\diag(1,x)) \notin \wt{Z(G)}\wt{T}_o.$$
However, it follows from Lemma \ref{Jaut} that there exists $a \in F^\times$ such that $xJ=a^2J$. This implies that for $t=\s(a1_G)$, we have $\chi^t=\chi^g$  for every genuine character $\chi$ of $\wt{M}$. The crucial point here is that for $\s(g'), \s(g'') \in \wt{T}$ and a character $\chi$ of $\wt{M}$,  we have $\chi^{g'}=\chi^{g''}$ if and only if $\det(g') J=\det(g'') J$.
\end{rmk}

We now fix a $d$-Lagrangian subgroup $K\subset F^\times$. In view of Lemma \ref{extlag}, we assume that $K \subseteq J$; furthermore, we assume that there exists a Lagrangian decomposition
$$(K^\dag,\widehat{K^\dag})$$
of $F^\times/F^{\times d}$ where $K^\dag$ is the projection of $K$ to $F^\times/F^{\times d}$. We set
$$A=\begin{cases}
\set{\diag(a,at): a\in J, t \in K } & \text{ if $n$ is odd}; \\
 \set{ \diag(a,a^{-1}) \cdot \diag(b,b): a\in K, b \in J }  & \text{ otherwise}.
\end{cases}$$

\begin{lm} \label{Tfacts}
The following hold.
\begin{enumerate}
\item[(i)] $\wt{A} \subset \wt{T}$ is a maximal abelian subgroup of $\wt{T}$. In particular  $[\wt{T}:\wt{A}]=n_c n\cdot  \val{n_c n}^{-1/2}$.
\item[(ii)] $\wt{A} \cap \wt{Z(G)}=\wt{M}$.
\item[(iii)] $A_o= \set{\diag(a,a^{-1}): a \in K }$ where $A_o := A\cap G_o$.
\item[(iv)] $\wt{A}_o \subset \wt{T}_o$ is a maximal abelian subgroup of $\wt{T}_o$; every maximal abelian subgroup of  $\wt{T}_o$ is the inverse image of $\{\diag(a,a^{-1}): a \in K' \}$ where $K' \subseteq F^\times$ is a $d$-Lagrangian subgroup.
\end{enumerate}
\end{lm}
\begin{proof}
Assume first that $n$ is odd. In this case, the assertion (ii) is clear. We prove (iii). Suppose that $\diag(a,at) \in A$ with $a^2t=1$. Then $a^2 \in K$, and Lemma \ref{Jaut} implies that $a\in K$. Conversely, if $a \in K$, then we have $t:=a^{-2} \in K$. Hence,
$$\diag(a,a^{-1})=\diag(a,a t) \in A_o.$$

The statement (iv) follows from (iii) and \cite[Lemma 3.1]{Szp6}. We now prove (i). It follows from \eqref{compeq} that for $g=\diag(x,xb), g'=\diag(z,zd)$, we have
$$[g,g'] =(x,z)_n^{8c+2}(x,d)_n^{4c+1}(b,z)_n^{4c+1}(b,d)_n^{2c}.$$
This implies that $\wt{A}$ is abelian. Suppose that $g=\diag(x,xb)$ is such that $[g, a]=1$ for all $a\in A$, we need to show that $g \in A$. However, we have already shown that for every $z\in K$,
$$g'=\diag(z,z^{-1}) \in A.$$
Using  \eqref{compeq} again, we obtain
$$[g,g'] =(b,z)_n.$$
This enforces $b \in K$. For every $d \in J$ and $g''=\diag(d,d) \in T$, we have
$$[g,g'']=(x,d)^{8c+2}_n(b,d)_n^{4c+1}=(x,d)^{8c+2}_n,$$
which implies that $x \in J$. Therefore, it follows that if $g=\diag(x,xb)$ satisfies $[g, a]=1$ for every  $a\in A$, then $g \in A$. This completes the proof when $n$ is odd.

Now, assume $n$ is even. The assertions (ii) and (iii) are clear. As in the odd case, (iv) follows from \cite[Lemma 3.1]{Szp6}. We are thus left to prove (i). From (ii) and (iii) along with Lemma \ref{Dcent}, it follows that $\wt{A}$ is abelian. Let $g'=\diag(c,d)$. We show that if $\s(g')$ commutes with every $\s(g)\in \wt{A}$,  then $g' \in A$. Indeed, by considering $g=(a,a^{-1}) \in A$, we deduce from \eqref{compeq}  that
$$(a,dc^{-1})_n=1 \text{ for all } a\in K.$$
Lemma \ref{nicesat2} implies $dc^{-1} \in K^2$. Similarly, by considering $g=\diag(x,x) \in T$, we deduce that $(x,dc)_n^{4c+1}=1$ for all $x\in J$. It then follows from Lemmas \ref{nicesat2} and \ref{oldlem} that $dc \in J^2$. Therefore, there exist $x \in K, y \in J$ such that
$$dc^{-1}=x^2, dc=y^2;$$
or equivalently,
 $$c^2=(xy)^2, d^2=(yx^{-1})^2.$$
 Since $-1 \in F^{\times d}$ and $F^{\times d} \subseteq J\cap K$, we may change $x$ to $-x$ and $y$ to $-y$ if necessary, and conclude that $c=xy$ and $d=yx^{-1}$. Thus,
 $$g'=\diag(xy,yx^{-1})=\diag(x,x^{-1}) \cdot \diag(y,y) \in A.$$
This completes the proof when $n$ is even.
\end{proof}

\begin{lm} \label{detj} One has the equality
$$\{ g \in G: \det(g) \in J^{n/d} \}=A G_o.$$
\end{lm}
\begin{proof} It is sufficient to prove
\begin{equation} \label{forDuse}
\{ t \in T \mid \det(t) \in J^{n/d} \}=A T_o.
\end{equation}
As the inclusion $\supseteq$ is clear, it suffices to  prove the converse inclusion.

Assume first that $n$ is even. If $\diag(a,b) \in A$ is such that $ab =x^2 \in J^2$, then
$$\diag(a,b)=\diag(x,x) \cdot \diag(x^{-1}a,xa^{-1}) \in A T_o.$$
Suppose now $n$ is odd. If $\diag(a,b) \in T$ is such that $ab =x \in J$, then we obtain
$$\diag(a,b)= \diag\bigl(x^{-2c},x^{-2c}(x^{4c+1})\bigr) \cdot \diag(ax^{2c},a^{-1}x^{-2c}).$$
The proof is completed once we show that $J^{4c+1} \subseteq K$. Indeed,
since $K \subseteq J= J^{\perp}_{(d_c)}$, it follows that for all $j \in J, k\in K$, we have
$$(j^{4c+1},k)_d=(j,k)_d^{4c+1}=1.$$
Thus, $J^{4c+1} \subseteq  K^{\perp}_{(d)}=K$, as desired.
\end{proof}

\subsection{Principal series representations}
The group $\wt{G}$ (resp. $\wt{G}_o$) splits uniquely over the unipotent radical $U$ of the Borel subgroup $B=TU$ (resp. $B_o=T_o U$). In fact,  it follows from the cocycle formula \eqref{rao} that the splitting is simply given by $u \mapsto \s(u)$.

Recall from \S \ref{SS:gps} that every genuine irreducible representation of $\wt{T}$ (resp. $\wt{T}_o$) is constructed as follows.  Let $\chi$ be a character of  $\wt{A}$ (or $\wt{A}_o$). We define ${\rm Ind}_ {\wt{A}}^{\wt{T}} \chi$ (resp. ${\rm Ind}_ {\wt{A}_o}^{\wt{T}_o} \chi_o$ ) to be the space of complex functions on $\wt{T}$ (resp. on $\wt{T}_o$) such that
$$f(ah)=\chi(a) \cdot f(h)$$
for all $a \in\wt{A}, \, h \in \wt{T}$ (resp. $a \in\wt{A}_o, h \in \wt{T}_o)$.  We denote by $i(\chi)$  (resp. by $i_o(\chi)$) the  above representation of $\wt{T}$ (of $\wt{T}_o$) acting on ${\rm Ind}_ {\wt{A}}^{\wt{T}} \chi$ (resp. ${\rm Ind}_ {\wt{A}_o}^{\wt{T}_o} \chi$) by right translations.

The isomorphism class of  $\sigma \in \Irr(\wt{T})$ (resp. of $\sigma_o \in \Irr(\wt{T}_o)$) is determined by its central character $\chi_\sigma$ (resp. $\chi_{\sigma_o}$). Moreover, a realization of $\sigma$  (resp. $\sigma_o$) is given by $i(\chi')$ (resp. $i_o(\chi'_o)$) where $\chi'$ is a character of  $\wt{A}$ (resp.  $\wt{A}_o$) which extends $\chi_\sigma$ (resp. $\chi_{\sigma_o}$). In particular, one has
$$\dim \sigma= [\wt{T}: \wt{A}]=(nn_c) \val{n n_c}^{-1/2}, \quad \dim \sigma_o= [\wt{T}_o:\wt{A}_o]=d \val{d}^{-\half}.$$

For every $\sigma \in \Irr(\wt{T})$, we have the principal series representation 
$$I(\sigma):={\rm Ind}_{\wt{B}}^{\wt{G}} \sigma$$
of $\wt{G}$. Similarly, one has 
$$I(\sigma_o):= {\rm Ind}_{\wt{T}_o}^{\wt{G}_o} \sigma_o$$ 
for every $\sigma_o \in \Irr(\wt{T}_o)$.  We note that $I(\sigma)$ could be realized in another equivalent formulation. First, for fixed $\sigma$, let
$$\chi_\sigma': \wt{A} \to \C^\times$$
be an extension of $\chi_\sigma: Z(\wt{T}) \to \C^\times$. Consider the induced representation
$$I(\chi_\sigma'):={\rm Ind}_{\wt{A} U}^{\wt{G}} (\chi_\sigma' \otimes \mbm{1}_U),$$
which consists of smooth functions
$$f:\wt{G} \rightarrow \C$$
such that
$$f(tug)=\delta_B(t)^{1/2}  \cdot \chi_\sigma'(t) f(g)$$
for all $t \in  \wt{A}, u\in U, \, g\in \wt{G}$. The action of $\wt{G}$ on $I(\chi_\sigma')$ is by right translations. Using induction by stages, we have
$$I(\sigma) \simeq I(\chi'_\sigma).$$
Similarly, let $\chi_{\sigma_o}'$ be  an extension of $\chi_{\sigma_o}$ to $\wt{A}_o$. We define in an analogous way  $I(\chi_{\sigma_o}')$, which gives $I(\sigma_o) \simeq I(\chi_{\sigma_o}')$.

\subsection{Restriction of $I(\sigma)$ to $\wt{G}_o$}
Let $\chi$  be a genuine character of $\wt{A}$. For a genuine character $\xi$ of $\wt{M}$, denote by
$$I(\chi)_\xi \subseteq I(\chi)$$
the subspace of $I(\chi)$ on which  $\wt{M}$ acts by $\xi$. Since $\wt{Z(G)}$ is the centralizer of $\wt{G}_o$ inside $\wt{G}$, it follows that $I(\chi)_\xi$ is a $\wt{G}_o$-space. If
$$\xi \notin \msc{E}(\wt{M},\chi|_{Z(\wt{G})}),$$
then $I(\chi)_\xi$ is trivial. Thus, we have a decomposition of $I(\chi)$ over the finite abelian group $\wt{M}/Z(\wt{G})$ as follows:
$$I(\chi)= \bigoplus_{\xi \in  \msc{E}(\wt{M},\chi|_{Z(\wt{G})})} I(\chi)_\xi.$$

\begin{lm} \label{deco}
Write $\sigma= i(\chi)$.
If $n$ is even, then
\begin{equation} \label{decoeven}
I(\chi)=\bigoplus_{t \in \wt{T}/ \wt{M}\wt{T}_o}  I(\chi)_{\chi^t|_{\wt{M}}}= \bigoplus_{t \in \wt{T}/ \wt{M}\wt{T}_o}  \sigma(t) I(\chi)_{\chi|_{\wt{M}}}.
\end{equation}
If $n$ is odd, then
\begin{equation} \label{decoodd}
I(\chi)=\bigoplus_{t \in \wt{Z(G)}/ \wt{M}}  I(\chi)_{\chi^t|_{\wt{M}}}= \bigoplus_{t \in \wt{Z(G)}/ \wt{M}}  \sigma(t) I(\chi)_{\chi|_{\wt{M}}}.
\end{equation}
\end{lm}
\begin{proof}
Assume first $n$ is even.  We observe that
$$\chi|_{\wt{M}} \in \msc{E}(\wt{M}, \chi|_{Z(\wt{G})}).$$
The first equality in \eqref{decoeven} follows from Lemma \ref{crucia} . To prove the second equality in \eqref{decoeven}, we note that if
$f \in  I(\chi)_{\chi|_{\wt{M}}}$, then for $m \in \wt{M}, t \in \wt{T}$ one has
$$I(\chi)(m) \circ I(\chi)(t^{-1})f=I(\chi)(t^{-1}) \circ \sigma(m^t)f=\chi(m^t) \cdot I(\chi)(t^{-1})f.$$
Thus, $f \mapsto \sigma(t^{-1})f$ is a linear isomorphism from $I(\chi)_{\chi|_{\wt{M}}}$ to $I(\chi)_{\chi^t|_{\wt{M}}}$. The proof for the odd case follows along the same line, and we omit the details.
\end{proof}

For $t \in \wt{T}$, let
$$I(\chi)_t \subseteq I(\chi)$$ be the subspace of functions in  $I(\chi)$ supported on the set $$\{g\in \wt{G} \mid \det(g) \in  \det(t)\cdot J^{n/d} \},$$
which is just equal to $t\wt{A}\wt{G}_o$ by Lemma \ref{detj}.

\begin{lm} \label{eigencoset}
For $t \in \wt{T}$, we have $I(\chi)_{\chi^t|_{\wt{M}}}=I(\chi)_t$.
\end{lm}
\begin{proof}
If $f \in  I(\chi)_{\chi^g|_{\wt{M}}}$, then for  all $m \in \wt{M}$, $g \in \wt{G}$ we have
$$(\sigma(m)f)(g)=\chi^t(m)f(g).$$
On the other hand,
$$(\sigma(m)f)(g)=f(gm)=f(m^gg)=\chi^{g}(m) f(t).$$
Thus, if $\chi^{t}|_{\wt{M}} \neq \chi^{g}|_{\wt{M}} $, then $f(t)=0$. It follows from Lemma \ref{detstab} that
$$I(\chi)_{\chi^t|_{\wt{M}}} \subseteq I(\chi)_t.$$

We now prove the converse inclusion. Suppose that $f \in I(\chi)$ is supported on $t\wt{T}\wt{G}_o.$
Since
$$t\wt{A}\wt{G}_o\wt{M}=t\wt{A}\wt{G}_o,$$
it suffices to show that if $g \in t\wt{A}\wt{G}_o$, then for all $m \in \wt{M}$ we have $f(gm)=\chi^t(m)\cdot f(g)$. Indeed, from Lemma \ref{Dcent} it follows that for $m \in \wt{M}$ and $g \in t\wt{A}\wt{G}_o$, we have $gm=m^t g$. This completes the proof.
\end{proof}

\begin{lm} \label{relem}
Let $\chi$ be a genuine character of $\wt{A}$ and $\chi_o$  be its restriction to $\wt{A}_o$. For $t \in \wt{T}$, we have
$$I(\chi)_t \simeq I_o((\chi_o)^t)$$
as representations of $\wt{G}_o$.
\end{lm}
\begin{proof}
Define
$$R_{\chi,t}: I(\chi)_t \rightarrow  I_o((\chi_o)^t)$$
by
$$R_{\chi, t}(f)(g_o)=f(tg_o).$$
Clearly, $R_{\chi,t} \in \Hom_{\wt{G}_o} \bigl(I(\chi)_t , I_o((\chi_o)^t)\bigr).$
We also define
$$Y_{\chi,t}: I_o((\chi_o)^t ) \rightarrow I(\chi)_t$$
by
$$\bigl((Y_{\chi,t})f\bigr)(g)=
\begin{cases}
(\chi)^t(a)f(g_o) & \text{ if } g=tag_o, a\in \wt{A},\ g_o \in \wt{G}_o; \\
0 &  \text{ if }  g \notin t \wt{A}\wt{G}_o.
\end{cases}
$$
We check that $Y_{\chi,t}$ is well-defined. Indeed, suppose that $ag_o=a'g_o'$ where $a, a' \in \wt{A}, g_o, g_o' \in \wt{G}_o$. Denote  $a_o=aa'^{-1}=g_o'g_o^{-1}$ and observe that $a_o  \in \wt{A} \cap \wt{G}_o=\wt{A}_o$.  For $f \in I_o((\chi_o)^t) $ we have
$$(\chi)^t(a')f(g_o')=(\chi)^t(a'a_o^{-1})f(a_og_o)=(\chi)^t(aa_o^{-1})(\chi_o)^t(a_o)f(g_o)=(\chi)^t(a)f(g_o).$$
Lastly, by a straightforward computation, one shows that
$$Y_{\chi,t} \in \Hom_{\wt{G}_o} \bigl( I_o((\chi_o)^t), I(\chi)_t)$$
and $Y^{-1}_{\chi,t}=R_{\chi,t}$. This completes the proof.
\end{proof}

\begin{prop} \label{decoprop}
Let $\chi$ be a genuine character of $\wt{A}$ and $\chi_o$ its restriction to $\wt{A}_o$.
If $n$ is odd, then
\begin{equation} \label{sldecoodd}
I(\chi)|_{\wt{G}_o} \simeq \bigoplus_{t \in \wt{Z(G)}/ \wt{M}}  I_o\bigl((\chi_o)^t\bigr) = n_c \val{n_c}^{-1/2} \cdot I_o(\chi_o).
\end{equation}
If $n$ is even, then
\begin{equation} \label{slddecoeven}
I(\chi)|_{\wt{G}_o} \simeq \bigoplus_{t \in \wt{T}/ \wt{M}\wt{T}_o}   I_o\bigl((\chi_o)^t\bigr) =   \bigoplus_{t \in \wt{T}/ {\wt{Z(G)}}\wt{T}_o}  d_c \val{d_c}^{-1/2} \cdot  I_o\bigl((\chi_o)^t\bigr)
\end{equation}
\end{prop}
\begin{proof}
We prove first when $n$ is odd. The isomorphism in \eqref{sldecoodd} follows  from \eqref{decoodd} combined with Lemmas \ref{eigencoset} and \ref{relem}. To prove the equality in  \eqref{sldecoodd}, we note that Lemma \ref{Dcent} implies that
$$I_o(\chi_o)=  I_o\bigl((\chi_o)^t\bigr)$$
 for all $t\in \wt{Z(G)}$.

Assume now that $n$ is even. The isomorphism in \eqref{slddecoeven} follows from a similar argument we used in the odd case. We have
$$\bigoplus_{t \in \wt{T}/ \wt{M}\wt{T}_o}  I_o\bigl((\chi_o)^t\bigr) \simeq \bigoplus_{t \in  \wt{T}/\wt{Z(G)} \wt{T}_o}\bigoplus_{g \in \wt{Z(G)} \wt{T}_o/\wt{M}\wt{T}_o}    I_o\bigl((\chi_o)^{tg}\bigr).$$
To prove the  equality in \eqref{slddecoeven}, observe that from Lemma \ref{Dcent} it follows that $I_o\bigl((\chi_o)^{tg}\bigr)=I_o\bigl((\chi_o)^t\bigr)$ for every $g\in \wt{Z(G)}$, and that by  Lemma \ref{evenslcent} one has
$$[\wt{Z(G)} \wt{T}_o: \wt{M}\wt{T}_o]=d_c \val{d_c}^{-1/2}.$$
This completes the proof.
\end{proof}

\begin{thm} \label{decopropcor}
Let $\sigma \in \Irr(\wt{T})$.
\begin{enumerate}
  \item[(i)] Suppose that $n$ is odd. Let $\sigma_o \in \Irr(\wt{T}_o)$ be the genuine smooth irreducible  representation  of $\wt{T}_o$ determined by the relation $\chi_{\sigma_o}=\chi_\sigma|_{Z(\wt{T}_o)}.$ Then
  $$I(\sigma)\mid_{\wt{G}_o} \simeq  n_c \val{n_c}^{-1/2} \cdot I(\sigma_o).$$
  \item[(ii)] Suppose that $n$ is even. Then,
      $$I(\sigma)|_{\wt{G}_o} \simeq \bigoplus_{ \substack{ \sigma_o \in  \Irr(\wt{T}_o) \\ \chi_{\sigma_o} \in E_o(\chi_{\sigma}|_{Z(\wt{T})\cap \wt{T}_o} )  }}  d_c \val{d_c}^{-1/2}  \cdot I_o(\sigma_o).$$
  \item[(iii)] Suppose that $n$ is even. Fix a $ \sigma_o \in  \Irr(\wt{T}_o)$ such that $\chi_{\sigma_o}$ agrees with  $\chi_{\sigma}$ on $Z(\wt{T})\cap \wt{T}_o$. Then we have
      $$I(\sigma)|_{\wt{G}_o} \simeq \bigoplus_{x\in F^\times/ F^{\times 2}}  d_c \val{d_c}^{-1/2}\cdot  I_o(\eta_{x,(n)} \otimes\sigma_o),$$
      where $\eta_{x,(n)}$ is viewed as a (non-genuine) character of $\wt{T}_o$ given by
      $$\s(\alpha^\vee(a)) \cdot \zeta \mapsto \eta_{x,(n)}(a) \text{ for all } \zeta \in \bbmu_n.$$
      Furthermore, if $n \equiv 2 \ (\text{mod } 4)$, then
      $$I(\sigma)|_{\wt{G}_o} \simeq  \bigoplus_{x\in F^\times/F^{\times 2}} d_c \val{d_c}^{-1/2} \cdot I_o(\eta_{x,(2)} \otimes\sigma_o).$$
\end{enumerate}
\end{thm}
\begin{proof}
The first assertion (i) follows immediately from Proposition \ref{decoprop}. To prove  (ii), one uses Proposition \ref{decoprop} and Lemma \ref{evenslcent}. For (iii), we note that as a set of representatives of  $Z(\wt{T})\cap \wt{T}_o$, one may pick a set of elements in $\wt{T}$ such that the determinants of the elements inside vary over  $F^\times/F^{\times 2}$. In addition, by Lemma \ref{evenslcent},
$$\chi^t_{\sigma_o}=\eta_{\det(t),(n)}|_{Z(\wt{T}_o)} \otimes \chi_{\sigma_o}.$$
For the last statement in (iii), note that $\eta_{x,(n)}|_{Z(\wt{T}_o)}$ is a quadratic character. Thus, if $d$ is odd, then
$$\eta_{x,(n)}|_{Z(\wt{T}_o)}=\eta_{x,(n)}^d |_{Z(\wt{T}_o)}=\eta_{x,(2)} |_{Z(\wt{T}_o)}.$$
This completes the proof.
\end{proof}

The results above for $I(\sigma)$ have a parallel for restricting the representation $i(\chi)$ to $\wt{T}_o$.

\begin{prop} \label{decoproph}
Let $\chi$ be a genuine character of $\wt{A}$ and let  $\chi$  be the restriction of $\chi_o$ to $\wt{A}_o$. Let $\sigma \in \Irr(\wt{T})$.
\begin{enumerate}
  \item[(i)] If $n$ is odd, then
$$i(\chi)|_{\wt{T}_o} =\bigoplus_{t \in \wt{Z(G)}/ \wt{M}}  i_o\bigl((\chi_o)^t\bigr) = n_c \val{n_c}^{-1/2} \cdot i_o(\chi_o).$$
\item[(ii)] If $n$ is even, then
$$i(\chi)|_{\wt{T}_o}=\bigoplus_{t \in \wt{T}/ \wt{M}\wt{T}_o}   i_o\bigl((\chi_o)^t\bigr) =   \bigoplus_{t \in \wt{T}/ {\wt{Z(G)}}\wt{T}_o}  d_c \val{d_c}^{-1/2} \cdot i_o\bigl((\chi_o)^t\bigr).$$
\item[(iii)] Suppose that $n$ is odd. Let $\sigma_o$ be the genuine smooth irreducible  representation  of $\wt{T}_o$ such that $\chi_{\sigma_o}=\chi_\sigma|_{Z(\wt{T}_o)}.$ Then,
$$\sigma|_{\wt{T}_o} \simeq  n_c \val{n_c}^{-1/2} \cdot \sigma_o.$$
\item[(iv)] If $n$ is even, then
      $$\sigma|_{\wt{T}_o} \simeq  \bigoplus_{ \substack{ \sigma_o \in  \Irr(\wt{T}_o) \\ \chi_{\sigma_o} \in E_o(\chi_{\sigma}|_{Z(\wt{T})\cap \wt{T}_o})  }}  d_c \val{d_c}^{-1/2} \cdot \sigma_o.$$
  \item[(v)] Suppose that $n$ is even. Fix a $ \sigma_o \in  \Irr(\wt{T}_o)$ such that $\chi_{\sigma_o}$ agrees with  $\chi_{\sigma}$ on $Z(\wt{T})\cap \wt{T}_o$. Then,
      $$\sigma|_{\wt{T}_o} \simeq \bigoplus_{x\in F^\times/F^{\times 2}}  d_c \val{d_c}^{-1/2} \cdot \eta_{x,(n)} \otimes\sigma_o.$$
      If $n \equiv 2 \ (\text{mod } 4)$, then
 $$\sigma|_{\wt{T}_o} \simeq \bigoplus_{x\in F^\times/F^{\times 2}}  d_c \val{d_c}^{-1/2} \cdot \eta_{x,(2)} \otimes\sigma_o.$$
\end{enumerate}

\end{prop}
\begin{rmk} The main difference between Proposition \ref{decoprop} and Theorem \ref{decopropcor} on one side and  Proposition \ref{decoproph} on the other side is that in the latter we write an irreducible $\sigma \in \Irr(\wt{T})$ as a direct sum of irreducible representations of $\wt{T}_o$. However, the principal series representation in Proposition \ref{decoprop} and Theorem \ref{decopropcor} may be reducible. In fact, if $n$ is odd, then it is possible that $I(\sigma)$ is irreducible, while $I(\sigma_o)$ becomes reducible.To see more of this, one may compare Corollary \ref{irrres} below with \cite[Proposition 5.5]{Szp6}. This phenomenon does not occur if $n$ is even.
\end{rmk}

\begin{rmk}
The statements  (i) and (ii) in Theorem \ref{decopropcor} and (iii) and (iv) in Proposition \ref{decoproph} are proven as Propositions 4.4.1, 4.4.2, 4.48 and 4.49 in \cite{Kar} using Mackey theory under the assumption that $\gcd(p,n)=1$ for the untwisted Kazhdan-Patterson covering $\wt{\GL}_2$ (i.e., when $c=0$). In our paper, we have used results such as in Lemmas \ref{deco}, \ref{eigencoset} and  \ref{relem} to prove the general case. In fact, these lemmas will play a crucial role in the computation of local coefficients matrices of $\wt{G}$ in the next section.
\end{rmk}

\subsection{The unramified case}
In this subsection we assume that
$$\gcd(p,n)=1,$$
and recall part of \S \ref{SS:unrep}. Under this assumption $\wt{G}$ splits over $K=\GL_2(O_F)$, and we fix such a splitting $s_K$.
For $\sigma \in \Irr(\wt{T})$, the representation $I(\sigma)$ is called $s_K$-unramified if it contains a  non-zero $s_K(K)$-fixed vector, in which case one has
$$\dim I(\sigma)^{s_K(K)} = 1.$$
We view $K_o= K\cap G_o$ as a subgroup of $\wt{G}_o$ via the splitting, and the notion of unramified $I(\sigma_o), \sigma\in \Irr(G_o)$ is defined in a similar way.

We note that $s_K$ is not unique, and thus the notion of $s_K$-unramified genuine principal series representation of $\wt{G}$ depends on the splitting. However, the restriction of $s_K$ to $K_o$ is unique; therefore, the notion of unramified genuine principal series representation of $\wt{G}_o$ is defined canonically. In fact, one can check that for $ t \in T_o \cap K_o$, we have
$$s_K(t)= \s(t).$$
Since  $\gcd(p,n)=1$, the restriction of the $n$-th power Hilbert symbol to $O_F^\times \times O_F^\times$ is trivial. Consequently, \eqref{rao} implies that the section $\s$ gives a splitting of $\wt{G}$ over $T \cap K$.

We also note that $I(\sigma_o)$ is unramified if and only if the restriction of $\chi_{\sigma_o}$ to $Z(\wt{T}_o) \cap K_o$ is trivial; equivalently, if and only if
$$\chi_{\sigma_o}( \s (\alpha^\vee(a)) \cdot \zeta)=\zeta$$
for all $a \in O_F^{\times d}$.

\begin{prop} \label{firstunramprop}
Assume that $\gcd(n,p)=1$. Let $\sigma \in \Irr(\wt{T})$. Then, the restriction of $\sigma$ to $\wt{T}_o$ contains an unramified $\sigma_o \in  \Irr(\wt{T}_o)$ (and thus $I(\sigma_o)$ is unramified) if and only if the restriction of $\chi_{\sigma}$ to $Z(\wt{T}) \cap G_o$ is trivial; equivalently, if and only if
$$\chi_{\sigma}(\s(\alpha^\vee(a) ) \cdot \zeta) =\zeta$$
for all $a \in O_F^{\times n}.$
\end{prop}
\begin{proof}
From Lemma  \ref{centers} we have
$$Z(\wt{T}) \cap K_o \subseteq Z(\wt{T}_o) \cap K_o$$
and thus it follows from the preceding discussion that if $\sigma$ contains an unramified $\sigma_o \in \Irr(\wt{T}_o)$, then the restriction of $\chi_{\sigma}$ to $Z(\wt{T}) \cap K_o$ is trivial.

We prove the converse. If $n$ is odd, then Lemma \ref{centers} gives
$$Z(\wt{T}) \cap K_o = Z(\wt{T}_o) \cap K_o.$$
Hence, the assertion follows in this case. Suppose now that $n$ is even and assume that
$$\chi_{\sigma_o}(\s(\alpha^\vee(a)) \cdot \zeta)=\zeta \text{  for all } a \in O_F^{\times n}.$$
Pick $\sigma_o \in \Irr(\wt{T}_o)$ which occurs in $\sigma$. If $\chi_{\sigma_o}$ is trivial on $Z(\wt{T}_o) \cap K_o$, we are done. Otherwise, we prove that $\eta^{-1}_{\varpi,(n)} \otimes \sigma_o \in \Irr(\wt{T}_o)$ is unramified; equivalently,  $I_o(\eta^{-1}_{\varpi,(n)} \otimes \sigma_o)$ is unramified.
Indeed, we have in this case
$$[Z(\wt{T}) \cap K_o :Z(\wt{T}_o) \cap K_o]= [O_F^{\times d}: O_F^{\times n}].$$
Since by Lemma \ref{dualcenter} we have
$$[F^{\times d}: F^{\times n}]=[F^\times: F^{\times 2}]=4,$$
it follows that
$$[O_F^{\times d}: O_F^{\times n}]=2.$$
Thus, for $a \in O_F^{\times d}$, we have
$$\chi_{\sigma_o}(\s(\alpha^\vee(a))\cdot \zeta )=\zeta \cdot
 \begin{cases} 1  &  \text{if } a \in O_F^{\times n}; \\
 -1  &   \text{otherwise}.
 \end{cases}$$
Lastly, we note that if $a \in O_F^{\times d}$, then
$$(a,\varpi)_n=
\begin{cases}
1  &  \text{if } a \in O_F^{\times n}; \\
-1  &   \text{otherwise}.
\end{cases}$$
This shows that the restriction of $\chi_{\eta^{-1}_{\varpi,(n)} \otimes \sigma_o}$ to $Z(\wt{T}_o) \cap K_o$ is trivial, and thus the proof is completed.
\end{proof}

We remark that if  $I(\sigma)$ is unramified, then $\sigma$ satisfies the assumption of Proposition \ref{firstunramprop}. However, the converse is not true.

\begin{prop}  \label{secondunramprop}
Suppose that $n$ is even. Let  $\sigma \in \Irr(\wt{T})$, and let $\sigma_o \in \Irr(\wt{T}_o)$ be an unramified constituent in  $\sigma|_{\wt{T}_o}$. Let $u \in O_F^\times$ be a non-square element.
\begin{enumerate}
\item[(i)] One has
  $$\sigma|_{\wt{T}_o} = \set{ \sigma_o, \  \eta_{u,(n)} \otimes \sigma_o, \  \eta_{\varpi^{-1},(n)} \otimes \sigma_o, \  \eta_{u\varpi^{-1},(n)} \otimes \sigma_o } \subseteq \Irr(\wt{T}_o).$$
Moreover, if  $n \equiv 2 \ (\text{mod } 4)$, then
$$\eta_{x,(n)} \otimes \sigma_o \cong  \eta_{x,(2)} \otimes \sigma_o.$$
\item[(ii)] The principal series $I(\eta_{u,(n)} \otimes \sigma_o)$ is  unramifed. On the other hand, $I(\eta_{\varpi^{-1},(n)} \otimes \sigma_o)$ and $I(\eta_{u\varpi^{-1},(n)} \otimes \sigma_o)$ are both ramified and contain one-dimensional  $K_o^{\s(e_2(\varpi))}$-fixed subspace.
\end{enumerate}
\end{prop}
\begin{proof} The first assertion (i) follows from Lemma \ref{pn1} and Corollary \ref{decopropcor}. We prove (ii). Since $\eta_{u,(n)}$ is an unramified character, it follows that the restriction of
$\chi_{\eta_{u,(n)} \otimes \sigma_o}$ to $Z(\wt{T}_o) \cap K_o$ is trivial. This shows that $\eta_{u,(n)} \otimes \sigma_o$ is unramified. Using a similar argument as the proof of Proposition \ref{firstunramprop},  we deduce that $I(\eta_{\varpi^{-1},(n)} \otimes \sigma_o)$ and $I(\eta_{u\varpi^{-1},(n)} \otimes \sigma_o)$ are ramified. We finally note that
$$I(\eta_{a\varpi^{-1},(n)} \otimes \sigma_o) \simeq I(\eta_{a,(n)} \otimes \sigma_o)^{ \s(e_2(\varpi))}$$
for $a\in F^\times$. This completes the proof.
\end{proof}

\subsection{Lower bound on $\dim \Wh_\psi(\pi)$}
In Proposition \ref{decoprop} and Theorem \ref{decopropcor}, we have proven that if $\wt{Z(G)}$ is not abelian,  then the restriction of a genuine principal series of $\wt{G}$ to $\wt{G}_o$ is never multiplicity-free. We generalize this result.

\begin{prop} \label{no1}
Let $\pi$ be a genuine smooth representation of $\wt{G}$ and let $\pi_o$ be a genuine smooth representation of $\wt{G}_o$.
\begin{enumerate}
\item  If ${\rm Hom}_{\wt{G}_o}(\pi, \pi_o)$ is not trivial, then
$$\dim {\rm Hom}_{\wt{G}_o}(\pi, \pi_o) \geq \sqrt{[\wt{Z(G)}: Z(\wt{Z(G)})]}=d_c \cdot \val{d_c}^{-1/2}.$$
\item If ${\rm Hom}_{\wt{G}_o}(\pi_o, \pi)$ is not trivial, then
$$\dim {\rm Hom}_{\wt{G}_o}(\pi_o, \pi) \geq \sqrt{[\wt{Z(G)}: Z(\wt{Z(G)})]}=d_c \cdot \val{d_c}^{-1/2}.$$
\end{enumerate}
\end{prop}
\begin{proof} We prove only the first assertion as the second assertion is proven similarly. Fix $T \in \Hom_{\wt{G}_o}(\pi, \pi_o)$ and $g \in \wt{Z(G)}$. It follows from Lemma \ref{Dcent} that
$$T \circ \pi(g) \in  \Hom_{\wt{G}_o}(\pi, \pi_o).$$
This implies that  $\wt{Z(G)}$ acts on the non-zero space $\Hom_{\wt{G}_o}(\pi_o, \pi)$, giving rise to a smooth and genuine representation of $\wt{Z(G)}$. Since  $\wt{Z(G)}$ is a Heisenberg-type group, it follows from the Stone-von Neumann Theorem that any genuine smooth representation of $\wt{Z(G)}$ is of dimension at least $d_c \val{d_c}^{-1/2}$.
\end{proof}

The restriction problem of some special irreducible genuine principal series representations from the double cover $\wt{\GSp}_{2r}^{(2)}$ to $\wt{\Sp}_{2r}^{(2)}$ is investigated in \cite{Szp4-1}, when $r$ is odd. In addition, some other examples are given in \cite{PatPr1} for the high multiplicity which occurs in the restriction from a double cover of $G$ to the double cover of $G_o$. However, in these examples, the inverse image of the center of the underlying linear group is always abelian; hence, they are of a different nature compared to Proposition \ref{no1}.

\begin{prop} \label{whifail}
Let $\pi$ be a generic genuine smooth representation of $\wt{G}$. Then
$$\dim \Wh_\psi(\pi) \ge d_c \val{d_c}^{-1/2}.$$
\end{prop}
\begin{proof} By definition, the space of  Whittaker functionals on $\pi$ is $\Hom_U(\pi, \C_\psi)$. Fix $\xi \in \Hom_U(\pi, \C_\psi)$. Since $U \subseteq \wt{G}_o$, it follows that for $g\in \wt{Z(G)}$, $\xi \circ \pi(g) \in \Hom_U(\pi, \C_\psi)$. In other words, $\wt{Z(G)}$ acts genuinely on $\Hom_U(\pi, \C_\psi)$. Thus, exactly as in Proposition \ref{no1}, we deduce that if $\Hom_U(\pi, \C_\psi)$ is not trivial then its dimension is at least $\sqrt{[\wt{Z(G)}: Z(\wt{Z(G)})]}=d_c \cdot \val{d_c}^{-1/2}.$

%it follows that if $\pi$ is generic, then at least one irreducible smooth representation of $\wt{G}_o$ which occurs in $\pi$ is generic. We now use Proposition \ref{no1} to conclude our proof.
\end{proof}

For a generic representation of a quasi-split linear reductive group, Whittaker model is unique (see \cite{Shal, GK, BZ1}). Thus, Proposition \ref{whifail} implies a complete contrast for $\wt{G}$, when $\wt{Z(G)}$ is not abelian.

%%%
\section{Local coefficients matrix for $\wt{\SL}_2$} \label{S:LCM-Go}

%\section{Local coefficients matrices and related invariants for $\wt{G}_o$.}
In this section, we first review some relevant results from \cite{GoSz}, \cite{Szp6} and \cite{GSS1}. In fact we complement these results wherever necessary, notably regarding the trace and the case where $4|n$. The goal is to determine the invariants $\Tr(\mca{M}(w, \sigma_o, s, \psi))$ and $\det(\mca{M}(w, \sigma_o, s, \psi))$ from an explicit local coefficients matrix $\mca{M}(w, \sigma_o, s, \psi)$ for $\wt{\SL}_2$, even for ramified data.

For  $\sigma_o \in \Irr(\wt{T}_o)$, we denote by $\chi'_{\sigma_o}$ a fixed extension of   $\chi_{\sigma_o}$ to $\wt{A}_o$. As a consequence of the Stone-von Neumann Theorem,  we may assume without loss of generality that
$$\sigma_o=i(\chi'_{\sigma_o}).$$

\subsection{A convenient model} \label{commsl}
For  $s \in \C$, we denote
$$\sigma_{o,s}=\delta_{B_o}^{s/2} \otimes \sigma.$$
We view the representation space of $I(\sigma_{o,s})$  as the space of functions
$$f: \wt{T}_o \times \wt{G}_o \rightarrow \C,$$
which are smooth from the right in the right argument and satisfy
$$f(at_o,tug)=\delta_{B_o}^{\frac {s+1}{2}}(t)  \cdot \chi'_{\sigma_o}(a) \cdot f(t_ot,g)$$
for all $t,t_o \in \wt{T}_o, a \in \wt{A}_o,  u\in U$ and  $g\in \wt{G}_o$.
Similarly, the representation space of $I\bigl((\sigma_{o,s})^w\bigr)$ consists of functions
$$f: \wt{T}_o \times \wt{G}_o \rightarrow \C,$$
which are smooth from the right in the right argument satisfying
$$f(at_o,tug)=\delta_{B_o}^{\frac {-s+1}{2}}(t) \cdot \chi'_{\sigma_o}(a) \cdot f(t_ot^w,g)$$
for all $t,t_o \in \wt{T}_o,  a \in \wt{A}_o,  u\in U$ and $g\in \wt{G}_o.$ The group $\wt{G}_o$ acts on both spaces by right translations on the right argument. Denote
$$\chi'_{\sigma_{o,s}}=\delta_{B_o}^{s/2}|_{\wt{A}_o} \otimes \chi'_{\sigma_o}.$$

Recall that using induction by stages, we have
$$I(\sigma_{o,s}) \simeq I(\chi'_{\sigma_{o,s}}).$$
Local coefficients matrices for $\wt{G}_o$  are computed in \cite{Szp6}  when $n \not \equiv 0 \ (\text{mod }4)$, by realizing the principal series $I(\sigma_{o,s})$ as $I(\chi'_{\sigma_{o,s}})$. In \S \ref{0mod4mat} below, we shall deal with the case $n \equiv 0 \ (\text{mod }4)$. In \S \ref{mainproof}, we shall use a similar realization for $\wt{G}=\wt{\GL}_2$.

Let $T(w, \sigma_{o,s}): I(\sigma_{o,s}) \to I((\sigma_{o,s})^w)$ be the standard intertwining operator which is given by
\begin{equation} \label{Inter-R}
T(w, \sigma_{o,s})(f_s)(t,g)=\int_{F} f_s(t,\wt{w} u(x) g) \ d_\psi x,
\end{equation}
where $u(x)=e_\alpha(x) \in U$.% We define analogously $$T(w, \chi_{\sigma_{o,s}}'): I(\chi_{\sigma_{o,s}}') \to I( (\chi_{\sigma_{o,s}}')^w)$$
 %as follows: for $h_s  \in I\bigl(\chi'_{o,s}\bigr)$ and $g \in \wt{G}_o$,  $ T(w, \chi_{\sigma_{o,s}}')(h_s)(g)$ is the meromorphic continuation of
%$$ \int_F h_s\bigl(\wt{w}u(x)g \big) \ d_\psi x.$$
%This integral converges absolutely wherever $T(w, \sigma_{o,s})$  converges absolutely.

The following Lemma is proven as \cite[Lemma 4.6]{Szp6} for the cases where  $n \not \equiv 0 \ (\text{mod }4)$; however, the same proof works for all $n$.

\begin{lm} \label{comdo}
The following diagram of $\wt{G}_o$-maps commutes:
$$\begin{tikzcd}
I \bigl(\chi'_{\sigma_{o,s}}\bigr) \ar[d, "{ T\bigl(w, \chi'_{\sigma_{o,s}}\bigr) }"']  &  I \bigl(\sigma_{o,s}\bigr)   \ar[l, "{M_{o,s}}"']  \ar[d, "{ T(w, \sigma_{o,s})}"]  \\
I \bigl((\chi'_{\sigma_{o,s}})^w\bigr)  \ar[r, "{N_{o,s}}"]  & I \bigl((\sigma_{o,s})^w\bigr) ,
\end{tikzcd} $$
where $M_{o,s}$ and $N_{o,s}$ are the $\wt{G}_o$ isomorphisms defined by
$$\bigl(M_{o,s}(f)\bigr)(g) = f\bigl(1,g\bigr), \quad \bigl(N_{o,s}(h)\bigr)(t,g) =\delta_{B_o}(t)^{\frac{1-s}{2}} h(t^wg),$$
with inverse given by
$$\bigl((M_{o,s})^{-1}(h)\bigr)(t,g)=\delta_{B_o}(t)^{\frac{-1-s}{2}} h(tg), \quad \bigl((N_{o,s})^{-1}(f)\bigr)(g) = f\bigl(1,g\bigr),$$
and where for $h_s=M_{o,s}(f_s) \in I\bigl(\chi'_{o,s}\bigr)$ and $g \in \wt{G}_o$,  $ T(w, \chi_{\sigma_{o,s}}')(h_s)(g)$ is the meromorphic continuation of
$$ \int_F h_s\bigl(wn(x)g \big) \, d_\psi x.$$
This integral converges absolutely wherever $T(w, \sigma_{o,s})$  converges absolutely.

\end{lm}

Recall that given two representations $\pi$ and $\varsigma$ of $\wt{G}_o$ and   $T \in \text{Hom}_{\wt{G}_o} \bigl(\pi, \varsigma \bigr)$, one obtains by duality
$$T^*:\text{Wh}_\psi (\varsigma)\rightarrow \text{Wh}_\psi (\pi), \text{ where } T^*(l)=l \circ T.$$
Fix $\sigma_o \in \Irr(\wt{T}_o)$ and an ordered basis $\mfr{R}$ of $\sigma_o^\vee$. Let
$$\mfr{B}_{\sigma_s}(\mfr{R}) \subset \Wh_\psi(I(\sigma_s))$$
and
$$\mfr{B}_{(\sigma_s)^w}(\mfr{R}) := \mca{C}( \mfr{B}_{\sigma_s}(\mfr{R}) ) \subset \Wh_\psi(I( (\sigma_s)^w  ))$$
be the two ordered bases.

\begin{prop} \label{ForRuse}
The following three matrices are equal:
\begin{enumerate}
\item[(i)] the local coefficients matrix representing $ \mca{T}\bigl(w, \sigma_{o,s}\bigr)^*$ with respect to the ordered basis $\mfr{B}_{\sigma_s}(\mfr{R})$;
\item[(ii)] the matrix representing $T\bigl(w, \sigma_{o,s}\bigr)^*$ with respect to the two ordered bases $\mfr{B}_{(\sigma_s)^w}(\mfr{R}) $ and $\mfr{B}_{\sigma_s}(\mfr{R})$;
\item[(iii)] the matrix representing  $T\bigl(w, \chi'_{\sigma_{o,s}}\bigr)^*$ with respect to the two ordered bases $N_{o,s}^* (\mfr{B}_{(\sigma_s)^w}(\mfr{R}) )$ and $({M_{o,s}^{-1}})^* (\mfr{B}_{\sigma_s}(\mfr{R}))$.
\end{enumerate}
\end{prop}
\begin{proof}
The identity between (i) and (ii) is just Remark \ref{R:LCM}, and that between (ii) and (iii) follows from Lemma \ref{comdo}.
\end{proof}

\begin{rmk} \label{2str}
Our strategy adopted in this section and the next for computing the (entries of the) local coefficients matrix for $\wt{\SL}_2$ and $\wt{\GL}_2$ is different from that in \S \ref{S:PG}, where we concentrate on unramified representations of a general $\wt{G}$. Indeed, in view of the diagram in \eqref{Strat}, we computed in the unramified setting  explicitly the two matrices $\mca{S}_\mfr{R}(w, i(\chi); r_w^{\rm un})$ and $\mca{C}(\mfr{B}_{{}^\w \chi}, \mfr{B}_\chi; r_w^{\rm un})$ to obtain the local coefficients matrix.

However, in this section and the next, we do not start with an explicit isomorphism $r_w$. Instead, by using the above convenient model, we have essentially exhibited an isomorphism $N_{o}:  I(i({}^\w \chi)) \to I( {}^w i(\chi) )$ directly in the notation of \eqref{Strat}. That is, $N_o^*$ takes the place of the $(r_w^*)^{-1}$ in diagram \eqref{Strat}. Fix a basis $\mfr{B}$ for $\Wh_\psi(I(i(\chi)))$, which then gives rise to a basis $N_o^* \circ \mca{C}(\mfr{B})$ of $\Wh_\psi(I(i({}^\w \chi)))$. We will essentially compute the matrix representing  $$T(w, i(\chi))^* \circ (N_{o}^*)^{-1}: \Wh_\psi(I( i({}^\w \chi) ))  \to   \Wh_\psi(I( i(\chi) ))  $$
with respect to the bases $N_o^* \circ \mca{C}(\mfr{B})$ and $\mfr{B}$ on the two sides. This matrix is just the local coefficients matrix $\mca{M}_\mfr{B}(w, i(\chi))$ with respect to $\mfr{B}$, as asserted in Proposition \ref{ForRuse}.
\end{rmk}

\subsection{Parametrization of genuine characters} \label{parchar}
\subsubsection{$n$ is odd}
Assume that $n$ is odd. In this case the restriction of the cocycle to $A_o\times A_o$ is trivial. In particular the sets of genuine characters of both $Z(\wt{T}_o)$ and $\wt{A}_o$ are canonically  parameterized by the sets of characters of the underlying linear groups (i.e., the image of $Z(\wt{T}_o)$ and $\wt{A}_o$ in $T_o$ respectively). More precisely, $\chi'_{\sigma_o}$ takes the form
$$\s(\alpha^\vee(a)) \cdot \zeta) \mapsto \zeta \cdot \chi(a)$$
where $\chi: F^\times \to \C^\times$ is a character of $F^\times$.
Clearly, $\chi$ is not uniquely determined by $\sigma_o$,  since $\chi'_{\sigma_o}$ is determined only by the restriction of $\chi$ to $K$. Moreover, the isomorphism class of $\sigma_o$ is determined by $\chi_{\sigma_o}$, which depends only on the restriction of $\chi$ to $F^{\times d}$.

\subsubsection{$n \equiv 2 \ (\text{mod }4)$}

Assume now that  $n \equiv 2 \ (\text{mod }4)$. In this case the restriction of the cocycle to $Z(\wt{T}_o) \times Z(\wt{T}_o)$ is not trivial. Thus, although the sets of genuine characters of both $Z(\wt{T}_o)$ and $\wt{A}_o$ are  parameterized by the sets of characters of the underlying linear groups, such parametrization is not canonical. Fix a non-trivial character $\psi'=\psi_b$ of $F$ and a character of $\chi$ of $F^\times$.
Define
$$\chi_{\psi'}: F^\times \rightarrow \C$$
 by
 $$\chi_{\psi'}(a)=\chi(a) \cdot \omega_{\psi'}(a)^{-1}.$$
Note that
\begin{equation} \label{cparpsi}
\chi_{\psi'}(a)=\chi(a) \cdot \eta_{b,(2)}(a) \cdot \omega_\psi(a)^{-1}.
\end{equation}
Every $\chi'_{\sigma_o}$ has the form
$$\s(\alpha^\vee(a)) \cdot \zeta  \mapsto \zeta  \cdot \chi_{\psi'}(a)$$
for suitable choices of $\chi$ and $\psi'$.  Furthermore, since $d$ is odd in this case,  $\chi_{\sigma_o}$ is not independent of $\psi'$. We note that for a fixed $\psi'$,  $\chi_{\sigma_o}$ is determined only on the restriction of $\chi$ to $F^{\times d}$. However, the restriction of $\chi_{\sigma_o}$ to $Z(\wt{T}) \cap \wt{A}_o$ is independent of $\psi'$.

Henceforth, in the $n \equiv 0 \ (\text{mod } 4)$ case we will always choose $\psi'=\psi$ to parameterize the set of genuine characters of $\wt{A}_o$ and $Z(\wt{T}_o)$. Here $\psi$ is the additive character of $F$ used in defining $\Wh_\psi(\pi)$.

\subsubsection{$n \equiv 0 \ (\text{mod } 4)$} \label{par4}
In this case, since $n$ divides $d^2$,  it follows that
$$(x,y)_n=1 \text{ for all } x,y \in F^{\times d}.$$
Consequently, the set of  genuine characters of $Z(\wt{T}_o)$ are canonically parameterized by $\widehat{F^{\times d}}$. Thus, on the one hand, the situation is similar to the odd case, i.e., 
\begin{equation} \label{cenchar4}
\chi_\sigma (\s(\alpha^\vee(a)) \cdot \zeta)=\zeta\cdot \chi(a)
\end{equation}
where again $\chi$ is a character of $F^\times$ defined up to twisting by elements of $\widehat{F^\times/F^{\times d} }$.

On the other hand, the cocycle on $A_o\times A_o$ is not trivial. Therefore, the set of genuine characters of a maximal abelian subgroup of $\wt{T}_o$ is not canonically parameterized by characters of the underlying subgroup of $T_o$. In fact, we do not know how to parameterize this set in the case where $\gcd(n,p)\ne 1$. Since at this moment  our methods for the computation of the local coefficients matrix rely on this parametrization we shall need to assume that $\gcd(n,p)=1$ wherever we compute the local coefficient matrices and related invariants in  the case $n \equiv 0 \ (\text{mod }4)$.

In contrast to the case $n \not \equiv 0 \ (\text{mod }4)$,  we now explicitly choose $K$ by
$$K=K_d.$$
Under the assumption that $\gcd(n,p)=1$, the group $K \subset F^\times$ is a $d$-Lagrangian subgroup (see Lemma \ref{pn1}). Thus, the associated group $\wt{A}_o$ is again a maximal abelian subgroup of $\wt{T}_o$.

Fix a non-trivial character $\psi'$ of $F$. For  a character $\chi$ of $F$, we define
$$\chi_{\psi'}: K \rightarrow \C^\times$$
 by
$$ \chi_{\psi'}(x)=
\chi(x)\omega_{\psi'}^{-1}(x)
\cdot
\begin{cases}
1  & \text{ if } x \in K_n; \\
\omega_{\psi'}(\varpi)^{-1} \cdot \eta^d_{\varpi,(n)}(x)  &  \text{ if } x  \not \in K_n.
\end{cases}. $$
It is shown in \cite[\S 3.2]{GoSz} that $\chi'_{\sigma_o}$ has the form
$$\s(\alpha^\vee(a)) \cdot \zeta \mapsto \zeta \cdot \chi_{\psi'}(a)$$
for a suitable choice of $\chi$. We note that since $4|n$, it follows that $K \subseteq O_F^\times  F^{\times 2}$ and  $-1 \in F^{\times 2}$. Thus, by Lemma \ref{for0mod4}, we may pick $\psi'$ such that
\begin{equation} \label{par0mod4}
\chi_{\psi'} \bigr( \s(\alpha^\vee(a)) \cdot \zeta \bigl)=\zeta  \chi(x) \cdot
\begin{cases}
1  & \text{ if } x \in K_n; \\
\eta^d_{\varpi,(n)}(x)  &  \text{ if } x  \not \in K_n.
\end{cases}
\end{equation}
We emphasize that contrary to the case $n \equiv 2 \ (\text{mod }4)$ case, the character $\psi'$ is fixed independently of $\psi$.

\subsection{Linear character associated with an element of $\Irr(\wt{T}_o)$.} \label{linchasec}
In light of the discussion in Section \ref{parchar} we define a linear character $\chi$ associated with  $\sigma_o \in \Irr(\wt{T}_o)$ to be a character of $F^\times$ satisfying the following properties:
\begin{enumerate}
\item[$\bullet$] if $n \not \equiv 2 \, (\text{mod }4)$, then  $\chi_{\sigma_o}\bigl(\s(\alpha^\vee(a))\bigr)=\chi(a)$;
\item[$\bullet$] if $n \equiv 2 \, (\text{mod }4)$, then  $\chi_{\sigma_o}\bigl(\s(\alpha^\vee(a))\bigr)=\chi_\psi(a)$.
\end{enumerate}
Thus, in the case $n \equiv 2 \, (\text{mod }4)$, the character $\chi$ depends also on $\psi$. In all cases of $n$,  $\chi$ is determined only up to twisting by elements of $\widehat{ F^\times/F^{\times d} }$. If $\gcd(n,p)=1$ and $\sigma_o$ is unramified, then we may assume that $\chi$ is unramified. Furthermore, by Proposition \ref{secondunramprop}, if  $n  \equiv 2 \, (\text{mod }4)$ and
$$\sigma_o=\eta_{\varpi,(2)} \otimes \sigma_{oo}$$
 where $\sigma_{oo}$ is unramified, then we may assume that $\chi$ is ramified and that $\chi^2$ is unramified.

We will use $\chi$ to describe some invariants related to $\sigma_o$.

\subsection{An explicit local coefficients matrix}
Fix a linear character $\chi$ associated with $\sigma_o$. For $k \in F^\times$, we set
$${\rm nor}_{\chi,\psi}(k)=
\begin{cases} \chi(k)^{-1}  &  \text{ if } n \not \equiv 2 \, (\text{mod }4), \\
 \chi_\psi(k)^{-1}  &  \text{ if } n \equiv 2 \, (\text{mod }4),
 \end{cases} $$
and define $\xi_{\chi,\psi,k} \in i(\chi'_{\sigma_o})^\vee$ by
$$\xi_{\chi,\psi,k}(f)={\rm nor}_{\chi,\psi}(k) \cdot f\bigl(\s(\alpha^\vee(k))).$$
In the cases where $n \not \equiv 0 \, (\text{mod }4)$, it is proven in \cite[\S 4.2]{Szp6} that
$$\mfr{R}_{o,\chi,\psi}=\{\xi_{\chi,\psi,k}:  k \in \widehat{K^\dag}\}$$
is a well defined ordered basis for $i(\chi_{\sigma_o})^\vee$. The same argument applies to the  case $n \equiv 0 \ (\text{mod }4)$ as well.

Let
$$\mca{M}(-,-,\chi,s,\psi):\widehat{K^\dag} \times \widehat{K^\dag} \rightarrow \C(q^{-s})$$
be the local coefficient matrix associated with $\sigma_o$  representing $ T\bigl(w, \sigma_{o,s}\bigr)^*$ with respect to $\mfr{B}_{(\sigma_{o,s})^w}(\mfr{R}_{o,\chi,\psi})$ and $\mfr{B}_{\sigma_{o,s}}(\mfr{R}_{o,\chi,\psi})$, see Proposition \ref{ForRuse}.

\subsubsection{The $n \not \equiv 0 \ (\text{mod }4)$ cases}
For $\widehat{k} \in \widehat{K^\dag}$, we define

\begin{equation} \label{defpargam}
\gamma_{K}(s,\chi,\psi,\widehat{k})=[F^\times: F^{\times n}]^{-1/2} \cdot \sum_{a \in K^\dag} \gamma(s,\chi\eta_{a,(n)},\psi) \cdot \eta_{a,(n)}(\widehat{k}^{-1})
\end{equation}
and
$$\tilde{\gamma}_{K}(s,\chi,\psi,\widehat{k})=[F^\times: F^{\times d}]^{-1/2} \cdot \sum_{a \in K^\dag} \tilde{\gamma}(s,\chi\eta_{a,(d)},\psi) \cdot \eta_{a,(d)}(\widehat{k}^{-1}).$$

\begin{prop}[{\cite[Theorem 4.12]{Szp6}}] \label{matnot4} For $a, b \in \widehat{K^\dag}$, one has
$$\mca{M}(a,b,\chi,s,\psi)=
\begin{cases}
\gamma_{K}(1-s,\chi^{-1}\eta_{ab,(n)},\psi,ab^{-1})  & \text{if  $n$ is odd};\\
\tilde{\gamma}_{K}(1-s,\chi^{-1}\eta^{(d+1)/{2}}_{ab,(d)},\psi,ab^{-1})  &   \text{if } n \equiv 2 \ (\text{mod }4).
\end{cases}$$
\end{prop}
In \cite[\S 4.4]{Szp6},  explicit formulas for the entries of $\mca{M}(a,b,\chi,s,\psi)$ are given in the case where $\gcd(n,p)=1$ and $K=O_F^\times \cdot F^{\times d}$.

\subsubsection{The $n  \equiv 0 \, (\text{mod }4)$ case. } \label{0mod4mat}
Assume that $n  \equiv 0 \ (\text{mod }4)$ and $\gcd(n,p)=1$. Since the appearance of the $\tilde{\gamma}$-factors rather than ${\gamma}$-factors in Proposition \ref{matnot4} above in the $n\equiv 2 \, (\text{mod }4)$ case is explained by the non-canonical parametrization of  the set of genuine characters of $Z(\wt{T}_o)$,  it is reasonable to expect that in the case  $n\equiv 0 \ (\text{mod }4)$, one could describe the local coefficients  matrices and the related invariants using ${\gamma}$ factors. The scattering matrix is computed in \cite{GoSz} under the assumption that $\gcd(n,p)=1$. However, that matrix actually involves the metaplectic $\tilde{\gamma}$-factor. In fact, one can use both $\gamma$ and $\tilde{\gamma}$ to describe the local coefficients  matrices and related invariants in the case $n  \equiv 0 \ (\text{mod }4)$. We first use $\gamma$-factors below, and then in \S \ref{4remark} we will explain how to relate the same objects to the $\tilde{\gamma}$-factors.

We now give an analogue for Proposition \ref{matnot4} and the results in \cite[\S 4.4]{Szp6}, namely, explicit formulas for the entries of $\mca{M} (\cdot,\cdot, \chi, s, \psi)$.

We note here that while  $K_n$ is an $n$-Lagrangian subgroup of $F^\times$, it is not true in general that a Lagrangian decomposition of $F^\times/F^{\times n}$ exists. Nevertheless, as explained in \cite[Remark 2.13]{Szp6}, $\gamma_{K_n}(s,\chi,\psi,\cdot)$ and $\tilde{{\gamma}}_{K_n}(s, \chi, \psi, \cdot)$ are still defined  as functions on $F^\times/K_n \simeq \widehat{K_n^\dag}$. More precisely,
$$\gamma_{K_n}(s,\chi,\psi,\varpi^j)=[F^\times: F^{\times n}]^{-1/2} \cdot \sum_{k \in K_n^\dag} \gamma(s,\chi\eta_{k,(n)},\psi) \cdot \eta_{k,(n)}(\varpi^{-j}),$$
and we have
$$\gamma_{{K_n}}(s,\chi,\psi,\varpi^j)={\gamma}_{{K_n}}(s,\chi,\psi,\varpi^{n+j}).$$
When computing local coefficients matrices for the case $n  \equiv 0 \ (\text{mod }4)$,  we write
$$\mca{M}(i,j,\chi,s,\psi):= \mca{M}(\varpi^i,\varpi^j,\chi,s,\psi).$$

\begin{prop} \label{P:M-3C}
 One has
\begin{equation} \label{tau0mod4}
\mca{M}(i,j,\chi,s,\psi)= {\gamma}_{{K_n}}(1-s,\chi^{-1}\eta^{(i+j)}_{\varpi,(n)},\psi, \varpi^{i-j})
+{\gamma}_{{K_n}}(1-s,\chi^{-1}\eta^{d+(i+j)}_{\varpi,(n)},\psi, \varpi^{d+i-j}).
\end{equation}
\end{prop}

\begin{proof}
Arguing as in \cite[Proposition 4.9]{Szp6}, one shows that $\mca{M}(\cdot,\cdot,\chi,s,\psi)$ is given by the relation
$$T\bigl(w, \chi'_{\sigma_{o,s}}\bigr)^* \bigl(\lambda_{\s(\alpha^\vee(\varpi^{-i})),\chi^{-1},\psi,-s}\bigr)=\sum_{j=0}^{d-1} \mca{M}(i,j,\chi ,s,\psi)  \cdot \lambda_{\s(\alpha^\vee(\varpi^{-i})),\chi,\psi,s}.$$
Here, for $j \in \Z $ the Whittaker functional
$$\lambda_{\s(\alpha^\vee(\varpi^{-j})),\chi,\psi,s} \in \Wh_\psi \bigl(I\bigl(\chi'_{\sigma_{o,s}}\bigr) \bigr)$$
is the analytic continuation of
$$f_s \mapsto \delta^{\frac {-s-1}{2}}(\varpi^{j})\chi^{-1}(\varpi^{j})\int_{\mfr{p}^{-r}} f_s\bigl(  \s(\alpha^\vee(\varpi^{-i})) \wt{w}u(x)\bigr)\psi^{-1}(x) \ d_{\psi}x.$$
We have
$$\lambda_{\s(\alpha^\vee(\varpi^{j})),\chi,\psi,s}=\lambda_{\s(\alpha^\vee(\varpi^{d+j})),\chi,\psi,s}.$$
 Furthermore, similar to the proof of \cite[Theorem 4.12]{Szp6}, one shows using a suitable test function that $\mca{M}(i,j,\chi,s,\psi)$ is the meromorphic continuation of
\begin{equation} \label{0mod4rep}
q^{s(i-j)}\chi^{j-i}(\varpi) \cdot
\lim_{r \rightarrow \infty} \int_{K_d \varpi^{j-i}\cap \mfr{p}^{-r}} \chi'_s(z\omega^{i-j}) \eta^{-(i+j)}_{\varpi,(n)}(z)_n \psi(z) \, d^\times_\psi z.
\end{equation}
The limit above exists for ${\rm Re}(s) \gg 0$. Note that
$$K_d \varpi^{j-i}=K_n \varpi^{j-i} \sqcup K_n \varpi^{j-i+d}.$$
Equation \eqref{par0mod4} along with the fact that
$$\eta^d_{\varpi,(n)}(\varpi)=1$$
  imply that for $z \in K_d \varpi^{j-i}$ we have
$$\chi'(z\varpi^{i-j})=\chi(z)\chi^{i-j}(\varpi) \cdot
\begin{cases} 1  & \text{ if } z \in K_n=K_n \varpi^{j-i} ; \\
 \eta^d_{\varpi,(n)}(z)  &  \text{ if }  z   \in K_n \varpi^{j-i+d}.
 \end{cases} $$
Thus, we can write \eqref{0mod4rep} as
$$\lim_{r \rightarrow \infty} \int_{K_n \varpi^{j-i}\cap \mfr{p}^{-r}} \chi_s(z)   \eta^{-(i+j)}_{\varpi,(n)}(z) \psi(z) \, d^\times_\psi z +
\lim_{r \rightarrow \infty} \int_{K_n \varpi^{d+j-i}\cap \mfr{p}^{-r}} \chi_s(z)   \eta^{-d-(i+j)}_{\varpi,(n)}(z) \psi(z) \, d^\times_\psi z.$$
Comparing the integrals above with the integral representation of $\gamma_{K_n}$ given in \cite[Theorem 2.12]{Szp6}, the proposition follows.
\end{proof}

\begin{lm} \label{odd sha comp}
Suppose that  $\gcd(p,n)=1$. Let $\psi$ be a normalized character of $F$. If $\chi$ is ramified, then
\begin{equation} \label{ram gama}
\gamma_{K_n}(1-s,\chi^{-1},\psi, \varpi^{t})=
\begin{cases}
\varepsilon(1-s,\chi^{-1},\psi)  &  \text{ if }  t\equiv \mfr{f}(\chi) \ (\text{mod } n) ; \\
0  & \text{ otherwise}.
\end{cases}
\end{equation}
If $\chi$ is unramified, then
\begin{equation} \label{ram gama un}
\gamma_{K_n}(1-s,\chi^{-1},\psi,\varpi^{-t} )=
\bigl(q^{-s}\chi(\varpi)\bigr)^t \cdot
\begin{cases}
(1-q^{-1}) L(ns,\chi^n)  & \text{ if } 0 \leq t \leq n-2 ; \\
\gamma(1-ns,\chi^{-n},\psi)& \text{ if } t=n-1.
\end{cases}
\end{equation}
\end{lm}
\begin{proof} For odd $n$, this was proven as \cite[Proposition 4.16]{Szp6}. The proof works for the even cases as well, once one replaces all the indices denoted by $d$ to $n$.
\end{proof}

We now give explicit formulas for the local coefficients matrix for all possible elements of $\Irr(\wt{T}_o)$. For this purpose, note that since $4|n$ we have $-1\in F^{\times 2}$ and thus $\omega_\psi(\varpi) = \pm 1$ by Lemma \ref{for0mod4}. We write
$$d_\omega:= d\cdot \omega_\psi(\varpi) \in \set{d, -d}$$
and define
\begin{equation} \label{betadef}
 \pmb{\beta}(\sigma_o,s,\psi)=  \frac{L((\chi\eta_{ u ,(n)})^{d_\omega},d_\omega s) L(\chi^{-d_\omega},-d_\omega s)}{L((\chi\eta_{ u ,(n)})^{-d_\omega},\half-d_\omega s) L(\chi^{d_\omega},d_\omega s+\half )} .
 \end{equation}

\begin{prop} \label{explicit4} Assume that $\psi$ is normalized. If $\chi$ is unramified, then

\begin{equation} \label{unramgam}
\begin{aligned}
& \mca{M}(i,j,\chi,s,\psi) \\
= &
\begin{cases} (1-q^{-1}) L(ns,\chi^n)  & \text{ if }  (i,j)=(0,0) ; \\
 \bigl(\chi(\varpi)q^{-s}\bigr)^{n-2i}(1-q^{-1}) L(ns,\chi^n) & \text{ if } j=d-i, \, 1 \leq i \leq d-1; \\
 \varepsilon(1-s,\chi^{-1}\eta^{-1}_{\varpi,(n)},\psi) & \text{ if } (i,j)=(0,d-1) ; \\
 \varepsilon(1-s,\chi^{-1}\eta^{2i-1}_{\varpi,(n)},\psi)  & \text{ if } j=i-1, \, 1 \leq i \leq d-1;  \\
 0 &  \text{ otherwise.}
 \end{cases}
\end{aligned}
\end{equation}
Also,
\begin{equation} \label{otherunramgam}
\begin{aligned}
& \mca{M}(i,j,\chi\eta^{-1}_{\varpi,(n)} ,s,\psi) \\
= & \begin{cases}
\bigl(q^{-s}\chi(\varpi)\bigr)^{n-1}\gamma(1-ns,\chi^{-n},\psi)  & \text{ if } (i,j)=(0,d-1) ; \\
 -\gamma_\psi(\varpi)\bigl(\chi^{-1}q^s\bigr) \cdot  \pmb{\beta}(\sigma_o,s,\psi)  & \text{ if } (i,j)=(\frac d 2, \frac d 2 -1) ; \\
 \bigl( q^{-s}\chi(\varpi)\bigr)^{n-2i-1} (1-q^{-1}) L(ns,\chi^n) & \text{ if } j=d-1-i, \, 1 \leq i \leq d-1, \, i \neq \frac d 2; \\ \varepsilon(1-s,\chi^{-1}\eta^{2i}_{\varpi,(n)},\psi)  & \text{ if } j=i-1, \, 1 \leq i \leq d-1 \, i \neq \frac d 2 ;  \\
 0 &  \text{ otherwise}.
\end{cases}
\end{aligned}
\end{equation}
If $\chi^n$ is ramified, then
$$\mca{M}(i,j,\chi,s,\psi)=\chi^{i-j}(\varpi) \cdot
\begin{cases}
\varepsilon(1-s,\chi^{-1}\eta^{(i+j)}_{\varpi,(n)},\psi)  & \text{ if } i-j\equiv \mfr{f}(\chi) \, (\text{mod }n) ;\\
\varepsilon(1-s,\chi^{-1}\eta^{d+(i+j)}_{\varpi,(n)},\psi)  & \text{ if } i-j+d\equiv \mfr{f}(\chi) \, (\text{mod }n); \\ 
0 &  \text{ otherwise.}
\end{cases}$$
\end{prop}

\begin{proof}
We first prove \eqref{unramgam}. Since $\chi$ is unramified, it follows from \cite[Lemma 4.15]{Szp6} that
$$\mfr{f}(\chi^{-1}\eta^{t}_{\varpi,(n)})=1$$
 unless $t \equiv 0 \ (\text{mod }n)$. Thus, by Lemma \ref{odd sha comp} we have
$${\gamma}_{_{K_n}}(1-s,\chi^{-1}\eta^{(i+j)}_{\varpi,(n)},\psi, \varpi^{i-j})=
\begin{cases}
(1-q^{-1}) L(ns,\chi^n)  & \text{ if }  (i,j)=(0,0) ; \\
\varepsilon(1-s,\chi^{-1}\eta^{2i-1}_{\varpi,(n)},\psi)  & \text{ if }  j=i-1, \, 1 \leq i \leq d-1;  \\
0 &  \text{ otherwise}.
\end{cases}$$
Also,
\begin{equation}
\begin{aligned}
&  {\gamma}_{_{K_n}}(1-s,\chi^{-1}\eta^{d+(i+j)}_{\varpi,(n)},\psi, \varpi^{d+i-j}) \\
= & \begin{cases}
\bigl(q^{-s}\chi(\varpi)\bigr)^{n-2i}  (1-q^{-1}) L(ns,\chi^n) &  \text{ if } j=d-i, \, 1 \leq i \leq d-1; \\
\varepsilon(1-s,\chi^{-1}\eta^{-1}_{\varpi,(n)},\psi) & \text{ if } (i,j)=(0,d-1) ;   \\
0 &  \text{ otherwise.}
\end{cases}
\end{aligned}
\end{equation}
Equation \eqref{unramgam} now follows.

We now prove \eqref{otherunramgam}. By \eqref{tau0mod4} we have
\begin{equation} \label{tauotherun}
\mca{M}(i,j,\chi\eta^{-1}_{\varpi,(n)} ,s,\psi)={\gamma}_{_{K_n}}(1-s,\chi^{-1}\eta^{(1+i+j)}_{\varpi,(n)},\psi, \varpi^{i-j})
+{\gamma}_{_{K_n}}(1-s,\chi^{-1}\eta^{d+i+j+1}_{\varpi,(n)},\psi, \varpi^{d+i-j}).
\end{equation}
Proposition \ref{ram gama} gives
$${\gamma}_{_{K_n}}(1-s,\chi^{-1}\eta^{(1+i+j)}_{\varpi,(n)},\psi, \varpi^{i-j})
=
\begin{cases}
\varepsilon(1-s,\chi^{-1}\eta^{2i}_{\varpi,(n)},\psi)  & \text{ if } j=i-1, \, 1 \leq i \leq d-1,  \\
 0 &  \text{ otherwise};
 \end{cases}$$
and
$$\begin{aligned}
& {\gamma}_{_{K_n}}(1-s,\chi^{-1}\eta^{d+i+j+1}_{\varpi,(n)},\psi, \varpi^{d+i-j})\bigr) \\
= &  \bigl(\chi(\varpi)q^{-s}\bigr)^{n-2i-1} \cdot
\begin{cases}
\gamma(1-ns,\chi^{-n},\psi)  & \text{ if } (i,j)=(0,d-1), \\
(1-q^{-1}) L(ns,\chi^n) & \text{ if } j=d-1-i, \, 1 \leq i \leq d-1,\\
0 &  \text{ otherwise.}
\end{cases}
\end{aligned} $$
Note that the two summands in the right hand side of \eqref{tauotherun} do not vanish simultaneously if and only if $(i,j)=(\frac d 2, \frac d 2 -1)$. For $i=\frac d 2$, we have
$$\varepsilon(1-s,\chi^{-1}\eta^{2i}_{\varpi,(n)},\psi)=\varepsilon(1-s,\chi^{-1}\eta_{\varpi,(2)},\psi)=q^{s-\half}\chi^{-1}(\varpi) \omega_\psi^{-1}(\varpi).$$
The second equality above follows from \eqref{epsilon old twist}, \eqref{epsilon twist} and \eqref{epsilon weil}. Hence, so far we have shown that
$$\begin{aligned}
& \mca{M}(i,j,\chi\eta^{-1}_{\varpi,(n)} ,s,\psi) \\
& =
\begin{cases}
\bigl(q^{-s}\chi(\varpi)\bigr)^{n-1}\gamma(1-ns,\chi^{-n},\psi)  & \text{ if } (i,j)=(0,d-1) ; \\
\bigl(\chi(\varpi)q^{-s}\bigr)^{d-1}(1-q^{-1}) L(ns,\chi^n)    & \text{ if } (i,j)=(\frac d 2, \frac d 2 -1) ; \\
\qquad  \quad +q^{s-\half}\chi^{-1}(\varpi)\omega^{-1}_\psi(\pi)  &   \\
\bigl( q^{-s}\chi(\varpi)\bigr)^{n-2i-1} (1-q^{-1}) L(ns,\chi^n) & \text{ if } j=d-1-i, \, 1 \leq i \leq d-1, \, i \neq \frac d 2; \\
\varepsilon(1-s,\chi^{-1}\eta^{2i}_{\varpi,(n)},\psi)  & \text{ if } j=i-1, \, 1 \leq i \leq d-1, \ i \neq \frac d 2 ;  \\
0 &  \text{ otherwise}.
\end{cases}
\end{aligned} $$
It thus remains to show that
$$\bigl(\chi(\varpi)q^{-s}\bigr)^{d-1}(1-q^{-1}) L(ns,\chi^n)  +q^{s-\half}\chi^{-1}(\varpi)\omega^{-1}_\psi(\varpi)=-\omega_\psi(\varpi)\bigl(\chi^{-1}q^s\bigr)  \pmb{\beta}(\sigma_o,s,\psi),$$
which however follows from a straightforward computation.

The ramified case follows immediately from Proposition \ref{P:M-3C} and Equation \eqref{ram gama}.
\end{proof}

\subsection{Invariants from $\mca{M}(w, \sigma_o, s, \psi)$}
Let $\sigma_o \in \Irr(\wt{T}_o)$. Let $\chi$ be a linear character associated with $\sigma_o$. If  $n  \equiv 0 \ (\text{mod }4)$, then we also assume $\gcd(n,p)=1$ unless stated otherwise.

\subsubsection{Trace} \label{trace}
Denote by $\Tr(\mca{M}(w, \sigma_o, s,\psi))$ the trace of a local coefficients matrix $\mca{M}(w, \sigma_o, s,\psi)$ associated with $\sigma_o$ and $\psi$.

\begin{thm} \label{trcodd24}
We have
\begin{equation} \label{trfor4}
\begin{aligned}
&  \Tr(\mca{M}(w, \sigma_o, s,\psi)) \\
= &  (\dim \sigma_o)^{-1} \cdot
 \begin{cases}
 \sum_{\eta \in \widehat{F^\times/ F^{\times d}}} \gamma(1-s,\chi^{-1}\eta,\psi)  & \text{ if }  n \not \equiv 2 \, (\text{mod }4) ; \\
 \sum_{\eta \in \widehat{F^\times/ F^{\times d}}} \tilde{\gamma}(1-s,\chi^{-1}\eta,\psi) & \text{ if }   n \equiv 2 \, (\text{mod }4). \end{cases}
 \end{aligned}
 \end{equation}
For ${\rm Re}(s) \gg 0$, we have
$$\Tr(\mca{M}(w, \sigma_o, s,\psi))= (\dim \sigma_o)  \cdot \lim_{r \rightarrow \infty} \int_{\mfr{p}^{-r} \cap F^{\times d}}    \chi_{\sigma_{o,s}}\bigr( \s(\alpha^\vee(x)) \bigl) \psi(x) \, d_\psi^\times x .$$
\end{thm}
\begin{proof}
This proposition is proven for the $n  \not \equiv 0 \ (\text{mod }4)$ case in \cite[Corollary 4.14]{Szp6}. We now prove for $n  \equiv 0 \, (\text{mod }4)$ starting with \eqref{trfor4}.

Let $u$ be generator of $K_n/F^{\times n}$. Note that $\eta_{u,(n)}$ is unramified and
$$\eta_{u,(n)}(\varpi)=\xi,$$
 where $\xi$ is a primitive $n$-th root of 1. By Proposition \ref{tau0mod4},
$$\begin{aligned}
&  \Tr(\mca{M}(w, \sigma_o, s,\psi))   \\
 = & \sum_{i=0}^{d-1}\mca{M}(i,i,\chi,s,\psi) \\
 = & \sum_{i=0}^{d-1}\bigl({\gamma}_{_{K_n}}(1-s,\chi^{-1}\eta^{2i}_{\varpi,(n)},\psi, 1)
+{\gamma}_{_{K_n}}(1-s,\chi^{-1}\eta^{d+2i}_{\varpi,(n)},\psi, \varpi^{d})\bigr).
\end{aligned} $$
By \eqref{defpargam}  we have
\begin{equation} \label{tr1}
\begin{aligned}
& \Tr(\mca{M}(w, \sigma_o, s,\psi)) \\
 = &\  \frac 1 n \sum_{i=0}^{d-1}\sum_{j=0}^{n-1} \bigl({\gamma}(1-s,\chi^{-1}\eta^{2i}_{\varpi,(n)}\eta^{j}_{u,(n)},\psi)+{\gamma}(1-s,\chi^{-1}\eta^{2i+d}_{\varpi,(n)}\eta^{j}_{u,(n)},\psi)\xi^{dj} \bigr).
 \end{aligned}
 \end{equation}
Recall that $\Tr(\mca{M}(w, \sigma_o, s,\psi))$ depends only on $\chi_\sigma$. In particular, since $\eta^{d}_{\varpi,(n)}$ is trivial on $F^{\times d}$, we may replace $\chi$ by $\chi\eta^{d}_{\varpi,(n)}$ and obtain
\begin{equation} \label{tr2}
\begin{aligned}
& \Tr(\mca{M}(w, \sigma_o, s,\psi)) \\
= & \ \frac 1 n\sum_{i=0}^{d-1}\sum_{j=0}^{n-1} \bigl({\gamma}(1-s,\chi^{-1}\eta^{2i+d}_{\varpi,(n)}\eta^{j}_{u,(n)},\psi)+{\gamma}(1-s,\chi^{-1}\eta^{2i}_{\varpi,(n)}\eta^{j}_{u,(n)},\psi)\xi^{dj} \bigr).
\end{aligned}
\end{equation}
Summing \eqref{tr1} and \eqref{tr2} together gives
$$\Tr(\mca{M}(w, \sigma_o, s,\psi)) =\frac {1} {2n} \sum_{i=0}^{d-1}\sum_{j=0}^{n-1} \alpha(i,j)(1+\xi^{dj})$$
where $$\alpha(i,j)={\gamma}(1-s,\chi^{-1}\eta^{2i+d}_{\varpi,(n)}\eta^{j}_{u,(n)},\psi)+{\gamma}(1-s,\chi^{-1}\eta^{2i}_{\varpi,(n)}\eta^{j}_{u,(n)},\psi).$$
Since $\xi^{dj}=(-1)^j$, we have
$$\begin{aligned}
& \Tr(\mca{M}(w, \sigma_o, s,\psi)) \\
=& \ \frac 1 n\sum_{i=0}^{d-1}\sum_{j=0}^{d-1} {\gamma}(1-s,\chi^{-1}\eta^{2i+d}_{\varpi,(n)}\eta^{2j}_{u,(n)},\psi)+{\gamma}(1-s,\chi^{-1}\eta^{2i}_{\varpi,(n)}\eta^{2j}_{u,(n)},\psi) \\
=& \ \frac 1 n\sum_{i=0}^{d-1}\sum_{j=0}^{d-1} {\gamma}(1-s,\chi^{-1}\eta^{d}_{\varpi,(n)}\eta^{i}_{\varpi,(d)}\eta^{j}_{u,(d)},\psi)+{\gamma}(1-s,\chi^{-1}\eta^{i}_{\varpi,(d)}\eta^{j}_{u,(d)},\psi),
\end{aligned} $$
where the last equality follows from \eqref{FV fact}. The proof of the equality in \eqref{trfor4} is completed once we note that in the right hand side of the above equality, each summand of the form ${\gamma}(1-s,\chi^{-1} \eta,\psi), \eta \in \widehat{F^\times/F^{\times d}}$ appears exactly twice.

Finally, the integral formula given in the proposition for the case $n  \equiv 0 \ (\text{mod }4)$ follows from \eqref{trfor4} by the same argument used in the proof of \cite[Corollary 4.14]{Szp6}.
\end{proof}

To reconcile the expression for $\Tr(\mca{M}(w, \sigma_o, s, \psi))$ given in Theorem \ref{trcodd24} with that in Proposition \ref{P:trace}, we have the following result.

\begin{cor} \label{exptrun}
Suppose that $n\ge 3$ and $\gcd(n,p)=1$. Assume that $\sigma_o$ is unramified and $\psi$ is normalized. Denote by $\chi$ the linear character associated with $\sigma_o$. Then
\begin{equation} \label{explicitT}
\Tr(\mca{M}(w, \sigma_o, s, \psi))
=(1-q^{-1}) \cdot
 \begin{cases}
 L(ns,\chi^n)  &   \text{ if } n \not \equiv 0 \  (\text{mod } 4); \\
 L(ds,\chi^d)  &   \text{ if } n \equiv 0\  (\text{mod } 4).
  \end{cases}
 \end{equation}
If $n \equiv 2 \ (\text{mod }4)$,  then
\begin{equation} \label{explicitT2}
\Tr(\mca{M}(w, \eta_{\varpi,(2)} \otimes \sigma_o, s, \psi))
=(1-q^{-1}) (q^{-s}\chi(\varpi))^{d-1} \cdot \varepsilon(\half-s,\eta_{\varpi,(2)} \chi^{-1},\psi)^{-1}\cdot L(ns,\chi^n).
\end{equation}
\end{cor}
\begin{proof}
 We first consider the case where $n \not \equiv 2 \ (\text{mod }4)$. In view of \eqref{trfor4} it is sufficient to show that if $\chi$ is unramified, then
 $$\frac{1}{d} \sum_{\eta \in \widehat{F^\times/F^{\times d}}} \gamma(1-s,\chi^{-1}\eta,\psi)=(1-q^{-1}) \cdot L(ds,\chi^d).$$
Since $\gcd(n,p)=1$, we have
$$\widehat{F^\times/F^{\times d}}=D \times U.$$
Here $D$ is a cyclic group generated by $\eta_{\varpi,(d)}$ and $U$ is a cyclic group of unramified characters generated by $\eta_{u,(d)}$, where $u \in O_F^\times$ such that
$$\xi:=\eta_{u,(d)}(\varpi)$$
is a primitive $d$-th root of unity. Both groups $D$ and $U$ are of order $d$. We have
$$\begin{aligned}
& \frac{1}{d} \sum_{\eta \in \widehat{F^\times/F^{\times d}}} \gamma(1-s,\chi^{-1}\eta,\psi) \\
= & \frac{1}{d} \sum_{i=0}^{d-1} \sum_{j=0}^{d-1} \gamma(1-s,\chi^{-1}\eta_{u,(d)}^i\eta_{\varpi,(d)}^j,\psi) \\
= & \frac{1}{d} \sum_{i=0}^{d-1} \gamma(1-s,\chi^{-1}\eta_{u,(d)}^i,\psi)+\frac{1}{d} \sum_{i=0}^{d-1} \sum_{j=1}^{d-1} \gamma(1-s,\chi^{-1}\eta_{u,(d)}^i\eta_{\varpi,(d)}^j,\psi).
\end{aligned} $$
We first show that
\begin{equation} \label{unramvan}
\sum_{i=0}^{d-1} \sum_{j=1}^{d-1} \gamma(1-s,\chi^{-1}\eta_{u,(d)}^i\eta_{\varpi,(d)}^j,\psi)=0. \end{equation}
First note all the summands in \eqref{unramvan} are $\varepsilon$-factors. Since $\gcd(d,p)=1$, it follows that  $1+\mfr{p} \subseteq F^{\times d}$. This implies that all the ramified elements in $\widehat{F^\times/F^{\times d}}$ have conductor 1. Thus, it follows from \eqref{epsilon twist} that
$$\sum_{i=0}^{d-1} \sum_{j=1}^{d-1} \gamma(1-s,\chi^{-1}\eta_{u,(d)}^i\eta_{\varpi,(d)}^j,\psi)= \varepsilon(1-s,\chi^{-1}\eta_{\varpi,(d)}^j,\psi) \cdot \sum_{i=0}^{d-1} \eta_{u,(d)}^i(\varpi) .$$
The equality \eqref{unramvan} then follows from
$$\sum_{i=0}^{d-1} \eta_{u,(d)}^i(\varpi)=\sum_{i=0}^{d-1} \xi^i=0.$$
The proof of \eqref{explicitT} for the case $n \not \equiv 2 \ (\text{mod }4)$ is completed if we have
$$\frac{1}{d} \sum_{\alpha \in U} \gamma(1-s,\chi^{-1}\alpha,\psi)=(1-q^{-1}) \cdot L(ds,\chi^d),$$
which however follows from the equality
$$\frac{1}{d} \sum_{\alpha \in U} \gamma(1-s,\chi^{-1}\alpha,\psi)=\frac 1 d  \sum_{l=0}^{d-1} \frac{1-q^{s-1}\chi^{-1}(\varpi)\xi^l}  {1-q^{-s}\chi(\varpi)\xi^{-l}}$$
coupled with Lemma \ref{L:aux}. We remark that the above computation is similar to that of $\gamma_{_J}(1-s,\chi^{-1},\psi, k^{0})$ in \cite[Proposition 4.16, Equation (4.6)]{Szp6}.

The proof for the case $n  \equiv 2 \ (\text{mod }4)$ is achieved along the same lines, by replacing the $\gamma$-factors by $\tilde{\gamma}$-factors and using \eqref{meta gama formula}.
\end{proof}

\begin{rmk}
 For $n  \not \equiv 0  \ (\text{mod }4)$, Corollary \ref{exptrun} could be proven also by examining the matrices in Propositions 4.17 and 4.19 of \cite{Szp6}. Similarly, for $n  \equiv 0  \ (\text{mod }4)$, a careful examination of the matrices in Proposition \ref{explicit4} above also gives an alternative proof for
 Corollary \ref{exptrun}.
 \end{rmk}

\begin{prop} \label{rumexptr}
Suppose that $n\ge 3$ and $\gcd(n,p)=1$. Assume that $\sigma_o$ is ramified. If $n \equiv 2 \ (\text{mod }4)$, then we also assume that $\eta_{\varpi,(2)} \otimes \sigma_o$ is ramified. Let $\chi$ be the linear character associated with $\sigma_o$. Then the following hold.
 \begin{enumerate}
 \item[(i)] The conductor $\mfr{f}(\chi)$ of $\chi$ is determined by $\sigma_o$.
 \item[(ii)] If $\mfr{f}(\chi)  \nequiv \mfr{f}(\psi) \ (\text{mod }d)$,  then $\Tr(\mca{M}(w, \sigma_o,s, \psi))=0$.
 \item[(iii)] If $\mfr{f}(\chi)  \equiv \mfr{f}(\psi) \ (\text{mod }d)$, then
 $$\begin{aligned}
 & \Tr(\mca{M}(w, \sigma_o,s, \psi)) \\
 = &
 \begin{cases}
 \sum_{\beta \in \widehat{O^\times/O^{\times ^d}}}   \varepsilon(1-s,\chi^{-1}\beta,\psi)  & \text{ if } n \nequiv 2 \ (\text{mod }4); \\
\sum_{\beta \in \widehat{O^\times/O^{\times ^d}}}    \omega(\psi) \cdot \varepsilon(\half-s,(\chi\beta)^{-1},\psi)^{-1} \cdot \varepsilon(2s,\chi^2\beta,\psi_2)^{-1} & \text{ if } n \equiv 2 \ (\text{mod }4).
 \end{cases}
 \end{aligned} $$
 \item[(iv)] There exists $\psi$ such that $\Tr(\mca{M}(w, \sigma_o,s, \psi))  \neq 0.$
 \end{enumerate}
 \end{prop}
\begin{proof}
Since  $\sigma_o$ is ramified we deduce that $\chi^d$ is ramified. The assertion (i) then follows from the fifth item in \cite[Lemma 4.15]{Szp6}.

We now prove the other assertions when $n \not \equiv 2 \ (\text{mod }4)$. By \eqref{trfor4}, proving (ii) amounts to showing that if $\mfr{f}(\chi) \nequiv \mfr{f}(\psi) \ (\text{mod }d)$, then
$$\sum_{\eta \in \widehat{F^\times/F^{\times d}}} \varepsilon(1-s,\chi^{-1}\eta,\psi)=0.$$
Keeping the notation in the proof of Corollary \ref{exptrun} and using the fact that $\chi^d$ is ramified, it follows from \eqref{trfor4} that
$$ \Tr(\mca{M}(w, \sigma_o,s, \psi)) = \frac{1}{d} \sum_{i=0}^{d-1} \sum_{j=0}^{d-1} \varepsilon(1-s,\chi^{-1}\eta_{u,(d)}^i\eta_{\varpi,(d)}^j,\psi).$$
By the fourth assertion in \cite[Lemma 4.15]{Szp6} we have
$$\mfr{f}(\chi)=\mfr{f}(\chi \beta) \text{ for every } \beta \in D.$$
Thus, by  \eqref{epsilon twist} we have
 \begin{equation} \label{for23}
\Tr(\mca{M}(w, \sigma_o,s, \psi)) = \frac{1}{d} \left(\sum_{i=0}^{d-1}  \xi^{i \cdot (\mfr{f}(\chi)- \mfr{f}(\psi))}   \right) \cdot \left( \sum_{\beta \in D} \varepsilon(1-s,\chi^{-1}\beta,\psi) \right).
\end{equation}
Now, (ii) and (iii) follow from the fact that
$$\sum_{i=0}^{d-1}  \xi^{i \cdot (\mfr{f}(\chi)-\mfr{f}(\psi))}
=\begin{cases}
 d  &  \text{ if }  \mfr{f}(\chi)  \equiv  \mfr{f}(\psi) \ (\text{mod }d) \\
  0  &  \text{ otherwise}.
  \end{cases}$$
To prove (iv),  we pick $\psi$ such that $\mfr{f}(\chi)  \equiv  \mfr{f}(\psi) \ (\text{mod }d) $. We also fix $a \in O_F^\times$.
It follows from  \eqref{for23} and \eqref{changepsi} that
$$\Tr(\mca{M}(w, \sigma_o, s, \psi_a))=\sum_{j=1}^{d-1} \varepsilon(1-s,\chi^{-1}\eta_{\varpi,(d)}^j,\psi_a)=\sum_{j=1}^{d-1} \chi^{-1}\eta_{\varpi,(d)}^j(a) \cdot \varepsilon(1-s,\chi^{-1}\eta_{\varpi,(d)}^j,\psi).$$
Thus,
$$\begin{aligned}
\sum_{a \in D} \chi^{-1}(a) \cdot \Tr(\mca{M}(w, \sigma_o, s, \psi_a))
= & \sum_{i=0}^{d-1} \sum_{j=1}^{d-1} \eta_{\varpi,(d)}^j(u^i) \varepsilon(1-s,\chi^{-1}\eta_{\varpi,(d)}^j,\psi) \\
= &  \sum_{j=1}^{d-1}\varepsilon(1-s,\chi^{-1}\eta_{\varpi,(d)}^j,\psi) \cdot \sum_{i=0}^{d-1} \xi^{-ij}.\end{aligned} $$
We have shown
$$\frac{1}{d} \sum_{a \in D} \chi^{-1}(a) \cdot \Tr(\mca{M}(w, \sigma_o, s, \psi_a))  =\varepsilon(1-s,\chi^{-1},\psi).$$
Since $\varepsilon(1-s,\chi^{-1},\psi) \neq 0$ for every $s\in \C$, we must have $\Tr(\mca{M}(w, \sigma_o, s, \psi_a))  \ne 0$ for some $a$.

The proof for the  case $n \equiv 2 \ (\text{mod }4)$ can be completed by  similar arguments replacing $\gamma$ by $\tilde{\gamma}$ and using \eqref{meta gama formula}.
\end{proof}

\subsubsection{Plancherel measure} \label{plansec}
In the case where ${\rm gcd}(n,p)=1$, it is proven in \cite[Theorem 5.1]{GoSz} that for every $n$ we have
\begin{equation} \label{myold result}
\mu(\sigma_o,s)^{-1}= q^{\mfr{f}(\psi)-\mfr{f}(\chi^n)} \frac{L \bigl(ns,\chi^n \bigr)L \bigl(-ns,\chi^{-n}\bigr)}{L \bigl(1-ns,\chi^{-n} \bigr)L \bigl(1+ns,\chi^{n}\bigr)}.
\end{equation}
In fact, if $n  \not \equiv 0 \, (\text{mod }4)$, then in general (i.e. without the constraint $\gcd(n,p)=1$) we have
$$\mu(\sigma_o,s)^{-1}= c(\sigma) \cdot \frac{L \bigl(ns,\chi^n \bigr)L \bigl(-ns,\chi^{-n}\bigr)}{L \bigl(1-ns,\chi^{-n} \bigr)L \bigl(1+ns,\chi^{n}\bigr)}$$
where $c(\sigma) \in \R_{>0}$ is a positive constant given by
$$c(\sigma)=q^{\mfr{f}(\psi_{n/d})} \cdot
\begin{cases}
[F^\times: F^{\times d}]^{-1} \sum_{\eta \in \widehat{F^\times/ F^{\times d}}} q^{-\mfr{f}(\chi^{n/d}\eta^{n/d})} &  \text{ if  $\chi^n$ is ramified},  \\
\val{d}  &  \text{ if  $\chi^n$ is unramified}.
\end{cases}$$
See \cite[Theorem 5.7]{Szp6}. It is discussed in \cite[\S 8.5]{Ga1} that for all $n$ and $p$, $\mu(\sigma_o,s)$ has the same analytic properties as the right hand side of \eqref{myold result}. Since a priori $\mu(\sigma_o,s)$ is a rational function in $q^{-s}$ depending only on $\chi_{\sigma_o}$, one deduces that  in the case $n  \equiv 0 \, (\text{mod }4)$ and  ${\rm gcd}(n,p)>1$, $\mu(\sigma_o,s)$ differs from the right hand side of \eqref{myold result} by a multiplicative factor of the form $aq^{-ks}$ where $a \in \C, \, k \in d\Z$. In fact, it follows from Corollary \ref{plancor} and Remark \ref{plancorr}  that  $k \in n\Z$.

\subsubsection{ $\det(\mca{M}(w, \sigma_o,s, \psi))$ in the unramified case} \label{unramsldet}

\begin{thm} \label{detunramnot4}
Assume $n \not \equiv 0 \ (\text{mod }4)$,  $\gcd(n,p)=1$, the character $\psi$ is normalized, and that $\sigma_o$ occurs in an unramified element of $\Irr(\wt{T})$. Then,
$$ \det(\mca{M}(w, \sigma_o,s, \psi))=\mu(\sigma_o,s)^{\frac{1-d}{2}} \cdot
\begin{cases}
\gamma(1-ds,\chi^{-d},\psi)  &  \text{ if $n$ is odd}; \\
\tilde{\gamma}(1-ds,\chi^{-d},\psi) &   \text{ if } n \equiv 2 \ (\text{mod }4) .
\end{cases}$$
\end{thm}
\begin{proof}
If $n$ is odd, then it follows from Proposition \ref{firstunramprop} that $\sigma_o$ is unramified. Thus, the result in this case is proven in \cite[Theorem 3.14]{GSS1}.

If $n \equiv 2\ (\text{mod }4)$, then by Lemma \ref{secondunramprop} and the last item in Proposition \ref{decoproph}, either $\sigma_o$ is unramified or
$$\sigma_o=\eta_{\varpi,(2)} \otimes \sigma_{oo},$$
 where $\sigma_{oo}\in \Irr(\wt{T}_o)$ is unramified. Again, the first case for unramified $\sigma_o$ is proven  in \cite[Theorem 3.14]{GSS1}. We prove the second case.

 As noted in \S \ref{linchasec} we may assume that $\chi$ is ramified but $\chi^2$ is unramified.   Let $\mca{M}_L(\cdot,\cdot,\chi, s, \psi)$ be the local coefficients matrix given in  \cite[Proposition 4.18]{Szp6} whose rows and columns are numbered from $0$ to $d-1$. Similar to the proof of \cite[Theorem 3.1.4]{GSS1}, we consider the matrix $M(\chi,s,\psi)$ produced from $\mca{M}_L(\cdot,\cdot,\chi, s, \psi)$ by swapping the $i$-th row with the $(d-1-i)$-th row (with a total of $(d-1)/2$ swaps). We have
$$\det(\mca{M}_L(w, \sigma_o,s, \psi))=(-1)^{(d-1)/2} \det\bigl(M(\chi,s,\psi)\bigr).$$
Here
$$M(\chi,s,\psi)=\begin{pmatrix}
A_1 &  &  &   &   &   & B_1 \\
  & \ddots &   &   &   & \upddots &  \\
  &   & A_{(d-1)/2} &   & B_{(d-1)/2} &   &  \\
  &   &   & C &   &   &  \\
  &   & B'_{(d-1)/2} &   & A'_{(d-1)/2} &   &  \\
  & \upddots &  &   &  & \ddots &  \\
B'_1 &   &   &   &   &   & A'_1
\end{pmatrix},$$
where
$$C=\bigl((q^{-2s}\chi(\varpi^{2})\bigr)^{\frac{d-1}{2}} \cdot \tilde{\gamma}(1-ds,\chi^{-d},\psi)$$
and for $1\leq i\leq \frac{d-1}{2}$ we have
$$ \begin{aligned}
A_i  & = \bigl((q^{-s}\chi(\varpi)\bigr)^{2j+d-3} (1-q^{-1}) \cdot \varepsilon(\half-s,\chi^{-1},\psi)^{-1}L(ns,\chi^n), \\
 A'_i &  = \bigl((q^{-s}\chi(\varpi)\bigr)^{-2j+d-1} (1-q^{-1}) \cdot \varepsilon(\half-s,\chi^{-1},\psi)^{-1}L(ns,\chi^n), \\
B_i  & = \varepsilon(\half-s,\chi^{-1}\eta_i^{-1},\psi)^{-1} \cdot \varepsilon(2s,\chi^2 \eta_i^2, \psi_2)^{-1}, \\
B'_i  & = \varepsilon(\half-s,\chi^{-1}\eta_i,\psi)^{-1}  \cdot \varepsilon(2s,\chi^2 \eta_i^{-2},\psi_2)^{-1}.
\end{aligned} $$
Here $\eta_i$ is a certain character of $F^\times$ which vanishes on $F^{\times d}$  for $1\leq i\leq \frac{d-1}{2}$. 

Thus,
$$ \begin{aligned}
&  \det(\mca{M}(w, \sigma_o,s, \psi))  \\
=\  & (-1)^{(d-1)/2}  \cdot C \prod_{i=1}^{(d-1)/2} (A_{i}A'_{i}-B_{i}B'_{i}) \\
 = \ & \bigl((q^{-2s}\chi(\varpi^{2})\bigr)^{\frac{d-1}{2}} \cdot \tilde{\gamma}(1-ds,\chi^{-d},\psi)   \cdot \prod_{i=1}^{(d-1)/2} (B_{i}B'_{i}-A_{i}A'_{i}) \\
 = \ & \tilde{\gamma}(1-ds,\chi^{-d},\psi) \cdot \prod_{i=1}^{(d-1)/2} \bigl((q^{-s}\chi(\varpi)\bigr)^2 \Bigl(B_{i}B'_{i}-A_{i}A'_{i} \Bigr).
 \end{aligned} $$
 To finish the  proof, it is sufficient to show that for every $1\leq i \leq \frac{d-1}{2}$, we have
\begin{equation} \label{det is plan}
\bigl((q^{-s}\chi(\varpi)\bigr)^2 \Bigl(B_{i}B'_{i}-A_{i}A'_{i} \Bigr)=\mu(\sigma_o,s)^{-1}.
\end{equation}
It follows from Lemma \ref{gen epsilon} that
\begin{equation} \label{A det is plan}
A_{i}A'_{i}= \bigl((q^{-s}\chi(\varpi)\bigr)^{n-2}(1-q^{-1})^2 L(ns,\chi^n)^2
\end{equation}
and that
$$B_{i}B'_{i}= q^{-1}\bigl(q^{s}\chi^{-1}(\varpi)\bigr)^{2}.$$
Therefore, in view of the explicit formula for $\mu(\sigma_o,s)^{-1}$ given in Section \ref{plansec}, the equality \eqref{det is plan} is equivalent to
\begin{equation} \label{techeq}
q^{-1}-\bigl((q^{-s}\chi(\varpi)\bigr)^{n}(1-q^{-1})^2 L(ns,\chi^n)^2=\frac{L \bigl(ns,\chi^n \bigr)L \bigl(-ns,\chi^{-n}\bigr)}{L \bigl(1-ns,\chi^{-n} \bigr)L \bigl(1+ns,\chi^{n}\bigr)},
\end{equation}
which can be proven easily by a direct computation.
\end{proof}

We now wish to give an analogue to Theorem \ref{detunramnot4} for the case $n  \equiv 0 \, (\text{mod }4)$. By Lemma \ref{secondunramprop}, the following result addresses all possible $\sigma_o \in \Irr(\wt{T}_o)$  occurring in an unramified representation of $\Irr(\wt{T})$.

\begin{thm}  \label{T:SL-c3}
Assume that $n \equiv 0 \, (\text{mod }4)$,  $\gcd(n,p)=1$,  $\mfr{f}(\psi)=0$ and $\sigma_o$ is unramified. Then,
\begin{equation} \label{unramdet4}
\det(\mca{M}(w, \sigma_o,s, \psi))=\bigl(\eta_{u,(n)}\chi^{-1}(\varpi)q^s \bigr)^d  \cdot \mu^{-\frac d 2}(\sigma_o,s)\end{equation}
and
\begin{equation} \label{semiunramdet4}
\begin{aligned}
& \det(\mca{M}(w, \eta^{-1}_{\varpi,(n)} \otimes \sigma_o,s, \psi)) \\
= \ & (q^{-s}\chi(\varpi))^{d} \cdot \gamma(1-ns,\chi^{-n},\psi) \mu(\sigma,s)^{1-\frac d 2 } \cdot \omega_\psi(\varpi) \pmb{\beta}(\sigma_o,s,\psi).
\end{aligned}
\end{equation}
\end{thm}
\begin{proof}
We first prove \eqref{unramdet4}. Let $\mca{M}(i,j,\chi,s,\psi)$ be the local coefficients matrix given in \eqref{unramgam} and let $M(\chi,s,\psi)$ be the matrix obtained from $\mca{M}(i,j,\chi,s,\psi)$ by swapping the $i$-th row with the $(d-i)$-th row  for $i=1,2,\cdots, d-1$ (with a total of $\frac{d-2}{2}$ swaps). We have
 \begin{equation} \label{nastycase}
 M(\chi,s,\psi)=
\begin{pmatrix}
E_0 &   &    &   &   & F_0 \\
  & \ddots &   &   & \upddots &  \\
  &   & E_{\frac {d}{2}-1 }   & F_{\frac {d}{2}-1 } &   &  \\
  &   & F_{\frac {d}{2}+1 }  &  E_{\frac {d}{2}-1 } &   &  \\
  &  \upddots  &  &  & \ddots &  \\
  F_{d-1}  &   &   &   &   & E_{d-1}
\end{pmatrix}
\end{equation}
where
$$\begin{aligned}
 {E_i}  & = \bigl(\chi(\varpi)q^{-s}\bigr)^{-2i} (1-q^{-1}) L(ns,\chi^n),\\
{F_i}   & =   {\varepsilon(1-s,\chi^{-1}\eta_{\varpi,(n)}^{-2i-1},\psi)} .
\end{aligned} $$
Thus,
$$\begin{aligned}
&  \det(\mca{M}(w, \sigma_o,s, \psi)) \\
 = \ & (-1)^{\frac{d-2}{2}}\det M'(\chi,s,\psi)  \\
 = \ &  -\prod_{i=0}^{\frac{d-2}{2}} (F_i F_{d-i}-E_i E_{d-i}) \\
 = \ &  -\prod_{i=0}^{\frac{d-2}{2}}\Bigl(\varepsilon(1-s,\chi^{-1}\eta_{\varpi,(n)}^{-2i-1},\psi)\varepsilon(1-s,\chi^{-1}\eta_{\varpi,(n)}^{2i+1},\psi)  \\
  & \qquad \qquad \qquad \qquad  - \bigl(\chi(\varpi)q^{-s}\bigr)^{n-2}(1-q^{-1})^2 L(ns,\chi^n)^2 \Bigr).
 \end{aligned}$$
Lemma \ref{gen epsilon} now gives
$$\det(\mca{M}(w, \sigma_o,s, \psi))=-\bigl(\chi^{-1}(\varpi)q^s \bigr)^d\prod_{i=0}^{\frac{d-2}{2}} \Bigl(q^{-1}-(\chi(\varpi)q^{-s} \bigr)^n(1-q^{-1})^2 L(ns,\chi^n)^2  \Bigr).$$
The proof of \eqref{unramdet4} is completed by using \eqref{techeq}.

We now prove \eqref{semiunramdet4}. Let $\mca{M}(i,j,\chi\eta^{-1}_{\varpi,(n)} ,s,\psi)$ be the local coefficients matrix given in \eqref{otherunramgam}. We expand $\det(\mca{M}(i,j,\chi\eta^{-1}_{\varpi,(n)} ,s,\psi))$ along the first row and then we expand the determinant of  the minor obtained along the $(d/2)$-th row. This gives
$$\det(\mca{M}(w, \eta^{-1}_{\varpi,(n)} \otimes \sigma_o,s, \psi)) =\omega_\psi(\varpi)( \chi(\varpi) q^{-s})^{n-2}\gamma(1-ns,\chi^{-n},\psi)  \pmb{\beta}(\sigma_o,s,\psi)  \cdot \det M'',$$
where $M''$ is the following $(d-2) \times (d-2)$ matrix:
$$M''=\begin{pmatrix}
G_1 &   &    &   &   & H_1 \\
  & \ddots &   &   & \upddots &  \\
  &   & G_{\frac {d}{2}-1 } & H_{\frac {d}{2}-1 } &   &  \\
  &   & \widehat{H}_{\frac {d}{2}-1 } & \widehat{G}_{\frac {d}{2}-1 } &   &  \\
  & \upddots &  &  & \ddots &  \\
\widehat{H}_{1} &   &   &   &   & \widehat{G}_{1}
\end{pmatrix} $$
with
$$G_i= \varepsilon(1-s,\chi^{-1}\eta^{2i}_{\varpi,(n)},\psi), \quad \widehat{G}_i= \varepsilon(1-s,\chi^{-1}\eta^{-2i}_{\varpi,(n)},\psi),$$
and $$H_i=\bigl(\chi(\varpi)q^{-s}\bigr)^{n-2i-1}  (1-q^{-1}) L(ns,\chi^n), \quad \widehat{ H_i}= \bigl(\chi(\varpi)q^{-s}\bigr)^{2i-1}  (1-q^{-1}) L(ns,\chi^n).$$
By similar computations to those already used in the proof of Theorem  \ref{detunramnot4},  one finds that
$$\det(M'')=\mu^{-(\frac d 2 -1)}(\sigma,s) \cdot (q^{s}\chi^{-1})^{d-2}.$$
This completes the proof.
\end{proof}

As expected, Theorem \ref{detunramnot4} and Theorem \ref{T:SL-c3} agree with Theorem \ref{T:M1} (or Example \ref{eg-SL}).

\begin{rmk} Equation \eqref{unramdet4} was already proven in \cite[\S 3.6]{GSS1}. The proof there uses a modification of  the scattering matrices given in \cite[Lemma 4.3]{GoSz}. Since no details are given regarding this modification and the modified scattering matrix is different from the local coefficient matrix computed in Proposition \ref{explicit4}, we have included an independent computation here.
\end{rmk}

\subsection{A remark about $\tilde{\gamma}$ and the $n \equiv 0 \, (\text{mod }4)$ case} \label{4remark}
In  \S \ref{0mod4mat}, \S \ref{trace} and \S \ref{unramsldet}  we have described the local coefficients matrix and the related invariants for the $n \equiv 0 \, (\text{mod }4)$ case using $\gamma$-factors and partial  $\gamma$-factors. In fact, we may describe these using $\tilde{\gamma}$-factors and partial  $\tilde{\gamma}$-factors as follows. We find this peculiar phenomenon to be another manifestation of the $\wt{G}_o$ trichotomy.

Comparing the integral representations of
$${\gamma}_{_{K_n}}(s,\chi,\psi,\varpi^j)$$ and
$$\tilde{\gamma}_{_{K_n}}(s,\chi,\psi,\varpi^j)$$
 given in \cite[Theorem 2.12, 2.14]{Szp6}, and  taking into account that both $K_n$ and $K_n \varpi^d$ are contained in $F^{\times 2}$, we may write \eqref{tau0mod4} as
$$\begin{aligned}
& \mca{M}(i,j,\chi,s,\psi) \\
= \ & \omega_\psi(\varpi)^{i-j} \cdot \bigl(\tilde{\gamma}_{_{K_n}}(1-s,\chi^{-1}\eta^{(i+j)}_{\varpi,(n)},\psi, \varpi^{i-j})
+\tilde{\gamma}_{_{K_n}}(1-s,\chi^{-1}\eta^{d+(i+j)}_{\varpi,(n)},\psi, \varpi^{d+i-j}) \bigr).
\end{aligned} $$
Moreover, since in the case $n \equiv 0 \ (\text{mod }4)$ one has
$$\omega_\psi(\varpi) \in \{\pm 1\},$$
we may conjugate this matrix $\mca{M}(\cdot, \cdot ,\chi,s,\psi)$ by a diagonal matrix whose $(i,i)$ entry is $(-1)^i$ and remove the term $\omega_\psi(\varpi)^{i-j}$ in  $\mca{M}(i,j,\chi,s,\psi)$. Same goal can be archived by changing  the additive character $\psi'$ fixed in \S \ref{par4} to the character $\psi$.

Since $d$ is even, we may use the third item in Lemma \ref{dualcenter}  and  rewrite \eqref{trfor4} for this case as
$$\Tr(\mca{M}(w, \sigma_o,s, \psi))=\frac{1}{d} \cdot \sum_{a \in F^\times/F^{\times d/2}} \sum_{b \in F^\times/F^{\times 2}} {\gamma}(1-s,\chi^{-1} \eta_{a,(d)}\eta_{b,(2)},\psi).$$
By using Lemma \ref{sumlemma} and then applying the third item in Lemma \ref{dualcenter} again, we deduce that
$$\Tr(\mca{M}(w, \sigma_o,s, \psi))= \sum_{\eta \in \widehat{F^\times/F^{\times d}}} \tilde{\gamma}(1-s,\chi^{-1}\eta,\psi).$$

Lastly, we can rewrite \eqref{semiunramdet4} as
$$\det( \mca{M}(w, \eta^{-1}_{\varpi,(n)} \otimes \sigma_o,s,\psi))=(q^{-s}\chi(\varpi))^{d} \cdot \tilde{\gamma}(1-ns,\chi^{-n},\psi) \mu(\sigma,s)^{1-\frac d 2 } \omega_\psi(\varpi) \cdot  \pmb{\beta}'(\sigma_o,s,\psi),$$
where
$$ \pmb{\beta}'(\sigma_o,s,\psi)=\tilde{\gamma}(1-ns,\chi^{-n},\psi){\gamma}(1-ns,\chi^{-n},\psi)^{-1}  \cdot \pmb{\beta}(\sigma_o,s,\psi)$$
is also a quotient of $L-$functions, similar to $\pmb{\beta}(\sigma_o,s,\psi)$.

\section{Local coefficients matrices for $\wt{\GL}_2$ and invariants via restrictions}  \label{S:LCM-G}
In this section, we compute the invariants associated to a local coefficients matrix of a genuine principal series of $\wt{\GL}_2$. Our method is to relate the local coefficients matrix to that of $\wt{\SL}_2$. In fact, our analysis gives an explanation of the disappearance of the trichotomy in $\det(\mca{M}(w, \sigma_o, s, \psi))$ for $\wt{\SL}_2$, when we consider the analogous $\det(\mca{M}(w, \sigma, s, \psi))$ for $\wt{\GL}_2$.

\subsection{Local coefficients matrix for $\wt{G}$ and $\wt{G}_o$}
\begin{thm} \label{Smain}
Let $\sigma \in \Irr(\wt{T})$. Then there exists a local coefficients matrix $\mca{M}(w, \sigma, \psi)$  associated with $\sigma$ which is a diagonal-block matrix satisfying the following:
\begin{enumerate}
\item[(i)] If $n$ is odd, then the matrix $\mca{M}(w, \sigma, \psi)$ consists of $ n_c \val{n_c}^{-1/2}$ diagonal blocks and each block is a local coefficient matrix associated with the unique $\sigma_o \in \Irr(\wt{T}_o)$ occurring in $\sigma$.
\item[(ii)] If $n$ is even, then the matrix $\mca{M}(w, \sigma, \psi)$ consists of $2n_c \val{2n_c}^{-1/2}$ diagonal blocks and each  $\sigma_o \in \Irr(\wt{T}_o)$ occurring in $\sigma$ contributes exactly $ d_c \val{d_c}^{-1/2}$ identical blocks, each of which is a local coefficients matrix associated with $\sigma_o$.
\end{enumerate}
\end{thm}
Theorem \ref{Smain} essentially follows from Proposition \ref{decoprop} and from the simple observation given in Lemma \ref{simpobser} below. We will postpone a full proof for Theorem \ref{Smain} to \S \ref{mainproof}. In \S \ref{glinv} below, we will first give some important consequence following from Theorem \ref{Smain}.

\begin{cor} \label{plancor}
Let $\sigma \in \Irr(\wt{T})$ and let $\sigma_o \in \Irr(\wt{T}_o)$ be occurring in $\sigma$. Then,
$$\mu(\sigma,s)=\mu(\sigma_o,s).$$
\end{cor}
\begin{proof}
Recall that
$$\big(T_o(w^{-1}, (\sigma_s)^w) \circ T_o(w, \sigma_s) \bigr)^*=\mu(\sigma_o,s)^{-1}\cdot {\rm id}.$$
By Theorem \ref{Smain}, a local coefficients matrix associated $\sigma$ can be chosen to be a block-diagonal matrix where each block is a local coefficient matrix associated with a principal series of $\wt{G}_o$ appearing in the restriction. Thus, a matrix representing $\big(T(w^{-1}, (\sigma_s)^w) \circ T(w, \sigma_s) \bigr)^*$ can be chosen to be a diagonal-block matrix with each block being a scalar matrix where the scalar is a Plancherel measure. On the other hand, one has
$$\big(T(w^{-1}, (\sigma_s)^w) \circ T(w, \sigma_s) \bigr)^*=\mu(\sigma,s)^{-1}\cdot {\rm id}.$$
The assertion now follows.
\end{proof}

\begin{rmk} \label{plancorr}
Let $\sigma_o \in \Irr(\wt{T}_o)$. Explicit formulas for $\mu(\sigma_o,s)$ are given in \S \ref{plansec}. These formulas show that $\mu(\sigma_o,s)$ depends only on
$$\chi_{\sigma_o}|_{Z(\wt{T}) \cap \wt{T}_o}.$$
This observation is non-trivial when $n$ is even, since in this case $Z(\wt{T}) \cap \wt{T}_o$ is strictly contained in $Z(\wt{T}_o)$. Theorem \ref{decopropcor} and Corollary \ref{plancor} together explain this phenomenon.
\end{rmk}

Remark \ref{plancorr} and the explicit formulas in \S \ref{plancorr} motivate the following definition.

\begin{dfn}
A character $\chi$ of $F^\times$ is said to be associated with $\sigma \in \Irr(\wt{T})$ if
$$\chi_\sigma( \s(\alpha^\vee(a) \cdot \zeta) \mapsto \zeta\cdot \chi(a);$$
equivalently, if $\chi$ is  associated with an $\sigma_o \in \Irr(\wt{T}_o)$ occurring in $\sigma$.
\end{dfn}

We emphasize that $\chi$ is independent of any choice of additive character $\psi$ of $F$ and that $\chi$ is unique up to twisting by elements of $\widehat{F^\times/F^{\times n}}$.

\begin{cor} \label{irrres}
Let $\sigma \in \Irr(\wt{T})$ be a unitary representation. Then, $I(\sigma)$  is always irreducible. Moreover, $I(\sigma_s)$ is reducible if and only if the trivial character of $F^\times$ is  associated with $\sigma$ and $q^{-1 \pm ns}=1$.
\end{cor}
\begin{proof}
We note that the above result follows from the general theory of Knapp-Stein R-groups, which is extended to covering groups by Caihua Luo \cite{Luo3}. On the other hand,  a weaker result applicable only for maximal parabolic induction is proven by Savin using a different method, see \cite[Appendix]{Tan}. We elaborate on the latter approach.

Since $\sigma$ is unitary, it follows from Propositions 6.2 and 6.3 in \cite[Appendix]{Tan} that $I(\sigma)$ is reducible if and only if $\sigma \simeq \sigma^w$ and  that $\mu(\sigma,s)^{-1}$  is analytic at $s=0$. By  the Stone-von Neumann Theorem,  $\sigma \cong \sigma^w$ if and only if $\chi_{\sigma}=\chi_{\sigma^w}.$ In particular, $\sigma \cong \sigma^w$ implies that for all $a \in F^{\times n}$, one has
 $$\chi_{\sigma}\big(\s(\alpha^\vee(a) )\big) \cdot \chi_{\sigma^w} \big(\s(\alpha^\vee(a) )^{-1}\big)=1.$$
  This is equivalent to the fact that the trivial character of $F^\times$ is associated with $\sigma$. In this case, Corollary \ref{plancor} together with the formulas in \S \ref{plansec} imply that $\mu(\sigma,s)^{-1}$ has a pole at $s=0$. Thus, $I(\sigma)$ is irreducible.

 If $s \in i\cdot \R$,  then $\sigma_s$ is again unitary. By what we have already proven, $I(\sigma_s)$ is irreducible.    Suppose now $s \notin i\cdot \R$. Then $\sigma_s$ is not unitary. It follows from Proposition 6.2  in \cite[Appendix]{Tan}  that $I(\sigma_s)$ is reducible if and only if $\mu(\sigma,s)^{-1}=0.$  Using Corollary \ref{plancor}  and the formulas in \S \ref{plansec} again, we conclude the proof.
\end{proof}

\subsection{Invariants from $\mca{M}(w, \sigma, s, \psi)$} \label{glinv}

\begin{thm} \label{trgl}
Fix $\sigma \in \Irr(\wt{T})$. If $n  \equiv 0 \, (\text{mod }4)$, then we assume that $\gcd(p,n)=1$. Let $\chi$  be a linear character of $F^\times$ associated with $\sigma$. One has
$$\Tr( \mca{M}(w, \sigma,s,\psi))=\frac{\val{\gcd(n,4c+1)}^{\half}}{\gcd(n,4c+1)} \cdot \sum_{\eta \in \widehat{F^\times/F^{\times n}} } \gamma(1-s,\chi^{-1}\eta,\psi). $$
For ${\rm Re}(s) \gg 0$,  we also have
$$\Tr( \mca{M}(w, \sigma,s,\psi)) =(\dim \sigma)  \cdot \lim_{r \rightarrow \infty} \int_{\mfr{p}^{-r} \cap F^{\times n} } \chi_{\sigma_s}\bigr(  \s(\alpha^\vee(x)) \bigl) \psi(x) \ d_\psi^\times x .$$
\end{thm}
\begin{proof}
We only prove the first assertion, as the second assertion follows from the first one by the same argument used in the proof of \cite[Corollary 4.14]{Szp6}.  Let $\sigma_o \in \Irr( \wt{T}_o)$ be occurring in $\sigma$.

Suppose first that $n$ is odd. In this case, it follows from Theorem \ref{Smain} that
$$\Tr( \mca{M}(w, \sigma,s,\psi))  =n_c \val{n_c}^{-1/2} \cdot \Tr( \mca{M}(w, \sigma_o,s,\psi)) .$$
The result now follows from \eqref{trfor4}.

Assume now $n\equiv 2 \ (\text{mod }4)$. By Theorem \ref{Smain} and Corollary \ref{decopropcor}, we have
$$\Tr( \mca{M}(w, \sigma,s,\psi)) =d_c \val{d_c}^{-1/2} \cdot \sum_{a \in F^\times/F^{\times 2}} \Tr(\mca{M}(w, \eta_{a,(2)} \otimes \sigma_o,s,\psi)).$$
We now deduce from \eqref{trfor4}  that
$$\begin{aligned}
& \Tr( \mca{M}(w, \sigma,s,\psi)) \\
=\ & d_c\val{d_c}^{-1/2} \cdot d^{-1} \val{d}^{\half} \cdot \sum_{\eta' \in \widehat{F^\times/F^{\times 2}}} \sum_{\eta \in \widehat{F^\times/F^{\times d}}} \tilde{\gamma}(1-s,\chi^{-1} \eta' \eta,\psi) \\
=\ & d_c\val{d_c}^{-1/2} \cdot d^{-1} \val{d}^{\half} \cdot  \sum_{\eta \in \widehat{F^\times/F^{\times d}}} \sum_{\eta' \in \widehat{F^\times/F^{\times 2}}} \gamma(1-s,\chi^{-1} \eta' \eta,\psi)  \text{ by Lemma \ref{sumlemma}}.
\end{aligned} $$
Since $4c+1$ is odd, we have
$$\frac{d}{d_c}=\gcd(n,4c+1).$$
The result for the case $n\equiv 2 \ (\text{mod }4)$ now follows from the third item in Lemma \ref{dualcenter}.

Lastly, we consider $n\equiv 0 \ (\text{mod }4)$.  In this case, Theorem \ref{Smain} and Theorem \ref{decopropcor} imply
$$\begin{aligned}
& \Tr( \mca{M}(w, \sigma,s,\psi))  \\
=\ & d_c\val{d_c}^{-1/2} \cdot \sum_{a \in F^\times/F^{\times 2}} \Tr(\mca{M}(w, \eta_{a,(n)} \otimes \sigma_o,s,\psi)) \\
=\ & \frac{d_c}{d} \cdot \sum_{a \in F^\times/F^{\times 2}} \sum_{b \in F^\times/F^{\times d}} {\gamma}(1-s,\chi^{-1} \eta_{a,(n)} \eta_{b,(d)},\psi) \text{ by \eqref{trfor4}}.
\end{aligned} $$
Using the third assertion in Lemma \ref{dualcenter}, the proof is completed.
\end{proof}

\begin{cor} \label{exptrungl}
Suppose $n>1$ and $\gcd(n,p)=1$. Let $\chi$ be a linear character associated with $\sigma$.  If $\sigma$ is unramified and $\psi$ is normalized,  then
$$\Tr(\mca{M}(w, \sigma, s, \psi))=n_c \cdot (1-q^{-1}) \cdot  L(ns,\chi^n).$$
\end{cor}

\begin{cor} \label{rumexptrgl}
Suppose $n>1$ and $\gcd(n,p)=1$. Denote by $\chi$ the linear character associated with $\sigma_o$. If we assume that $\sigma$ is ramified, then the following hold:
 \begin{enumerate}
 \item[(i)] The conductor $\mfr{f}(\chi)$ is determined by $\sigma_o$.
 \item[(ii)] If $\mfr{f}(\chi)  \nequiv \mfr{f}(\psi) \ (\text{mod }d)$, then $\Tr(\mca{M}(w, \sigma, s, \psi))=0$.
  \item[(iii)] If $\mfr{f}(\chi)  \equiv \mfr{f}(\psi) \ (\text{mod }d)$, then
 $$\Tr(\mca{M}(w, \sigma, s, \psi))=    n_c  \cdot \sum_{\beta \in \widehat{O^\times /O^{\times n}} }  \varepsilon(1-s,\chi^{-1}\beta,\psi).$$
 \item[(iv)] There exists $\psi$ such that $\Tr(\mca{M}(w, \sigma, s, \psi)) \ne 0.$
 \end{enumerate}
 \end{cor}
Using the first assertion in Theorem \ref{trgl}, the proof of Corollaries \ref{exptrungl} and \ref{rumexptrgl} goes word for word  as that of Corollaries \ref{exptrun} and \ref{rumexptr}.

\begin{thm} \label{detformgl}
Assume that $\gcd(p,n)=1$. Let $\sigma \in \Irr(\wt{T})$ be an unramified representation.  Let $\chi$ be a linear character associated with $\sigma$. If $\mfr{f}(\psi)=0$, then
$$\det(\mca{M}(w, \sigma, s, \psi))=\tau(n) \cdot \mu(\sigma, s)^{\frac{(1-n)n_c}{2}}  \cdot \gamma(1-ns, \chi^{-n}, \psi)^{n_c},$$
where $$
\tau(n)=
\begin{cases} 1 &  \text{if } -1 \in F^{\times n/d}; \\
-1  &   \text{otherwise}
\end{cases}
=
\begin{cases}
(-1,\varpi)_2 &  \text{if } n \equiv 2 \ (\text{mod }4); \\
1  &   \text{otherwise}.
\end{cases}$$
\end{thm}
\begin{proof}
Throughout the proof we repeatedly use Corollary \ref{plancor} which gives
$$\mu(\sigma, s)= \mu(\sigma_o, s)$$
for every $\sigma_o \in \Irr(\wt{T}_o)$ occurring in $\sigma$. By Proposition \ref{firstunramprop}, there exists an unramified $\sigma_o \in \Irr(\wt{T}_o)$ which occurs in $\sigma$.

Suppose first that $n$ is odd. It then follows from Theorem \ref{Smain} that
$$\det(\mca{M}(w, \sigma, s, \psi))=\det(\mca{M}(w, \sigma_o, s, \psi))^{n_c}.$$
The result for odd $n$ then follows from Theorem \ref{detunramnot4}.

Assume now $n$ is even. In view of Proposition \ref{secondunramprop} we may assume $\sigma_o$ is unramifed, and that $\chi$ is an unramified linear character associated with $\sigma_o$. If $n\equiv 2 \ (\text{mod }4)$, then Theorem \ref{Smain} and Theorem \ref{decopropcor} imply that
$$\begin{aligned}
& \det(\mca{M}(w, \sigma, s, \psi)) \\
=\ &  \prod_{a \in F^\times /F^{\times 2}} \det(\mca{M}(w, \eta_{a,(2)} \otimes \sigma_o,s,\psi)) ^{d_c} \\
=\ & {\mu_n(\sigma,s)}^{2d_c(1-d)}  \cdot \prod_{a \in F^\times /F^{\times 2}} \tilde{\gamma}(1-ds,\eta_{a,(2)}\chi^{-d},\psi)^{d_c }  \text{ by  Proposition \ref{detunramnot4}}.
\end{aligned} $$
By comparing Corollary \ref{verybicelem} with the explicit formula for the Plancherel measure given in \S \ref{plansec}, we deduce that
$$\prod_{a \in F^\times /F^{\times 2}} \tilde{\gamma}(1-ds,\eta_{a,(2)}\chi^{-d},\psi)=(\varpi,-1)_2 \cdot \gamma(1-ns,\chi^{-n},\psi)^2 \mu_n(\sigma,s)^{-1}.$$
Note that
$$(\varpi,-1)_2^{d_c}=(\varpi,-1)_2$$
 as $d_c$ is odd. Hence the result follows in the case $n\equiv 2 \ (\text{mod }4)$.

Last, suppose that $n\equiv 0 \, (\text{mod }4)$. By Theorem \ref{Smain} and Corollary \ref{decopropcor},  we have
$$\det(\mca{M}(\sigma, s, \psi))=\prod_{a \in F^\times /F^{\times 2}} \det(\mca{M}(\eta_{a,(n)} \otimes \sigma_o,s,\psi)) ^{d_c}.$$
It is thus sufficient to show that
\begin{equation} \label{tem-eqn1}
\begin{aligned}
& \gamma(1-ns,\chi^{-n},\psi)^2 \mu(\sigma,s)^{1-n} \\
=\ & \det(\mca{M}(\sigma_o,s,\psi)) \cdot  \det(\mca{M}(\eta_{u,(n)} \otimes \sigma_o,s,\psi)) \\
   &  \qquad \qquad \cdot \det(\mca{M}(\eta_{\varpi,(n)} \otimes \sigma_o,s,\psi)) \cdot  \det(\mca{M}(\eta_{u\varpi,(n)} \otimes \sigma_o,s,\psi)),
\end{aligned}
 \end{equation}
where $u$ is such that $\eta_{u,(n)}$ is unramified and
$$(u,\varpi)_n^d=-1.$$

For this purpose, we note that $\det(\mca{M}(\sigma_o,s,\psi))$ is given in \eqref{unramdet4}; moreover, since $\eta_{u,(n)} \otimes \sigma_o$ is unramified, $\det(\mca{M}(\eta_{u,(n)} \otimes \sigma_o,s,\psi))$ is also given by \eqref{unramdet4} by replacing $\chi$ by $\chi \eta_{u,(n)}$. Note further that
 $\det(\mca{M}(\eta_{\varpi,(n)} \otimes \sigma_o,s,\psi))$ is given in \eqref{semiunramdet4}, and by the same argument just used,  $\det(\mca{M}(\eta_{u\varpi,(n)} \otimes \sigma_o,s,\psi))$ is also given by \eqref{semiunramdet4} with $\chi$ replaced by $\chi \eta_{u,(n)}$.  Thus, the right hand side of \eqref{tem-eqn1} equals
$$\mu(\sigma,s)^{2-n} \cdot \gamma(1-ns,\chi^{-n},\psi)^2 \cdot \pmb{\beta}(\sigma_o,s,\psi) \pmb{\beta}(\eta_{(u),n}\chi,s,\psi).$$
Finally, a straightforward computation shows that
$$\pmb{\beta}(\sigma_o,s,\psi)  \cdot \pmb{\beta}(\eta_{(u),n}\chi,s,\psi)=\mu(\sigma,s)^{-1}.$$
This completes the proof.
\end{proof}

It is clear that the proof of Theorem \ref{detformgl} holds under the weaker assumption that $\sigma$ contains an unramified element of $\Irr(\wt{T}_o)$. Aslo, it is clear that Theorem \ref{detformgl} agrees with Theorem \ref{T:M1} (see also Example \ref{eg-GL}).

\subsection{Proof of Theorem \ref{Smain}} \label{mainproof}
Given with  $\sigma \in \Irr(\wt{T})$, we realize it as $i(\chi'_\sigma)$ where $\chi'_\sigma$ is an extension of $\chi_\sigma$ to $\widetilde{T}$. Let $\sigma_o \in \Irr(\wt{T}_o)$ be such that
$$\chi_{\sigma_o} = \chi'_\sigma|_{Z(\wt{T}_o)}.$$
Note that $\sigma_o$ occurs in $\sigma$. We shall realize $\sigma_o$ as $i(\chi'_{\sigma_o})$ where
$$\chi'_{\sigma_o}= \chi'_\sigma|_{\wt{A}_o}.$$
We fix a linear character $\chi$ associated with $\sigma_o$. For $g \in \wt{G}$,  we define
$$\chi^g=\chi \cdot \eta_{\det g, (n)}^{4c+1}.$$
Lemma \ref{justcomp} implies that $\chi^g$ is a linear character associated with  $(\sigma_o)^g$.

For $\wt{G}$, we define
$$\sigma_s, \quad  \chi'_{\sigma_{s}}, \quad T\bigl(w, \chi'_{\sigma_{s}}\bigr), \quad M_{s}, \text{ and } N_{s}$$
exactly in the same way as in \S \ref{commsl} for $\wt{G}_o$. For instance, $T\bigl(w, \chi'_{\sigma_{s}}\bigr)$ is given as follows: for
$$h_s=M_{s}(f_s) \in I(\chi'_{s})$$
and $g \in \wt{G}$, we have $ T\bigl(w, \chi'_{s}\bigr)(h_s)(g)$ is the meromorphic continuation of
\begin{equation} \label{iomgl}
\int_F h_s\bigl(\wt{w}u(x)g \big) \ d_\psi x.
\end{equation}
This integral converges absolutely wherever $T(w, \sigma_{s})$  converges absolutely. By the same argument used for  Lemma \ref{comdo} (see also \cite[Lemma 4.6]{Szp6}), one obtains a commutative diagram of $\wt{G}$-maps:
 $$\begin{tikzcd}
I(\chi'_{\sigma_{s}})  \ar[d, "{ T(w, \chi'_{\sigma_{s}})}"']   &    I\bigl(\sigma_{s}\bigr)   \ar[l, "{ {M_{s}} }"']   \ar[d, "{ T(w, \sigma_{s})}"] \\
I\bigl((\chi'_{\sigma_{s}})^w\bigr)  \ar[r, "{N_{s}}"]   &  I\bigl((\sigma_{s})^w\bigr) .
\end{tikzcd} $$

\begin{lm}\label{simpobser}
Let $\chi'_{\sigma_{s}}$ be an extension of $\chi_{\sigma_{s}}$ to $\wt{A}$. Then for  $t \in \wt{T}$,  we have
$$T(w, \chi'_{\sigma_{s}}) \bigl(I_o(\chi'_{\sigma_{s}})_t \bigr) \subseteq I_o\bigl((\chi'_{\sigma_{s}})^w\bigr)_t.$$
\end{lm}
\begin{proof}
Since $w \in \wt{G}_o$ and $U \subseteq \wt{G}_o$, it follows that for all $u(x) \in U$, $g \in \wt{G}$ one has
$$\det(g)=\det(\wt{w} u(x)g).$$
Thus, it follows from \eqref{iomgl} that if $h_s \in I\bigl(\chi'_{\sigma_{s}}\bigr)_t$, then
$$T\bigl(w, \chi'_{\sigma_{s}}\bigr)(h_s) \in  I\bigl(\chi'_{\sigma_{s}}\bigr)_t.$$
This gives the desired inclusion.
\end{proof}

Consider the diagram
\begin{equation} \label{CDtem}
\begin{tikzcd}
 I\bigl({\chi_{\sigma_s}'}\bigr)_{g}   \ar[d, "{ T(w, {\chi_{\sigma_s}'})  }"']   &  &  I_o\bigl((\chi_{\sigma_{o,s}}')^g\bigr)   \ar[ll, "{  Y'_{{\chi_{\sigma_s}'},g}  }"']    \ar[d, "{ T(w, (\chi_{\sigma_{o,s}}')^g)  }"] \\
I\bigl(({\chi_{\sigma_s}'})^w\bigr)_{g}   \ar[rr, "{ R'_{({\chi_{\sigma_s}'})^w,g^{(w^{-1})}}   }"']   &  & I_o\bigl((\chi_{\sigma_{o,s}}')^{g^{(w^{-1})}}\bigr)
\end{tikzcd}
\end{equation}
where the maps are given as follows. First,
$$Y'_{{\chi_{\sigma_s}'},g}: I_o\bigl((\chi_{\sigma_{o,s}}')^g\bigr) \rightarrow I\bigl({\chi_{\sigma_s}'}\bigr)_{g}$$
is defined by
$$\bigl((Y'_{{\chi_{\sigma_s}'},g})f\bigr)(h)= \delta_B^{\frac{s+1}{2}}(g) \cdot
 \begin{cases}
 {\chi_{\sigma_s}'}^{g}(t)f(h_o) & \text{ if }  h=gth_o, \, t\in \wt{A},\ h_o \in \wt{G}_o; \\
 0 & \text{ if }  h \notin g \wt{A}\wt{G}_o.
 \end{cases}$$
It is a normalization of the $\wt{G}_o$-isomorphism $Y_{\chi, t}$ given in Lemma \ref{relem}. Second,
$$R'_{({\chi_{\sigma_s}'})^w,g^{(w^{-1})}}: I\bigl(({\chi_{\sigma_s}'})^w\bigr)_{g}\rightarrow I_o\bigl((\chi_{\sigma_{o,s}}')^{g^{(w^{-1})}}\bigr)$$
is given by
$$\bigl(R'_{({\chi_{\sigma_s}'})^w,g^{(w^{-1})}}(f)\bigr)(h_o)=\delta_B^{\frac{-s+1}{2}}(g)f(g^w h_o),$$
which is a normalization of the $\wt{G}_o$-isomorphism $R_{\chi, t}$ in Lemma \ref{relem}. Since
$$({\chi_\psi}^g)^w=({({\chi_\psi})^w})^{g^{(w^{-1})}},$$
Lemma \ref{simpobser} shows that the vertical map in \eqref{CDtem} is well-defined.

\begin{lm}
The diagram \eqref{CDtem} commutes.
\end{lm}
\begin{proof}
Fix $f_s \in  I_o\bigl((\chi_{\sigma_{o,s}}')^g\bigr) $, $g_o \in \wt{G}_o$. We compute
$$\begin{aligned}
&  \Bigl(R'_{({\chi_{\sigma_s}'})^w,g^{(w^{-1})}} \circ   T(w, {\chi_{\sigma_s}'}) \circ Y'_{{\chi_{\sigma_s}'},g} (f_s)\Bigr)(g_o) \\
=\ & \delta^{\frac{s+1}{2}}(g)\Bigl(  A_w\bigl({\chi_{\sigma_s}'}\bigr) \circ Y'_{{\chi_{\sigma_s}'},g} (f_s)\Bigr)(g^{(w^{-1})}g_o) \\
=\ & \delta^{\frac{s+1}{2}} \int_F \bigl( Y'_{{\chi_{\sigma_s}'},g} (f_s)\bigr)\bigl(\wt{w}n(x)g^{(w^{-1})}g_o \bigr) \, d_\psi x.
\end{aligned} $$
Observe that
$$\wt{w} u(x)g^{(w^{-1})}=g \bigl( \wt{w} {(n(x))}^{w^{-1}g^{-1}w} g_o\bigr).$$
Denoting
$$g=\s(\diag(a,b)),$$
we have from the definition of $Y'_{{\chi_{\sigma_s}'},g}$ that
$$\Bigl(R'_{({\chi_{\sigma_s}'})^w,g^{(w^{-1})}} \circ T\bigl(w, {\chi_{\sigma_s}'}\bigr) \circ Y'_{{\chi_{\sigma_s}'},g} (f_s)\Bigr)(g_o)=\val{a b^{-1}} \cdot \int_F  f\bigl(\wt{w}u(a b^{-1} x)g_o \bigr) \, d_\psi x.$$
The result follows from a change of variable $a b^{-1} x \mapsto x$.
\end{proof}

For $k \in \widehat{K}$ and $g \in \wt{T}$ we define
$$\xi_{g,k} \in i(\eta)^\vee$$
 by
$$\xi_{g,k}(f)={\rm nor}_{\chi,\psi}(k) \cdot f\bigl(g \s(\alpha^\vee(k))\bigr).$$
The same as in the $\wt{G}_o$ case,
$$(g,k)\mapsto \xi_{g,k}$$ gives rise to a well-defined map on
$$\widehat{K^\dag} \times \wt{T}.$$
 If $n$ is odd (resp. even), then we fix a set $S$ of representatives of $\wt{Z(G)}/ \wt{M}$  (resp. $\wt{T}/ \wt{M}\wt{T}_o$). Consider
$$\mfr{R}_{\chi,\psi}=\{\xi_{g,k} \mid g \in S, \, k \in \widehat{K^\dag} \}.$$

\begin{lm}
The set $\mfr{R}_{\chi,\psi}$  is a basis for $i(\chi'_{\sigma_{s}})^\vee.$
\end{lm}
\begin{proof}
We need to show that
$$\{(d,\s(\alpha^\vee(k)) ):  g\in S, k \in \widehat{K^\dag}\}$$
is a set of representatives of $\wt{T}/\wt{A}$. By Lemma \ref{crucia} and \eqref{Dindex} we have
$$\val{S}= n_c \val{n_c}^{-1/2} \cdot
\begin{cases}
1&  \text{ if $n$ is odd}; \\
2\val{2}^{-1/2} &  \text{ otherwise}.
\end{cases}$$
Also, by \cite[Remark 2.6]{Szp6} (see also \cite[\S 4.2]{Szp6})) one has
$$\widehat{K^\dag} \simeq \wt{T}_o/\wt{A}_o,$$
and thus in particular
$$\widehat{K^\dag}=d_c \val{d_c}^{-\half}.$$
It then follows from the first item in Lemma \ref{Tfacts} that
$$\val{S \times \widehat{K^\dag} }=[\wt{T}: \wt{A}].$$
Thus we only need to show that for any $g,g' \in S$ and $k,k' \in \widehat{K^\dag}$, if
$$gg'^{-1} \cdot \s(\alpha^\vee(kk'^{-1})) \in \wt{A},$$
then $g=g'$ and $k=k'$. However, the equality $g=g'$ follows from \eqref{forDuse}. Moreover, as $\wt{A}_o=\wt{T}_o \cap T$, we deduce that $k=k'$. This completes the proof.
\end{proof}

Let
$$\mca{M} \bigl((\cdot,\cdot),(\cdot,\cdot),{\chi_{\sigma}'},\chi,s,\psi \bigr):(S \times \widehat{K^\dag} ) \times (S \times \widehat{K^\dag} ) \rightarrow \C(q^{-s})$$
be the local coefficients matrix associated with $\sigma$, representing $ T\bigl(w, \sigma_{s}\bigr)^*$ with respect to $\mfr{B}_{(\sigma_s)^w,\psi}(\mfr{R}_{\chi,\psi})$ and $\mfr{B}_{\sigma_s,\psi}(\mfr{R}_{\chi,\psi})$. Similar to Proposition \ref{ForRuse}, it is the same matrix representing  $T\bigl(w, \chi'_{\sigma_{s}}\bigr)^*$ with respect to the two basis
$$N_{s}^\psi \circ \mfr{B}_{(\sigma_s)^w,\psi}(\mfr{R}_{\chi,\psi})$$
 and
 $$({M_{s}^{-1}})^\psi \circ \mfr{B}_{\sigma_s,\psi}( \mfr{R}_{\chi,\psi}).$$

Theorem \ref{Smain} is a direct consequence of Proposition \ref{decoprop} and the Proposition \ref{lastone} given below. Though the proof is technical, the focus lies in the transformations of Whittaker functionals arising from the two commutative diagrams above given in this subsection.

For every $g \in S$, we define
$$\mca{M}_g \bigl((\cdot,\cdot),{\chi_{\sigma}'},\chi,s,\psi \bigr): \widehat{K^\dag}\times \widehat{K^\dag} \rightarrow \C(q^{-s})$$
by
$$\mca{M}_g \bigl(a,b),{\chi_{\sigma}'},\chi,s,\psi \bigr)=\mca{M} \bigl((g,a),(g,b),{\chi_{\sigma}'},\chi,s,\psi \bigr).$$

\begin{prop} \label{lastone}
\begin{enumerate}
\item[(i)] Suppose that $g,h \in S$ and that $g \neq h$. Then for all $k_1,k_2 \in \widehat{K^\dag}$, one has
$$\mca{M} \bigl((g,k_1),(h,k_2),{\chi_{\sigma}'},\chi,s,\psi \bigr)=0.$$
\item[(ii)] We have $$\mca{M}_g \bigl(a,b),{\chi_{\sigma}'},\chi,s,\psi \bigr)=\mca{M}(a,b,\chi^g,s,\psi).$$
In particular, $\mca{M}_g \bigl(a,b),{\chi_{\sigma}'},\chi,s,\psi \bigr)$ is a local coefficients matrix associated with $\sigma_o ^g$.
\end{enumerate}
\end{prop}
 \begin{proof}
 In order to shorten the notations in this proof, we shall use $k$ to denote either an element of $\widehat{K^\dag}$ or $\s(\alpha^\vee(k)) \in \wt{T}_o$, and no confusion will arise from the context. We have (see \cite[\S 4]{Szp6})
 $$\mfr{B}_{\sigma_s,\psi}(\mfr{R}_{\chi,\psi}) =\set{  J_{\sigma, s}(\xi_{g, k}): \  \xi_{g, k} \in \mfr{R}_{\chi,\psi}  },$$
where
$$J_{\sigma,s}: \sigma^\vee \rightarrow \Wh_\psi \bigl(I(\sigma_s))\bigr),$$
is the isomorphism such that $J_{\sigma,s}(\xi_{g,k})$ is given by the meromorphic continuation of
$$f_s \mapsto {\rm nor}_{\chi,\psi}(k)\int_F f_s \bigl(gk, \wt{w}u(x) \bigr) \psi^{-1}(x) \ dx.$$
Thus,
$$J_{\sigma,s}(\xi_{g,k}) \circ M_s^{-1} (h_s)$$
 is the meromorphic continuation of
$$h_s \mapsto {\rm nor}_{\chi,\psi}(k) \cdot \delta_B^{\frac{-s-1}{2}}(gk)\int_F h_s(g k  \wt{w}u(x)) \psi^{-1}(x) \, dx.$$
Denote
$$\beta_{{\chi_{\sigma}'},\chi,\psi,s,g,k}:=J_{\sigma,s}(\xi_{g,k}) \circ M_s^{-1}.$$
Similarly, we have the isomorphism
$$J_{\sigma^w,-s}: \sigma^\vee \rightarrow \Wh_\psi \bigl(I\bigl((\sigma_s)^w\bigr)\bigr).$$
A similar consideration shows that $J_{\sigma^w,-s}(\xi_{g,k})$ is the meromorphic continuation of
$$f_{-s} \mapsto {\rm nor}_{\chi,\psi}(k)\int_F f_{-s}(gk, \wt{w}u(x)) \psi^{-1}(x) \ dx,$$
and that  $J_{\sigma,s}(\xi_{g,k}) \circ N_s (h_s)$ is the meromorphic continuation of
$$h_{-s} \mapsto {\rm nor}(\chi^{-1},\psi)(k^{-1}) \cdot \delta_B^{\frac{s-1}{2}}(g^wk^{-1})\int_F h_{-s}(g^wk^{-1} \wt{w}u(x)) \psi^{-1}(x) \, dx.$$
This implies that
$$J_{\sigma,s}(\xi_{g,k}) \circ N_s=\beta_{\chi_{\sigma}'^w,\chi^{-1},\psi,-s,g^w,k^{-1}}.$$
Thus, $\mca{M}\bigl((\cdot,\cdot),(\cdot,\cdot),{\chi_{\sigma}'},\chi,s,\psi \bigr)$ is determined by the relation
$$ T\bigl(w, {\chi_{\sigma_s}'}\bigr)^*\bigl(\beta_{{\chi_{\sigma}'}^w,\chi^{-1},\psi,s,g^w,a^{-1}}\bigr)=
\sum_{h \in S}  \sum_{b \in \widehat{K^\dag}} \mca{M} \bigl((g,a),(h,b),{\chi_{\sigma}'},\chi,s,\psi \bigr) \cdot \beta_{{\chi_{\sigma}'},\chi,\psi,s,h,b}.$$
Observe that $\beta_{{\chi_{\sigma}'}^w,\chi^{-1},\psi,s,g^w,a^{-1}}$ is supported on $I(\chi_{\sigma_s}')_g$. Therefore, Lemma \ref{simpobser} implies (i). In fact, it also implies that $\mca{M}_g \bigl(a,b),{\chi_{\sigma}'},\chi,s,\psi \bigr)$ is given by the relation
$$ T\bigl(w, {\chi_{\sigma_s}'}\bigr)^*\bigl(\beta_{\chi_{\sigma}'^w,\chi^{-1},\psi,s,g^w,a^{-1}}\bigr)=
\sum_{b \in \widehat{K^\dag}} \mca{M}_g \bigl(a,b),{\chi_{\sigma}'},\chi,s,\psi \bigr)  \cdot \beta_{{\chi_{\sigma}'},\chi,\psi,s,g,b}.$$

By \cite[Proposition 4.9]{Szp6},  the local coefficients matrix
$$\mca{M}(\cdot,\cdot,\chi^g,s,\psi):\widehat{K^\dag} \times \widehat{K^\dag} \rightarrow \C(q^{-s})$$
associated with $(\sigma_o)^g$ is determined by the relation
$$ T\bigl(w, (\chi_{\sigma_{o,s}}')^g\bigr)^*\bigl(\lambda_{a^{-1},\chi^{-1},\psi,-s}\bigr)=\sum_{b \in K^\dag} \mca{M}(a,b,\chi_{},s,\psi) \cdot \lambda_{b,\chi,\psi,s}, $$
where for $k \in \widehat{K^\dag}$,  the Whittaker functional $\lambda_{k,\chi,\psi,s} \in \Wh_\psi \bigl(I\bigl(\sigma_{s}\bigr)  \bigr)$ is the analytic continuation of
$$h_s \mapsto {\rm nor}_{\chi,\psi}(k) \cdot \int_{\mfr{p}^{-r}} h_s\bigl(k \wt{w} u(x)\bigr)\psi^{-1}(x) \, d_{\psi}x.$$
(See also Proposition \ref{tau0mod4} for the  case $n \equiv 0 \, (\text{mod }4)$.) A straightforward computation shows that
$$\beta_{{\chi_{\sigma}'},\chi,\psi,s,g,b} \circ Y'_{{\chi_{\sigma_s}'},g}=\lambda_{b,\chi^g,\psi,s}$$ 
and thus
$$\bigl(\beta_{\chi_{\sigma}'^w,\chi^{-1},\psi,s,g^w,a^{-1}}\bigr) \circ {R'}^{-1}_{({\chi_{\sigma_s}'})^w,g^{(w^{-1})}}=\lambda_{a^{-1},{\chi^g}^{-1},\psi,-s}.$$
This completes the proof.
\end{proof}

\begin{rmk}
For the intertwining operator in \eqref{Inter-R} (and everything that follows), we have used $\wt{w}=\wt{w}_\alpha$ instead of $\wt{w}^{-1}=\wt{w}_\alpha^{-1}$ (compared to \S \ref{SS:para}) in the integral form, as this is better-aligned with the work in \cite{Szp6}. However, we briefly justify below that such discrepancy is harmless in all our consideration.

The essential point is the effect of $\wt{h}_\alpha(-1) := \wt{w}_\alpha(-1) \cdot \wt{w}_\alpha(-1) = \wt{w}_\alpha^{-2}$ (in the notation of \eqref{E:h}) on the operator
$$\mca{T}(w, \chi)^*:   \Wh_\psi(I(\chi)) \to \Wh_\psi(I(\chi))$$
as in \eqref{End-T}. More precisely, it is easy to see $\wt{h}_\alpha(-1) \in \wt{Z(G)}$ and $\wt{h}_\alpha(-1) \in Z(\wt{G}_o) = \wt{Z(G_o)}$; thus $\wt{h}_\alpha(-1)$ acts naturally on $\Wh_\psi(I(\chi))$. Moreover, as $\wt{h}_\alpha(-1)$ is fixed by $W$, by tracing through the definition of $\mca{T}(w, \chi)^*$, we see that it is $\wt{h}_\alpha(-1)$-equivariant, i.e.,
$$\mca{T}(w, \chi)^* ( \wt{h}_\alpha(-1) \cdot v ) = \wt{h}_\alpha(-1) \cdot ( \mca{T}(w, \chi)^*(v) ) \text{ for every } v\in \Wh_\psi(I(\chi)).$$
Therefore, the two local coefficients matrices arising from using $w$ or $w^{-1}$ are conjugate and thus have the same invariants.

Alternatively, we see that for $\wt{G}_o$ the change of Weyl representative from $w$ to $w^{-1}$ amounts to multiplying both $J_{(\sigma_{o,s})^w,\psi}$ and $J_{\sigma_{o,s},\psi}$ by $\chi_{\sigma_o}( \wt{h}_\alpha(-1) )$. Thus, the functional equation defining the local coefficients matrices is unchanged.  Moreover, Theorem \ref{Smain} shows that a local coefficients matrix for $\wt{G}$ can be realized as a block matrix, where each block is a local coefficient matrix for $\wt{G}_o$. This also justifies that the choice of representatives $w$ or $w^{-1}$ is immaterial to our final result.
\end{rmk}

\vskip 30pt

%\cite{BZ2}
%%%%%%%%%%%%%%% %%%%%%%%%%%%%%%%%%%%%%%%%%%%%%%%%%%%%%%%%

\end{document}